\documentclass{my_thesis}
\def\backgroundshade{beach!70}

\def\theocolor{pond}

\def\tableheadcolor{pond}

\def\tableheadcoloralt{beach!80}

\def\genarraystretch{1.1} 

\newcolumntype{C}{>{\columncolor{\backgroundshade}}c}
\newcolumntype{L}{>{\columncolor{\backgroundshade}}l}

\newcommand{\rmv}[1]{}
\newcommand{\ktv}[1]{k_{t,v}(#1)}
\newcommand{\Ntv}[1]{N_{t,v}(#1)}
\def\qmin{q_{min}}
\def\qmax{q_{max}}
\def\qcol{q_{cycl}}

\def\Z{\mathbb{Z}}
\def\J{\mathcal{J}}
\def\Rcal{\mathcal{R}}
\def\Bcal{\mathcal{B}}
\def\C{\mathcal{C}}

\def\S{\mathcal{S}}
\newcommand{\eqcl}[2]{\mathcal{S}_{#1}(#2)} 
\def\AA{\mathcal{A}}
\def\Zwx{\mathbb{Z}_{[t]_q}^{\times}}
\def\ZWx{\mathbb{Z}_{w}^{\times}}
\def\cipw{C_{p,[t]_q}^{i}}
\def\cipW{C_{p,w}^{i}}
\def\cjpw{C_{p,[t]_q}^{j}}
\def\cjpW{C_{p,w}^{j}}
\def\gpw{\Gamma_{p,[t]_q}}
\def\gpW{\Gamma_{p,w}}
\def\ord{\mathrm{ord}}
\def\neck{\mathrm{neck}}
\def\bin{\mathrm{bin}}
\def\gtst{\mathrm{GetSet}}
\def\charv{\mathrm{char}}
\def\im{\mathrm{Im}}
\def\spn{\mathrm{span}}

\def\nr{N_r(\a,I)}
\def\hr{h(\a,I,v,L; r)}
\def\lcm{\mathrm{lcm}}
\newcommand{\Mod}[1]{\ (\mathrm{mod}\ #1)} 
\def\bigo{\mathcal{O}}
\def\o{\omega}
\def\logw{\mathrm{log}_{\o}}

\def\extend{\textsc{ComputeCand}}
\newcommand{\cand}[1]{\mathrm{Candidates}(#1)} 
\def\xv{\chi_{v}}
\def\xov{\chi_{\o,v}}
\def\xvj{\chi_{v}^{j}}
\def\xji{\chi^{j_i}}
\def\xjone{\chi^{j_1}}
\def\xjtwo{\chi^{j_2}}
\def\xjthree{\chi^{j_3}}

\newcommand{\A}[1]{\mathcal{A}(#1)}
\newcommand{\ZA}[1]{\mathcal{A}_{\bm{0}}(#1)}

\newcommand{\M}[1]{\mathcal{A}_{q^t/q}(#1)}
\newcommand{\MV}[1]{\mathcal{A}_{q^t/q}(#1;v)}
\newcommand{\ZM}[1]{\mathcal{A}_{\bm{0}}(#1)}

\def\a{\alpha}
\def\b{\beta}
\def\bfa{\mathbf{a}}
\def\bfb{\mathbf{b}}
\def\bfc{\mathbf{c}}
\def\bfC{\mathbf{C}}
\def\bfl{\bm{\lambda}}
\def\bfs{\mathbf{s}}
\def\bfx{\mathbf{x}}
\def\bfy{\mathbf{y}}

\def\fqt{\mathbb{F}_{q^t}}
\def\fqtstar{\mathbb{F}_{q^t}^{\times}}
\def\fqx{\mathbb{F}_{q}[x]}
\def\fqthree{\mathbb{F}_{q^3}}
\def\fqfour{\mathbb{F}_{q^4}}

\def\fq{\mathbb{F}_q}
\def\fp{\mathbb{F}_p}
\def\fqstar{\mathbb{F}_q^{\times}}
\def\ftwo{\mathbb{F}_{2}}
\def\fnine{\mathbb{F}_{9}}
\def\ftwox{\mathbb{F}_{2}[x]}
\def\fthreex{\mathbb{F}_{3}[x]}
\def\ffour{\mathbb{F}_{4}}
\def\feight{\mathbb{F}_{8}}
\def\fseven{\mathbb{F}_{7}}
\def\fnine{\mathbb{F}_{9}}
\def\fsixteen{\mathbb{F}_{16}}

\def\ftwot{\mathbb{F}_{2^t}}
\def\fthree{\mathbb{F}_{3}}
\def\f{\mathbb{F}}

\def\T{\mathrm{Tr}}

\def\Tt{\mathrm{Tr}_{q^t/q}}

\newcommand{\seq}[1]{\mathrm{Seq}_{q^t/q}\left(#1\right)}
\newcommand{\sq}[1]{\mathrm{Seq}\left(#1\right)}
\newcommand{\seqv}[1]{\mathrm{Seq}_{q^t/q}\left(#1; v\right)}
\newcommand{\sqv}[1]{\mathrm{Seq}\left(#1; v\right)} 
\newcommand{\w}[1]{[#1]_q}
\def\OA{\mathrm{OA}}
\def\CA{\mathrm{CA}}
\def\CAN{\mathrm{CAN}}
\def\Sj{\sum_{j=1}^{v-1}}

\newcommand\restr[2]{{
  \left.\kern-\nulldelimiterspace 
  #1 
  \vphantom{\big|} 
  \right|_{#2} 
  }}

\allowdisplaybreaks

\usepackage{bm}
\usepackage{afterpage} 
\usepackage{pagecolor}
\usepackage{imakeidx}
\usepackage[font=small]{idxlayout}
\makeindex
\usepackage[disable,chapter]{easy-todo} 

\titlespacing*{\chapter}{0cm}{1.7cm}{2cm} 

\bibliography{references}

\begin{document}
\pagenumbering{gobble} 

\newpagecolor{beach!40}
\begin{tikzpicture}[remember picture, overlay, transform shape]
  \node [anchor=north west, inner sep=0pt]
    at (current page.north west)
    {
        \begin{tikzpicture}[scale=1.11]
        \node (first) at (0,0) [black!50,draw,minimum width=1.5cm,minimum height=1.5cm,fill=aoi] {};
        \node [right = 0cm of first,black!50,draw,minimum width=1.5cm,minimum height=1.5cm,fill=pond] {};
        \node [right = 1.5cm of first,black!50,draw,minimum width=1.5cm,minimum height=1.5cm,fill=goldfish] {};
        \node [right = 3cm of first,black!50,draw,minimum width=1.5cm,minimum height=1.5cm,fill=goldfish] {};
        \node [right = 4.5cm of first,black!50,draw,minimum width=1.5cm,minimum height=1.5cm,fill=aoi] {};
        \node [right = 6cm of first,black!50,draw,minimum width=1.5cm,minimum height=1.5cm,fill=goldfish] {};
        \node [right = 7.5cm of first,black!50,draw,minimum width=1.5cm,minimum height=1.5cm,fill=pond] {};
        \node [right = 9cm of first,black!50,draw,minimum width=1.5cm,minimum height=1.5cm,fill=beach] {};
        \node [right = 10.5cm of first,black!50,draw,minimum width=1.5cm,minimum height=1.5cm,fill=beach] {};
        \node [right = 12cm of first,black!50,draw,minimum width=1.5cm,minimum height=1.5cm,fill=goldfish] {};
        \node [right = 13.5cm of first,black!50,draw,minimum width=1.5cm,minimum height=1.5cm,fill=beach] {};
        \node [right =15cm of first,black!50,draw,minimum width=1.5cm,minimum height=1.5cm,fill=pond] {};
        \node [right =16.5cm of first,black!50,draw,minimum width=1.5cm,minimum height=1.5cm,fill=aoi] {};
        \node [right =18cm of first,black!50,draw,minimum width=1.5cm,minimum height=1.5cm,fill=aoi] {};
        \node [right =19.5cm of first,black!50,draw,minimum width=1.5cm,minimum height=1.5cm,fill=beach] {};

        \node (second) [below = 0cm of first,black!50,draw,minimum width=1.5cm,minimum height=1.5cm,fill=pond] {};
        \node [right =    0cm of second,black!50,draw,minimum width=1.5cm,minimum height=1.5cm,fill=goldfish] {};
        \node [right =  1.5cm of second,black!50,draw,minimum width=1.5cm,minimum height=1.5cm,fill=goldfish] {};
        \node [right =    3cm of second,black!50,draw,minimum width=1.5cm,minimum height=1.5cm,fill=aoi] {};
        \node [right =  4.5cm of second,black!50,draw,minimum width=1.5cm,minimum height=1.5cm,fill=goldfish] {};
        \node [right =    6cm of second,black!50,draw,minimum width=1.5cm,minimum height=1.5cm,fill=pond] {};
        \node [right =  7.5cm of second,black!50,draw,minimum width=1.5cm,minimum height=1.5cm,fill=beach] {};
        \node [right =    9cm of second,black!50,draw,minimum width=1.5cm,minimum height=1.5cm,fill=beach] {};
        \node [right = 10.5cm of second,black!50,draw,minimum width=1.5cm,minimum height=1.5cm,fill=goldfish] {};
        \node [right =   12cm of second,black!50,draw,minimum width=1.5cm,minimum height=1.5cm,fill=beach] {};
        \node [right = 13.5cm of second,black!50,draw,minimum width=1.5cm,minimum height=1.5cm,fill=pond] {};
        \node [right =   15cm of second,black!50,draw,minimum width=1.5cm,minimum height=1.5cm,fill=aoi] {};
        \node [right = 16.5cm of second,black!50,draw,minimum width=1.5cm,minimum height=1.5cm,fill=aoi] {};
        \node [right =   18cm of second,black!50,draw,minimum width=1.5cm,minimum height=1.5cm,fill=beach] {};

        \node (third) [below = 0cm of second,black!50,draw,minimum width=1.5cm,minimum height=1.5cm,fill=goldfish] {};
        \node [right =      0 of third,black!50,draw,minimum width=1.5cm,minimum height=1.5cm,fill=goldfish] {};
        \node [right =  1.5cm of third,black!50,draw,minimum width=1.5cm,minimum height=1.5cm,fill=aoi] {};
        \node [right =    3cm of third,black!50,draw,minimum width=1.5cm,minimum height=1.5cm,fill=goldfish] {};
        \node [right =  4.5cm of third,black!50,draw,minimum width=1.5cm,minimum height=1.5cm,fill=pond] {};
        \node [right =    6cm of third,black!50,draw,minimum width=1.5cm,minimum height=1.5cm,fill=beach] {};
        \node [right =  7.5cm of third,black!50,draw,minimum width=1.5cm,minimum height=1.5cm,fill=beach] {};
        \node [right =    9cm of third,black!50,draw,minimum width=1.5cm,minimum height=1.5cm,fill=goldfish] {};
        \node [right = 10.5cm of third,black!50,draw,minimum width=1.5cm,minimum height=1.5cm,fill=beach] {};
        \node [right =   12cm of third,black!50,draw,minimum width=1.5cm,minimum height=1.5cm,fill=pond] {};
        \node [right = 13.5cm of third,black!50,draw,minimum width=1.5cm,minimum height=1.5cm,fill=aoi] {};
        \node [right =   15cm of third,black!50,draw,minimum width=1.5cm,minimum height=1.5cm,fill=aoi] {};
        \node [right = 16.5cm of third,black!50,draw,minimum width=1.5cm,minimum height=1.5cm,fill=beach] {};

        \node (fourth) [below =      0 of third,black!50,draw,minimum width=1.5cm,minimum height=1.5cm,fill=goldfish] {};
        \node [right =    0cm of fourth,black!50,draw,minimum width=1.5cm,minimum height=1.5cm,fill=aoi] {};
        \node [right =  1.5cm of fourth,black!50,draw,minimum width=1.5cm,minimum height=1.5cm,fill=goldfish] {};
        \node [right =    3cm of fourth,black!50,draw,minimum width=1.5cm,minimum height=1.5cm,fill=pond] {};
        \node [right =  4.5cm of fourth,black!50,draw,minimum width=1.5cm,minimum height=1.5cm,fill=beach] {};
        \node [right =    6cm of fourth,black!50,draw,minimum width=1.5cm,minimum height=1.5cm,fill=beach] {};
        \node [right =  7.5cm of fourth,black!50,draw,minimum width=1.5cm,minimum height=1.5cm,fill=goldfish] {};
        \node [right =    9cm of fourth,black!50,draw,minimum width=1.5cm,minimum height=1.5cm,fill=beach] {};
        \node [right = 10.5cm of fourth,black!50,draw,minimum width=1.5cm,minimum height=1.5cm,fill=pond] {};
        \node [right =   12cm of fourth,black!50,draw,minimum width=1.5cm,minimum height=1.5cm,fill=aoi] {};
        \node [right = 13.5cm of fourth,black!50,draw,minimum width=1.5cm,minimum height=1.5cm,fill=aoi] {};
        \node [right =   15cm of fourth,black!50,draw,minimum width=1.5cm,minimum height=1.5cm,fill=beach] {};
        \node [right = 16.5cm of fourth,black!50,draw,minimum width=1.5cm,minimum height=1.5cm,fill=aoi] {};

    \end{tikzpicture}
    };
\end{tikzpicture}
\hspace{1.9cm}
\begin{minipage}{0.9\linewidth}
\raggedleft
\vspace{5cm}
    {\Huge\bfseries Covering arrays from maximal sequences over finite fields \par}
    \vspace{2cm}
    {\huge \bfseries {\color{black}Georgios Tzanakis\par}}
    \vspace{0.4cm}
    {\LARGE {\color{black}Ph.D. Thesis}}
\end{minipage}
\afterpage{\restorepagecolor}

\afterpage{\thispagestyle{empty}}\cleardoublepage

\frontmatter
\pagenumbering{roman}
\begin{center}
\begin{minipage}{0.85\linewidth}
    \centering
    \vspace{1.18cm}
    {\huge Covering arrays from maximal sequences over finite fields \par}
    \vspace{1cm}
    {\Large Georgios Tzanakis\par}
    \vspace{1cm}
    {\large A thesis submitted to the Faculty of Graduate and Postdoctoral Affairs\\
        in partial fulfillment of the requirements for the degree of 
        \\ \vspace{1cm} \Large Doctor of Philosophy
        \\ \vspace{0.3cm} \large in
        \\ \vspace{0.3cm} \Large Mathematics
    \par}
    \vspace{9.1cm}
    {School of Mathematics and Statistics \\
     Ottawa-Carleton Institute for Mathematics and Statistics\\
     Carleton University\\
     Ottawa, Ontario}\\
    {\vspace{0.5cm} 
    \copyright 2017 Georgios Tzanakis}
\end{minipage}
\end{center}
\thispagestyle{empty}

\def\chapterblockcolor{white} 

{

\chapter*{Abstract}
\thispagestyle{empty}
\textsc{The focus of this thesis} is the study and construction of covering arrays, relying on maximal period sequences and other tools from finite fields.
A \emph{covering array of strength $t$}, denoted $\CA(N; t, k,v)$, is an $N\times k$ array with entries from an alphabet $A$ of size $v$, with the property that in the $N\times t$ subarray defined by any $t$ columns, each of the $v^t$ vectors in $A^t$ appears at least once as a row.
Covering arrays generalize orthogonal arrays, which are classic combinatorial objects that have been studied extensively.
Constructing covering arrays with a small row-to-column ratio is important in the design of statistical experiments, however it is also a challenging mathematical problem.

\emph{Linear feedback shift register (LFSR) sequences} are sequences of elements from a finite field that satisfy a linear recurrence relation.
It is well-known that these are periodic; LFSR sequences that attain the maximum possible period are \emph{maximal (period) sequences}, often abbreviated to \mbox{\emph{m-sequences}} in the literature.
Arrays constructed from cyclic shifts of maximal sequences possess strong combinatorial properties and have been previously used to construct orthogonal and covering arrays  \cite{moura2016finite}, although only one of the known constructions is for covering arrays that are not orthogonal arrays \cite{raaphorst2014construction}. 
In this thesis we present several new such constructions.

The cornerstone of our results is a study of the combinatorial properties of arrays constructed from maximal sequences, where we make fundamental connections with concepts from diverse areas of discrete mathematics, such as orthogonal arrays, error-correcting codes, divisibility of polynomials and structures of finite geometry.

One aspect of our work involves concatenating arrays corresponding to different maximal sequences and finding subarrays that are covering arrays.
We express this as an optimization problem, to which we give an algorithmic solution based on backtracking, an underlying finite field theory and connections to other combinatorial objects.
The results of our experiments include 37 new covering arrays of strength $4$ and one of strength $5$.

For integers $v\geq 2$, we introduce cyclic trace arrays modulo $v$, a variation of arrays from maximal sequences
that we study using finite field characters - homomorphisms from the finite field to the unit circle of complex numbers. 
In particular, we use well-known bounds on character sums to derive conditions subject to which cyclic trace arrays modulo $v$ are covering arrays, and we present new infinite families of covering arrays of strengths $3$ and $4$, as well as one of arbitrary strength which appears to be the second such family in the known literature \cite{colbourn2010covering}.
We also express the number of times that different vectors appear in the rows of a cyclic trace array modulo $v$ as the solution of a linear program.


}

\begin{dedication}
{\large
  To my parents, Maro and Nikos
}
\end{dedication}

{

\chapter*{Acknowledgments}
\textsc{First and foremost}, I would like to thank Daniel Panario. 
Having him as my advisor was a privilege; the guidance, opportunities and friendship that he offered will always be deeply appreciated.
I am also most grateful for having met and worked with Lucia Moura and Brett Stevens.
I owe this piece of work and much more to the support, knowledge and enthusiasm of these three people.

I would like to extend my thanks to the members of my defence committee, Charlie Colbourn, Evangelos Kranakis, Mike Newman and David Thomson for their time, effort and insightful comments.
Special thanks to Charlie Colbourn whose work on covering arrays has been essential to my research.
I also wish to thank my undergraduate supervisor, Theodoulos Garefalakis, who introduced me to the area of finite fields and who was the one to encourage and help me to pursue my studies in Canada.

Completing my Ph.D.\ was a long process full of bright highs and dark lows; I would not have gotten through the latter without the endless love and support of my family.
It is hard to express what they mean to me without making an understatement.

I was extremely fortunate to be surrounded by numerous wonderful friends throughout my studies.
Ali Al-Bashabsheh has been a true and most invaluable friend.
Lorena Ibacache was the closest thing that I had to family while being away from home.
John Talboom and Ming-Ming Zhang where my companions along the way, from our first assignments to the celebration after my defence; I do not want to imagine how my long days and nights in Herzberg labs would be without them.
Also essential to surviving my Ph.D.\ was having Ioannis Nomidis as my housemate and much needed Greek company, along with Aris Leivadeas and Manos Skevakis.
Kate Trotter and Dragica Stanivuk have been my dearest friends since our university residence years and I cannot think of my time in Canada without thinking of them.
Whether it was a lazy evening or a hike in the middle of the Canadian winter, hanging out with Chris Dionne, Aras Erzurumluo\u{g}lu, Nevena Franceti\'{c} and Elizabeth Maltais was always a blast.
I also greatly appreciate the friendship of Caio Fernandes, Camelia Karimianpour, Jos\'{e} Matos, Patrick Niesink and Anna Rousso. 
Finally, I am thankful to Stelios Charitakis, the Gorgoraptis brothers, Isis Mouza and Christos Spatharakis, for being my beloved friends despite our distance.

\thispagestyle{empty}
}

\def\chapterblockcolor{goldfish} 

\tableofcontents

\mainmatter

\pagenumbering{arabic}
\chapter{Introduction}
\label{chapter:Introduction}

%

\textsc{This thesis is concerned with} the study and construction of covering arrays, a combinatorial object that lies in the intersection of diverse fields such as discrete mathematics, statistics, computer science and software engineering, and is the topic of active research, ranging from purely theoretical mathematical studies to software development and industrial applications.

The origins of covering arrays are orthogonal arrays, which were first studied by Rao in the second half of the 1940s \cite{rao1946hypercubes,rao1947factorial,rao1949class}, and since then have became a classic mathematical object due to their natural definition and fundamental connections to other areas, such as finite fields, finite geometry and error-correcting codes.
An \emph{orthogonal array of strength $t$ and index $\lambda$}, denoted $\OA_{\lambda}(t, k,v)$, is a $\lambda v^t\times k$ array with entries from an alphabet $A$ of size $v$, with the property that in the $\lambda v^t \times t$ subarray defined by any $t$ columns, each of the $v^t$ vectors in $A^t$ appears as a row precisely $\lambda$ times.
The requirement that every row appears $\lambda$ times is very restrictive and, in the context of certain applications, also unnecessary; a covering array is the result of relaxing this part of the definition.
More precisely, a \emph{covering array of strength~$t$}, denoted $\CA(N; t, k,v)$, is an $N\times k$ array with entries from an alphabet $A$ of size $v$, with the property that in the $N\times t$ subarray defined by any $t$ columns, each of the $v^t$ vectors in $A^t$ appears \emph{at least} once as a row.

The primary application of covering arrays is the design of experiments.
In areas such as software development and manufacturing, it is often infeasible to perform exhaustive system tests. 
However, empirical research shows that in many types of systems, errors are triggered only when a small number of factors interact
\cite{kuhn2004software}.
In these cases, a practical alternative is \emph{$t$-way interaction testing}, where the objective is to check every $t$-combination of factors.
This approach can dramatically reduce the number of tests that need to be performed, while still being extremely effective in detecting errors
\cite{cohen1996combinatorial,kuhn2004software}.
A $\CA(N; t, k,v)$ can be considered as a test suite for $t$-way interaction testing, where the rows provide a collection of $N$ configurations for a system with $k$ factors, represented by the columns, with each factor admitting $v$ possible values.
Performing experimental runs for all $N$ configurations corresponding to the rows of the covering array, guarantees that every possible interaction of $t$ factors is tested at least once.

In \Cref{table:ExampleOfSoftwareSystem} we give an example of 2-way interaction testing  using a covering array.
The software system in \Cref{table:ExampleOfSoftwareSystem_Factors} has 4 factors, namely a payment server, a web server, a browser, and a database.
Each of them admits 3 possible values as shown in the table.
Using a $\CA(10; 4, 3,2)$  by representing each factor value with $0$, $1$ or $2$, we obtain a collection of $10$ tests corresponding to its rows, as shown in \Cref{table:ExampleOfSoftwareSystem_TestingSuite}.
Testing these 10 configurations guarantees that every interaction of two factors has been considered.
The downside is that, if there exists a combination of $3$ or $4$ factor values that causes the system to fail, then this type of testing is not guaranteed to reveal it; for this to be the case, up to $3^4=81$ configurations should be tested.

\begin{figure}[t]
    \renewcommand{\arraystretch}{\genarraystretch}
    \begin{subfigure}[t]{\textwidth}\centering
        \begin{tabular}[]{LlLl}
            \rowcolor{\tableheadcolor}
            Payment Server & Web server & Browser & Database \\
            Visa (0)& IIS (0)& Mozilla (0)& MySQL (0)\\
            Master card (1)& Apache (1)& Chrome (1)& Oracle (1)\\
            American express (2)& Sun (2)& Opera (2)& PostgreSQL (2)
        \end{tabular}
        \caption{Factors of the system and their possible values}
        \label{table:ExampleOfSoftwareSystem_Factors}
    \vspace{1em}
    \end{subfigure}
    \begin{subfigure}[b]{\textwidth}\centering
    \rowcolors{1}{\backgroundshade}{white}
        $
        \begin{array}{cccc}
        \rowcolor{white}
            &&&\\
            2&1&0&1\\
            1&0&2&2\\
            0&2&2&1\\
            2&2&0&2\\
            2&0&1&1\\
            2&1&2&0\\
            1&2&1&0\\
            1&1&0&1\\
            0&1&1&2\\
            0&0&0&0
        \end{array}
        \quad
        \longrightarrow
        $
        \quad
        \begin{tabular}[]{lllll}
\rowcolor{\tableheadcolor}
\cellcolor{white}  & Payment Server & Web server & Browser & Database \\
\cellcolor{white} Test 1  & American express &Apache  &Mozilla &Oracle\\                                                                                              
\cellcolor{white} Test 2  & Master card      &IIS     &Opera   &PostgreSQL\\                                                                                     
\cellcolor{white} Test 3  & Visa             &Sun     &Opera   &Oracle\\                                                                                  
\cellcolor{white} Test 4  & American express &Sun     &Mozilla &PostgreSQL\\                                                                                          
\cellcolor{white} Test 5  & American express &IIS     &Chrome  &Oracle\\                                                                                              
\cellcolor{white} Test 6  & American express &Apache  &Opera   &MySQL\\                                                                                               
\cellcolor{white} Test 7  & Master card      &Sun     &Chrome  &MySQL\\                                                                                          
\cellcolor{white} Test 8  & Master card      &Apache  &Mozilla &Oracle\\                                                                                         
\cellcolor{white} Test 9  & Visa             &Apache  &Chrome  &PostgreSQL\\                                                                              
\cellcolor{white} Test 10 & Visa             &IIS     &Mozilla &MySQL                                                                                     
        \end{tabular}
        \caption{Using a $\CA(10; 2, 4,3)$ to obtain the configurations for
        2-way interaction testing of the system.}
        \label{table:ExampleOfSoftwareSystem_TestingSuite}
    \end{subfigure}
    \caption{Testing of a software system using a covering array.}
    \label{table:ExampleOfSoftwareSystem}
\end{figure}

Although testing 81 configurations for the system in the previous example may be feasible, for systems with a large number of values per factor, exhaustive testing can be an enormous task.
More precisely, a system with $k$ factors each admitting $v$ possible values requires $v^k$ tests, so the number of tests required in exhaustive testing increases exponentially with respect to the number of factors.
On the other hand, performing $t$-way interaction testing for the same system requires $N$ tests, where $N$ is the number of rows of a $\CA(N; t, k,v)$, which can be significantly smaller than $v^k$.
For example, for fixed $t$ and $v$, it is proven that a $\CA(N; t, k,v)$ with $N=\bigo(\log(k))$ can be constructed in polynomial time using a greedy algorithm \cite{bryce2009density}.
At the same time, empirical evidence has shown that $t$-way interaction testing, even for very small values of $t$, is sufficient for many types of systems in real life.
Indeed, a review of fifteen years of medical device recall data by the US Food and Drug Administration (FDA), shows that $t$-way interaction testing with $t$ equal to $2,3$ or $4$, could detect $97,99$ and $100$ percent of the defects, respectively~\cite{wallace2001failure}.
For studies on the effectiveness of $t$-way interaction testing, we refer to \cite{cohen1996combinatorial,kuhn2004software}.
We also note that the National Institute of Standards and Technology (NIST) maintains an online database of covering arrays as part of its Automated Combinatorial Testing for Software (ACTS) program~\cite{nistcoveringarrays}.  

It is evident from our above discussion that constructing covering arrays with a small row-to-column ratio is an important problem.
This leads to the definition of the \emph{covering array number} $\CAN(t,k,v)$, which is the smallest $n$ such that a $\CA(n;t,k,v)$ exists.
For a $\CA(N; t, k,v)$, comparing $N$ to $\CAN(t, k,v)$ provides a measure of how good the covering array is.
A covering array that attains the minimum number of rows is \emph{optimal}.
A trivial lower bound for $\CAN(t,k,v)$ is $v^t$, which is achieved by $\OA_{1}(t, k,v)$. 
However, such orthogonal arrays are rare and, beyond that, few optimal families of covering arrays are known \cite{bush1952orthogonal,katona1973two,kleitman1973families}. 
New constructions aim instead to improve upon the best currently known upper bounds for covering array numbers.
The previously mentioned greedy algorithm implies a logarithmic upper bound on the covering array number with respect to the number of columns.
Other upper bounds follow from numerous methods for obtaining covering arrays that exist in the literature. 
These include combinatorial and algebraic constructions, greedy and metaheuristic computer algorithms, and recursive methods for obtaining new covering arrays from existing ones.
Colbourn actively maintains an online database of the best known upper bounds for $\CAN(t,k,v)$, where $2\leq t\leq 6$ and $2\leq v\leq 25$~\cite{colbournwebsite}.
Other aspects of the research on covering arrays include existence results of probabilistic nature, and generalizations of covering arrays such as \emph{mixed-level covering arrays}, whose columns contain entries from different alphabets.
A very comprehensive survey on the subject is given by Colbourn \cite{colbourn2004combinatorial}; for a more recent albeit more brief presentation, we refer to Kuliamin and Petukhov \cite{kuliamin2011survey}.
We also give a short overview of many important results in \Cref{chapter:Preliminaries}.

In this thesis we focus on the following two problems related to covering arrays:
\begin{enumerate}
    \item find covering arrays of strength more than $3$ that improve upon the currently best known upper bounds for covering array numbers, and
    \item give a theoretical construction of covering arrays such that, for arbitrary $t,k,v$, a $\CA(N;t,k,v)$ can be constructed for some $N$.
\end{enumerate}
With regards to the first problem, we note that an increasing number of the best known upper bounds for covering array numbers for strengths $4$ or higher are obtained recursively from other bounds \cite{colbournwebsite}, while the direct theoretical constructions in Colbourn's survey \cite{colbourn2004combinatorial} are almost exclusively for strengths $2$ and $3$.
Regarding the second problem, it appears that prior to this work and to the best of our knowledge, the construction by Colbourn \cite{colbourn2010covering} is the only direct (not algorithmic, recursive or probabilistic) answer.
We address these two problems relying on powerful tools from the theory of finite fields.

A finite field - or Galois field, in honor of \'{E}variste Galois - is an algebraic field with a finite number of elements.
The theory of finite fields is a branch of algebra with a wide variety of applications in combinatorics, error-correcting codes, finite geometry, cryptography and other fields of mathematics. 
One area of finite fields that has gained a lot of interest in recent years is concerned with sequences of finite field elements that satisfy a linear recurrence relation.
Although the study of such sequences can be traced back to the works of Lagrange in the 18\textsuperscript{th} century, it came to fore again in the 1950s with the dawn of the digital age, when they were found to have applications of great importance in wireless communications.
Due to their straightforward and widespread implementation in digital circuits using feedback shift registers, they have come to be known as linear feedback shift register (LFSR) sequences.
\index{LFSR sequence}
From both a mathematical and engineering standpoint, the works of Golomb \cite{golomb1955sequences,golomb1967shift} played a crucial role in the revival of the topic during that time, and nowadays the algebraic properties of LFSR sequences have been studied extensively \cite{golomb1967shift,golomb2005signal}. 
For instance, it is well-known that they are periodic; LFSR sequences that attain the maximum possible period are \emph{maximal (period) sequences} - often abbreviated to \emph{m-sequences} in the literature.

Maximal sequences have been previously employed for the construction of orthogonal arrays.
As we show in \Cref{chapter:CombinatorialArraysFromMSequences}, some classic orthogonal array constructions can be translated into the language of maximal sequences.
However, the first construction that relies primarily on using their properties appears to be a 1998 work by Munemasa \cite{munemasa1998orthogonal}, who constructs binary orthogonal arrays of strength $2$ that are ``close to'' strength $3$.
This was followed by two generalizations about a decade later \cite{dewar2007division,panario2012divisibility}, that provide binary and ternary covering arrays of strength $2$ or $3$.
These results are based on the study of the divisibility properties of certain types of polynomials with coefficients from the finite fields with $2$ or $3$ elements.
A different approach was adopted by Raaphorst et al.\ \cite{raaphorst2014construction} (see also \cite{raaphorst2013variable}) who exploited connections with combinatorial designs to construct a family of covering arrays of strength~$3$. 
This is the only construction, prior to our work, that is based on LFSR sequences and produces covering arrays which are not orthogonal arrays.
Hence, despite the rich algebraic structure and many combinatorial properties of LFSR sequences, their use for the construction of covering arrays is a rather unknown territory.
This thesis is a step towards this promising direction.

Our main contributions can be summarized in the following.
\begin{itemize}
\item We give a comprehensive study of the connections between maximal sequences and diverse concepts related to orthogonal arrays, error-correcting codes, divisibility of polynomials and structures from finite geometry.
While these connections have been used or implied in previous research, we present them in a unified framework for the first time.
This also allows us to describe the previous and new constructions of orthogonal and covering arrays from maximal sequences using the language of different areas of discrete mathematics, which we hope will facilitate further research on this topic.

\item We develop a backtracking algorithm for the construction of covering arrays of arbitrary strength, that relies on an underlying finite field theory, properties of maximal sequences and connections to other combinatorial objects.
    Among our experimental results there are 37 new covering arrays of strength $4$ and one of strength $5$ that improve upon previously best known bounds listed in the online database by Colbourn \cite{colbournwebsite}.
    This work appears in \cite{tzanakis2016constructing}.

\item Relying on properties of maximal sequences and an argument involving character sums, we give new infinite families of covering arrays of strengths $3$ and $4$, as well as an infinite family of covering arrays of arbitrary strength.
This is one of the two known direct constructions of covering arrays of arbitrary strength along with
a recent one due to Colbourn \cite{colbourn2010covering}.
Furthermore, our constructions and that of Raaphorst et al.\ \cite{raaphorst2014construction} are the only ones that rely on LFSR sequences and yield covering arrays that are not orthogonal arrays.

%

\end{itemize}

We conclude by outlining the structure of the thesis.

\begin{itemize}
\item \Cref{chapter:Preliminaries} contains the necessary preliminaries to make our work self-contained.
The first part is dedicated to several aspects of the theory and applications of finite fields, such as the existence and construction of finite fields, characters, maximal sequences, finite geometry, combinatorial designs, and coding theory.
The second part gives the necessary background on orthogonal and covering arrays, as well as a presentation of the state of the current research on covering arrays.

\item \Cref{chapter:CombinatorialArraysFromMSequences} serves both to provide tools necessary for the rest of the thesis, as well as to present in detail previously established constructions of orthogonal and covering arrays using maximal sequences.
We start by proving fundamental connections between maximal sequences and concepts from other areas of discrete mathematics, and then we explore the previously established constructions in view of the above connections.

\item \Cref{chapter:CAsFirstPaper} is concerned with the construction of covering arrays of any strength from maximal sequences, from the point of view of an optimization problem to which we give an algorithmic solution.
The chapter starts with an in-depth look of the optimization problem we wish to solve, and continues with the development of the theory behind the algorithm, that requires several results from finite fields and the study of other combinatorial objects.
Finally, we present our experimental results and discuss their importance.

\item In \Cref{chapter:CAsFromMSequencesAndCharacterSums} we introduce cyclic trace arrays modulo $v$, a new type of array based on maximal sequences, discrete logarithms and integer congruences that we study using characters - homomorphisms from the finite field to the unit circle of complex numbers.
In particular, we use ideas from \cite{colbourn2010covering} to translate the problem of checking whether a cyclic trace array modulo $v$ is a covering array into the problem of finding a lower bound to an expression that involves character sums.
We then use properties of maximal sequences and well-known bounds of character sums to obtain conditions subject to which a cyclic trace array modulo $v$ is a covering array. 
Based on that, we present new infinite families of covering arrays.

\item \Cref{chapter:Index} is preliminary work on some aspects of the arrays introduced in the previous chapter, using different techniques.
More precisely, we count the number of different $t$-tuples that appear in the rows, and we express this count as the solution of a linear program. 
By providing computational lower bounds for this number, we hope to find cyclic trace arrays modulo $v$ that are covering arrays which are not revealed by the conditions in \Cref{chapter:CAsFromMSequencesAndCharacterSums}.
An extended abstract of this ongoing work has been submitted \cite{kalamata}.

\item \Cref{chapter:FutureDirections} discusses potential future directions related to the thesis.

\end{itemize}

\chapter{Preliminaries}
\label{chapter:Preliminaries}

\textsc{In this chapter we give} the minimum necessary background to make our thesis self-contained.
This consists of well-known definitions and results from various topics related to finite fields as well as orthogonal and covering arrays.
In \Cref{section:FiniteFields} we present some fundamental theory of finite fields, and then we touch on characters, sequences, finite geometry and linear codes.
In \Cref{section:OrthogonalAndCoveringArrays} we define orthogonal and covering arrays, and we give an overview of the research in this area, which includes a discussion about their applications and a presentation of some important results that have already been established.

\section{Finite fields}
\label{section:FiniteFields}

\subsection{Review of some basic finite field theory}
In this section we give a number of fundamental definitions and facts about finite fields.
Our presentation is brief as we only include things that we use in this thesis.
For a complete and thorough introduction to finite fields, we refer the reader to \cite{lidl1997finite}.
We also assume some basic knowledge of ring and field theory.

\subsubsection{Existence and uniqueness of finite fields}

It follows from elementary algebra that if $L$ is an extension field over a subfield $K$, then $L$ can be seen as a vector space over $K$, under the usual field operations.
For finite fields in particular, we have the following lemma.

\begin{lemma}
    {A finite field is a vector space over any subfield}
    {FFIsAVSpaceOverSubfield}
    Let $F$ be a finite field and $K$ be a subfield containing $q$ elements.
    Then $F$ is a vector space over $K$, and $|F|=q^t$, where $t$ is the dimension of $F$, viewed as a vector space over $\fq$.
\end{lemma}
\index{Finite field!as vector space}
Using \Cref{lemma:FFIsAVSpaceOverSubfield}, we can characterize the cardinality of finite fields.

\begin{theorem}
    {Every finite field has cardinality a prime power
     \cite[Theorem 2.2]{lidl1997finite}}
{CardinalityOfFFieldIsPrimePower}
    Let $F$ be a finite field.
    Its prime subfield $K$ has cardinality $p$, a prime number, which is also the characteristic of $F$.
    Then, $|F|=p^n$ where $n$ is the dimension of $F$ over $K$.
\end{theorem}
    \index{Finite field!cardinality of}

The converse of \Cref{theorem:CardinalityOfFFieldIsPrimePower} is also true.

\begin{theorem}
{Existence and uniqueness of finite fields
\cite[Theorem 2.5]{lidl1997finite}}
{ForEveryPrimeThereIsAFFieldWithThatCardinality}
    For every prime $p$ and positive integer $n$, there exists a finite field with $p^n$ elements.
    Furthermore, any finite field with $p^n$ elements is isomorphic to the splitting field of $x^{p^n}-x$ over $\Z_p$.
\end{theorem}
    \index{Finite field!existence and uniqueness of}

From
\Cref{theorem:CardinalityOfFFieldIsPrimePower,theorem:ForEveryPrimeThereIsAFFieldWithThatCardinality}
we conclude that every finite field has cardinality a prime power and, conversely, for every prime power $q$ there exists a finite field with $q$ elements.
Furthermore, all finite fields with $q$ elements are isomorphic since splitting fields of the same polynomial are isomorphic.
Hence, we speak of \emph{the} finite field with $q$ elements, which we denote $\fq$.

\subsubsection{Properties and construction of finite fields}
Having established the existence and uniqueness of finite fields, we list some of their basic properties and discuss their construction.
We begin by a characterization of the subfields of a finite field.
\begin{theorem}
{Subfield criterion \cite[Theorem 2.6]{lidl1997finite}}
{SubfieldCriterionOfFiniteFields}
    Let $p$ be a prime and $n$ be a positive integer.
    Then, every subfield of $\f_{p^n}$ has order $p^d$, where $d$ is a positive divisor of $n$.
    Conversely, if $d$ is a positive divisor of $n$, then there is exactly one subfield of $\fq$ with $p^d$ elements.
\end{theorem}
    \index{Subfield!criterion}

To every field element, we associate a unique polynomial as follows.
\begin{definition}
{Minimal polynomial of a finite field element}
{MinimalPolynomialOfFiniteFieldElement}
    Let $K$ be a subfield of $F$, and $\a \in F$ be algebraic over $K$.
    Then, the uniquely determined monic irreducible polynomial $g \in K[x]$ generating the ideal
    $J=\{ f\in K[x] \mid f(\a)=0 \}$
    is the \emph{minimal polynomial of $\a$ over $K$}.
    We denote this polynomial $m_{\a,K}$, however we simply write $m_{\a}$ when it is clear from the context what $K$ is.
\end{definition}
\index{Polynomial!minimal|see{Minimal polynomial}}
\index{Minimal polynomial!of element}

We use the following lemma regularly in the next chapters.

\begin{lemma}{\cite[Lemma 2.12]{lidl1997finite}}
    {AlphaRootOfPolynomialIffItIsDivisibleByMinimal}
    Let $t$ be a positive integer and $\a\in \fqt$.
    Then, for a polynomial $f\in \fqx$, we have that $\a$ is a root of $f$ if and only if $f$ is divisible by $m_{\a}$.
\end{lemma}

For integers $a,b$ with $a<b$, we denote
$[a,b]=\{ a, a+1, \dots, b \}$.
For a prime $p$, the field $\fp$ is isomorphic to $\Z_p=\Z/p\Z$, so we identify $\fp$ with the set $[0,p-1]$ along with the usual addition and multiplication modulo $p$.
Next, we consider the representation of finite fields whose cardinality is a prime power.

\begin{theorem}
{Roots of irreducible polynomials over finite fields
    {\cite[Theorem 2.14]{lidl1997finite}}
}
{RootsOfIrreduciblePolynomialsOverFFields}
If $f$ is an irreducible polynomial in $\fqx$ of degree $t$, then $f$ has a root $\a$ in $\fqt$.
Furthermore, all the roots of $f$ are simple and are given by the $t$ distinct elements $\a, \a^{q}, \dots, \a^{q^t}$.
\end{theorem}
\index{Irreducible polynomial!roots of}
\index{Polynomial!irreducible|see {Irreducible polynomial}}

\begin{corollary}
{{Construction of finite fields \cite[Corollary 2.15]{lidl1997finite}}}
{ConstructionOfFiniteFields}
Let $f\in\fqx$ be irreducible with degree $t$.
Then $\fqt$ is the splitting field of $f$ over $\fq$.
\end{corollary}
\index{Finite field!construction of}

We know (see for instance \cite[Corollary 2.11]{lidl1997finite}) that for every prime power $q$ and positive integer $t$, there exists an irreducible polynomial $f\in \fqx$ of degree $t$.
It follows from \Cref{theorem:RootsOfIrreduciblePolynomialsOverFFields} and \Cref{corollary:ConstructionOfFiniteFields} that for a root $\a\in \fqt$ of $f$, we have that
$\fqt=\fq(\a,\a^q, \ldots, \a^{q^t}) = \fq(\a)$,
therefore we can write
\begin{equation}
\label{equation:FundamentalRepresentationOfFiniteField}
\fqt
=
\left\{ c_{t-1}\a^{t-1}+c_{t-2}\a^{t-2}+\cdots+c_1\a+c_0 \mid c_{0},\ldots, c_{t-1} \in \fq \right\}.
\end{equation}

\begin{definition}
    {Additive and multiplicative groups of a finite field}
    {GroupsOfFiniteField}
    We refer to $(\fq,+)$ as the \emph{additive group of $\fq$}, and $(\fqstar,\cdot)$ as the \emph{multiplicative group of $\fq$}, where $\fqstar=\fq\setminus \{ 0 \}$.
    We denote $(\fqstar, \cdot)$ simply as $\fqstar$.
\end{definition}
    \index{Finite field!additive group of}
    \index{Finite field!multiplicative group of}

Since $\fq$ is a field, every nonzero element has a multiplicative inverse, hence $\fqstar$ consists of the $q-1$ nonzero elements of $\fq$.
The following theorem gives an important property of $\fqstar$.

\begin{theorem}
    {The multiplicative group of a finite field is cyclic \cite[Theorem 2.8]{lidl1997finite}.}
    {MultiplicativeGroupOfFQIsCyclic}
    For every prime power $q$, the multiplicative group $\fqstar$ of nonzero elements of $\fq$ is cyclic.
\end{theorem}
    \index{Finite field!multiplicative group of}

The next lemma gives a way to determine if an element belongs to a given finite field.

\begin{lemma}
{}
{PowerQEqualsElementMeansBelongingToFQ}
    For a prime power $q$ and a positive integer $t$, we have that $a\in \fqt$ if and only if $a^{q^t}=a$.
\end{lemma}
\begin{proof}
The case $a=0$ is trivial.
If $a\neq 0$, then $a \in \fqtstar$ and therefore the multiplicative order of $a$ divides the group order $q^t-1$, hence
$a^{q^t-1}=1$; multiplying both sides by $a$ gives $a^{q^t}=a$.
Conversely, if $a$ satisfies $a^{q^t}=a$, then it is a root of $x^{q^t}-x$ and as such it belongs to its splitting field; by \Cref{theorem:ForEveryPrimeThereIsAFFieldWithThatCardinality}, this is $\fqt$.
\end{proof}

Generators of the multiplicative group of a finite field and their minimal polynomials are important in the area of finite fields and have been studied extensively in various contexts; see for example \cite[Chapter 4]{mullen2013handbook} for an overview of related topics.

\begin{definition}
{Primitive elements and polynomials}
{PrimitiveElement}
    For a prime power $q$, a generator of the cyclic group $\fqstar$ is a \emph{primitive element of $\fq$}.
    The minimal polynomial of a primitive element is a \emph{primitive polynomial}.
\end{definition}
    \index{Primitive element}
    \index{Element!primitive|see{Primitive element}}
    \index{Primitive polynomial}
    \index{Polynomial!primitive|see{Primitive polynomial}}
\begin{remark}
{Primitive polynomials are irreducible}
{PrimitivePolysAreIrreducible}
    Any primitive polynomial is also irreducible.
    Indeed, for a root $\a$ of a primitive polynomial of degree $t$, we have that $\fqt = \fq(\a)$, hence $f$ has to be irreducible.
\end{remark}

It follows from \Cref{definition:PrimitiveElement} that for primitive $\a\in\fq$, we have that 
\[\fqstar=\left\{ \a^i \mid i \in [0,q-2] \right\}.\] 
 
\begin{example}
{Construction of the finite fields with $4$ and $16$ elements}
{ConstructionFFourAndFSixteen}
For the  polynomial $f(x)=x^2+x+1 \in \ftwo[x]$ we have that $f(0)=f(1)=1$, so $f$ does not have a root in $\ftwo$ and therefore it is irreducible over $\ftwo$, as a polynomial of degree $2$ without roots.
Let $\a$ be a root of $f$ in $\ffour$.
From the above, we have that the field with $4$ elements satisfy
\begin{equation}
    \label{equation:ExampleConstructionOfFFour}
    \ffour
    =\ftwo(\a)
    =\left\{ c_0+c_1\a \mid c_0,c_1 \in \ftwo \right\}
    =\left\{ 0,1,\a,\a+1 \right\}.
\end{equation}
We note that since $\a$ is a root of $x^2+x+1$, we have that $\a^2+\a+1=0$, hence 
\[\a^2=-\a-1=\a+1,\] which we can use to
write a product of elements from $\ffour$ in the form $c_0+c_1\a$, $c_0, c_1 \in \ftwo$.
For example, the  product of $\a$ and $\a+1$ is given by 
\[\a(\a+1)=\a^2+\a=(\a+1)+\a=2\a+1=1.\]
Working this way, we create the multiplication table of $\ffour$, which we display along with the addition table below.
\[
\begin{array}[]{c|cccc}
+ & 0 & 1 & \a & \a+1\\
\hline
0&0&1&\a&\a+1\\
1&1&0&\a+1&\a\\
\a&\a&\a+1&0&1\\
\a+1&\a+1&\a&1&0
\end{array}
\quad
\quad
\quad
\begin{array}[]{c|cccc}
\cdot & 0 & 1 & \a & \a+1\\
\hline
0&0&0&0&0\\
1&0&1&\a&\a+1\\
\a&0&\a&\a+1&1\\
\a+1&0&\a+1&1&\a
\end{array}
\]

We also note that, since $\a^0=1$ and $\a^2=\a+1$, we have that
\[
    \left\{ \a^i \mid i \in [0,2] \right\}
=\left\{ 1,\a,\a+1 \right\} = \ffour^{\times},\]
hence $\a$ is a primitive element of $\ffour$ and $f$ is a primitive polynomial in $\ftwo[x]$.

We now construct $\f_{16}$ as an extension of $\ffour$.
Let $\beta$ be a root of $g(x)=x^2+x+\a \in \ffour[x]$.
We calculate that $g(0)=g(1)=\a$, and $g(\a)=g(\a+1)=\a+1$, therefore no element of $\ffour$ is a root of $g$ and hence $g$ is irreducible over $\ffour$, as a polynomial of degree $2$ without roots in $\ffour$.

Next, we show that $\beta$ is a primitive element of $\f_{16}$.
Since $\beta$ is a root of $g$, we have that $\beta^2+\beta+\a=0$, which means that $\beta^2=\beta+\a$.
Then,
\[\beta^3
    =\beta\beta^2
    =\beta(\beta+\a)
    =\beta^2 +\a\beta
    = \beta(\a+1)+\a
= \a + (\a+1)\beta.\]
Also,
\[\beta^5
    =\beta^2\beta^3
=(\beta+\a)(\a+(\a+1)\beta).\]
Expanding the last product and working out the powers of $\alpha$ and $\beta$ as before, we find that $\beta^5=\a$.
Since $|\f_{16}^{\times}|=15=3\cdot5$ and the multiplicative order of $\beta$ is neither $3$ or $5$, then it must be $15$.
This means that $\beta$ is a generator of $\f_{16}^{\times}$, that is, a primitive element of $\f_{16}$.
We present the first 15 powers of $\beta$ in \Cref{table:PowersOfBetaInExample}; we observe that every element of $\f_{16}^{\times}$ appears in the table.
\end{example}

\begin{table}
    \renewcommand{\arraystretch}{\genarraystretch}
    \centering
    \rowcolors{1}{\backgroundshade}{white}
    \[
    \begin{array}[]{c*{5}{!{\color{\backgroundshade}\vrule}c}}
    i&0&1&2&3&4\\
    \beta^i & 1& \beta& \a+\beta& \a+(\a+1)\beta& 1+\beta\\
    i&5&6&7&8&9\\
    \beta^i &\a& \a\beta &\a+1+\a\beta& (\a+1)+\beta& \a+\a\beta\\
    i&10&11&12&13&14\\
    \beta^i & \a+1& (\a+1)\beta& 1+(\a+1)\beta& 1+\a\beta&(\a+1)+(\a+1)\beta
    \end{array}
    \]
    \caption[The powers of a primitive element of $\f_{16}$]{The first 15 powers of the primitive element $\beta \in \f_{16}$
described in \Cref{example:ConstructionFFourAndFSixteen}.
These are also precisely the nonzero elements of $\f_{16}$.
}
\label{table:PowersOfBetaInExample}
\end{table}

For a positive integer $n$, \emph{Euler's function} $\phi(n)$ is the number of integers $k\in [1,n]$, with $\gcd(k,n)=1$.
\index{Euler's function}
\index{Function!Euler's|see{Euler's function}}
The following lemma follows from well-known facts in group theory.
\begin{lemma}
    {Number of primitive elements in $\fq$}
    {NumberOfPrimitiveElementsOfFqm}
    Let $q$ be a prime power, and $\a\in \fq$ be primitive.
    There exist exactly $\phi(q-1)$ primitive elements in $\fq$, given by $\a^i$, $i \in [0, q-2]$ such that $\gcd(i,q-1)=1$.
\end{lemma}

\subsubsection{The trace function}
An important function that we use extensively in this thesis, is defined as follows.

\begin{definition}
{The trace function over finite fields}
{TraceFunction}
    Let $q$ be a prime power, $t$ be a positive integer, and $a \in \fqt$.
    The \emph{trace of $a$ over $\fq$} is defined by
    \[
        \Tt(a) = a+a^q+\dots +a^{q^{t-1}}.
    \]
\end{definition}
\index{Element!trace\ of}
\index{Trace function}
\index{Function!trace|see{Trace function}}
We have that
$\Tt(a)\in \fq$ for all $a\in \fqt$
(see \cite[Section 2.3]{lidl1997finite}).
Furthermore, by \Cref{lemma:FFIsAVSpaceOverSubfield} we have that $\fqt$ is a vector space over $\fq$, and hence we can view the trace as a vector space mapping defined by

\begin{equation}
    \label{equation:TraceIsAVSpaceMapping}
    \begin{array}[]{ll}
        \Tt : & \fqt  \longrightarrow \fq \\
              &x \mapsto x+x^q+\dots+x^{q^{t-1}}.
    \end{array}
\end{equation}

\begin{theorem}
    {Properties of the trace function \cite[Theorem 2.23]{lidl1997finite}}
    {PropertiesOfTrace}
    For a prime power $q$ and positive integer $t$, the trace $\Tt$ satisfies the following properties:
    \begin{enumerate}
        \item $\Tt$ is a linear transformation from $\fqt$ onto $\fq$, where both $\fqt$ and $\fq$ are viewed as vector spaces over $\fq$;
            \label{item:TraceIsLinearTransformation}
        \item $\Tt(a)=ta$, for all $a\in\fq$;
        \item $\Tt(a^q)=\Tt(a)$, for all $a\in\fqt$.
    \end{enumerate}
    \index{Trace function!properties of}
\end{theorem}
\begin{example}
{}
{CalculateSomeTraces}
    For $\a$ and $\beta$ as in \Cref{example:ConstructionFFourAndFSixteen}, we calculate
    $\T_{16/2}(\a\beta)$ and $\T_{16/4}(\a\beta)$.
    Using \Cref{table:PowersOfBetaInExample} we calculate
    \begin{align*}
        \T_{16/2}(\a\beta)
        &=
        \a\beta+ (\a\beta)^2+ (\a\beta)^{2^2}+ (\a\beta)^{2^3}
        \\&=
        \beta^6+ \beta^{12}+ \beta^{24}+ \beta^{48}
        \\&=
        \beta^6+ \beta^{12}+ \beta^{9}+ \beta^{3}
        \\&=
        \a\beta+1+(\a+1)\beta+\a+\a\beta+\a+(\a+1)\beta
        \\&= 1,
    \end{align*}
    and similarly we can calculate
    \[
        \T_{16/4}(\a\beta)=\a.
    \]
    As expected from \Cref{theorem:PropertiesOfTrace}, we observe in the above that $\Tt(a) \in \fq$, for all $a \in \fqt$.
\end{example}

The fact that the trace defines a linear transformation over $\fq$ implies the following proposition.

\begin{proposition}
    {\cite[Corollary 5.1]{golomb2005signal}}
    {TraceDefinesQMinusOneToOneMapping}
    Let $q$ be a prime power, and $t$ be a positive integer.
    Then, the mapping defined in \Cref{equation:TraceIsAVSpaceMapping} is a $(q-1)$-to-$1$ mapping.
    Furthermore, for every $c\in\fq$ and $\beta\in\fqt$, the equation $\Tt(\beta x)=c$ has exactly $q^{t-1}$ solutions $x\in\fqt$.
\end{proposition}

\subsection{Characters}
Characters are an important class of group homomorphisms.
For groups related to finite fields, they are powerful tools which have been studied extensively and used in many applications \cite[Chapter 6]{mullen2013handbook}.

\subsubsection{Characters of arbitrary groups}
\begin{definition}
    {Group character and related notions}
    {Character}
    A \emph{character} of a finite Abelian group $G$ is a homomorphism from $G$ into the multiplicative group $\mathbb{C}^*$ of complex roots of unity.
    The set of all characters of $G$ is denoted $\widehat{G}$.
    The \emph{trivial character} $\chi_{0}$ is defined by $\chi_0(g)=1_G$, for all $g\in G$; all other characters are referred to as \emph{non-trivial}.
    For each character $\chi$ there is a conjugate character $\overline{\chi}$ defined by $\overline{\chi}(g) = \overline{\chi(g)}$ for all $g \in G$.
\end{definition}
    \index{Character}
    \index{Character!multiplicative}
    \index{Character!additive}
    \index{Character!trivial}
    \index{Character!non-trivial}
    \index{Character!conjugate}
    \index{Character!group}

Let $G$ be a finite Abelian group, and $\chi_1, \dots, \chi_n \in \widehat{G}$.
The \emph{product of $\chi_1, \dots, \chi_n$} is the character defined by
\[(\chi_1 \cdots \chi_n)(g)=\chi_1(g)\cdots\chi_n(g),\]
and $\widehat{G}$ is a group under this multiplication.
If $\chi_1, \dots, \chi_n$ are the same character $\chi$, then we write $\chi_1 \cdots \chi_n = \chi^n$.
When $G$ is cyclic, we can describe the characters of $\widehat{G}$ as
follows.

\begin{proposition}
    {Characters of a finite cyclic group}
    {DescriptionOfCharactersOfAFiniteCyclicGroup}
    Let $G$ be a finite cyclic group with a generator $g$.
    We consider the mapping on $G$ defined by
    \begin{equation}
        \label{equation:GeneratorOfMultiplicativeGroup}
        \chi(g^k)
        = e^{\frac{2\pi ik}{|G|}}, \quad k \in \Z,
    \end{equation}
    so that, for an integer $j$, we have
    \begin{equation}
        \label{equation:CharactersOfMultiplicativeGroup}
        \chi^j(g^k)= e^{\frac{2\pi ijk}{|G|}}, \quad k\in \Z.
    \end{equation}
    Then, $\widehat{G}$ is a cyclic group of order $|G|$, and
    \[\widehat{G}=\left\{ \chi^j \mid j \in [0, |G|-1]\right\}.\]
\end{proposition}
\begin{proof}
    For a fixed $j \in [0, |G|-1]$, the mapping in \Cref{equation:CharactersOfMultiplicativeGroup} is a character of $G$.
    Conversely, for any character $\psi \in \widehat{G}$ and any $g \in G$, $\psi(g)$ is a $|G|$-th root of unity and thus $\psi(g)=e^{2\pi ij/|G|}$, for some $j\in [0,|G|-1]$.
    Since $\psi$ is multiplicative and $g$ generates $G$, it follows that $\psi = \chi^j$, hence $\chi$ generates the cyclic group $\widehat{G}$.
\end{proof}

The following theorem is a classic result regarding character sums.
\begin{theorem}{Orthogonality relations for characters
    \cite[Section~5]{lidl1997finite}}
    {CharSumIsZeroClassicResult}
    Let $\psi$ be a character of a group $G$.
    Then
    \[ \sum_{g\in G}\psi(g)=
        \begin{cases}
            |G|& \text{ if } \psi \text{ is trivial;}\\
              0& otherwise.
        \end{cases}
    \]
   Furthermore, if $g \in G$ and $\widehat{G}$ is the group of all characters of $G$, we have
   \[ \sum_{\psi\in \widehat{G}}\psi(g)=
       \begin{cases}
         |G| & \text{ if } g=1_G;\\
           0 & otherwise.
       \end{cases}
   \]
\end{theorem}

\subsubsection{Characters of a finite field}

There are two kinds of characters associated to a finite field, corresponding to its additive and multiplicative groups; in this thesis we use exclusively characters of the latter kind.

\begin{definition}
    {Multiplicative character of a finite field}
    {MultiplicativeCharacterOfAFiniteField}
    For a prime power $q$, a character of the multiplicative group $\fqstar$ is a \emph{multiplicative character of $\fq$}.
    A multiplicative character $\chi$ of $\fq$ is extended to a function on $\fq$ by setting $\chi(0)=0$.
    The multiplicative character $\chi_0$ that maps every nonzero element to $1$ is the \emph{trivial multiplicative character of $\fq$}.
\end{definition}
    \index{Character!multiplicative}
    \index{Character!trivial}
    \index{Character!group}

    As per \Cref{theorem:MultiplicativeGroupOfFQIsCyclic}, the multiplicative group of a finite field is cyclic.
    Therefore, by \Cref{proposition:DescriptionOfCharactersOfAFiniteCyclicGroup}, we have the following characterization of multiplicative finite field characters.

\begin{corollary}
    {Characterization of multiplicative characters of a finite field}
    {DescriptionOfMultiplicativeCharactersOfAFiniteField}
    Let $q$ be a prime power, and $\a$ be a primitive element of $\fq$.
    We consider the mapping
    \begin{equation}
        \label{equation:GeneratorOfMultiplicativeCharactersOfAFF}
        \chi(\a^k)
        = e^{\frac{2\pi ik}{q-1}}, \quad k\in \Z,
    \end{equation}
    so that, for an integer $j$, we have
    \begin{equation}
        \label{equation:CharactersOfMultiplicativeCharacterGroupOfAFF}
        \chi^j(\a^k)
        = e^{\frac{2\pi ijk}{q-1}}, \quad k\in \Z.
    \end{equation}
    Then, the group $\widehat{\fqstar}$ of multiplicative characters of $\fq$ is a cyclic group of order $q-1$, and
    \[
        \widehat{\fqstar}
        =\left\{ \chi^j \mid j \in [0, q-2] \right\}.
    \]
\end{corollary}

We now look at products of multiplicative characters.
\begin{lemma}
{}
{directproductofcharacters}
Let $v$ be an integer with $v\geq 2$ and $j_0, j_1, \dots, j_{n-1} \in [1, v-1]$.
Then the mapping $X:( \fqstar )^n \rightarrow \mathbb{C}^*$ defined by 
\[
    X(x_1,\dots,x_n)=\prod_{i=0}^{n-1}\chi^{j_i}(x_i)
\]
is a nontrivial character of the group $( \fqstar )^n$, and 
\begin{equation}
    \label{equation:sumofproductofcharactersiszero}
    \sum_{x_0, \dots, x_{n-1} \in \fqstar} \prod_{i=0}^{n-1} \chi^{j_i}(x_i) = 0.
\end{equation}
\end{lemma}
\begin{proof}
    We have that $X$ is a homomorphism whose order is $\lcm(\ord(\chi^{j_0}),\dots, \ord(\chi^{j_{n-1}}))$, which is greater than $1$ since the orders of $\chi^{j_i}$, $i\in [0,n-1]$ are all greater than $1$.
    Thus $X$ is a non-trivial character, and \Cref{equation:sumofproductofcharactersiszero} follows from \Cref{theorem:CharSumIsZeroClassicResult}.
\end{proof}

Finally, we introduce Jacobi sums and a related important result.

\index{Jacobi sum}
\begin{definition}{Jacobi sum}{JacobiSum}
    For $\chi_1, \dots, \chi_n$ multiplicative characters on $\fq$, the \emph{Jacobi sum $J(\chi_1, \dots, \chi_n)$} over $\fq$ is defined by
    \[
      \J(\chi_1, \dots, \chi_n)
      = \sum_{\substack{ s_1, \dots, s_n \in \fq\\ s_1+\dots+s_n=1}}
        \chi_1(s_1)\dots \chi_n(s_n).\]
\end{definition}

\begin{theorem}
{Absolute value of Jacobi sum {\cite[Theorem 6.1.38]{mullen2013handbook}}}
{JacobiSumBound}
    Let $\chi_1, \dots, \chi_n$ be characters of $\fqstar$ such that $\chi_1, \dots, \chi_n$, and $\prod_{i=1}^n \chi_i$ are all nontrivial.
    Then,
    \[|\J(\chi_1, \dots, \chi_n)|=q^{\frac{n-1}{2}}.\]
\end{theorem}
\index{Jacobi sum!absolute value of}

\subsection{Sequences over finite fields}
In this section we give the background required in this thesis about sequences over finite fields.
For a comprehensive presentation we refer the reader to \cite[Chapter 5]{lidl1997finite}, as well as \cite{golomb2005signal} for a textbook on the subject.

\subsubsection{LFSR sequences and their minimal polynomials}

\begin{definition}{(Ultimately) periodic sequence}{PeriodicSequence}
    Let $(s_i)_{i\geq 0}$ be a sequence of elements from a finite set.
    If there exist integers $r,u$ such that $r>0, u\geq 0$ and $s_{i+r} = s_i$, for all $i \geq u$, then the sequence is \emph{ultimately periodic}, with \emph{period} $r$.
    The smallest $r$ with that property is the \emph{least period} of the sequence.
    If $u=0$, the sequence is just called \emph{periodic}.
\end{definition}
    \index{Sequence!periodic}
    \index{Sequence!ultimately\ periodic}

\begin{definition}
{Linear feedback shift register (LFSR) sequence}
{LFSRSequence}
    Let $t$ be a positive integer and let $I=(c_0, c_1, \dots, c_{t-1})$ be a fixed $t$-tuple of elements in $\fq$.
    A sequence $\bfs=(s_0, s_1, \dots)$ of elements in $\fq$ satisfying the homogeneous linear recurrence relation
    \[
        s_{j+t}
        = \sum_{i=0}^{t-1} c_i s_{j+i}, \quad j =0, 1, \dots
    \]
    is a \emph{linear feedback shift register (LFSR) sequence over $\fq$}.
    The $t$-tuple $I$ is the \emph{initial state vector} of $\bfs$, the integer $t$ is its \emph{order}, and the polynomial
    \[f(x)=x^t- c_{t-1}x^{t-1} - \dots - c_1 x- c_0  \in \fq[x]\]
    is a \emph{characteristic polynomial} of $\bfs$.
    We say that $\bfs$ \emph{is generated by $f$ with initial state $I$}.
    If $\bfs$ is generated by $f$ with some initial state, then we simply say that \emph{$f$ generates $\bfs$}.
\end{definition}
\index{Sequence!LFSR|see{LFSR sequence}}
\index{Sequence!linear\ feedback\ shift\ register|see{LFSR sequence}}
\index{LFSR sequence!initial state vector of}
\index{LFSR sequence!characteristic polynomial of}
\index{LFSR sequence!order of}

In the following, we denote by $\overline{(s_0, \dots, s_{r-1})}$ the infinite periodic sequence with period $r$, whose first $r$ elements are $s_0, \dots, s_{r-1}$, that is,  the sequence $(s_0, \dots, s_{r-1}, s_0, \dots, s_{r-1}, \dots).$

\begin{example}
{}
{SomeLinearSequencesOverFtwo}
    The linear recurrence relation in $\ftwo$ given by
    \begin{equation}
    \label{equation:LinearReccurenceRelationExample}
        s_{i+3}=s_{i+2}+s_{i+1}+s_i, \quad i=0, 1, \dots
    \end{equation}
     generates the following sequences.
     \begin{itemize}
         \item
             The zero sequence
             $\overline{(0)}=(0,0,\dots)$, with
             initial state $(0,0,0)$;
         \item
             the constant sequence
             $\overline{(1)}=(1,1,\dots)$, with
             initial state $(1,1,1)$;
         \item
             the sequence
             $\overline{(0,0,1,1)}$, with
             initial state $(0,0,1)$, and
         \item
             the sequence
             $\overline{(1,0)}$, with
             initial state $(1,0,1)$.
     \end{itemize}
     All other initial states generate one of the sequences above, possibly after a shift.
     For example, the initial state $(1,1,0)$ yields the sequence $\overline{(1,1,0,0)}$, which is equal to the sequence constructed by $\overline{(0,0,1,1)}$ with the first two elements removed.
     We conclude that the polynomial $x^3-x^2-x-1=x^3+x^2+x+1\in\ftwox$ generates the sequences $\overline{(0)}$, $\overline{(1)}$, $\overline{(1,1,0,0)}$, $\overline{(0,0,1,1)}$, and $\overline{(1,0)}$ among others, so it is a characteristic polynomial for all of them.
     On the other hand, we observe that the sequence $\overline{(1,0)}$ is also
     generated by $x^2+1$ with initial state $(1,0)$, therefore both $x^3+x^2+x+1$ and $x^2+1$ are characteristic polynomials of the sequence $\overline{(1,0)}$ over $\ftwo$.
\end{example}

It follows from \Cref{example:SomeLinearSequencesOverFtwo} that a characteristic polynomial of an LFSR sequence is not unique.
However, for any LFSR sequence, a unique polynomial associated to it does exist; this is described in the following result, along with one of its most important properties.

\begin{theorem}
{\cite[Theorem 4.5 and Lemma 4.2]{golomb2005signal}}
{UniquenessOfMinimalPolynomial}
    Let $\bfs$ be an LFSR sequence over $\fq$, and $f\in \fqx$ be a characteristic polynomial of $\bfs$ that has the lowest degree.
    Then, the following statements hold:
    \begin{enumerate}
        \item The polynomial $f$ is unique.
        In other words, if $g\in\fqx$ is a monic polynomial of the lowest degree that generates $\bfs$,  then we have that $g=f$.
        \item A polynomial $h\in \fqx$ generates $\bfs$ if and only if it is divisible by $f$.
    \end{enumerate}
\end{theorem}

\begin{definition}
{Minimal polynomial of an LFSR sequence}
{MinimalPolynomialOfLinearSequence}
    Let $\bfs$ be an LFSR sequence over $\fq$.
    The (unique) monic characteristic polynomial of $\bfs$ that has the lowest degree is the \emph{minimal polynomial of $\bfs$}.
\end{definition}
    \index{LFSR sequence!minimal polynomial of}
    \index{Minimal polynomial!of LFSR sequence}

\begin{example}
{}
{MinimalPolynomial}
    As stated in \Cref{example:SomeLinearSequencesOverFtwo}, both $x^2+1$ and $x^3+x^2+x+1$ are characteristic polynomials of $\overline{(1,0)}$.
    Since $x$ and $x+1$ do not generate $\overline{(1,0)}$, then $x^2+1$ is the characteristic polynomial of $\overline{(1,0)}$ that has the lowest degree, i.e. it is the minimal polynomial of $\overline{(1,0)}$.
    Furthermore, we have that $x^3+x^2+x+1=(x^2+1)(x+1)$, which agrees with the second statement of \Cref{theorem:UniquenessOfMinimalPolynomial}.
\end{example}

\begin{theorem}
{Period of an LFSR sequence \cite[Theorem 4.8]{golomb2005signal}, \cite[Theorems 8.7, 8.29]{lidl1997finite}}
{PeriodOfALinearSequence}
    Let $\bfs$ be an LFSR sequence over $\fq$ of order $t$, and minimal polynomial $f$.
    Then, we have the following results regarding the periodicity of $\bfs$:
    \begin{enumerate}
        \item $\bfs$ is ultimately periodic with least period $r\leq q^t-1$;
        \item $\bfs$ is periodic if and only if $f(0)\neq 0$;
        \item if $f$ is irreducible, then $\bfs$ is periodic and $r| q^t-1$;
        \item $\bfs$ is periodic with $r=q^t-1$ if and only if $f$ is primitive.
    \end{enumerate}
\end{theorem}
    \index{LFSR sequence!period}

\begin{example}
    {Periods of LFSR sequences}
    {PeriodsOfLFSRSequences}
    We examine the periodicity of some LFSR sequences and compare with \Cref{theorem:PeriodOfALinearSequence}.
    \begin{enumerate}
        \item
            The polynomial $f(x)=x^3+x^2+x \in \ftwox$ has a zero root and is the minimal polynomial of the sequence $(0,0,1,1,0,1,1,0,1,1,0,1,1,\dots)$, which is ultimately periodic, but not periodic.  On the other hand, the polynomial $x^2+1 \in \ftwo$ does not have a zero root and generates $\overline{(1,0)}$ which is periodic.
        \item
            The converse of the third statement of \Cref{theorem:PeriodOfALinearSequence} is not true.
            For example the sequence $\overline{(1,0)}$ over $\fthree$ is generated by $x^2+1 \in \fthree[x]$.
            This is also its minimal polynomial, since none of $x,x+1,x+2$ generate $\overline{(1,0)}$.
            Now $x^2+1=(x+1)(x+2)$, so that minimal polynomial is reducible, while the period of $\overline{(1,0)}$ is $2$, which divides $3^2-1$.
        \item
            The polynomial $g(x)=x^2+1 \in \fthreex$ is irreducible but not primitive; it generates the sequence $\overline{(1,0)}$, whose period $2$ divides $3^2-1$.  On the other hand, the polynomial $h(x)=x^2+x+2 \in \fthreex$ is primitive and, with initial state $(1,0)$, it generates the sequence $\overline{(1,0,1,2,2,0,2,1)}$, which has period exactly $3^2-1$.
    \end{enumerate}
\end{example}

\subsubsection{Maximal period sequences}
We now turn our focus on a type of sequence fundamental to our work.

\begin{definition}
{Maximal period sequence}
{MaximalPeriodSequence}
    An LFSR sequence over $\fq$ whose minimal polynomial is primitive is a \emph{maximal (period) sequence}.
    A maximal sequence is often referred to as an \emph{m-sequence} in the literature.
\end{definition}
\index{LFSR sequence!period}
\index{Maximal sequence}
\index{Maximal period sequence|see{Maximal sequence}}
\index{m-sequence|see{Maximal sequence}}

It follows from the last statement of \Cref{theorem:PeriodOfALinearSequence} that a maximal period sequence over $\fq$ of order $t$ is periodic with period $q^t-1$ which, as the name suggests, is the maximum that the period of an LFSR sequence of order $t$ can attain.
An example of a maximal sequence of order $2$ over $\fthree$ is given in part 3 of \Cref{example:PeriodsOfLFSRSequences}.
We present another example of a maximal sequence that we also use later in the thesis.

\begin{example}
{A maximal sequence over $\ffour$}
{AMaximalSequenceOverFFour}
    We recall that in \Cref{example:ConstructionFFourAndFSixteen} the finite field $\ffour$ is constructed using a root $\a$ of $x^2+x+1 \in \ftwox$, as shown in \Cref{equation:ExampleConstructionOfFFour}.
    Then, the polynomial $g(x)=x^2+x+\a$ is shown to be primitive and, as such, it is the minimal polynomial of a maximal sequence $\bfs=(s_i)_{i\geq 0}$, that has period $4^2-1$, and is generated by the linear recursion
    \[
        s_{i+2} =-s_{i+1}-\a s_i =s_{i+1}+\a s_i, \quad i \geq 0.
    \]
    With initial state $(1,0)$, this yields the maximal sequence over $\ffour$ of period $4^2-1$ given by
    \[
        \overline{(1,0,\a,\a,1,\a,0,\a+1,\a+1,\a,\a+1,0,1,1,\a+1)}.
    \]
\end{example}
The following theorem provides a very useful way of representing maximal sequences.

\begin{theorem}
{\cite[Corollary 4.6]{golomb2005signal},\cite[Theorem 8.24]{lidl1997finite}}
{TraceRepresentationOfMSequence}
    An LFSR sequence $\bfs=(s_i)_{i\geq 0}$ over $\fq$ is a maximal sequence with period $q^t-1$ if and only if the elements of $\bfs$ can be represented by
    \begin{equation}
        \label{equation:TraceRepresentation}
        s_{i} = \Tt(\theta \a^i), \quad i \geq 0,
    \end{equation}
    where $\a, \theta \in \fqt$, with $\a$ being primitive and $\theta$ nonzero.
    For each maximal sequence $\bfs$, the $\theta$ in \Cref{equation:TraceRepresentation} is unique, and $\a$ is a root of the minimal polynomial of $\bfs$ in $\fqx$.
\end{theorem}

\begin{definition}
{The trace representation of a maximal sequence}
{TraceRepresentationOfMSequence}
    For a maximal sequence $\bfs=(s_i)_{i\geq 0}$ of order $t$ over $\fq$, the representation of $\bfs$ given by \Cref{equation:TraceRepresentation} is the \emph{trace representation of $\bfs$}.
\end{definition}
\index{Maximal sequence!Trace representation of}
\index{Trace representation}

The trace representation also applies to LFSR sequences that are not maximal, however we do not discuss this since we only use maximal sequences in this thesis.
For an in-depth presentation of the trace representation in general, we refer the reader to \cite[Section 6.2]{golomb2005signal}.

\begin{example}
{The trace representation of a maximal sequence over $\fthree$}
{TraceRepresentation}
    The recursion corresponding to the primitive polynomial $x^2+x+2 \in \fthreex$ is
    \[s_{i+2}=2s_{i+1}+s_i,\quad i \geq 0.\]
    With initial state $(0,1)$, this yields the maximal sequence
    \[\bfs=(s_i)_{i\geq 0}=\overline{(0,1,2,2,0,2,1,1)}.\]
    Let $\a$ be a root of $x^2+x+2$, so that $\a^2=2\a+1$.
    Denoting $\T_{3^2/3}=\T$, we compute
    \begin{align*}
        \T(\a)
        = \a+\a^3
        =\a+\a\a^2
        =\a+\a(2\a+1)
        =2\a+2\a^2
        =2\a+2(2\a+1)
        =6\a+2=2.
    \end{align*}
    Similarly, we have that
    \[ \T(\a^2)=0, \T(\a^3)=2, \T(\a^4)=1, \T(\a^5)=1, \T(\a^6)=0, \T(\a^7)=1,\] 
    and also $\T(\a^0)=\T(1)=2$.
    By \Cref{lemma:PowerQEqualsElementMeansBelongingToFQ}, we have that $\a^9=\a$ and hence $\T(a^{i+9})=\T(\a^i)$, for all $i\geq 0$.
    Therefore, the sequence over $\fthree$ given by $(\T(\a^i))_{i\geq 0}$ is periodic with period $9$, and it is equal to $\overline{(2,2,0,2,1,1,0,1)}$.
    We observe that this is a cyclic shift of $\bfs$ with $s_{i}=\T(\a^{i+6})$, $i\geq 0$.
    Thus, setting $\theta=\a^6$, we have that
    \[ \T(\theta \a^i)=\T(\a^6\a^i)=\T(\a^{i+6})=s_i,\]
    which gives the trace representation for $\bfs$.
\end{example}

In the following example we give the trace representation of the maximal sequence over $\ffour$ discussed in \Cref{example:AMaximalSequenceOverFFour}.  This is obtained by working similarly to \Cref{example:TraceRepresentation}; we omit the details and we just state the result for future reference.

\begin{example}
{The trace representation of a maximal sequence over $\ffour$}
{TraceRepresentationOfSequenceOverFFour}
    We assume the notation of \Cref{example:AMaximalSequenceOverFFour} and the results established thereof.
    Setting $\theta=\beta^{13}=1+\a\beta$, we have that the maximal sequence $\bfs=(s_i)$, $i\geq 0$ over $\ffour$ defined by
    \[
        \bfs=\overline{( 1,0,\a,\a,1,\a,0,\a+1,\a+1,\a,\a+1,0,1,1,\a+1),}
    \]
    has trace representation given by
    $s_i=\T_{4^2/4}(\theta\a^i)$, $i\geq 0$.
\end{example}

Maximal sequences have a rich algebraic structure and a variety of statistical properties that make them ideal for use in applications such as radar, sonar, and wireless communications \cite{golomb1982shift,golomb2005signal}; and cryptography \cite[Chapter 6]{menezes1996handbook}.
One of their most fundamental properties is the following.

\begin{proposition}
{Balance property of maximal sequences \cite[Property 5.3]{golomb2005signal}}
{BalancePropertyMSequences}
    Let $\bfs=(s_i)_{i\geq 0}$ be a maximal sequence over $\fq$ of order $t$ and, for every $a\in \fq$, let
    \[
        N_{\bfs}(a)
        = \left| \left\{ i| i \in [0,q^t-1], s_i=a \right\} \right|.
    \]
    Then, we have that
    \[
        N_{\bfs}(a)=
    \begin{cases}
        q^{t-1}, & \text{ if } a \neq 0;\\
        q^{t-1}-1, & \text{ otherwise }.
    \end{cases}
\]
\end{proposition}
\index{Maximal sequence!Balance property of}
The balanced distribution of the elements of a maximal sequence goes beyond
\Cref{proposition:BalancePropertyMSequences}; for example, the following well-known generalization shows a similar property for pairs of elements.
We denote 
\[\w{t}=\frac{q^t-1}{q-1}, \label{equation:DefinitionW} \] 
a notation that we use extensively from now on.

\begin{proposition}
{Two-tuple balance property of maximal sequences \cite[Corollary 4.6]{golomb2005signal}}
{TwoTupleBalancePropertyOfMSequences}
    Let $\bfs=(s_i)_{i\geq 0}$ be a maximal sequence over $\fq$ of order $t$.
    Then, for any integer $\tau$ and pair $(a,b)$ of elements $a,b\in\fq$, we denote
    \[
        N_{\bfs,\tau}(a,b)
        = \left| \left\{ i \mid i\in [0,q^t-1], (s_i, s_{i+\tau})=(a,b) \right\} \right|
    \]
    Then, we have that
    \[
        N_{\bfs,\tau}(a,b)=
        \begin{cases}
            q^{t-2}     & \text{ if } \tau\not\equiv 0\Mod{\w{t}}\text{ and } (a,b)\neq (0,0);\\
            q^{t-2}-1   & \text{ if } \tau\not\equiv 0\Mod{\w{t}}\text{ and } (a,b) = (0,0);\\
            q^{t-1}     & \text{ if } \tau\equiv 0\Mod{\w{t}}\text{ and } b=ca \text{ for some } c \in \fqstar;\\
            0           & \text{ otherwise}.
        \end{cases}
    \]
\end{proposition}
\index{Maximal sequence!Two-tuple balance property of}

It follows from \Cref{proposition:TwoTupleBalancePropertyOfMSequences} that elements in a maximal sequence whose indexes are multiples of $\w{t}$ are constant multiples of each other.
This property is described in more detail in \Cref{proposition:ProjectivePropertyOfMaximalSequences}.
First, we need a lemma.

\begin{lemma}
    {Characterization of constant multiples in $\fqt$}
    {CharacterizationOfConstantMultiplesInFQM}
    Let $q$ be a prime power, $t$ be a positive integer, and $\a$ be a primitive element of $\fqt$.
    Then, $\a^{\w{t}}$ is a primitive element of $\fq$, and for any $i,j \in [0,q^t-2]$, we have that $\a^j=c\a^i$ for some $c \in \fqstar$ if and only if $ j \equiv i\Mod {\w{t}}$.
\end{lemma}
\index{Primitive element}
\begin{proof}
    First, we show that $\a^{\w{t}}$ is a primitive element of $\fq$.
    Indeed, we have that
    \begin{align}
        \label{equation:OrderOfAlphaToTheW}
        \ord(\a^{\w{t}})
        =\frac{\ord(\a)}{\gcd\left(\ord(\a), \w{t}\right)}
        =\frac{q^t-1}{\gcd\left(q^t-1, \frac{q^t-1}{q-1}\right)}
        =\frac{q^t-1}{\frac{q^t-1}{q-1}}
        =q-1,
    \end{align}
    which proves our claim;
    it follows that $c\in \fqstar$ if and only if $c=\a^{l\w{t}}$ for some integer $l$.

    If $j\equiv i \Mod{\w{t}}$, then $j=i+l\w{t}$ for some $l$, thus 
    $\a^j=\a^{i}\a^{l\w{t}}=c\a^i$, for $c=\a^{l\w{t}} \in \fqstar$.
    Conversely, we suppose that $\a^j=c\a^i$, for some $c\in \fqstar$.
    Since $\a^{\w{t}}$ is a primitive element of $\fqstar$, $c=(\a^{\w{t}})^l=\a^{l \w{t}}$, for some $l$.
    Therefore, 
    $\a^j=c\a^i=\a^{l \w{t}}\a^i=\a^{i+l \w{t}},$
    which implies that $\a^{i-j+l \w{t}}=1$.
    Since $\a$ is a primitive element of $\fqtstar$, its order is $q^t-1$, so the last equation implies that $q^t-1$ divides $i-j +l \w{t}$.
    Furthermore, since $\w{t}$ divides $q^t-1$, we also have
    that $\w{t}$ divides $i-j +l \w{t}$, and therefore $i \equiv j \Mod{\w{t}}$.
\end{proof}

\begin{proposition}
{Projective property of maximal sequences \cite[Section 5.3.2]{golomb2005signal}}
{ProjectivePropertyOfMaximalSequences}
    Let $\bfs=(s_i)_{i\geq 0}$ be a maximal sequence of order $t$ over $\fq$, and let $\a\in \fqt$ be a root of its minimal polynomial.
    Denoting $w=\w{t}$ and
    \[
       P_i
       = \left( s_{o+iw}, s_{1+iw}, \dots, s_{w-1+iw} \right), \quad i\in [0,q-2],
    \]
    we have that $P_i=\a^{iw}P_0$, for all $i\in [0,q-2]$.
\end{proposition}
    \index{Maximal sequence!projective property of}

In other words, \Cref{proposition:ProjectivePropertyOfMaximalSequences} states that if we arrange the sequence in a $(q-1)\times w$ array, and denoting $c_i=\a^{wi}$ so that
$\fqstar = \{ c_i\mid i\in [0,q-2] \}$,
we have that

\begin{align*}
    \left[
    \begin{array}{llll}
        s_0&s_1&\cdots&s_{w-1}\\
        s_{w}&s_{w+1}&\cdots&s_{2w-1}\\
        s_{2w}&s_{2w+1}&\cdots&s_{3w-1}\\
        \multicolumn{1}{l}{\vdots}&
        \multicolumn{1}{l}{\vdots}&
         \multicolumn{1}{l}{\ddots}&
         \multicolumn{1}{l}{\vdots}
         \\
         s_{(q-2)w}&s_{(q-2)w+1}&\cdots&s_{(q-2)w-1}
    \end{array}
    \right]
    &=
    \left[
    \begin{array}{rrrr}
        s_0&s_1&\cdots&s_{w-1}\\
        \a^{w} s_0&\a^{w} s_1&\cdots&\a^{w} s_{w-1}\\
        \a^{2w} s_0&\a^{2w} s_1&\cdots&\a^{2w} s_{w-1}\\
        \multicolumn{1}{r}{\vdots}&
        \multicolumn{1}{r}{\vdots}&
         \multicolumn{1}{c}{\ddots}&
         \multicolumn{1}{r}{\vdots}
         \\
         \a^{(q-2)w} s_0&\a^{(q-2)w} s_1&\cdots&\a^{(q-2)w} s_{w-1}
    \end{array}
    \right]
    \\
    &=
    \left[
    \begin{array}{r}
        P_0\\
        \a^w P_0\\
        \a^{2w} P_0\\
        \multicolumn{1}{r}{\vdots}\\
        \a^{w(q-2)}P_0
    \end{array}
    \right]
    =
    \left[
    \begin{array}{r}
        P_0\\
        c_1 P_0\\
        c_2 P_0\\
        \multicolumn{1}{r}{\vdots}\\
        c_{q-2}P_0
    \end{array}
    \right].
\end{align*}

\begin{example}
{}
{PropertiesOfMaximalSequences}
    We assume the notation used in \Cref{proposition:BalancePropertyMSequences,proposition:TwoTupleBalancePropertyOfMSequences,proposition:ProjectivePropertyOfMaximalSequences}, and consider the sequence
    \[
        \bfs=(s_i)_{i\geq 0}=\overline{(0,1,2,2,0,2,1,1)}
    \]
    which is the maximal sequence discussed in \Cref{example:TraceRepresentation}.
    We observe that $N_{\bfs}(1)=N_{\bfs}(2)=3^{2-1}$, and $N_{\bfs}(0)=3^{2-1}-1$, as per \Cref{proposition:BalancePropertyMSequences}.
    To demonstrate \Cref{proposition:TwoTupleBalancePropertyOfMSequences}, we construct the table below.
    \[
        \renewcommand{\arraystretch}{\genarraystretch}
        \begin{array}[]{c*{8}{!{\color{\tableheadcoloralt}\vrule}c}}
        \rowcolor{\tableheadcoloralt}
        i&0&1&2&3&4&5&6&7\\
        s_i& 0&1&2&2&0&2&1&1\\
        s_{i+1}& 1&2&2&0&2&1&1&0\\
        s_{i+4}& 0&2&1&1&0&1&2&2.
        \end{array}
    \]
    Using the first two rows of the table to count the pairs of the form $(s_i,s_{i+1})$, we see that $N_{\bfs,1}(a,b)=3^{2-2}=1$ for all nonzero $(a,b)$, and $N_{\bfs, 1}(0,0)=3^{2-2}-1=0$, as expected from \Cref{proposition:TwoTupleBalancePropertyOfMSequences} with $\tau=1$, which is not divisible by $[2]_3=4$.
    On the other hand, using the first and third rows of the table we count that $N_{\bfs,4}(1,2)=3=3^{2-1}$ and $N_{\bfs,4}(0,2)=0$, as expected from \Cref{proposition:TwoTupleBalancePropertyOfMSequences} with $\tau=4$, which is divisible by $[2]_3=4$.
    Finally, we arrange $\bfs$ in a $(q-1) \times \w{t} = 2\times 4$ array below.
    \[
        \left[
        \begin{array}[]{cccc}
        0&1&2&2\\
        0&2&1&1
        \end{array}
        \right]
        =
        \left[
        \begin{array}[]{cccc}
        s_0&s_1&s_2&s_3\\
        s_4&s_5&s_6&s_7
        \end{array}
        \right]
        =
        \left[
        \begin{array}[]{c}
            P_0\\P_1
        \end{array}
        \right].
    \]
    This shows that $P_0=(0,1,2,2)$, and $P_1=(0,2,1,1)=2P_0$.
    We note that a root $\a$ of the minimal polynomial $x^2+x+2 \in \fthreex$ satisfies $\a^{[2]_3}=\a^4=2$, hence $P_1=\a^{[2]_3}P_0$, as per \Cref{proposition:ProjectivePropertyOfMaximalSequences}.
\end{example}

\begin{example}
{}
{PropertiesOfMaximalSequencesAnotherExample}
    We assume the notation used in \Cref{proposition:BalancePropertyMSequences,proposition:TwoTupleBalancePropertyOfMSequences,proposition:ProjectivePropertyOfMaximalSequences}, and define the sequence
    \[\bfs=\overline{( 1,0,\a,\a,1,\a,0,\a+1,\a+1,\a,\a+1,0,1,1,\a+1)}\]
    which is the maximal sequence discussed in \Cref{example:TraceRepresentationOfSequenceOverFFour}.
    We observe that
    $N_{\bfs}(1)=N_{\bfs}(\a)= N_{\bfs}(\a+1)=4=4^{2-1}$,
    and
    $N_{\bfs}(0)=3=4^{2-1}-1$,
    as per \Cref{proposition:BalancePropertyMSequences}.
    To demonstrate \Cref{proposition:TwoTupleBalancePropertyOfMSequences}, we construct the table below.
    \[
        \begin{array}[]{c*{8}{!{\color{\tableheadcoloralt}\vrule}c}}
            \rowcolor{\tableheadcoloralt}
            i&0&1&2&3&4&5&6&7\\
            s_i&1&0&\a&\a&1&\a&0\\
            s_{i+1}&0&\a&\a&1&\a&0&\a+1&\a+1\\
            s_{i+10}&\a+1&0&1&1&\a+1&1&0&\a\\
            \rowcolor{\tableheadcoloralt}
            i&8&9&10&11&12&13&14&\\
            s_i&\a+1&\a&\a+1&0&1&1&\a+1&\\
            s_{i+1}&\a&\a+1&0&1&1&\a+1&1&\\
            s_{i+10}&\a&1&\a&0&\a+1&\a+1&\a&\\
        \end{array}
    \]
    Using the first two rows of the table to count the pairs of the form $(s_i,s_{i+1})$, we see that $N_{\bfs,1}(a,b)=4^{2-2}=1$ for all nonzero $(a,b)$, and $N_{\bfs, 1}(0,0)=4^{2-2}-1=0$; this is as expected from \Cref{proposition:TwoTupleBalancePropertyOfMSequences} with $\tau=1$, which is not divisible by $[2]_4=5$.
    On the other hand, from the first and third rows of the table we count that $N_{\bfs,10}(1,\a+1)=4=4^{2-1}$ and $N_{\bfs,4}(1,\a)=0$, as expected from \Cref{proposition:TwoTupleBalancePropertyOfMSequences} with $\tau=10$, which is divisible by $[2]_4=5$.
    Now, we see in \Cref{table:PowersOfBetaInExample} that $\beta^{[2]_4}=\beta^5=\a$ and $\beta^{[2]_4}=\beta^{10}=\a+1$.
    We arrange $\bfs$ in a $(q-1) \times \w{t} = 3\times 5$ array below.
    \[
        \left[
            \begin{array}[]{ccccc}
            s_0&s_1&s_2&s_3&s_4\\
            s_5&s_6&s_7&s_8&s_9\\
            s_{10}&s_{11}&s_{12}&s_{13}&s_{14}
            \end{array}
        \right]
        =
        \left[
            \begin{array}[]{l*{4}{!{\color{beach!90!black}\vrule}l}}
            1&0&\a&\a&1\\
            \a&0&\a+1&\a+1&\a\\
            \a+1&0&1&1&\a+1
            \end{array}
        \right]
        =
        \left[
            \begin{array}{l}
                P_0\\P_1\\P_2
            \end{array}
        \right].
    \]
    Indeed, $P_1=\a P_0=\beta^{[4]_2} P_0$ and $P_2=(\a+1)P_0=\beta^{2[2]_2} P_0$, as per \Cref{proposition:ProjectivePropertyOfMaximalSequences}.
\end{example}

\subsection{Finite geometry and combinatorial designs}
\label{section:FiniteGeometry}

We continue with some aspects of finite geometry that we need in the next chapters.

\begin{definition}
{Finite projective space}
{FiniteProjectiveSpace}
    An \emph{incidence structure} or \emph{finite geometry} is a pair of finite sets of \emph{points} and \emph{lines}, along with a reflexive and symmetric relation on the set of points, called an \emph{incidence relation}.
    A \emph{finite projective space} $P$ is a finite incidence structure that satisfies all of the following statements.
    \begin{enumerate}
        \item Any two distinct points of $P$ belong to exactly one line.
        \item Let $A,B,C,D$ be four distinct points of which no three are collinear.
            If the lines $AB$ and $CD$ intersect each other, then the lines
            $AD$ and $BC$ also intersect each other.
        \item Any line has at least three points.
    \end{enumerate}
\end{definition}

\begin{definition}
{The $(t-1)$-dimensional projective space $PG(t-1,q)$}
{TheFiniteProjectiveSpacePGDQ}
    Let $t$ be a positive integer, $q$ be a prime power, and $V$ be a vector space of dimension $t$ over $\fq$.
    The finite geometry $PG(t-1,q)$ that has as its points and lines the $1$-dimensional and $2$-dimensional vector subspaces of $V$, respectively, is \emph{the $(t-1)$-dimensional projective space over $\fq$}.
\end{definition}
\index{Projective space}
\index{Projective space!dimension}
\index{Projective space!point}
\index{Projective space!line}

For $v \in V\setminus \{ 0 \}$, let
\begin{equation}
    \label{equation:RepresentationOfPointsInPGAsClassesOfVectors}
\index{Projective space!point}
    [v]= \{ cv \mid c \in \fqstar \},
\end{equation}
where $cv$ is the scalar multiplication of $v$ by $c$.
We use the symbols $[v]$ to represent the points of $PG(t-1,q)$ and we write
$ PG(t-1,q)=\{ [v] \mid v \in V\setminus \{ 0 \} \}.$
For any $u\in [v]$, $u$ is a \emph{representative vector for the point $[v]$}.
For a subspace $S$ of $V$ with dimension $k+1$, the set of points
$\{ [v]\mid v \in S\setminus\{ 0 \} \}$
is a \emph{$k$-flat} of $PG(t-1,q)$.
A $k$-flat with $k=0,1,2,t-2$ is a \emph{point, line, plane, hyperplane} of $PG(t-1,q)$, respectively.
The flats have properties which give a fundamental connection between finite projective spaces and the designs in the next definition.
\index{Projective space!point}
\index{Projective space!line}
\index{Projective space!plane}
\index{Projective space!flat}
\index{Projective space!hyperplane}

\begin{definition}
{Balanced incomplete block design (BIBD)}
{BIBD}
    For integers $v,s$ with $2<s<v$, a \emph{$(v,s,\lambda)$-BIBD (balanced incomplete block design)} is a pair $(V,\Bcal)$ where $V$ is a set with $|V|=v$ and $\Bcal$ is a multiset consisting of $s$-sets of $V$, such that for every $2$-set $T\subset V$, there are exactly $\lambda$ sets $B\in \Bcal$ such that $T\subset B$.
    The sets $V$ and $\Bcal$ are the \emph{points} and \emph{blocks} of the BIBD, whereas $\lambda$ is its \emph{index}.
\end{definition}
\index{Balance incomplete block design|see{BIBD}}
\index{BIBD}
\index{BIBD!points of}
\index{BIBD!blocks of}
For positive integers $a,b$ and prime power $q$, the number
\[
    {a \brack b}_q
    =
    \begin{cases}
        \frac{(q^a-1)(q^{a-1}-1) \cdots (q^{a-b+1})} {(q^b-1)(q^{b-1}-1)\cdots (q-1)}, \text{ if } & a\geq b;\\
        0 & \text{ otherwise.}
    \end{cases}
\]
is a \emph{Gaussian binomial coefficient}.
We note that this is the number of subspaces of dimension $b$ in a vector space of dimension $a$ over $\fq$.
\begin{proposition}
{Flats as blocks of a BIBD \cite[Proposition 2.35]{colbourn2006handbook}}
{FlatsAsBlocksOfABIBD}
    Let $d$ be a positive integer and $q$ be a prime power.
    For $k\in [0,d-2]$, the points of $PG(d-1,q)$ with the $k$-flats as blocks form a $(v,s,\lambda)$-BIBD with parameters
    \[  v =\w{d+1}, \, s=\w{k+1},\, \lambda= {d \brack i}_q. \]
\end{proposition}
\index{Projective space!flat}
\index{BIBD}

Then next lemma gives a translation of the notion of linear independence in a finite vector space over $\fq$ into the language of flats in the corresponding projective space.
\begin{lemma}
{}
{IndependenceAndGeneralPosition} 
Let $V$ be a $t$-dimensional vector space over $\fq$ and $v_0, \dots, v_{n-1} \in V$, where $v\leq t$.
Then, $v_0, \dots, v_{n-1}$ are linearly independent if and only if there is no $(n-2)$-flat in $PG(t-1,q)$ that contains $[v_0], \dots, [v_{n-1}]$.
\end{lemma}

\begin{proof}
    From the definition of flats, $[v_0], \dots, [v_{n-1}]$ belong in the same $(n-2)$-flat if and only if $v-0, \dots, v_{n-1}$ belong in the same $(n-1)$-dimensional subspace of $V$, which is true if and only if they are linearly dependent.
\end{proof}

In this thesis we use exclusively $\fqt$ as our $t$-dimensional vector space $V$ when constructing $PG(t-1,q)$.
To demonstrate this construction, we need the next lemma.

\begin{proposition}
{Construction of $PG(t-1,q)$ using $\fqt$}
{ConstructionOfPGDQFromFF}
    Let $q$ be a prime power, $t$ be a positive integer, and $\a$ be a primitive element of $\fqt$.
    Then, the set
    \[\{ [\a^i] \mid i \in [0,\w{t}-1] \}\]
    contains exactly all the $\w{t}$ points of $PG(t-1,q)$.
\end{proposition}
\index{Projective space}
\begin{proof}
    First, there are $q^t-1$ nonzero vectors in $\fqt$, considered as a vector space, and each has $q-1$ nonzero scalar multiples.
    Therefore, there are precisely $(q^t-1)/(q-1)= \w{t}$ distinct $1$-dimensional vector subspaces of $\fqt$, which means that $PG(t-1,q)$ has exactly $\w{t}$ points.
    It suffices to show that the elements $[\a^i]$, $i \in [0,\w{t}-1]$ are distinct.
    We assume by means of contradiction that this is not the case, so that there exist $i, j \in [0,\w{t}-1]$, $i\neq j$, such that $[\a^i]=[\a^j]$.
    By \Cref{equation:RepresentationOfPointsInPGAsClassesOfVectors}, this means that $\a^j=c\a^i$ for some $c\in \fqstar$ which, by \Cref{lemma:CharacterizationOfConstantMultiplesInFQM} is equivalent to $j\equiv i \Mod{\w{t}}$.
    Hence, we have that $\w{t}\mid j-i$, which implies that $\w{t}\leq j-i$ or $j-i=0$; this contradicts the definition of $i$ and $j$.
\end{proof}

\begin{example}
{The projective space $PG(2,3)$}
{ProjectiveSpacePG23}
    Let  $\a$ be a root of the primitive polynomial $x^3+2x+1$ over $\fthree$, so that $\a$ is a primitive element in $\f_{3^3}$.
    From \Cref{proposition:ConstructionOfPGDQFromFF} there are exactly $[3]_3=(3^3-1)/(3-1)=13$ points in $PG(2,3)$, given by $[\a^i], i \in [0,12]$.
    In the following table, we list these powers of $\a$ in the form $c_2\a^2+c_1\a+c_0$, as per \Cref{equation:FundamentalRepresentationOfFiniteField}, as well as the coordinate vector $(c_2,c_1,c_0)$ to emphasize how we view $\f_{3^3}$ as a $3$-dimensional vector space over $\fq$.
    \[
    \begin{array}{c*{7}{!{\color{\tableheadcoloralt}\vrule}c}}
    \rowcolor{\tableheadcoloralt}
    i&0& 1& 2& 3& 4& 5& 6\\
    \a^i&1& \a& \a^2& \a + 2& \a^2 + 2\a& 2\a^2 + \a + 2& \a^2 + \a + 1\\
    \text{Coord.}&(0,0,1)& (0,1,0)& (1,0,0)& (0,1,2)& (1,2,0)& (2,1,2)& (1,1,1)\\
    \rowcolor{\tableheadcoloralt}
    i&7& 8& 9& 10& 11& 12 &\\
    \a^i&\a^2 + 2\a + 2& 2\a^2 + 2& \a + 1& \a^2 + \a& \a^2 + \a + 2& \a^2 + 2 &\\
    \text{Coord.}&(1,2,2)& (2,0,2)& (0,1,1)& (1,1,0)& (1,1,2)& (1,0,2)& 
    \end{array}
\]

    Using the table, in \Cref{figure:ProjectiveSpacePG23} we provide a visual representation of the points and lines of $PG(2,3)$ that also demonstrates the incidence relations between the points.
    For example, the $1$-dimensional subspace of $\fthree^3$ spanned by $(0,0,1)$ corresponds to the point labeled $c$ in the figure.
    Similarly, the points labeled $b$ to $m$, correspond to the remaining $1$-dimensional subspaces of $\fthree^3$, as shown on the left of the figure.
    Next, we consider the $2$-dimensional subspace of $\fthree^3$, given by
    \[S=\{ (x,0,y) \mid x,y \in \fthree \}.\]
    This corresponds to the line ($1$-flat) of $PG(2,3)$ whose points correspond to the $1$-dimensional subspaces contained in $S$.
    There are four such subspaces, and they are spanned by the vectors
    $(0,0,1),(1,0,0),(1,0,2)$ and $(2,0,2)$, respectively.
    These are representative vectors for the points labeled $a,b,c$ and $k$, respectively, which are indeed shown to be collinear in the diagram.
\end{example}
\begin{figure}[t]
    \centering
    \begin{tikzpicture}[scale=1.8]
    \foreach \x in {0,...,2}
    \foreach \y in {0,...,2} {
          \pgfmathtruncatemacro\nk{\y*3+\x+1}
            \coordinate (\nk) at (\x,2-\y);
    }
    \coordinate (11) at ({3+sqrt(2)},1);
    \coordinate (13) at (1,{-1-sqrt(2)});
    \coordinate (10) at ($(2,2)+(45:2)$);
    \coordinate (12) at ($(2,0)+(-45:2)$);
    \coordinate (14) at (-3,0);
    \begin{knot}
    \strand (10) arc[radius={2+sqrt(2)},start angle=45,end angle=-90];
    \strand (1) -- (2) -- (3) to[out=0,in=135] (11);
    \strand (4) -- (5) -- (6) -- (11);
    \strand (7) -- (8) -- (9) to[out=0,in=-135] (11);
    \strand (1) -- (4) -- (7) to[out=-90,in=135] (13);
    \strand (2) -- (5) -- (8) -- (13);
    \strand (3) -- (6) -- (9) to[out=-90,in=45] (13);
    \strand (7) -- (5) -- (3) -- (10);
    \strand (9) .. controls +(-2.5,-2.5) and +(-2.5,-2.5) .. (4) -- (2) to[out=45,in=-180] (10);
    \strand (1) .. controls +(-2.5,-2.5) and +(-2.5,-2.5) .. (8) -- (6) to[out=45,in=-90] (10);
    \strand (1) -- (5) -- (9) -- (12);
    \strand (7) .. controls +(-2.5,2.5) and +(-2.5,2.5) .. (2) -- (6) to[out=-45,in=90] (12);
    \strand (3) .. controls +(-2.5,2.5) and +(-2.5,2.5) .. (4) -- (8) to[out=-45,in=180] (12);
    \end{knot}
        \tikzstyle{point}=[ball color=white, circle, draw=black,minimum width=24pt]
        \node[point] at (1)  {$a$};
        \node[point] at (2)  {$b$};
        \node[point] at (3)  {$c$};
        \node[point] at (4)  {$d$};
        \node[point] at (5)  {$e$};
        \node[point] at (6)  {$f$};
        \node[point] at (7)  {$g$};
        \node[point] at (8)  {$h$};
        \node[point] at (9)  {$i$};
        \node[point] at (10) {$j$};
        \node[point] at (11) {$k$};
        \node[point] at (12) {$l$};
        \node[point] at (13) {$m$};
        \node at (14) {
$
\begin{aligned}
{} 
[(0,0,1)]&=[1]	     =c\\
[(0,1,0)]&=[\a]      =m\\
[(1,0,0)]&=[\a^2]    =k\\
[(0,1,2)]&=[\a^{3}]  =i\\
[(1,2,0)]&=[\a^4]    =l\\
[(2,1,2)]&=[\a^{5}]  =h\\
[(1,1,1)]&=[\a^6]    =e\\
[(1,2,2)]&=[\a^7]    =d\\
[(2,0,2)]&=[\a^{8}]  =b\\
[(0,1,1)]&=[\a^9]    =f\\
[(1,1,0)]&=[\a^{10}] =j\\
[(1,1,2)]&=[\a^{11}] =g\\
[(1,0,2)]&=[\a^{12}] =a
\end{aligned}
$
};
\end{tikzpicture}
    \caption[Visual representation of the points and lines of $PG(2,3)$]{A visual representation of the points and lines of $PG(2,3)$, constructed using either $\fthree^3$ or $\f_{3^3}$ as discussed in \Cref{example:ProjectiveSpacePG23}.}
    \label{figure:ProjectiveSpacePG23}
\end{figure}
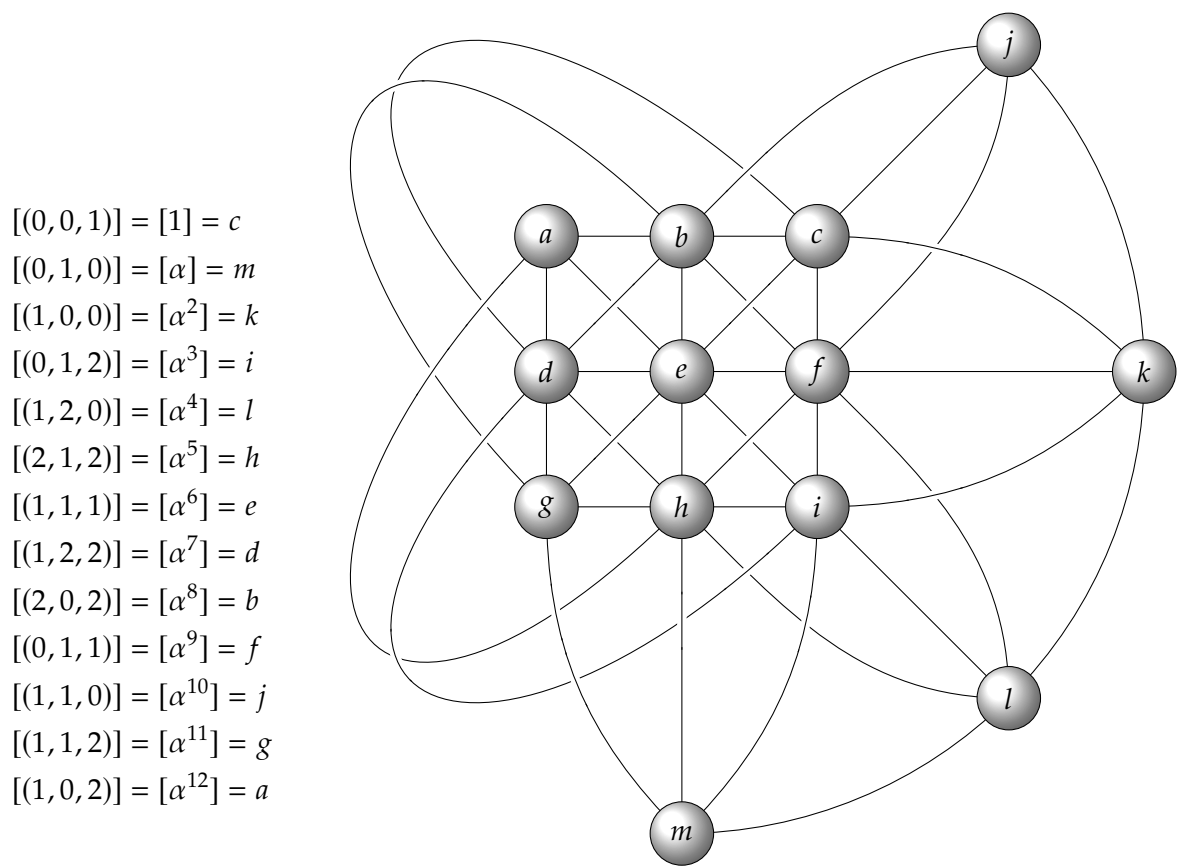

\begin{example}
{}
{GenPositionAndLinearIndepInPG23}
The points $a,b$ and $e$ in \Cref{figure:ProjectiveSpacePG23} are collinear, which means that they belong in the same $(3-2)$-flat.
    The vectors $(1,0,2), (2,0,2)$ and $(1,1,1)$, which are respective representative vectors for these points, are indeed linearly independent as per \Cref{lemma:IndependenceAndGeneralPosition}.
\end{example}
We close this section with a definition of a structure that is useful in the next chapter.
\begin{definition}
{Arcs, tracks and caps}
{NTrack}
    Let $S$ be a set of $k$ points in $PG(t-1,q)$. 
    \begin{itemize}
        \item $S$ is a \emph{$k$-arc} if there is no $(t-2)$-flat (hyperplane) in $PG(t-1,q)$ that contains $t$ points from $S$.
        \item $S$ is a \emph{$k$-track} if it is not a $k$-arc and there is no $(t-3)$-flat in $PG(t-1,q)$ that contains $t-1$ points from $S$. 
        \item $S$ is a \emph{$k$-cap} if no $3$ points in $S$ are collinear.
    \end{itemize}
\end{definition}
\index{Arc}
\index{Track}
\index{Cap}

\newpage
\begin{remark}
    {Arcs, tracks, caps and linear independence}
    {ArcsTracksCapsAndLI}
    Let $S=\{[v_0], \dots, [v_{k-1}]\}$ be a set of points from $PG(t-1,q)$. 
    The following are a straightforward implication of \Cref{lemma:IndependenceAndGeneralPosition} and \Cref{definition:NTrack}.
    \begin{itemize}
        \item $S$ is a $k$-arc if and only if every $t$ vectors among $v_0, \dots, v_{k-1}$ are linearly independent over $\fq$.
        \item $S$ is a $k$-track if and only if every $t-1$ vectors among $v_0, \dots, v_{k-1}$ are linearly independent over $\fq$ but there are some $t$ vectors that are linearly dependent.
        \item $S$ is a $k$-cap if and only if every $3$ vectors among $v_0, \dots, v_{k-1}$ are linearly independent over $\fq$.
    \end{itemize}
\end{remark}

\subsection{Linear codes}
In this section we give some background on linear codes that is essential for our thesis.

\begin{definition}
{Hamming distance and Hamming weight}
{HammingDistance}
    Let $\mathbf{x}=(x_0, \dots, x_{n-1}), \mathbf{y}=(y_0, \dots, y_{n-1}) \in \fq^n$.
    The \emph{Hamming distance between $\mathbf{x}$ and $\mathbf{y}$} is the cardinality of the set $\{ i \mid x_i \neq y_i \}$.
    The \emph{(Hamming) weight} of $\mathbf{x}$, denoted by $w(\mathbf{x})$, is the number of its nonzero coordinates.
\end{definition}
\index{Hamming!distance}
\index{Hamming!weight}

\begin{definition}
{Linear codes over a finite field and related notions}
{LinearCodeOverFF}
    Let $q$ be a prime power.
    A \emph{$q$-ary linear code} $C$ of \emph{length} $n$, \emph{dimension} $k$, and \emph{minimum distance} $d$, referred to as an $[n,k,d]_q$ code, is a subspace of $\fq^n$ with dimension $k$, such that the minimum Hamming distance between distinct elements is $d$.
    The elements of $C$ are referred to as the \emph{words} or \emph{codewords} of $C$.
    A $k\times n$ matrix whose rows are a basis for $C$ is a \emph{generating matrix} for $C$.
    The \emph{dual code} of $C$, denoted $C^{\perp}$, is the orthogonal complement of $C$ in $\fq^n$, i.e. it is the code of length $n$ and dimension $n-k$ given by
    \[
        C^{\perp}
        = \left\{ x \in \fq^{n} \mid x\cdot y =0 \text{ for all } y \in C \right\}.
    \]
    A $(n-k)\times n$ matrix that is a generator matrix for $C^{\perp}$ is a \emph{parity check matrix} for $C$.
\end{definition}
\index{Linear code}
\index{Linear code!length}
\index{Linear code!dimension}
\index{Linear code!minimum distance}
\index{Linear code!generating matrix}
\index{Linear code!dual of}
\index{Matrix!generating}
\index{Matrix!parity check}
\index{Dual code}

From the definition of the Hamming distance and weight it follows that for $\mathbf{x}, \mathbf{y} \in \fq^n$, we have
$d(\mathbf{x},\mathbf{y})= w(\mathbf{x}-\mathbf{y})$.
An immediate consequence is the following lemma.
\begin{lemma}
{}
{MinimumDistanceEqualsMinimumWeight}
    The minimum distance $d$ of a $q$-ary linear code satisfies
    \[d=\min\left\{w(\mathbf{x}) \mid x\in C, x\neq 0 \right \}.\]
\end{lemma}

An implication of \Cref{lemma:MinimumDistanceEqualsMinimumWeight} is the following classic result, that shows a connection between the minimum distance of a linear code with the linear independence of the columns of a parity check matrix.
\begin{proposition}
{Minimum weight and linear independence}
{MinWeightAndLinearIndependence}
    The minimum distance $d$ of a linear code $C$ over $\fq$ with parity check matrix $H$ satisfies
    \[
        d= \min\left\{ s \mid \text{there exist $s$ columns of $H$ that are linearly dependent} \right\}.
    \]
\end{proposition}
\begin{proof}
    Let $n$ be the length of $C$ and denote the column vectors of $H$ by $H_0, \dots, H_{n-1}$.
    From \Cref{lemma:MinimumDistanceEqualsMinimumWeight}, it suffices to show that, for every positive integer $s$, there exists a nonzero word $x \in C$ of weight at most $s$ if and only if there exist $s$ columns of $H$ that are linearly dependent.

    Assume that $x=(x_0, \dots, x_{n-1}) \in C$ has weight at most $s$, so that there exists an $s$-set $I\subseteq[0,n-1]$ such that $x_i=0$ if $i \not\in I$.
    Every row $s$ of $H$ is a codeword of $C^\perp$, thus it satisfies $r\cdot \mathbf{x}=0$.
    Since this is true for every row, the column vectors $H_0, \dots, H_{n-1}$ of $H$ satisfy $\sum_{i=0}^{n-1}x_iC_i=0$, which implies that $\sum_{i\in I}x_iC_i=0$; this means that the columns $H_i, i \in I$ are linearly dependent.

    Conversely, assume there are $s$ linearly dependent columns in $A$, say $H_{i_0}, \dots, H_{i_{s-1}}$, for some $\{ i_0, \dots, i_{s-1} \}\subseteq [0,n-1]$.
    Then, $\sum_{j=0}^{s-1}x_{i_j}H_{i_j}=0$ for some $x_{i_j} \in \fq$, not all zero.
    Hence, the vector $(x_0, \dots, x_{n-1})$ so that $x_i=0$ if $i \not\in \{ i_0, \dots, i_{s-1} \}$, is in $C$ and has weight at most~$s$.
\end{proof}

\section{Orthogonal and covering arrays}
\label{section:OrthogonalAndCoveringArrays}
\subsection{Introduction}
We begin with some necessary notions.
\begin{definition}
{Covered vectors}
{CoveredColumns}
    Let $V$ be a finite set, $t,n$ be positive integers, $S$ be a set of $t$ vectors from $V^n$, and let $A_S$ be a $n\times t$ array whose columns are the vectors in $S$. 
    We have the following definitions for $S$.
    \begin{itemize}
        \item If every $t$-tuple in $V^t$ appears \emph{at least} $\lambda$ times as a row of $A_S$, then $S$ is \emph{$\lambda$-covered}.
        \item If every $t$-tuple in $V^t$ appears \emph{exactly} $\lambda$ times as a row of $A_S$, then $S$ is \emph{uniformly $\lambda$-covered}.
        \item If $S$ is $1$-covered, we simply refer to it as \emph{covered}.
        \item If $S$ is uniformly $1$-covered, we simply refer to it as \emph{uniformly covered}.
    \end{itemize}
    The vectors $v_1, \dots, v_s$ are (uniformly) $\lambda$-covered, if the set $\{v_1, \dots, v_s \}$ is (uniformly) $\lambda$-covered.
\end{definition}
\index{Covered set|see{Set}}
\index{Uniformly covered set|see{Set}}
\index{Set!covered}
\index{Set!uniformly covered}

In his 1943 Master's thesis, Rao introduced hypercubes of strength $d$.
In a series of subsequent papers \cite{rao1946hypercubes,rao1947factorial,rao1949class}, he extended these to a wider class of combinatorial objects, which became known as orthogonal arrays.
These are closely connected with other combinatorial objects, such as mutually orthogonal latin squares, transversal designs, projective geometries, and linear codes; we refer to \cite{colbourn2006handbook} for more on the subject and to \cite{hedayat2012orthogonal} for a textbook treatment.

\begin{definition}
{Orthogonal array}
{OrthogonalArray}
    Let $A$ be an $N\times k$ array with entries from an alphabet of cardinality $v$, $t\leq k$ be a positive integer such that $v^t|N$, and $\lambda = N/v^t$.
    If every $t$-set of column vectors of $A$ is uniformly $\lambda$-covered, then $A$ is an \emph{orthogonal array of strength $t$ and index $\lambda$}, denoted $\OA_{\lambda}(t, k,v)$.
    When $\lambda=1$ the orthogonal array has \emph{index unity}, and we simply write $\OA_{}(t, k,v)$.
\end{definition}
\index{Array!orthogonal|see{Orthogonal array}}
\index{Orthogonal array}
\index{Orthogonal array!index}
\index{Orthogonal array!index unity}
\index{Index!of orthogonal array|see{Orthogonal array}}
\index{Strength!of orthogonal array|see{Orthogonal array}}

\begin{definition}
{Linear orthogonal array}
{LinearOA}
    Let $q$ be a prime power.
    An $\OA_{\lambda}(t, k,q)$ with elements from the finite field $\fq$ is  \emph{linear} if its row vectors are distinct and form a subspace of $\fq^k$.
\end{definition}
\index{Orthogonal array!linear}

\begin{figure}[t!]
\centering
\[
\begin{array}[]{ccCccccCccccc}
1&1&1&0&1&1&0&1&0&0&0\\
1&1&0&1&1&0&1&0&0&0&1\\
1&0&1&1&0&1&0&0&0&1&1\\
0&1&1&0&1&0&0&0&1&1&1\\
1&1&0&1&0&0&0&1&1&1&0\\
1&0&1&0&0&0&1&1&1&0&1\\
0&1&0&0&0&1&1&1&0&1&1\\
1&0&0&0&1&1&1&0&1&1&0\\
0&0&0&1&1&1&0&1&1&0&1\\
0&0&1&1&1&0&1&1&0&1&0\\
0&1&1&1&0&1&1&0&1&0&0\\
0&0&0&0&0&0&0&0&0&0&0
\end{array}
\]
$\OA_3(2,11,2)$
\vspace{1em}

    \begin{center}
    \begin{minipage}{.23\linewidth}
        \[
        \begin{array}[]{CC}
            1&1\\
            \rowcolor{goldfish}
            0&0\\
            1&0\\
            1&0\\
            0&1\\
            1&1\\
            0&1\\
            \rowcolor{goldfish}
            0&0\\
            0&1\\
            1&1\\
            1&0\\
            \rowcolor{goldfish}
            0&0\\
        \end{array}
        \]
        (0,0) appears 3 times
    \end{minipage}
    \begin{minipage}{.23\linewidth}
        \[
        \begin{array}[]{CC}
            1&1\\
            0&0\\
            1&0\\
            1&0\\
            \rowcolor{goldfish}
            0&1\\
            1&1\\
            \rowcolor{goldfish}
            0&1\\
            0&0\\
            \rowcolor{goldfish}
            0&1\\
            1&1\\
            1&0\\
            0&0\\
        \end{array}
        \]
        (0,1) appears 3 times
    \end{minipage}
    \begin{minipage}{.23\linewidth}
        \[
        \begin{array}[]{CC}
            1&1\\
            0&0\\
            \rowcolor{goldfish}
            1&0\\
            \rowcolor{goldfish}
            1&0\\
            0&1\\
            1&1\\
            0&1\\
            0&0\\
            0&1\\
            1&1\\
            \rowcolor{goldfish}
            1&0\\
            0&0\\
        \end{array}
        \]
        (1,0) appears 3 times
    \end{minipage}
    \begin{minipage}{.23\linewidth}
        \[
        \begin{array}[]{CC}
            \rowcolor{goldfish}
            1&1\\
            0&0\\
            1&0\\
            1&0\\
            0&1\\
            \rowcolor{goldfish}
            1&1\\
            0&1\\
            0&0\\
            0&1\\
            \rowcolor{goldfish}
            1&1\\
            1&0\\
            0&0\\
        \end{array}
        \]
        (1,1) appears 3 times
    \end{minipage}
    \end{center}
    \caption[Example of an $\OA_{3}(2, 11,2)$]{Above, an example of an $\OA_{3}(2, 11,2)$; 
        below, checking that the highlighted columns are uniformly $3$-covered,
        as per the definition of an orthogonal array.}
    \label{figure:ExampleOfOA}
\end{figure}

\begin{example}
{}
{OrthogonalArrayExample}
    An example of an $\OA_{3}(2, 11,2)$ is given in \Cref{figure:ExampleOfOA}.
    By definition, every two columns of such an array are uniformly $3$-covered.
    As an example, we verify this with the two highlighted columns in the same figure: since every pair from $\{ 0,1 \}^2$ appears exactly 3 times as a row of the subarray corresponding to these columns, then they are uniformly $3$-covered.
    The same is true for any other  pair of columns, which means that this is an orthogonal array of strength $t=2$ and index $\lambda=3$.
\end{example}

Since their introduction, orthogonal arrays have been studied extensively in the context of different mathematical areas, and have been used in a variety
of applications, notably the design of statistical experiments.
In particular, the rows of an $\OA_{}(t, k,v)$ provide a collection of configurations for a system of $k$ factors represented by the columns, where each factor admits $v$ possible values.
Performing experimental runs for all $v^t$ configurations corresponding to the rows, guarantees that every combination of $t$ factors and their values is tested exactly once.
However, most applications require that each such combination is tested \emph{at least} once instead.
Moreover, there exist many combinations of parameters $t,k,v$ for which an $\OA_{}(t, k,v)$ does not exist.
This motivates the generalization of orthogonal arrays to covering arrays, for which the requirement that every $t$-tuple appears the same number of times is relaxed.

\begin{definition}{Covering array}{CoveringArray}
    Let $A$ be an $N\times k$ array with entries from an alphabet of cardinality $v$, $\lambda$ and $t$ be positive integers such that $t\leq k$ and $\lambda v^t \leq N$.
    If every $t$-set of column vectors of $A$ is $\lambda$-covered, then $A$ is a \emph{covering array of strength $t$ and index $\lambda$}, denoted $\CA_{\lambda}( N;,t,k,v)$.
    When $\lambda=1$, we simply write $\CA(N; t, k,v)$.
\end{definition}
\index{Array!covering|see{Covering array}}
\index{Covering array}
\index{Covering array!index}
\index{Covering array!number}
\index{Index!of covering array|see{Covering array}}
\index{Strength!of covering array|see{Covering array}}

\begin{figure}[t!]
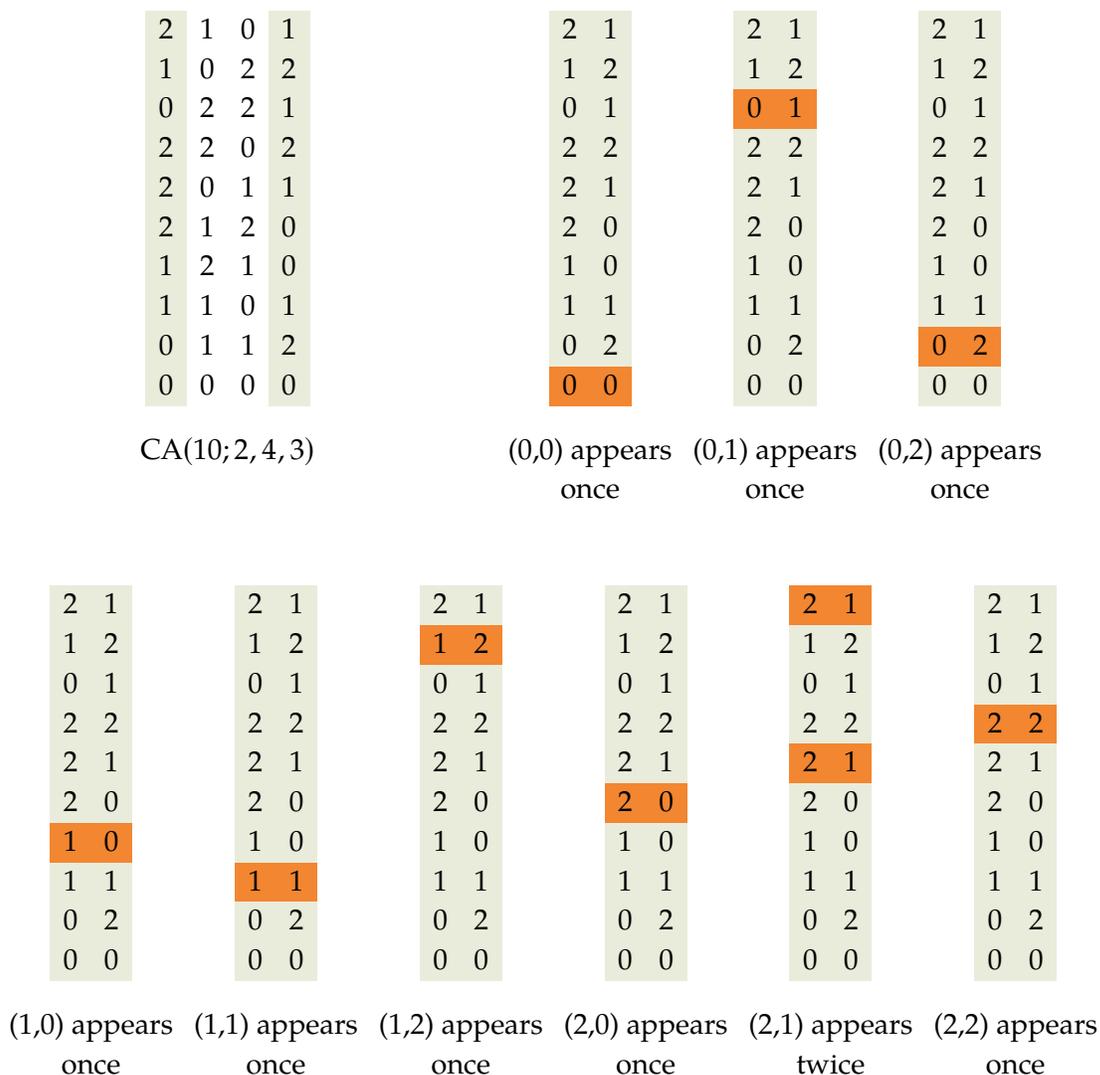

\begin{minipage}{.45\linewidth}
\centering
\[
\begin{array}[]{CccC}
2&1&0&1\\
1&0&2&2\\
0&2&2&1\\
2&2&0&2\\
2&0&1&1\\
2&1&2&0\\
1&2&1&0\\
1&1&0&1\\
0&1&1&2\\
0&0&0&0
\end{array}
\]
\noindent
$\CA(10;2,4,3)$

\mbox{}
\end{minipage}
\begin{minipage}{.15\linewidth}
    \centering
    \[
    \begin{array}[]{CC}
 2&1\\
 1&2\\
 0&1\\
 2&2\\
 2&1\\
 2&0\\
 1&0\\
 1&1\\
 0&2\\
 \rowcolor{goldfish}
 0&0
    \end{array}
    \]
    (0,0) appears once
\end{minipage}
\begin{minipage}{.15\linewidth}
    \centering
    \[
    \begin{array}[]{CC}
 2&1\\
 1&2\\
 \rowcolor{goldfish}
 0&1\\
 2&2\\
 2&1\\
 2&0\\
 1&0\\
 1&1\\
 0&2\\
 0&0
    \end{array}
    \]
    (0,1) appears
    once
\end{minipage}
\begin{minipage}{.15\linewidth}
    \centering
    \[
    \begin{array}[]{CC}
 2&1\\
 1&2\\
 0&1\\
 2&2\\
 2&1\\
 2&0\\
 1&0\\
 1&1\\
 \rowcolor{goldfish}
 0&2\\
 0&0
    \end{array}
    \]
    (0,2) appears
    once
\end{minipage}
\vspace{1em}
\begin{center}
\begin{minipage}{.15\linewidth}
    \centering
    \[
    \begin{array}[]{CC}
 2&1\\
 1&2\\
 0&1\\
 2&2\\
 2&1\\
 2&0\\
 \rowcolor{goldfish}
 1&0\\
 1&1\\
 0&2\\
 0&0
    \end{array}
    \]
    (1,0) appears
    once
\end{minipage}
\begin{minipage}{.15\linewidth}
    \centering
    \[
    \begin{array}[]{CC}
 2&1\\
 1&2\\
 0&1\\
 2&2\\
 2&1\\
 2&0\\
 1&0\\
 \rowcolor{goldfish}
 1&1\\
 0&2\\
 0&0
    \end{array}
    \]
    (1,1) appears
    once
\end{minipage}
\begin{minipage}{.15\linewidth}
    \centering
    \[
    \begin{array}[]{CC}
 2&1\\
 \rowcolor{goldfish}
 1&2\\
 0&1\\
 2&2\\
 2&1\\
 2&0\\
 1&0\\
 1&1\\
 0&2\\
 0&0
    \end{array}
    \]
    (1,2) appears
    once
\end{minipage}
\begin{minipage}{.15\linewidth}
    \centering
    \[
    \begin{array}[]{CC}
 2&1\\
 1&2\\
 0&1\\
 2&2\\
 2&1\\
 \rowcolor{goldfish}
 2&0\\
 1&0\\
 1&1\\
 0&2\\
 0&0
    \end{array}
    \]
    (2,0) appears
    once
\end{minipage}
\begin{minipage}{.15\linewidth}
    \centering
    \[
    \begin{array}[]{CC}
        \rowcolor{goldfish}
 2&1\\
 1&2\\
 0&1\\
 2&2\\
        \rowcolor{goldfish}
 2&1\\
 2&0\\
 1&0\\
 1&1\\
 0&2\\
 0&0
    \end{array}
    \]
    (2,1) appears
    twice
\end{minipage}
\begin{minipage}{.15\linewidth}
    \centering
    \[
    \begin{array}[]{CC}
 2&1\\
 1&2\\
 0&1\\
 \rowcolor{goldfish}
 2&2\\
 2&1\\
 2&0\\
 1&0\\
 1&1\\
 0&2\\
 0&0
    \end{array}
    \]
    (2,2) appears
    once
\end{minipage}
    \end{center}
    \caption[Example of a $\CA(10; 2, 4,3)$]{An example of a $\CA(10;2,4,3)$ and a demonstration that
    the highlighted columns are covered, but not uniformly covered.}
    \label{figure:ExampleOfCoveringArray}
\end{figure}

\begin{example}
{}
{CoveringArrayExample}
    An example of a $\CA(13; 3, 10,2)$ is given in \Cref{figure:ExampleOfCoveringArray}.
    By definition, every two columns of such an array are covered.
    We verify this with the two highlighted columns in the same figure.
    Since every pair from $\{ 0,1,2 \}^2$ appears at least once as a row of the subarray corresponding to these columns, then these are covered; however they are not uniformly covered, since the pair (0,0) appears twice.
    The same type of coverage is true for any other  pair of columns, which means that this is a covering array of strength $t=2$.
\end{example}

Similarly to orthogonal arrays, the rows of a $\CA(N; t, k,v)$ provide a collection of $N$ configurations for a system with $k$ factors, represented by the columns, where each factor admits $v$ possible values.
Performing experimental runs for all $N$ configurations corresponding to the rows of the covering array guarantees that every possible interaction of $t$ factors is tested at least once.
This is known as \emph{$t$-way interaction testing}.
An example of $2$-way interaction testing is given in \Cref{table:ExampleOfSoftwareSystem}; see also the related discussion in \Cref{chapter:Introduction}.
For studies on the effectiveness of $t$-way interaction testing, we refer the reader to \cite{cohen1996combinatorial,kuhn2004software}.

\subsection{Current state of research on covering arrays}
\label{section:OAsAndCAsCurrentStateOfResearch}
In order to discuss the research on covering arrays, we first need to explain one fundamental underlying problem.
First, we note that for any integers $t,k$, and $v$, a $\CA(N,t,k,v)$ can be constructed for large enough $N$.
Indeed, we consider the following naive construction: for each one of the $\binom{k}{t}$ $t$-sets of columns, and for each one of the $v^t$ $t$-tuples in $[0,v-1]^t$, we create a row whose elements in the position of these $t$ columns are the elements of the $t$-tuple, with the rest of the elements having any value.
This yields a $\CA(v^t\binom{k}{t}; t, k,v)$.

Having established the existence of a $\CA(N,t,k,v)$ for some $N$ which is a function of $t,k$ and $v$, a natural question to ask is how small this $N$ can be.
This is a question integral to the research on covering arrays, also because of the implications in testing applications, where a smaller number of rows means a smaller number of tests, and thus a reduction on the time and cost that is needed for a system to be tested.
This leads to the following definition.

\begin{definition}
{Covering array number}
{CAN}
    Let $t,k$ and $v$ be positive integers.
    The \emph{covering array number for $t$, $k$, and $v$}, denoted $\CAN(t,k,v)$, is the smallest integer $N$ for which a $\CA(N; t, k,v)$ exists.
    A covering array $\CA(N;t,k,v)$ with $N=\CAN(t,k,v)$ is \emph{optimal}.
\end{definition}

There is active research on covering arrays, where the focus is to improve upon previously known bounds for covering array numbers, either by providing constructions of covering arrays, or by establishing theoretical bounds, including asymptotic results.
In the following, we present an overview of some of the most important results from this research, that we categorize as follows:
\begin{itemize}
    \item Bounds and asymptotics
    \item Algebraic constructions
    \item Recursive constructions
    \item Algorithmic constructions
\end{itemize}
Part of our presentation is based on \cite{colbourn2004combinatorial}, a somewhat dated but very thorough survey on the subject.
We note that the vast majority of the research on covering arrays that are not orthogonal arrays, focuses on $\CA(N,t,k,v)$ rather than $\CA_{\lambda}(N; t, k,v)$ with $\lambda>1$.

\subsubsection{Bounds and asymptotics}
We start with some trivial bounds.
First, we have that
\[ v^t \leq \CAN(t,k,v)\leq \binom{k}{t} v^t, \]
where the lower bound comes from the covering array definition, and the upper bound comes from our discussion in the beginning of the section.

A rather straightforward recursive bound can be obtained as follows.
We choose any column $c$ of a $\CA(N; t, k,v)$ and any element $x$, and we create an array by removing that column and keeping only those rows whose element at the position of column $c$ is $x$.
Then, this array is a $\CA(N'; t-1, k-1,v)$, where $N'$ is the number of occurrences of $x$ in column $c$, which implies that
\[ \CAN(t-1,k-1,v) \leq \frac{1}{v}\CAN(t,k,v).  \]

Apart from the above, there exist more sophisticated bounds of an asymptotic nature.
For $t=2$ and $v>2$ Gargano et al.\ \cite{gargano1994capacities} show that
\[ N=\frac{v}{2} \log(k) (1+o(1)).  \]
Furthermore, the following classic result is obtained using a method by Stein \cite{stein1974two}, Lov\'{a}sz \cite{lovasz1975ratio} and Johnson \cite{johnson1974approximation}; see \cite[Theorem 1]{sarkar2016upper} for a proof.
\begin{equation*}
    \label{equation:odontopasta}
    \CAN(t, k, v) \leq \frac{t}{\log\frac{v^t}{v^t-1}}\log k(1+o(1)).
\end{equation*}
Improvements of this bound have been obtained using various techniques.
For example, 
Godbole et al.\ \cite{godbole1996t} give an improved bound by employing the Stein-Chen method on random $N\times k$ arrays with elements from an alphabet of size $v$, to obtain a Poisson approximation for the number of uncovered $t$-sets of columns.
For the binary case, an implication from the study of certain probability spaces \cite{azar1998approximating,naor1993small,naor1995splitters} implies that
\[ \CAN(t,k,2) \leq 2^t t^{\bigo(\log(t))}\log(k).  \]
A more recent improvement due to Franceti\'{c} and Stevens is based on the Lov\'{a}sz local lemma (see \cite{alon2008wiley}) -- also known as entropy compression.
Apart from a tighter bound, the authors give bounds in closed form for strengths $2$ and $3$.
A series of improvements on these bounds are given by Sarkar and Colbourn in \cite{sarkar2016upper}, where an overview of the recent advances on this topic is also presented.

Other results come from work on $t$-qualitatively independent partitions of sets, which are equivalent to a covering array of strength $t$ \cite{colbourn2004combinatorial}.
In \cite{gargano1993sperner} Gargano et al.\ determine the asymptotics of the largest size of a family of $2$-independent $k$-partitions of an $n$-set, which at the time was a longstanding problem in combinatorics posed in its generality by R{\'e}nyi in 1971 \cite{renyi1970foundations}.
In the context of covering arrays, their result implies that the ratio of $CAN(2,k,v)$ to $\log(k)$ is asymptotic to $v/2$.

For strengths higher than 2, it follows from \cite{poljak1983qualitatively,poljak1989maximum} that the largest $k$ for which a $\CA(N; t, k,v)$ exists satisfies
\[ \frac{et}{v}e^{\frac{N}{tv^t}} \leq k \leq \frac{K_{v,n}}{\sqrt{N}} 4^{\frac{N}{v^{t-1}}}, \]
where $K_{v,n}$ is a constant depending only upon $v$ and $N$.

\subsubsection{Orthogonal arrays}
\label{section:CurrentStateOAConstructions}
Since an $\OA_{\lambda}(t, k,v)$ is also a $\CA_{\lambda}(\lambda v^t;t,k,v)$, orthogonal arrays are also relevant here.
However, the research on orthogonal arrays is too big a subject to be in the scope of this section, so we limit our presentation to some important constructions that use finite fields.
For a thorough presentation on the subject, we refer the reader to \cite[Chapter III]{colbourn2006handbook}, as well as \cite{hedayat2012orthogonal} for a textbook treatment.

We recall that we denote $\w{n}=(q^n-1)/(q-1)$.
An early construction is for linear $\OA_{q^{n-2}}(2, \w{n},q)$ for all $n\geq 2$, where $q$ is a prime power.
These arrays, at the time named hypercubes of strength $2$, are due to Rao \cite{rao1946hypercubes} and they are optimal in the sense that, for all prime powers $q$ and $n\geq 2$, there do not exist $\OA_{q^{n-2}}(2, k,q)$ with $k > \w{n}$ \cite[Corollary 3.21]{hedayat2012orthogonal}.
Different equivalent constructions can be found in the literature, with a notable one being very similar to the construction of Hamming codes.
In fact, when constructed this way, the rows of such orthogonal arrays are precisely the codewords of the dual of a Hamming code.  Because of this, sometimes these arrays are referred to as \emph{Rao-Hamming} orthogonal arrays.
In \Cref{section:TwoClassicOAConstructions} we present a construction of Rao-Hamming orthogonal arrays using sequences over finite fields, which is fundamental to the results of this thesis.

For a prime power $q>t$, Bush \cite{bush1952orthogonal} constructs linear orthogonal arrays $\OA(t,k,q)$ for all $k \leq q+1$.  Such arrays are optimal, in the sense that they have the minimum possible number of rows.
In this construction, columns are indexed by the elements of $\fq$ and rows are indexed by the polynomials in $\fq$ of degree at most $t-1$; then, each entry of the array is the evaluation of the corresponding column polynomial at the corresponding row element.
The rows of such a linear orthogonal array are precisely the words of an extended Reed-Solomon code, although these codes were introduced several years after Bush's construction; see \cite[Section 5.5]{hedayat2012orthogonal}.

For an \emph{odd} prime power $q$ and integer $n \geq 2$, Addelman and Kempthorne \cite{addelman1961some} give a construction of $\OA_{2q^{n-2}}(2, 2 \w{n},q)$.  This is similar to Bush's construction, with certain functions being used instead of polynomials; however, the resulting arrays are not linear.

\subsubsection{Algebraic constructions}
R{\'e}nyi \cite{renyi1970foundations} determines covering array numbers for optimal binary covering arrays of strength 2 with an even number of rows.  Kleitman and Spencer \cite{kleitman1973families} and Katona \cite{katona1973two} independently generalized this to any number of rows.
In particular they show that, for any $N$, an optimal $\CA(N; 2, k,2)$ has $\binom{N-1}{\lceil N/2\rceil}$ columns.
The construction is given by forming an array whose columns consist of all distinct binary $N$-tuples of weight $\lceil N/2 \rceil$ that have zero in the first position.

Chateauneuf et al.\ \cite{chateauneuf1999covering} construct covering arrays of strength $3$ whose entries are the result of a finite group acting on the symbols, as follows.
Let $M$ be an $n\times k$ array with entries $M_{ij}$ from an alphabet $\Omega$ of size $v>2$, and $G$ be a subgroup of $Sym(\Omega)$, the symmetric group of permutations on the symbols in $\Omega$.
For $g \in G$, let $M^g$ be the $n\times k$ array whose $(i,j)$-th element is $M_{ij}^g$, the image of $M_{ij}$ under $g$.
Finally, let $M^G$ be the $(n|G|+v)\times k$ array that is the vertical concatenation of $M^g$, $g\in G$ and the constant row vectors $(x,...,x)\in \Omega^k$, for all $x\in \Omega$.
The array $M$ is the \emph{starter array} with respect to $G$.
The authors show that it is possible to choose a group $G$ and a starter array $M$ so that $M^G$ is a covering array of strength $3$.
Using further refinements, they construct $\CA(N; 3, 2v,v)$ for all $v>2$ and prime powers $q\geq v-1$.

Meagher and Stevens  \cite{meagher2005group} use a similar method to create covering arrays of strength $2$.
For a subgroup $G$ of $Sym(\Z_v)$, $k\geq 2$ and $a\in (\Z_v\cup \infty)^k$, they consider a circulant $k\times k$ matrix $M$ whose rows are precisely the cyclic shifts of $a$.
Then, $a$ is a \emph{starter vector} with respect to $G$ if $M$ has the property that, for any $k\times 2$ subarray, there exists as a row a representative of every orbit of the action of $G$ on $(\Z_v\cup \infty)^2$, where $\infty$ is fixed by the action.
The authors show that if a starter vector exists in $\Z_v^k$ with respect to $G=\langle (1,2,\dots, v-1) \rangle $, then the vertical concatenation of $M^g$, $g \in G$ is a $CA(k(v-1)+1; 2, k, v)$.
They provide starter vectors for various values of $v$ and $k$ using computer search.
This work has been generalized by Lobb et al.\ \cite{lobb2012cover} who use arbitrary groups up to certain size instead of $\Z_v$, as well as multiple fixed symbols (infinities).

A different effective technique that combines algebraic elements and computations is due to Sherwood et al.\ who introduce a type of array in a compacted form using permutation vectors, and connect the covering array definition with the existence of (covering) perfect hash families. 

\subsubsection{Recursive constructions}
A covering array can be obtained by combining one or more existing covering arrays, using recursive constructions.
The following product of strength 2 covering arrays is a construction of this nature.
Given $A=\CA(N; 2, k,v)$ and $B=\CA(M; 2, l,v)$, a $\CA(N+M; 2, kl,v)$ can be constructed by appending $k$ copies of $B$ to $l$ copies of $A$.
More precisely, denoting $A=(A_{ij})$ and $B=(B_{ij})$, this is the $(N+M)\times kl$ array $C=(C_{ij})$ given by
\[ C_{i,(f-1)k+g} = A_{i,g}, \text{ for } i\in [1,N], f\in [1,l], g\in [1,k] \]
and
\[ C_{N+i,(f-1)k+g} = B_{i,f}, \text{ for } i\in [1,M], f\in [1,l], g\in [1,k].  \]
This method is also the essence behind results on related combinatorial objects, such as qualitatively independent systems \cite{poljak1983qualitatively,cohen1998new} and transversal covers \cite{stevens1999new}; see also \cite[Section 4.1]{colbourn2004combinatorial} for extensions and further discussion on this recursive construction.

An important category of recursive constructions are the \emph{Roux-type constructions}.
In his PhD thesis, Roux \cite{roux1987k} shows that for integers $N,M,k$, arranging four covering arrays as in the following diagram, where the array at the bottom right is the bit complement of the array on the bottom left, is a $\CA(N+M; 3, 2k,2)$.
\[
    \begin{array}{|c|c|}
        \hline
        \CA(N; 3, k,2)&
        \CA(N; 3, k,2)\\ \hline
        &\\[-1.1em]
        \CA(M; 2, k,2)&
        \overline{\CA(M; 2, k,2)}\\
        \hline
    \end{array}
\]
This implies the following bound for covering array numbers.
\[\CAN(3,2k,2) \leq \CAN(3,k,2) + \CAN(2,k,2).\]
A generalization of this recursive construction that is not restricted to the binary alphabet is given by Chateauneuf and Kreher  \cite{chateauneuf2002state}.
Let $\pi$ be a cyclic permutation of $v$ symbols.
Then, the following arrangement of covering arrays yields a
$\CA(N+(v-1)M; 3, 2k,v)$.
\[
    \begin{array}{|c|r|}
        \hline
        \CA(N; 3, k,v)&
        \CA(N; 3, k,v)\phantom{)}\\
        \hline
        \CA(M; 2, k,v)&
        \pi\left(\CA(M; 2, k,v)\right)\\ \hline
        \CA(M; 2, k,v)&
        \pi^2\left(\CA(M; 2, k,v)\right)\\ \hline
        \vdots&\multicolumn{1}{c|}{\hspace{2.5em}\vdots}\\\hline
        \CA(M; 2, k,v)&
        \pi^{v-1}\left(\CA(M; 2, k,v)\right)\\ \hline
    \end{array}
\]
This implies the following bound for covering array numbers.
\[\CAN(3,2k,v) \leq \CAN(3,k,v) +(v-1)\CAN(2,k,v).\]
A number of generalizations exist for strengths greater than $3$ \cite{boroday1998determining,hartman2005software,hartman2004problems,martirosyan2004t} as well as a generalization due to Cohen et al.\ \cite{cohen2008constructing} that permits multiplying the number of columns by any $l\geq 2$.

In \cite{colbourn2010coveringandradius} Colbourn et al.\ give two recursive methods that they call \emph{fusion} and \emph{factor increase}.
\index{Fusion operation}
\index{Covering array!fusion operation on}
The first is a generalization of a result from \cite{colbourn2008strength} and is a method to obtain a $\CA(N-2; t, k,v-1)$ from a $\CA(N; t, k,v)$.
The procedure is as follows.
First we note that permuting the symbols in a column of a covering array does not affect the covering array property.
We assume without loss of generality that the alphabet is $[0,v-1]$.
For every column a suitable permutation of the symbols is applied so that a row consisting entirely of the symbol $v-1$ is created; this row is then dropped.
Then, another row $r=(r_0, \dots, r_{k-1})$ is chosen at random and subsequently for every column $c_i$, $i \in [1,k-1]$, every instance of $v-1$ is replaced by $r_i$ if $r_i\neq v-1$ and $0$ otherwise.
In the array resulting from the above, all the $t$-tuples that did not involve $v-1$ that were covered in row $r$, are covered in other rows.
Hence, removing row $r$ yields a $\CA(N-2,t,k,v-1)$.
This implies the following inequality of covering array numbers.
\begin{equation}
    \label{equation:FusionOperationInequality}
    \CAN(t,k,v-1) \leq \CAN(t,k,v)-2.
\end{equation}
We use the above inequality in \Cref{chapter:CAsFirstPaper} to improve upon previously best known bounds for covering array numbers.

For the factor increase construction, we start with two covering arrays $\CA(N;t,k,v)$ and $\CA(M; t-2, k-1,v)$, with columns $A_1, \dots, A_k$ and $B_1, \dots, B_{k-1}$, respectively.
We assume without loss of generality that the alphabet is $[0,v-1]$.
For $x \in [0,v-1]$, we denote $C_x$ to be the column vector with $M$ coordinates all equal to $x$.
Then, the vertical concatenation of the $N\times (k+1)$ arrays $[A_1|A_2|\dots | A_{k-1}|A_k|A_k]$ and $[B_1|B_2|\dots | B_{k-1}|C_x|C_y]$, for all pairs $(x,y) \in [0,v-1]^2$, $x\neq y$, is a $\CA(N+v(v-1)M; t, k,v)$.
This implies the following inequality of covering array numbers.
\[ \CAN(t, k + 1, v) \leq \CAN(t, k, v) + v(v-1)\CAN(t-2, k-1, v).  \]

Other combinatorial structures can also be used to obtain covering arrays recursively.  Cohen et al.\ \cite{cohen2008constructing} construct strength-3 covering arrays using ordered designs along with a covering array of strength 2; a refinement of this construction \cite[Section 4.3]{colbourn2004combinatorial} gives $\CA(N; 3, q+1,q+1)$ where $q$ is any prime power and $N=q^3-q+\binom{q+1}{2}\CAN(3,q+1,2)-q^2+1$.
See also \cite[Sections 4.4 and 4.5]{colbourn2004combinatorial} for a number of recursive constructions that rely on generalized Hadamard matrices and perfect hash families.

\subsubsection{Algorithmic constructions}
There exists a large number of greedy algorithms for computer generation of covering arrays.
As a detailed overview is beyond our scope, we only enumerate here some of the most well-known results.
For a survey we refer to \cite{colbourn2004combinatorial}; see also \cite{bryce2009density}
for a thorough analysis on some of these algorithms.
Greedy algorithms include AETG \cite{cohen1997aetg}, TCG \cite{tung2000automating}, DDA \cite{bryce2007density}, IPO \cite{lei1998parameter}, and BBA \cite{ronneseth2009merging}.
Notably in \cite{bryce2009density}, generalizing results from \cite{bryce2007density}, Bryce and Colbourn give an algorithm which, for fixed $t$ and $v$, constructs a $\CA(N;t,k,v)$ with $N=\bigo(\log(k))$ in polynomial time.
While the existence of covering arrays that satisfy this upper bound have been previously established \cite{colbourn2004combinatorial}, this is the first deterministic polynomial-time algorithm that provides such a construction.
Furthermore, there has been work with metaheuristic algorithms that use techniques such as tabu search \cite{nurmela2004upper}, and simulated annealing \cite{george2012constructing,stevens1998transversal,stardom2001metaheuristics,cohen2004designing,cohen2003variable,cohen2003constructing,cohen2003augmenting,cohen2008constructing}.

\titlespacing*{\chapter}{0cm}{0.87cm}{2cm} 
\chapter[Combinatorial arrays from maximal sequences]
{Combinatorial arrays from maximal sequences over finite fields}
\label{chapter:CombinatorialArraysFromMSequences}


\textsc{In this chapter we study} connections between maximal sequences and concepts from other areas of discrete mathematics, and we present previously established constructions of orthogonal and covering arrays using maximal sequences in view of these connections.


We first present two classic theorems that give the essence of the main results in this chapter.
The first theorem is stated explicitly by Bose \cite{bose1961some}, although its first half is due to Kempthorne \cite{kempthorne1947simple}.
It is also the special case of a generalization for the non-linear case due to Delsarte \cite{delsarte1973algebraic}; an in-depth presentation of Delsarte's result can be found in \cite[Section~4.4]{hedayat2012orthogonal}.

\begin{theorem}
{Linear codes and linear orthogonal arrays are equivalent objects \cite[Theorem 4.6]{hedayat2012orthogonal}}
{LinearCodesAndLinearOAsAreEquivalent}
    Let $q$ be a prime power, and $C$ be a linear code of length $k$ and dimension $n$, whose dual code has minimum distance $d^{\perp}$.
    Then, the $q^n \times k$ array whose rows are the codewords of $C$, is a linear $\OA_{q^{n-d^{\perp}+1}}(d^{\perp}-1, k,q)$.
    Conversely, the $q^n$ rows of a linear $\OA_{q^{n-t}}(t, k,q)$ are the codewords of a $q$-ary linear code with length $k$ and dimension $n$, whose dual code has minimum distance $d^{\perp}\geq t+1$.
    If the orthogonal array has strength $t$ but not $t+1$, then $d^{\perp}$ is precisely $t+1$.
    \index{Linear code}
\index{Orthogonal array}
\end{theorem}

The second theorem gives a necessary and sufficient condition for an array to be a linear orthogonal array, that is related to the linear independence of its columns.

\begin{theorem}
{Linear orthogonal arrays and linear independence \cite[Theorem 3.29]{hedayat2012orthogonal}}
{OANecessaryAndSufficientConditionRelatedToLI}
    Let $A$ be a $q^{n}\times k$ array whose rows form a $q$-ary linear code of length $k$ and dimension $n$.
    Then, $A$ is a linear $\OA_{q^{n-t}}(t, k,q)$ if and only if every $t$-set of its column vectors is linearly independent over $\fq$.
\index{Orthogonal array}
\end{theorem}

The structure of the rest of the chapter is as follows.
In \Cref{section:CyclicTraceArraysAndCombinatorialProperties} we give a method for constructing arrays from maximal sequences, with the property that their rows are the words of a linear code.
We study these arrays and present results closely related to \Cref{theorem:LinearCodesAndLinearOAsAreEquivalent,theorem:OANecessaryAndSufficientConditionRelatedToLI} applied to these arrays, that also demonstrate how certain properties of maximal sequences can be translated into the language of orthogonal arrays, linear codes, divisibility of polynomials, and finite geometry.
In \Cref{section:ConstructionOfOrthogonalArraysFromCyclicTraceArrays} we show how arrays from maximal sequences can be used to construct several previously known families of orthogonal arrays.
In \Cref{section:ConstructionOfCoveringArraysFromCyclicTraceArrays} we present a method due to Raaphorst et al.~\cite{raaphorst2014construction} that uses arrays from maximal sequences as building blocks to construct larger covering arrays of strength 3.

\pagebreak
The main contribution of this chapter is a framework that allows us to express previous results in a consistent manner and emphasize the ties between maximal sequences and other areas of discrete mathematics.
Although these ties have been previously explored or implied, we offer a detailed overview of all the different points of view that, to the best of our knowledge, has not been given before as explicitly.

\section{Cyclic trace arrays and their properties}
\label{section:CyclicTraceArraysAndCombinatorialProperties}

We begin with the definition of a type of array that plays a fundamental role for the results in this thesis.
\begin{definition}
{Cyclic trace array of a primitive element}
{cyclicAlphaSArray}
\index{Cyclic trace array}
\index{Cyclic trace array!principal}
\index{Array!cyclic trace array|see{Cyclic trace array}}
\index{Principal cyclic trace array|see{Cyclic trace array}}
    Let $t,k$ be positive integers, $\a$ be a primitive element of $\fqt$ and $C=\{ c_0, \dots, c_{k-1} \}$ be an ordered subset of $[0,q^t-2]$.
    Then, the \emph{cyclic trace array corresponding to $\a$ and $C$}, denoted $\mathcal{A}_{q^t/q}(\a, C)$, is the $(q^t-1)\times k$ array with elements 
    \[ \mathcal{A}_{q^t/q}(\a, C)_{ij} = \Tt(\a^{i+c_j}), \quad (i,j) \in [0,q^t-2] \times [0,k-1].  \]
    We denote by $\mathcal{A}_{0,\,q^t/q}(\a, C)$ the $q^t\times k$ array obtained by appending a row of zeros to $\mathcal{A}_{q^t/q}(\a, C)$, i.e.
    \[
        \mathcal{A}_{0,\, q^t/q}(\a, C)_{ij} =
        \begin{cases}
            \mathcal{A}_{q^t/q}(\a, C)_{ij}= \Tt(\a^{i+c_j}), & \text{ if } i \in [0,q^t-2], j \in [0,k-1]; \\
            0, &   \text{ if } i = q^t-1, j \in [0,k-1].
        \end{cases}
    \]
    When it is clear from the context that the fields are $\fqt$ and $\fq$, we simply write $\mathcal{A}$ instead of $\mathcal{A}_{q^t/q}$.
    Furthermore, we simply write $\A{\a}$ to denote $\A{\a,[0,\w{t}-1]}$, $\ZA{\a}$ to denote $\ZA{\a,[0,\w{t}-1]}$, and we refer to $\ZA{\a}$ as the \emph{principal cyclic trace array of $\a$}.
\end{definition}
We use a dedicated name and notation for $\ZA{\a}=\ZA{\a, [0,\w{t}-1]}$ for reasons that we discuss at the end of this section.
Furthermore, there is a close connection between cyclic trace arrays and maximal sequences which we give in \Cref{remark:ColumnsOfCyclicTraceArrayAreCyclicShiftsOfMseq}.
First, we need to introduce some notation related to sequences.

\begin{definition}
{Sequence associated with finite field elements}
{SequenceAssociatedWithFFElement}
\index{LFSR sequence!associated to element}
    Let $q$ be a prime power and $t$ be a positive integer.
    For $\a \in \fqt$, the \emph{(LFSR) sequence over $\fq$ associated to $\a$} is
    \[
        \mathrm{Seq}_{q^t/q}(\a)=
        \left( \mathrm{Seq}_{q^t/q}(\a)_i \right)_{i \geq 0} =
        \left( \Tt(\a^i) \right)_{i\geq 0}.
    \]
    When it is clear that the underlying field is $\fq$, we simply write $\sq{\a}$ instead of $\mathrm{Seq}_{q^t/q}(\a)$.
\end{definition}

We note that, by \Cref{theorem:TraceRepresentationOfMSequence}, if $\a$ is primitive, then $\seq{\a}$ is a maximal sequence with period $q^t-1$.

\begin{definition}
{Left shift operator}
{LeftShiftOperator}
\index{LFSR sequence!shift}
\index{Left shift operator}
    For a sequence $\bfs=(s_i)_{i\geq 0}$ we define the \emph{left shift operator on $\bfs$}, denoted $L$, by 
    \[ L^j(\bfs)= (s_{i+j})_{i\geq 0}, \quad j \in \Z.  \]
    We say that $L^j(\bfs)$ is the \emph{left cyclic shift of $\bfs$ by $j$} and we simply write $L(\bfs)$ to denote $L^1(\bfs)$.
    Furthermore, for a positive integer $n$, we denote
    \[ L^j_n(\bfs)= (s_{i+j})_{i=0}^{n-1}, \quad j \in \Z.  \]
\end{definition}

\begin{remark}
{}
{ColumnsOfCyclicTraceArrayAreCyclicShiftsOfMseq}
    The next statements are straightforward implications of \Cref{definition:cyclicAlphaSArray,definition:SequenceAssociatedWithFFElement,definition:LeftShiftOperator}.
    \begin{enumerate}
        \item The columns of $\A{\a, C}$ are the left cyclic shifts of $\seq{\a}$ by the elements in $C$, that is,
        \[
            \A{\a, C}=
            \left[ 
                L^{c_0}_{q^t-1}\left( \seq{\a}\right) \mid
                L^{c_1}_{q^t-1}\left( \seq{\a} \right)\mid
                \dots\mid
                L^{c_{k-1}}_{q^t-1} \left( \seq{\a}\right)
            \right].
        \]
        \item If $C=\{ c_0, \dots, c_{k-1} \}$, then for $(i,j) \in [0,q^t-2]\times[0,k-1]$, the $(i,j)$-th element of $\M{\a, C}$ is
            \begin{equation*}\label{equation:IJthElementOfTraceArrayInTermsOfSeqq}
                \M{\a, C}_{i,j}= \seq{\a}_{i+c_j}.
            \end{equation*}
        \item 
        The rows of $\A{\a}$ are the vectors $L_{\w{t}}^i( \seq{\a} )$, $i \in [0,q^t-2]$.
    \item 
        For $C\subseteq C'$, we have that $\ZA{\a, C}$ is a subarray of $\ZA{\a, C'}$.
    \end{enumerate}
\end{remark}

\newcolumntype{b}{>{\columncolor{\backgroundshade}}r}
\def\ph{\phantom{1}}
\setlength{\tabcolsep}{2pt}
\renewcommand{\arraystretch}{0.8}
\begin{table}
\small
\centering
\begin{tabular}{rrrrrrrbrbrrrbrrrrrrrrrrrrrr}
\rowcolor{\tableheadcolor}
\cellcolor{white} &\ph0&\ph1&\ph2&\ph3&\ph4&\ph5& \cellcolor{\tableheadcolor!50!beach}\ph6 &\ph7 &\cellcolor{\tableheadcolor!50!beach}\ph8 &\ph9&10&11 &\cellcolor{\tableheadcolor!50!beach}12 &13  & 14 & 15 & 16 & 17 & 18 & 19 & 20 & 21 & 22 & 23 & 24 & 25\\
                    &0&0&1&0&1&2&1&1&2&0&1&1&1 & 0&0&2&0&2&1&2&2&1&0&2&2&2\\
                    &0&1&0&1&2&1&1&2&0&1&1&1&0 & 0&2&0&2&1&2&2&1&0&2&2&2&0\\
                    &1&0&1&2&1&1&2&0&1&1&1&0&0 & 2&0&2&1&2&2&1&0&2&2&2&0&0\\
                    &0&1&2&1&1&2&0&1&1&1&0&0&2 & 0&2&1&2&2&1&0&2&2&2&0&0&1\\
                    &1&2&1&1&2&0&1&1&1&0&0&2&0 & 2&1&2&2&1&0&2&2&2&0&0&1&0\\
                    &2&1&1&2&0&1&1&1&0&0&2&0&2 & 1&2&2&1&0&2&2&2&0&0&1&0&1\\
                    &1&1&2&0&1&1&1&0&0&2&0&2&1 & 2&2&1&0&2&2&2&0&0&1&0&1&2\\
                    &1&2&0&1&1&1&0&0&2&0&2&1&2 & 2&1&0&2&2&2&0&0&1&0&1&2&1\\
                    &2&0&1&1&1&0&0&2&0&2&1&2&2 & 1&0&2&2&2&0&0&1&0&1&2&1&1\\
                    &0&1&1&1&0&0&2&0&2&1&2&2&1 & 0&2&2&2&0&0&1&0&1&2&1&1&2\\
                    &1&1&1&0&0&2&0&2&1&2&2&1&0 & 2&2&2&0&0&1&0&1&2&1&1&2&0\\
                    &1&1&0&0&2&0&2&1&2&2&1&0&2 & 2&2&0&0&1&0&1&2&1&1&2&0&1\\
                    &1&0&0&2&0&2&1&2&2&1&0&2&2 & 2&0&0&1&0&1&2&1&1&2&0&1&1\\
$\A{\a,[0,q^t-2]}=$ &0&0&2&0&2&1&2&2&1&0&2&2&2 & 0&0&1&0&1&2&1&1&2&0&1&1&1\\
                    &0&2&0&2&1&2&2&1&0&2&2&2&0 & 0&1&0&1&2&1&1&2&0&1&1&1&0\\
                    &2&0&2&1&2&2&1&0&2&2&2&0&0 & 1&0&1&2&1&1&2&0&1&1&1&0&0\\
                    &0&2&1&2&2&1&0&2&2&2&0&0&1 & 0&1&2&1&1&2&0&1&1&1&0&0&2\\
                    &2&1&2&2&1&0&2&2&2&0&0&1&0 & 1&2&1&1&2&0&1&1&1&0&0&2&0\\
                    &1&2&2&1&0&2&2&2&0&0&1&0&1 & 2&1&1&2&0&1&1&1&0&0&2&0&2\\
                    &2&2&1&0&2&2&2&0&0&1&0&1&2 & 1&1&2&0&1&1&1&0&0&2&0&2&1\\
                    &2&1&0&2&2&2&0&0&1&0&1&2&1 & 1&2&0&1&1&1&0&0&2&0&2&1&2\\
                    &1&0&2&2&2&0&0&1&0&1&2&1&1 & 2&0&1&1&1&0&0&2&0&2&1&2&2\\
                    &0&2&2&2&0&0&1&0&1&2&1&1&2 & 0&1&1&1&0&0&2&0&2&1&2&2&1\\
                    &2&2&2&0&0&1&0&1&2&1&1&2&0 & 1&1&1&0&0&2&0&2&1&2&2&1&0\\
                    &2&2&0&0&1&0&1&2&1&1&2&0&1 & 1&1&0&0&2&0&2&1&2&2&1&0&2\\
                    &2&0&0&1&0&1&2&1&1&2&0&1&1 & 1&0&0&2&0&2&1&2&2&1&0&2&2\\
\rowcolor{white}
                    & & & & & & & & & & & & &  &  & & & & & & & & & & & & 
\end{tabular}

\hspace{3.7em}
\begin{tabular}{rrrrrrrrrrrrrr}
\arrayrulecolor{\backgroundshade}
\rowcolor{white}
          &\ph0&\ph0&\ph1&\ph0&\ph1&\ph2&\ph1&\ph1&\ph2&\ph0&\ph1&\ph1&\ph1\\
          &0&1&0&1&2&1&1&2&0&1&1&1&0\\
          &1&0&1&2&1&1&2&0&1&1&1&0&0\\
          &0&1&2&1&1&2&0&1&1&1&0&0&2\\
          &1&2&1&1&2&0&1&1&1&0&0&2&0\\
          &2&1&1&2&0&1&1&1&0&0&2&0&2\\
          &1&1&2&0&1&1&1&0&0&2&0&2&1\\
          &1&2&0&1&1&1&0&0&2&0&2&1&2\\
          &2&0&1&1&1&0&0&2&0&2&1&2&2\\
          &0&1&1&1&0&0&2&0&2&1&2&2&1\\
          &1&1&1&0&0&2&0&2&1&2&2&1&0\\
          &1&1&0&0&2&0&2&1&2&2&1&0&2\\
          &1&0&0&2&0&2&1&2&2&1&0&2&2\\
$\ZA{\a}=$&0&0&2&0&2&1&2&2&1&0&2&2&2\\
          &0&2&0&2&1&2&2&1&0&2&2&2&0\\
          &2&0&2&1&2&2&1&0&2&2&2&0&0\\
          &0&2&1&2&2&1&0&2&2&2&0&0&1\\
          &2&1&2&2&1&0&2&2&2&0&0&1&0\\
          &1&2&2&1&0&2&2&2&0&0&1&0&1\\
          &2&2&1&0&2&2&2&0&0&1&0&1&2\\
          &2&1&0&2&2&2&0&0&1&0&1&2&1\\
          &1&0&2&2&2&0&0&1&0&1&2&1&1\\
          &0&2&2&2&0&0&1&0&1&2&1&1&2\\
          &2&2&2&0&0&1&0&1&2&1&1&2&0\\
          &2&2&0&0&1&0&1&2&1&1&2&0&1\\
          &2&0&0&1&0&1&2&1&1&2&0&1&1\\
          &0&0&0&0&0&0&0&0&0&0&0&0&0\\
\end{tabular}
\quad
$\A{\a,C}=$
\begin{tabular}{bbb}
1&2&1\\
1&0&0\\
2&1&0\\
0&1&2\\
1&1&0\\
1&0&2\\
1&0&1\\
0&2&2\\
0&0&2\\
2&2&1\\
0&1&0\\
2&2&2\\
1&2&2\\
2&1&2\\
2&0&0\\
1&2&0\\
0&2&1\\
2&2&0\\
2&0&1\\
2&0&2\\
0&1&1\\
0&0&1\\
1&1&2\\
0&2&0\\
1&1&1\\
2&1&1\\
\rowcolor{white}
 & &  
\end{tabular}
\quad
$\ZA{\a,C}=$
\begin{tabular}{bbb}
1&2&1\\
1&0&0\\
2&1&0\\
0&1&2\\
1&1&0\\
1&0&2\\
1&0&1\\
0&2&2\\
0&0&2\\
2&2&1\\
0&1&0\\
2&2&2\\
1&2&2\\
2&1&2\\
2&0&0\\
1&2&0\\
0&2&1\\
2&2&0\\
2&0&1\\
2&0&2\\
0&1&1\\
0&0&1\\
1&1&2\\
0&2&0\\
1&1&1\\
2&1&1\\
\rowcolor{white}
0&0&0
\end{tabular}

\caption[Examples of cyclic trace arrays]{The arrays $\A{\a,[0,q^t-2]}$, $\ZA{\a}$, $\A{\a,C}$, $\ZA{\a,C}$, with $\a$ and $C=[6,8,12]$, as described in \Cref{example:cyclicarray}. 
         The highlighted columns of $\A{\a}$ are those with indexes in $C$.}
\label{table:fullcyclicarray}
\end{table}

\begin{example}{}{cyclicarray}
    For a root $\a$ of the primitive polynomial $x^3+2x+1$ over $\f_{3}$, we compute
    \[\seq{\a}=00101211201110020212210222.\]
    For $C=\{ 6,8,12\}$, the arrays $\A{\a,[0,q^t-2]}$, $\ZA{\a}$, $\A{\a, C}$ and $\ZA{\a, C}$, are shown in \Cref{table:fullcyclicarray}. 
\end{example}

The rows of a cyclic trace array form the words of a linear code. 
To prove this we need the next lemma.

\begin{lemma}
{}
{TraceBetaXIsZeroForAllXIffBetaIsZero}
    Let $t$ be a positive integer and $\beta \in \fqt$.
    Then, we have that $\Tt(\beta x)=0$ for all $x\in \fqt$, if and only if $\beta=0$.
\end{lemma}
\begin{proof}
    From \Cref{theorem:PropertiesOfTrace}, the trace is a linear transformation from $\fqt$ onto $\fq$ and, if $\beta\neq 0$, the mapping $x\mapsto \beta x$ permutes the elements of $\fqt$.
    Therefore, $\Tt(\beta x )$ maps $\fqt$ onto $\fq$ as well, hence it is not true that $\Tt(\beta x ) = 0$ for all $x \in \fqt$.
    The other direction is clear.
\end{proof}

\begin{theorem}
{The rows of $\ZA{\a, C}$ as a $q$-ary linear code}
{RowsOfAAreCode}
\index{Cyclic trace array}
\index{Linear code}
    Let $t$ be a positive integer, $\a$ be a primitive element of $\fqt$, and $C$ be a nonempty subset of $[0,q^t-2]$, with $|C|=k$.
    Then, the set of rows of $\ZA{\a, C}$ is a linear code of length $k$ and dimension 
    $d=\dim( \spn\{ \a^i\mid i \in C \} )$.
\end{theorem}

\begin{proof}
    Let $C=\{ c_0,\dots, c_{k-1}\}$.
    The set of rows of $\ZA{\a, C}$ is given by
    \begin{align*}
        \Rcal
        &= 
        \left\{ \left( \T(\a^{i}\a^{c_0}), \dots, \T(\a^{i}\a^{c_{k-1}}) \right) \mid i \in [0,q^t-2] \right\}
        \cup
        \left\{ (0,\dots,0) \right\}\\
        &=
        \left\{ \left( \T(x\a^{c_0}), \dots, \T(x\a^{c_{k-1}} \right) \mid x \in \fqt \right\}.
    \end{align*}
    Let $r_1, r_2 \in \Rcal$ and $c \in \fq$.
    Then there exist $x,y \in \fqt$ such that
    \begin{align*}
        r_1&=\left(\T(x\a^{c_0}), \dots, \T(x\a^{c_{k-1}})\right)\\
        r_2&=\left(\T(y\a^{c_0}), \dots, \T(y\a^{c_{k-1}})\right)
    \end{align*}
    and hence
    \begin{align*}
        r_1+cr_2
        &=
        \left(\T(x\a^{c_0}), \dots, \T(x\a^{c_{k-1}})\right) +c\left(\T(y\a^{c_0}), \dots, \T(y\a^{c_{k-1}})\right)\\
        &=
        \left( \T(x\a^{c_0})+c\T(y\a^{c_0}), \dots, \T(x\a^{c_{k-1}})+c\T(y\a^{c_{k-1}}) \right).
    \end{align*}
    By \Cref{item:TraceIsLinearTransformation} of \Cref{theorem:PropertiesOfTrace}, the trace is linear, so the latter is equal to 
    \[\left( \T((x+cy)\a^{c_0}), \dots, \T((x+cy)\a^{c_{k-1}}) \right),\]
    which is an element of $\Rcal$.
    We conclude that $\Rcal$ is a subspace of $\fq^k$, that is, linear code of length~$k$.

    Now, we observe that the dimension of $\Rcal$ is the dimension of the row space of $\ZA{\a, C}$, which is also the dimension of its column space.
    Hence, to complete the proof it suffices to show that the dimension of the column space is equal to $\dim( \spn\{ \a^c \mid c\in C \} )$.
    To do that, we show that for every $I \subseteq [0,k-1]$, the column vectors of $\ZA{\a, C}$ indexed by $c_i, i \in I$, are linearly dependent if and only if the elements $\a^{c_i}$, $i \in I$, are linearly dependent.
    We denote the $i$-th column vector of $\ZA{\a, C}$ by $C_i$, and consider nonempty $I \subseteq [0,k-1]$ and elements $y_i \in \fq$, $i \in I$ that are not all zero.
    Then, we have that $\sum_{i\in I}y_i C_i=0$ if and only if every row of the array consisting of the columns $C_i$, $i \in I$, satisfies the same linear dependence relation, that is, $\sum_{i\in I}y_i \T(x\a^{c_i})=0$, for all $x\in \fqt$.
    From the linearity of the trace, as per \Cref{theorem:PropertiesOfTrace}, the latter is equivalent to $\T(x\sum_{i\in I}y_i\a^{c_i})=0$, for all $x \in \fqt$.
    Furthermore, by \Cref{lemma:TraceBetaXIsZeroForAllXIffBetaIsZero}, this is equivalent to $\sum_{i \in I}y_i\a^{c_i}=0$.
    We conclude that $\sum_{i\in I}y_iC_i=0$ if and only if $\sum_{i \in I}y_i\a^{c_i}=0$.
    From our earlier discussion, this completes the proof.
\end{proof}

It follows from \Cref{theorem:LinearCodesAndLinearOAsAreEquivalent,theorem:RowsOfAAreCode} that a cyclic trace array $\ZA{\a, C}$ is a $q^t \times |C|$ orthogonal array, whose strength depends on the minimum distance of the dual code of the linear code whose codewords are the rows of the orthogonal array.
By \Cref{theorem:OANecessaryAndSufficientConditionRelatedToLI}, its strength is also the largest number $t$ such that every $t$ of its column vectors are linearly independent over $\fq$.
Therefore, to specify the strength it is useful to characterize the subsets of column vectors of cyclic trace arrays that are linearly independent.  This is done in the next proposition, which is a generalization of a result due to \mbox{Raaphorst et al.} \mbox{\cite[Theorem 2]{raaphorst2014construction}}.
We have added \Cref{item:b,item:d,item:e}, as well as the index $q^{t-s}$ in \Cref{item:c}, which is missing from \cite[Theorem 2]{raaphorst2014construction}.

\begin{proposition}
{}
{SebastianExtended}
    Let $t,s$ be positive integers with $s\leq t$, $\a$ be a primitive element of $\fqt$ with minimal polynomial $m_{\a}$, and $I=\{ i_0, \dots, i_{s-1} \}$ be a nonempty subset of $[0,q^t-2]$.
    Then, the following statements are equivalent:
    \begin{enumerate}
        \item \label{item:b}
            The $q^t\times s$ array $\ZA{\a,I}$ is an $\OA_{q^{t-s}}(t, s,q)$.
        \item \label{item:a}
            The set $\{ \a^{i_j} \mid j \in [0,s-1] \}$ is linearly independent over $\fq$.  
        \item \label{item:c}
          The set $\{ L_{q^t-1}^{i_j}(\seq{\a}) \mid j \in [0,s-1] \}$ is linearly independent over $\fq$.
        \item \label{item:d}
              For every $d_0, \dots, d_{s-1} \in \fq$ not all zero, $m_{\a}(x)$ does not divide $\sum_{j=0}^{s-1}d_j x^{i_j}$.
          \item \label{item:e}
              There is no $(s-2)$-flat of $PG(t-1,q)$ that contains the points $[\a^{i_0}], \dots, [\a^{i_{s-1}}]$.
    \end{enumerate}
    For the special case when $s=t$, the following is also equivalent to the statements above.
    \begin{enumerate}[resume]
        \item \label{item:f}
        There is no zero row in $\A{\a, C}$.
    \end{enumerate}
\end{proposition}

\begin{proof} 
    ``\ref{item:a} $\Rightarrow$ \ref{item:b}''.  
    We denote $\T=\Tt$.
    We assume that \Cref{item:a}  holds and we need to show that, for every $(b_0, \dots, b_{s-1}) \in \fq^s$, there exist exactly $q^{t-s}$ rows of $\ZA{\a,I}$ equal to $(b_0, \dots, b_{s-1})$, i.e. there exist exactly $q^{t-s}$  elements $i \in [0, q^t-1]$ such that 
    \begin{equation}
        \label{equation:RowEqualToBNTuple}
        \left( \ZA{\a, I}_{i,0}, \dots, \ZA{\a, I}_{i, {s-1}} \right) = (b_0, \dots, b_{s-1}).
    \end{equation}
    From \Cref{definition:cyclicAlphaSArray}, for all $(i,j) \in [0,q^t-1]\times [0,s-1]$ we have that 
    \[
        \ZA{\a}_{i,j}=
        \begin{cases}
            \T(\a^{i}\a^{i_j}), & \text{ if } i \in [0, q^t-2], j \in [0,s-1];\\
            0=\T(0\cdot\a^{i_j}), & \text{ if } i = q^t-1, j \in [0,s-1].
        \end{cases}
    \]
    Furthermore, since $\a$ is a primitive element of $\fqt$, we have that 
    \[ \fqt= \left\{ \a^i \mid i \in [0, q^t-2] \right\} \cup \left\{ 0 \right\}.\]
    Therefore, we equivalently need to show that, for every $(b_0, \dots, b_{s-1}) \in \fq^s$, there exist exactly $q^{t-s}$ elements $x \in \fqt$ such that 
    \begin{equation}\label{equation:TraceNTupleXEqualsBNtuple}
        \left( \T(x\a^{i_0}), \dots, \T(x\a^{i_{s-1}}) \right)
        = (b_0, \dots, b_{s-1}).
    \end{equation}
    We consider the mapping 
    \begin{align}\label{equation:RaaphorstMappingProof}
        \nonumber
        \varphi &: \fqt \longrightarrow \fq^s\\ 
            x   &\mapsto \left( \T(x\a^{i_0}), \dots, \T(x\a^{i_{s-1}}) \right).
    \end{align}
    From the linearity of the trace over $\fq$, it follows that $\varphi$ is a linear transformation between vector spaces over $\fq$.
    We show that $\varphi$ is surjective; assume by means of contradiction that this is not the case.
    Then $\im(\varphi)$ is a proper subspace of $\fq^s$ and therefore 
    $\im(\varphi)^{\perp} \neq \{ \bm{0} \}$.
    Let $\mathbf{d}=(d_0, \dots, d_{s-1}) \in \im(\varphi)^{\perp} \setminus \{ \bm{0} \}$, so we have that 
    \begin{align}
        \label{equation:DualElementTimesPhiIsZero}
        0 
        = \mathbf{d}\cdot\varphi(x) 
        = \sum_{k=0}^{s-1} d_k \T(x\a^{i_{k}})
        = \T\left(x\sum_{k=0}^{s-1}i_k \a^{i_k} \right), \quad \text{ for all } x \in \fqt.
    \end{align}
    \Cref{equation:DualElementTimesPhiIsZero} along with \Cref{lemma:TraceBetaXIsZeroForAllXIffBetaIsZero} implies that $\sum_{i=0}^{s-1}d_i\a^{s_i}=0$.
    Since $\mathbf{d} \neq \bm{0}$, this contradicts the assumption that $\a^{i_0}, \dots, \a^{i_{s-1}}$ are linearly independent, and completes the proof that $\varphi$ is surjective.
    It follows that the rank of $\varphi$ is $s$ and, hence, for all $(b_0, \dots, b_{s-1}) \in \fq^s$, there exist exactly $q^{t-s}$ elements $x \in \fq $ such that $\varphi(x) = (b_0, \dots, b_{s-1})$ or, equivalently, such that \Cref{equation:TraceNTupleXEqualsBNtuple} holds, as we needed to prove.
    \todoii
    {Daniel's comment: Not clear from surjectivity of how to get exact split and $q^{t-s}$ values. I think it needs the trace perfect split, right?}
    {Daniel's comment haven't looked yet.}
    
    ``\ref{item:b} $\Rightarrow$ \ref{item:a}'' 
    We assume that $\Cref{item:b}$ holds, that is, the columns of $\ZA{\a, I}$ are uniformly $q^{t-s}$-covered.
    From the previous part of the proof, this implies that $\varphi$ is surjective, hence we have $\im(\varphi) = \fq^s$ and, therefore, $\im(\varphi)^{\perp} = \{ \bm{0} \}$.
    Now, we assume by means of contradiction that $\{ \a^{i_j} \mid j \in [0,s-1] \}$ is linearly dependent over $\fq$.
    Then, there exist $d_0, \dots, d_{s-1} \in \fq$ not all zero, such that $\sum_{j=0}^{s-1}d_i \a^{i_j}=0$, which implies that $\sum_{j=0}^{s-1}d_i \a^{i_j} x = 0$ for all $x \in \fqt$.
    From the linearity of the trace over $\fq$ we have that, for all $x \in \fqt$,
    \begin{align*}
        0 &= \T(0) = \T\left( \sum_{j=0}^{s-1} d_j \a^{i_j} x \right)
           = \sum_{j=0}^{s-1} d_j \T\left(x\a^{i_j}\right)
        \\&= (d_0, \dots, d_{s-1}) \cdot (\T(x\a^{i_0}), \dots, \T(x\a^{i_{s-1}}))
        \\&= (d_0, \dots, d_{s-1}) \cdot \varphi(x).
    \end{align*}
    The above means that $(d_0, \dots, d_{s-1}) \in \im(\varphi)^{\perp}=\{ \bm{0} \}$, which contradicts the assumption that $d_0, \dots, d_{s-1}$ are not all zero.

    ``\ref{item:a} $\Leftrightarrow$ \ref{item:c}''
    We prove the contrapositive of this equivalence.
    We assume that the set in \Cref{item:a} is linearly dependent, i.e. there exist $d_0, \dots, d_{s-1} \in \fq$ not all zero, such that $\sum_{j=0}^{s-1}d_j\a^{i_j}=0$.
    By \Cref{lemma:TraceBetaXIsZeroForAllXIffBetaIsZero}, this is equivalent to $\T( x\sum_{j=0}^{s-1}d_j\a^{i_j} )=0$, for all $x \in \fqt$.
    Since $\a$ is a primitive element, the latter holds if and only if $\T( \a^i\sum_{j=0}^{s-1}d_j\a^{i_j})=0$, for all $i \in [0,q^t-2]$.
    From the linearity of the trace, as per \Cref{theorem:PropertiesOfTrace}, this is equivalent to 
    \begin{equation}\label{equation:lqlkjoijojis}
        \sum_{j=0}^{s-1}d_j \T(\a^{i_j+i})=0, \text{ for all } i \in [0,q^t-2].
    \end{equation}
    However, by \Cref{definition:LeftShiftOperator} we have that 
    \[
          \left( \T(\a^{i_j+i}) \right)_{i=0}^{q^t-2}
         =L_{q^t-1}^{i_j}\left(  \left( \T(\a^i) \right)_{i\geq 0} \right)
         =L_{q^t-1}^{i_j}\left(\seq{\a}\right),
    \]
    hence \Cref{equation:lqlkjoijojis} is equivalent to 
    \[ \sum_{j=0}^{s-1}d_jL_{q^t-1}^{i_j}\left(\seq{\a}\right)=0, \]
    i.e. the set in \Cref{item:c} is linearly dependent.

    ``\ref{item:a} $\Leftrightarrow$ \ref{item:d}''
    We show the contrapositive of this equivalence.
    We have that $\{ \a^{i_j} \mid j \in [0,s-1] \}$ is linearly dependent over $\fq$ if and only if there exist $d_0, \dots, d_{s-1} \in \fq$, not all zero, satisfying $\sum_{j=0}^{s-1} d_j \a^{i_j}=0$.
    From \Cref{lemma:AlphaRootOfPolynomialIffItIsDivisibleByMinimal} the latter holds if and only if there exists a nonzero polynomial $f(x) = \sum_{j=0}^{s-1}d_j x^{i_j} \in \fqx$ such that $m_{\a}(x) | f(x)$, where $m_{\a}$ is the minimal polynomial of $\a$.

    ``\ref{item:a} $\Leftrightarrow$ \ref{item:e}''
    We recall that in view of \Cref{proposition:ConstructionOfPGDQFromFF},
    $\{ [\a^i] \mid i \in [0,q^t-2] \}$
    is the set of points of $PG(t-1,q)$.
    Then, the proof for this part is a direct consequence of \Cref{lemma:IndependenceAndGeneralPosition} and \Cref{remark:ArcsTracksCapsAndLI}.

    ``\ref{item:b} $\Rightarrow$ \ref{item:f}''
    We assume that $s=t$ and \Cref{item:b} holds.
    Then, the rows of the $q^t\times t$ array $\ZA{\a, C}$ are precisely all the vectors of $\fq^t$, without repetitions.
    Since $\ZA{\a, C}$ is constructed by adding a row of zeros to $\A{\a, C}$,
    \Cref{item:b} implies that the rows of $\A{\a, C}$ are precisely all the nonzero vectors of $\fq^t$, hence \Cref{item:f} holds.

    ``\ref{item:f} $\Rightarrow$ \ref{item:b}''
    We assume that \Cref{item:f} holds.
    Then, there exists no $i \in [0,q^t-2]$ such that 
    \[
        \left( \T(\a^i\a^{i_0}), \ldots \T(\a^i\a^{i_{t-1}}) \right)
        = (0,\dots, 0).
    \]
    Since $\a$ is primitive, equivalently, there is no $x \in \fqtstar$ such that 
    \[
         \left( \T(x\a^{i_0}), \ldots \T(x\a^{i_{t-1}}) \right) 
        =(0,\dots, 0).
    \]
    We consider again the mapping 
    \[
    \begin{array}[]{ll}
    \varphi: &\fqt \longrightarrow \fq^t\\
    &x\mapsto 
    \left( \T(x\a^{i_0}), \dots, \T(x\a^{i_{t-1}}) \right).
    \end{array}
    \]
    From the above, there is no $x \in \fqstar$ such that $\varphi(x)=\bm{0} \in \fq^t$, hence $\ker(\varphi)=\{0\}$, which implies that $\varphi$ is injective, and hence also surjective.
    As we have argued in the proof for \ref{item:a} $\Rightarrow$ \ref{item:b}, the surjectivity of $\varphi$ implies that \Cref{item:b} holds.
\end{proof}

\begin{example}
{}
{DemonstrationOfSebastianExtended}
    Let $\a$ be a root of the primitive polynomial $x^3+2x+1\in\fthree[x]$, so that $\a$ is a primitive element in $\f_{3^3}$, and let $C=\{ 6,8,12 \}$.
    Considering \Cref{example:ProjectiveSpacePG23}, we have that the set $\{ \a^c\mid c\in C \}= \{ \a^6, \a^8, \a^{12}\}$ is linearly independent, therefore \Cref{item:a} holds for this case.
    We also observe that $\ZA{\a, C}=\ZA{\a,\{ 6,8,12 \}}$ which we show in \Cref{table:fullcyclicarray} is an $\OA(3, 3,3)$, as per \Cref{item:b}.
    Furthermore, the set $\{L^{c}_{3^3-1}\seq{\a}) \mid c \in C\}$ is the set of columns of $\A{\a, C}$ that is also shown in \Cref{table:fullcyclicarray}, and is also linearly independent, so \Cref{item:c} holds.
    Finally, the points $\a^{c}$, $c\in C=\{ 6,8,12 \}$ are the points $a,b$ and $e$ in \Cref{figure:ProjectiveSpacePG23}, which are not contained in the same $1$-flat (line), as per \Cref{item:e}.
\end{example}

In \Cref{theorem:EquivalenceOAandNTsetsExtended} that follows, we give several necessary and sufficient conditions for a cyclic trace array to be a linear orthogonal array of a given strength. 
The proof relies on \Cref{proposition:SebastianExtended} and the theorem is an extension of \Cref{theorem:OANecessaryAndSufficientConditionRelatedToLI} applied to cyclic trace arrays.  First, we need to introduce the following notion.

\begin{definition}{$(k,s)$-set}{NTSet}
\index{Set!$(k,s)$-set}
    Let $k,s$ be positive integers such that $s\leq k$, and $S$ be a set of $k$ elements from a vector space over a field $F$.
    If every subset of $S$ of size $s$ is linearly independent, then $S$ is a \emph{$(k,s)$-set over $F$}, or simply a $(k,s)$-set when it is clear from the context that the underlying field is $F$.
\end{definition}

\begin{theorem}
{}
{EquivalenceOAandNTsetsExtended}
Let $t,s$ be positive integers, $\a$ be a primitive element of $\fqt$ with minimal polynomial $m_{\a}$, and $C$ be a nonempty subset of $[0,q^t-2]$.
    Then, the following statements are equivalent.
    \begin{enumerate}
        \item \label{item:C}
          The $q^t\times |C|$ array $\ZA{\a, C}$ is a linear $\OA_{q^{t-s}}(s, |C|,q)$.
        \item \label{item:A}
            The set $\{ \a^c \mid c\in C \}$ is a $(|C|,s)$-set over $\fq$.
        \item \label{item:B}
              The set $\{ L_{q^t-1}^{c}(\seq{\a}) \mid c \in C \}$ is a $(|C|,s)$-set over $\fq$.
        \item \label{item:D}
            For every $c_0, \dots, c_{s-1} \in C$ and $d_0, \dots, d_{s-1} \in \fq$ not all zero, $m_{\a}(x)$ does not divide $\sum_{j=0}^{s-1}d_jx^{c_j}$.
        \item \label{item:E}
            For every $s$ points in $\{ [\a^c] \mid c \in C \}$, there is no $(s-2)$-flat in $PG(t-1,q)$ that contains them.
    \end{enumerate}
\end{theorem}

\begin{proof}
    In the following, we index the columns of $\ZA{\a, C}$ by the elements of $C$ and the columns of $\ZA{\a}$ by $[0,t-1]$.

    ``\ref{item:C} $\Leftrightarrow$ \ref{item:A}''
    We have that $\ZA{\a, C}$ is a linear $\OA_{q^{t-s}}(s, k,q)$ if and only if any $s$-set of its columns is uniformly $q^{s-2}$-covered.
    In other words, if and only if for any $I\subseteq C$ with $|I|=s$, the $q^t\times s$ array $\ZA{\a,I}$ is an $\OA_{q^{t-s}}(s, s,q)$.
    We conclude that $\ZA{\a, C}$ is an $\OA_{q^{t-s}}(s, |C|,q)$ if and only if for every $I\subseteq C$ with $|I|=t$ we have that 
    $\{\a^i \mid i \in I\}$ is linearly independent.
    However, the second statement is the definition of $\{\a^c\mid c\in C\}$ being a $(|C|,s)$-set, hence we have proven the equivalence of the first two statements.
    \todoii
    {Daniel: Problems: we need to prove 2 (not 1) and we need to prove 1 (not 2)}
    {Review Daniel's comment}

%
    ``\ref{item:C} $\Leftrightarrow$ \ref{item:D}''
    We show the contrapositive of this equivalence.
    From the equivalence of the first two statements, we have that \Cref{item:C} does not hold if and only if there exists a subset $\{ c_0, \dots, c_{s-1} \}\subseteq C$ so that $\{ \a^{c_j} \mid j\in [0,s-1] \}$ is linearly dependent.
    From \Cref{item:A,item:D} of \Cref{proposition:SebastianExtended}, this is true if and only if there also exist $d_0, \dots, d_{s-1} \in \fq$ not all zero such that $m_{\a}(x)$ divides $\sum_{j=0}^{s-1}d_jx^{c_j}$, which is the negation of \Cref{item:D}.

    Finally, the equivalence of \Cref{item:A} with \Cref{item:B,item:D,item:E} is a straightforward implication of the equivalence of the statements with the same numbers in \Cref{proposition:SebastianExtended}.
\end{proof}

\begin{remark}
{}
{WhyWeChoseColumnsFrom0W}
    For a primitive element $\a \in \fqt$ and integers $i,j$, \Cref{lemma:CharacterizationOfConstantMultiplesInFQM} states that $\a^i$ and $\a^j$ are constant multiples of each other if and only if $i\equiv j \Mod{\w{t}}$.
    In other words, for every $j \in [0,q^t-2]$ there exists some $i \in [0,\w{t}-1]$ such that $\a^i$ and $\a^j$ are linearly dependent.
    It follows from \Cref{item:a} of \Cref{proposition:SebastianExtended} and \Cref{theorem:EquivalenceOAandNTsetsExtended} that, for any $C$ that is a subset of $[0,q^t-2]$, there exists some $C'$ that is a subset of $[0,\w{t}-1]$ such that $\ZA{\a, C}$ and $\ZA{\a, C'}$ have the same number of column subsets with the orthogonal array property.
    Therefore, for the construction of an orthogonal array of the form $\ZA{\a, C}$, without loss of generality the set $ C$ can be considered to be a subset of $[0,\w{t}-1]$.
    This also justifies the definition of the principal cyclic trace array in \Cref{definition:cyclicAlphaSArray}.
\end{remark}

\section{Orthogonal arrays from cyclic trace arrays}
\label{section:ConstructionOfOrthogonalArraysFromCyclicTraceArrays}

In this section we present two classic orthogonal array constructions, as well as several more recent orthogonal and covering array constructions.
All these results are adapted in the context of the previous section; in particular, the orthogonal and covering arrays are defined using the notation of \Cref{definition:cyclicAlphaSArray}, and the results are justified using \Cref{theorem:EquivalenceOAandNTsetsExtended}.
We note that every proof that we present is one of several equivalent proofs that result by focusing on different statements of the theorem.

\subsection{Two classic orthogonal array constructions}
\label{section:TwoClassicOAConstructions}

In the next corollary we show that a principal cyclic trace array is equivalent to the Rao-Hamming construction that we discuss in \Cref{section:OAsAndCAsCurrentStateOfResearch}.  Rao was the first to study these arrays \cite{rao1946hypercubes,rao1947factorial,rao1949class} whose rows are exactly the words of the dual of $q$-ary Hamming codes \cite{hedayat2012orthogonal}, thus the references in the name.
An in-depth discussion of this result that includes its history and several different constructions is given in \cite[Section 3.4]{hedayat2012orthogonal}.

\begin{corollary}
{The Rao-Hamming construction using maximal sequences}
{RaoHammingOAs}
\index{Orthogonal array!Rao-Hamming}
    Let $t$ be a positive integer and $\a$ be a primitive element of $\fqt$.  Then $\ZA{\a}$ is an $\OA_{q^{t-2}}(2, \w{t},q)$.
    Furthermore, there is no $C\subset [0,q^t-2]$ with $|C|>\w{t}$ such that $\ZA{\a, C}$ is a linear orthogonal array of strength 2.
\end{corollary}
\begin{proof}
    By \Cref{proposition:ConstructionOfPGDQFromFF}, the points $[\a^i]$, $i \in [0,\w{t}-1]$ are exactly all the points in $PG(t-1,q)$, which means that every two of them are distinct.
    In other words, any two points do not belong in the same $0$-flat.
    Then, by \Cref{item:b,item:e} of \Cref{theorem:EquivalenceOAandNTsetsExtended}, we have that $\ZA{\a,[0,\w{t}-1]}=\ZA{\a}$ is an $\OA_{q^{t-2}}(2, \w{t},q)$.

    Now, we assume by means of contradiction that there exists $C \subset [0,q^{t}-2]$ such that $|C|>\w{t}$ and $\ZA{\a, C}$ is a linear orthogonal array of strength $2$.
    Then, the points $[\a^c]$, $c \in C$ are all distinct elements of $PG(t-1,q)$ which contradicts the fact that $PG(t-1,q)$ contains exactly $\w{t}$ elements.
\end{proof}

The Rao-Hamming construction using maximal sequences is the cornerstone of many previously established results.
In \cite{munemasa1998orthogonal,dewar2007division,panario2012divisibility} the authors consider subarrays of Rao-Hamming orthogonal arrays, generated by maximal sequences whose minimal polynomials generated certain conditions, and push the strength to $3$ (or ``nearly'' $3$ in the case of \cite{munemasa1998orthogonal}, as we discuss later). 
In \cite{raaphorst2013variable} (see also \cite{raaphorst2014construction}) Raaphorst et al.\ explore in more depth the properties of Rao-Hamming arrays using maximal sequences and use them as building blocks to construct covering arrays of strength $3$. 
We discuss these results in detail in the subsequent sections.

Next, we show a construction for orthogonal arrays of strength $3$.
To formulate this construction in terms of cyclic trace arrays, we rely on a theorem due to Ebert \cite{ebert1985partitioning}.
\begin{theorem}
{Partition of finite projective spaces into caps \cite[Theorem 3]{ebert1985partitioning}}
{EbertMainTheorem}
\index{Cap}
    Let $q$ be a prime power, $t$ be a positive integer divisible by $4$, and $\a$ be a primitive element of $\fqt$.
    We define 
    \[
        \Omega_i = \left\{ \a^{i+ j \w{t/2}} \mid j \in [0,q^{t/2}] \right\}, \quad i \in [0,\w{t/2}-1].
    \]
    Then, for every $i \in [0,\w{t/2}-1]$ we have that the set $\Omega_i$ is a $(q^{t/2}+1)$-cap and the points in $PG(t-1,q)$ are partitioned by
    \[
        PG(t-1,q)= \dot{\bigcup}_{i=0}^{\w{t/2}-1}\Omega_i.
    \]
\end{theorem}

\begin{corollary}
{Orthogonal arrays of strength $3$ from maximal sequences}
{OAsOfStrength3FromMSequences}
    Let $q$ be a prime power, $t$ be a positive integer divisible by $4$, and $\a$ be a primitive element of $\fqt$. 
    For 
    \[C = \left\{ j \w{t/2} \mid j \in  [ 0, q^{t/2}] \right\},\] we have that $\ZA{\a, C}$ is an $\OA_{q^{t-3}}(3, q^{t/2}+1,q)$.
\end{corollary}

\begin{proof}
    We observe that $\{ \a^c \mid c \in C \}= \Omega_0$, where $\Omega_0$ is defined in \Cref{theorem:EquivalenceOAandNTsetsExtended}.
    By the same theorem, we have that $\Omega_0$ is a $(q^{\w{t/2}+1})$-cap, which means that no $3$ points in $\Omega_0$ are collinear. 
    In other words, there is no $1$-flat in $PG(t-1,q)$ that contains any $3$ points from $\Omega_0$, so the proof is complete by \Cref{item:C,item:E} of \Cref{theorem:EquivalenceOAandNTsetsExtended}.
\end{proof}

\index{Ovoid}
\index{Projective space!ovoid}
The special case of \Cref{corollary:OAsOfStrength3FromMSequences} for $t=4$ is of particular interest.
A $k$-cap in $PG(3,q)$ is an \emph{ovoid}.
Early on, Bose \cite{bose1947mathematical} showed that the maximum size of an ovoid is $q^2+1$ when $q$ is odd, and soon after Qvist \cite{qvist1952some} showed that the same is true for even $q$.
It follows that the $\OA_q(3,q^2+1,q)$ obtained in \Cref{corollary:OAsOfStrength3FromMSequences} for $t=4$ has the maximum number of columns among linear orthogonal arrays.
We refer to \cite[Section 5.9]{hedayat2012orthogonal} for further discussion about this orthogonal array, and to \cite[Section 4]{hirschfeld2001packing} and \cite[Section VII.2.11]{colbourn2006handbook} for details about the problem of finding maximum $k$-caps in $PG(d,q)$.

\subsection{Orthogonal arrays from polynomials with few nonzero terms}

In this section we present three constructions of orthogonal arrays that are of the form $\ZA{\a, [0,2t]}$, where $\a$ is a primitive element of $\fqt$ with a minimal polynomial that satisfies certain divisibility properties. 

\subsubsection{Divisibility of trinomials by trinomials and orthogonal arrays of strength close to 3 in $\ftwo$}
A \emph{trinomial}
\index{Trinomial}
is a polynomial with exactly three nonzero coefficients, i.e. a polynomial of the form $ax^t+bx^l+cx^k$, where $a,b,c$ are nonzero.
The next theorem due to Munemasa \cite{munemasa1998orthogonal}, characterizes all the polynomials of degree up to $2t$ over $\ftwo$ that are divisible by trinomials of degree $t$ that satisfy certain conditions.

\begin{theorem}
{{\cite[Corollary 3.3]{munemasa1998orthogonal}}}
{MunemasaTrinomialDivision}
    Let $f(x)=x^t+x^l+1$ be a primitive trinomial over $\ftwo$ with $t\geq 4$.
    If $g(x)$ is a trinomial of degree at most $2t$ divisible by $f(x)$, then $g(x)=f(x)^2$ or $g(x)=x^{t-\deg(g)}f(x)$.
\end{theorem}
Using this theorem, Munemasa constructs orthogonal arrays of strength $2$ that are close to being strength $3$, in the sense that most but not all triples of columns have the orthogonal array property. 
To formalize this notion and give Munemasa's result, we need the following definition.

\begin{definition}
{}
{LambdaB}
    Let $A$ be a $n\times k$ array with elements $A_{ij}$ from a set $V$ with $|V|=v$.
    For an ordered set $J=\{ j_0, \dots, j_{s-1} \} \subseteq [0,k-1]$ and $s$-tuple $\bfb=(b_0, \dots, b_{s-1})\in V^s$, we define
    \[
        \lambda_{\bfb}^{J}(A)
        =\left| \left\{ i \mid (A_{ij_0}, \dots, A_{i{j_{s-1}}})= (b_0, \dots, b_{s-1}),\; i\in [0,n-1] \right\} \right|.
    \]
\end{definition}

\begin{theorem}
{\cite[Main Theorem]{munemasa1998orthogonal}}
{MunemasaOAConstruction}
    Let $t$ be a positive integer, $\a$ be a primitive element of $\ftwot$ with minimal polynomial $m_{\a}(x)=x^t+x^l+1$, and define
    \[ U= \left\{ (i,i+l,i+t) \mid i \in [0,t]\right \} \cup \left\{ (0,2l,2t) \right\}.\]
    Then, for any $J=\{ j_0, j_1, j_2 \} \subseteq [0,2t]$ and $(b_0, b_1, b_2) \in \ftwo^3$, we have that
    \[
        \lambda_{b}^J(\ZA{\a,[0,2t]})=
        \begin{cases}
            2^{t-3} & \text{ if } J \not\in U; \\
            2^{t-2} & \text{ if } J \in U \text{ and } b_1+b_2+b_3=0; \\
            0       & \text{ if } J \in U \text{ and } b_1+b_2+b_3\neq 0.
        \end{cases}
    \]
\end{theorem}
\begin{proof}
    In this proof we denote $\T=\T_{2^3/2}$.
    First we consider the case when $t\geq 4$.
    If $J\not \in U$, then by \Cref{theorem:MunemasaTrinomialDivision}, the polynomial $g(x)=x^{j_0}+x^{j_1}+x^{j_2}$ is not divisible by $m_{\a}(x)$.
    By Statements 3 and 4 of \Cref{proposition:SebastianExtended} this means that the column vectors of $\ZA{\a}$ with indexes in $J$ are uniformly $2^{t-3}$-covered, so in particular $\bfb=(b_0,b_1,b_2)$ appears at those columns in exactly $q^{t-3}$ rows; in other words, $\lambda_{\bfb}^{J}(\ZA{\a,[0,2t]})=2^{t-3}$.

    If $J\in U$, then $m_{\a}(x)$ divides $x^{j_0}+x^{j_1}+x^{j_2}$ which, by \Cref{lemma:AlphaRootOfPolynomialIffItIsDivisibleByMinimal}, means that $\a^{j_0}+\a^{j_1}+\a^{j_2}=0$, or
    \begin{equation}
        \label{equation:sjslsksjs}
        \a^{j_2}=\a^{j_0}+\a^{j_1}.
    \end{equation}
    We note that since $j_0, j_1< \w{t}$, the elements $\a^{j_0}, \a^{j_1}$ are linearly independent over $\ftwo$, hence, by \Cref{proposition:SebastianExtended} the columns of $\ZA{\a,[0,2t]}$ are uniformly $2^{t-2}$-covered. 
    From the definition of the entries of $\ZA{\a,[0,2t]}$ this means that for every triple $(b_0,b_1,b_2) \in \ftwo^3$, there exist exactly $2^{t-2}$ elements $x \in \ftwo^3$ such that 
    \[(\T(\a^{j_0}x),\T(\a^{j_1}x))=(b_0,b_1).\]
    By \Cref{equation:sjslsksjs}, we have that 
    \[\T(\a^{j_2}x)= \T(\a^{j_0}x+\a^{j_1}x)= \T(\a^{j_0}x)+\T(\a^{j_1}x)= b_0+b_1.\] 
    We conclude that 
    \[(\T(\a^{j_0}x),\T(\a^{j_1}x),\T(\a^{j_2}x))=(b_0,b_1,b_2)\] 
    if and only if $b_2=\T(\a^{j_2}x)$, and the proof follows.
\end{proof}

\subsubsection{Divisibility of trinomials by pentanomials and orthogonal arrays of strength 3 in $\ftwo$}

A \emph{pentanomial}
\index{Pentanomial}
is a polynomial with exactly five nonzero coefficients.  Dewar et al.\ \cite{dewar2007division} extend the idea of Munemasa to pentanomials over $\ftwo$, by describing the polynomials $g(x)$ of degree up to $2t$ that are divisible by pentanomials $f(x)=x^t+x^l+x^k+x^j+1$ with $\gcd(t,l,k,j)=1$.
The following theorem is a special case of \cite[Theorem 1.2]{dewar2007division}, where $f(x)$ is primitive and $g(x)$ is a trinomial.

\begin{theorem}
{A special case of \cite[Theorem 1.2]{dewar2007division}}
{DewarPentanomialDivision}
    Let $f(x)= x^t+x^l+x^k+x^j+1\in \ftwox$ be primitive with $\gcd(t,l,k,j)=1$.
    Then, $f(x)$ divides a trinomial $g(x)$ of degree up to $2t$ if and only if $g(x)=f(x)h(x)$ with $f(x)$ and $h(x)$ as shown in the table below.
\[
    \renewcommand{\arraystretch}{\genarraystretch}
    \rowcolors{1}{beach}{\backgroundshade}
    \begin{array}[]{ll}
        \rowcolor{\theocolor!70}
        \multicolumn{1}{c}{f(x)}&
        \multicolumn{1}{c}{h(x)}\\
        x^5 +x^4 +  x^3 + x^2 +1&   x^3  + x^2 + 1 \\
        x^5 +x^3 +  x^2 + x   +1&   x^3  + x   + 1 \\
        x^5 +x^3 +  x^2 + x   +1&   x^4  + x   + 1 \\
        x^5 +x^4 +  x^3 + x   +1&   x^2  + x   + 1 \\
        x^6 +x^4 +  x^3 + x   +1&   x^2  + x   + 1 \\
        x^7 +x^4 +  x^3 + x^2 +1&   x^3  + x^2 + 1 
    \end{array}
\]
\end{theorem}

\begin{corollary}
{}
{DewarOAConstruction}
    Let $\a$ be a primitive element of $\ftwot$ with minimal polynomial $m_{\a}(x)=x^t+x^l+x^k+x^j+1$, such that $t>l,k,j$ and $\gcd(t,l,k,j)=1$.
    If $m_{\a}(x)$ is not one of the polynomials $f(x)$ listed in \Cref{theorem:DewarPentanomialDivision}, then $\ZA{\a,[0,2t]}$ is an $\OA_{2^{t-3}}(3,2t+1,2)$.
\end{corollary}
\begin{proof}
    First, we observe that $\ZA{\a, [0,2t]}$ is an orthogonal array of strength $2$, as a subarray of $\ZA{\a}$ which has strength $2$ by \Cref{corollary:RaoHammingOAs}.
    By \Cref{item:C,item:D} of \Cref{theorem:EquivalenceOAandNTsetsExtended}, we have that for every $c_0,c_1 \in [0,2t]$, and $d_0,d_1\in \ftwo$ not all zero, $m_{\a}(x)$ does not divide $d_0x^{c_0}+d_{1}x^{c_1}$.
    Furthermore, since $m_{\a}(x)$ is not one of the polynomials listed in the table, it does not divide any trinomials of degree up to $2t$, i.e. for any $c_0,c_1,c_2 \in [0,2t]$, we have that $m_{\a}(x)$ does not divide $x^{c_0}+x^{c_1}+x^{c_2}$.

    From the above, we conclude that for every $c_0,c_1,c_2\in [0,2t]$ and $d_0,d_1,d_2\in \ftwo$ not all zero, we have that $m_{\a}(x)$ does not divide $d_0x^{c_0}+d_1x^{c_1}+d_2x^{c_2}$ and hence, by \Cref{item:C,item:D} of \Cref{proposition:SebastianExtended}, we have that $\ZA{\a,[0,2t]}$ is an $\OA_{2^{t-3}}(3, 2t+1,2)$.
\end{proof}

\subsubsection{Divisibility of trinomials by trinomials and orthogonal arrays of strength 3 in $\fthree$}
Another extension of Munemasa's work is given by Panario et al.~\cite{panario2012divisibility}, where they study the divisibility of trinomials over $\fthree$ and completely describe the trinomials that divide other trinomials of up to three times their degree, as follows.
\begin{theorem}
{\cite[Theorem 8]{panario2012divisibility}}
{OlgaTrinomialDivisibility}
    Let $f(x)=x^t+bx^k+a \in \fthree[x]$ with $t>k$, $a, b\neq 0$, and $g(x)$ be a monic trinomial of degree up to $3t$ divisible by $f(x)$.
    Then $g(x)=f(x)^3$, or $f(x),g(x)$ are among the polynomials listed below, or the reciprocals of the polynomials listed below, i.e.  $x^nf(1/x)$ and $x^ng(1/x)$.
\[
    \renewcommand{\arraystretch}{\genarraystretch}
    \rowcolors{1}{beach}{\backgroundshade}
    \begin{array}[]{ll!{\color{beach!90!black}\vrule}ll}
        \rowcolor{\theocolor!70}
        \multicolumn{1}{c}{f(x)}& \multicolumn{1}{c!{\color{beach!90!black}\vrule}}{g(x)}& \multicolumn{1}{c}{f(x)}& \multicolumn{1}{c}{g(x)}\\
        x^t + bx^{t/2} - 1 &x^{3t}  - bx^{t/2} -1 & x^t-  x^{t/3} +a &x^{8t/3} + x^{2t/3}  +1 \\
        x^t + bx^{t/2} + 1 &x^{5t/2}+  x^{t/2}  b & x^t+  x^{t/3} +a &x^{7t/3} + ax^{2t/3} +a \\
        x^t + bx^{t/2} - 1 &x^{5t/2}- bx^t      b & x^t-  x^{t/3} +a &x^{7t/3} - ax^{4t/3} +a \\
        x^t + bx^{t/2} - 1 &x^{5t/2}-  x^{3t/2}-b & x^t-  x^{t/3} +a &x^{7t/3} + x^{5t/3}  -a \\
        x^t + bx^{t/2} + 1 &x^{5t/2}+ bx^{4t/2} b & x^t+  x^{t/3} +a &x^{2t}   + ax^{5t/3} +1 \\
        x^t + bx^{t/2} + 1 &x^2t    +  x^t      1 & x^t-  x^{t/3} +a &x^{5t/3} + ax^{4t/3} +a \\
        x^t + bx^{t/2} - 1 &x^{3t/2}+  x^t      b & x^t+ bx^{t/4} -1 &x^{11t/4}+ bx^{6t/4} -b \\
        x^t + bx^{t/2} - 1 &x^{3t/2}- bx^t     -b & x^t+ bx^{t/4} +1 &x^{10t/4}+ bx^{9t/4} +1 \\
        x^t -  x^{t/3} + a &x^{3t}  -  x^{t/3} -a &                  & 
    \end{array}
\]
\end{theorem}

The following corollary is not stated explicitly in \cite{panario2012divisibility} although the application of \Cref{theorem:OlgaTrinomialDivisibility} for the construction of orthogonal arrays is discussed.
The proof is along the same lines with that of \Cref{corollary:DewarOAConstruction}.
\begin{corollary}
{}
{OlgaOAConstruction}
Let $\a$ be a primitive element of $\f_{3^t}$ with minimal polynomial $m_{\a}(x)=x^t+bx^k+a \in \fthree[x]$, where $t>k$ and $a,b \neq 0$.
    If $m_{\a}(x)$ is not among the polynomials $f(x)$ listed in \Cref{theorem:OlgaTrinomialDivisibility} or their reciprocals, then $\A{\a,[0,3t-1]}$ is an $\OA_{3^{t-3}}(3,3t,3)$.
\end{corollary}

The proof of \Cref{theorem:OlgaTrinomialDivisibility} requires examining a rather large number of cases for trinomial forms, and it is stated in \cite{panario2012divisibility} that the method becomes increasingly complicated to apply to larger finite fields.
The same method has been applied recently to study the binary tetranomials divisible by pentanomials with consecutive inner coefficients \cite{kim2016division}.
However, it is suggested in \cite{panario2012divisibility} that the next important step in this area of research would be to find a different method of studying the divisibility of polynomials of a certain weight that applies to arbitrary fields, rather than extending the method in \cite{panario2012divisibility} to other specific finite fields and polynomial weights.

\subsection{Orthogonal arrays from AMDS codes}

It follows from \Cref{theorem:LinearCodesAndLinearOAsAreEquivalent} that linear codes and orthogonal arrays are equivalent objects.
In this section we focus on a certain type of linear code and study the corresponding orthogonal arrays.
These orthogonal arrays are also relevant to \Cref{chapter:CAsFromMSequencesAndCharacterSums}, where we use them to construct an infinite family of covering arrays.

We begin with two auxiliary results related to linear codes.
The first is a connection between the minimum distance of a linear code and the sets of \Cref{definition:NTSet}.
\begin{proposition}
{}
{MinDistanceAndNTSets}
    A linear code $C$ of length $n$ and dimension $k$ has minimum distance $d$ if and only if the columns of its parity check matrix form an $(n,d-1)$-set of vectors in $\f_{q^{n-k}}$ that is not a $(n,d)$-set.
\end{proposition}
\begin{proof}
    Let $H$ be a parity check matrix of $C$.
    The columns of $H$ are a $(n,d-1)$-set that is not a $(n,d)$-set if and only if every $d-1$ columns are linearly independent, but some $d$ columns are not.
    This is equivalent to the following equation:
    \[ d= \min\left\{ x\mid \text{there exist $x$ linearly dependent columns in $H$} \right\},\]
    which is equivalent to $d$ being the minimum distance, by \Cref{proposition:MinWeightAndLinearIndependence}.
\end{proof}
The second auxiliary result gives an explicit method of using the parity check matrix of a linear code with minimum distance $d$ to construct a cyclic trace array that is an orthogonal array of strength $d-1$.

\begin{theorem}
{Orthogonal arrays from maximal sequences and linear codes}
{OAsFromMSequencesAndLinearCodes}
    Let $C$ be a linear code with parameters $[n,k,d]_q$ and parity check matrix $H=(H_{ij})$, $i\in [0,n-k-1]$, $j\in[0,n-1]$, and let $\a$ be a primitive element of $\f_{q^{n-k}}$.
    We define 
    \[C= \left\{ \log_{\a}\left( \sum_{i=0}^{k-1} H_{ij}\a^i \right) \mid j\in [0,n-1]\right\}.\]
    Then $\ZA{\a, C}$ is an $\OA_{q^{n-k-d+1}}(d-1, n,q)$ that is not an $\OA_{q^{n-k-d}}(d,n,q)$.
\end{theorem}
\begin{proof}
    Denoting the $j$-th column of $H$ by $H_j$, we have by \Cref{proposition:MinDistanceAndNTSets} that the set 
    \[S'=\left\{ H_j\mid j \in [0,n-1]\right\}\]
    is an $(n,d-1)$-set that is not an $(n,d)$-set over $\fq$.
    Consider the vector space isomorphism $\varphi$, defined by
    \[
        \begin{array}[]{ll}
        \varphi:&\fq^{n-k}\rightarrow \f_{q^{n-k}}\\
        &(x_0, \dots, x_{n-k-1})\mapsto\sum_{i=0}^{n-k-1}x_i\a^i.
        \end{array}
    \]
    Then, we have that 
    \[S=\left\{ \varphi(H_j)\mid j \in [0,n-1] \right\}=\varphi(S').\]
    This implies that $S$ is an $(n,d-1)$-set that is not an $(n,d)$-set, since this is true for $S'$ and $\varphi$ respects linear independence, as a vector space isomorphism.
    The proof is complete by Statements 1 and 3 of \Cref{theorem:EquivalenceOAandNTsetsExtended}.
\end{proof}

The following is a classic bound for linear codes.
\begin{lemma}
{Singleton bound \cite{singleton1964maximum}}
{SingletonBound}
For a $q$-ary linear $[n,k,d]$-code we have that \[n-k-d+1\geq 0.\]
\end{lemma}
\index{Singleton!bound}
\index{Singleton!defect}
\index{Linear code!MDS}
\index{Linear code!AMDS}
\index{Linear code!NMDS}
\index{AMDS|see{Linear code}}
\index{NMDS|see{Linear code}}
\index{MDS|see{Linear code}}
Linear codes with parameters that achieve or are close to equality in \Cref{lemma:SingletonBound} are of special interest.

\begin{definition}
{Singleton defect and related codes}
{SingletonDefectAndRelatedCodes}
    The \emph{Singleton defect} of a linear $[n,k,d]$-code is the number 
    \[s(C)=n-k+1-d.\]
    If $s(C)=0$, then $C$ is \emph{maximum distance separable (MDS)}.
    If $s(C)=1$, then $C$ is \emph{almost maximum distance separable (AMDS)}.
    If both $C$ and $C^{\perp}$ are AMDS, then $C$ is \emph{near maximum distance separable (NMDS)}.
\end{definition}

From \Cref{definition:SingletonDefectAndRelatedCodes} we have that a code is AMDS if it has parameters $[n,n-d,d]_q$ for some $n,d$.
Then, by \Cref{proposition:MinDistanceAndNTSets}, the columns $H_0, \dots, H_{n-1}$ of a parity check matrix form an $(n,d-1)$-set that is not an $(n,d)$-set, which by \Cref{remark:ArcsTracksCapsAndLI} is equivalent to $\{ [H_i] \mid i \in [0, n-1] \}$ being an $n$-track in $PG(d-1,q)$.
We conclude that the following are equivalent objects:
\begin{itemize}
    \item An $n$-track in $PG(d-1,q)$. \index{Track} \index{Linear code!AMDS}
    \item An AMDS code with parameters $[n,n-d,d]_q$.
    \item An $\OA_q(d-1,n,q)$.
    \label{item:computer}
\end{itemize}

Let $M(r,q)$ be the maximum number $n$ such that a maximum $n$-track exists in $PG(r,q)$.
Some exact values of $M(r, q)$ for general prime power $q$ are shown in \Cref{table:ExactValuesMDQgeneral}. 
In particular, we have that $M(2,q)=q^2+q+1$ and that $M(3,q)=q^2+1$ for $q>2$, which means that the orthogonal arrays in \Cref{corollary:RaoHammingOAs,corollary:OAsOfStrength3FromMSequences} have the maximum number of columns that can be obtained from our type of construction.
However, finding exact values for $M(r,q)$ for $r\geq 4$ is a difficult problem \cite{hirschfeld2001packing}.
In \Cref{table:ExactValuesMDQsporadic} we show the known values and bounds for small values of $r$ and $q$.
For other bounds on $M(r,q)$ that rely on bounds on the size of other objects in finite geometry, we refer the reader to the survey by Hirschfeld and Storme \cite{hirschfeld2001packing}.

Other than that, we have an infinite class of AMDS codes from elliptic curves. 
We recall that an NMDS code is an AMDS code whose dual is also AMDS.
The original rather challenging proof of the next theorem appears in \cite{tsfasman1991algebraic}; see \cite{giulietti2002near} for a more accessible proof.

\begin{theorem}
{\cite[Theorem 1.1]{giulietti2002near}}
{NMDSfromEllipticCurves}
    Let $q$ be a prime power and suppose that there exists an elliptic curve with $n$ rational points over $\fq$.
    Then, for every $k\in [2,n-1]$, there exists an $[n,k,d]_q$ NMDS code.
\end{theorem}

It is known, see for example \cite[Theorem 2.3.17]{tsfasman1991algebraic} that if $q=p^r$ then there exists and elliptic curve with $N_q$ rational points over $\fq$, where
\begin{equation}
    N_q=
    \begin{cases}
        q+\lfloor 2\sqrt{q}\rfloor,
        & \text{if $p|2\sqrt{q}$ and $r\geq 3$, $r$ odd};\\
        q+\lfloor 2\sqrt{q}\rfloor+1, & \text{otherwise}.
    \end{cases}
    \label{equation:RationalPointsNQ}
\end{equation}
Hence, from \Cref{theorem:NMDSfromEllipticCurves} there exists an AMDS (that is moreover NMDS) code with parameters $[N_q,N_q-r,r]_q$ which, from the previously mentioned equivalence of AMDS codes and orthogonal arrays, yields the following result.

\begin{corollary}
{}
{OAsFromNMDS}
    Let $q$ be a prime power and $\a$ be a primitive element of $\fqt$.
    Then there exists $C\subset [0, \w{t}-1]$ with $|C|=N_q$, such that $\ZA{\a, C}$ is an $\OA_q(t-1,N_q,q)$.
\end{corollary}

\begin{table}[t] 
\renewcommand{\arraystretch}{\genarraystretch}
\begin{subtable}{1\linewidth}\centering
{
\rowcolors{1}{\backgroundshade}{white}
    \centering
    \[
    \begin{array}{ccc}
    \rowcolor{\tableheadcolor}
        d&q&M(d,q)\\
        2   & \text{any}  &   q^2+q+1\\
        3   & q>2  &   q^2+1\\
        2q - 2& q > 3 & 2q+1\\
        2q - 1& q > 3 & 2q+2\\
        \geq 2q &\text{any} & 0
    \end{array}
    \]
    \caption{Known exact values of $M(d, q)$ \cite[Table 8.1]{hirschfeld2001packing}.}
    \label{table:ExactValuesMDQgeneral}
}
\end{subtable}

\begin{subtable}{1\linewidth}\centering
{
\rowcolors{1}{white}{\backgroundshade}
    \[
    \begin{array}{rrcccccccc}
    \rowcolor{\tableheadcolor}
    &q&2&3&4&5&7&8&9&11\\
    \rowcolor{\tableheadcolor}
    N&&&&&&&&&\\
    2&&7&13&21&31&57&73&91&133\\
    3&&8&10&17&26&50&65&82&122\\
    4&&&11&11&12-20&16-30&14-36&16-43&22-57\\
    5&&&12&12&12-14&15-31&15-37&17-44&23-58\\
    6&&&&9&10-15&13-28&14-34&17-39&18-49\\
    7&&&&10&11-16&13-20&14-35&18-40&18-50\\
    8&&&&&11&13-21&14-23&19-36&19-50\\
    9&&&&&12&13-22&14-24&20-26&20-51\\
    10&&&&&&14-23&14-25&16-27&18-44\\
    11&&&&&&15-24&15-26&16-28&18-32\\
    12&&&&&&15&15-27&16-29&18-33\\
    13&&&&&&16&16-28&17-30&18-34
    \end{array}
    \]
    \caption{Known values of $M(d, q)$ for small $d$ and $q$ \cite[Table 8.4]{hirschfeld2001packing}.}
    \label{table:ExactValuesMDQsporadic}
}
\end{subtable}
\label{table:ValuesMDQ}
\caption{Values of $M(d,q)$.}
\end{table}

\subsection{Orthogonal arrays from arcs}
\label{section:OAsFromArcs}

Every result that we have presented so far follows from \Cref{theorem:EquivalenceOAandNTsetsExtended} for some $s<t$.
In this section we briefly discuss the case $s=t$.
From the equivalence of \Cref{item:C,item:A} of \Cref{theorem:EquivalenceOAandNTsetsExtended} for $s=t$ and the definition of an arc, we have the following.
\begin{corollary}
    {Orthogonal arrays from arcs}
    {OAsFromArcs}
    \index{Arc}
    Let $\a$ be a primitive element of $\fqt$ and $C$ be a nonempty subset of $[0,q^t-2]$.
    Then, the following statements are equivalent:
    \begin{itemize}
        \item The $q^t\times |C|$ array $\ZA{\a, C}$ is a linear $\OA(t, |C|,q)$.
        \item The set $\{ [\a^c]\mid c\in C \}$ is a $|C|$-arc in $PG(t-1,q)$.
    \end{itemize}
\end{corollary}
Therefore, to obtain an orthogonal array with the maximum number of columns, we equivalently need to obtain a $k$-arc in $PG(t-1,q)$ where $k$ is the maximum possible.

Let $m(d,q)$ be the maximum $k$ for which a $k$-arc exists in $PG(d,q)$.
Determining $m(d,q)$ for various values of $d,q$ is a well-researched problem of finite geometry.
However, the currently known values of $m(d,q)$ correspond to orthogonal arrays with a very small number of columns compared to the number of rows.
In fact, a well-known conjecture \cite{hirschfeld2001packing} is that
\[
    m(d,q)=
    \begin{cases}
        d+2& \text{ if } d\geq q-1; \\
        q+2& \text{ if $q$ is even and } d\in [2,q-2];\\
        q+1& \text{ otherwise.}
    \end{cases}
    \label{fornomenclaturearc}
\]
We refer to Sections 2 and 3 of \cite{hirschfeld2001packing} for a survey on the currently known constructions and sizes of arcs, all of which support the above conjecture.
A discussion of the problem in the context of orthogonal arrays and linear codes can be found in Section 5.6 of \cite{hedayat2012orthogonal}.

\section{Covering arrays from cyclic trace arrays}
\label{section:ConstructionOfCoveringArraysFromCyclicTraceArrays}
Other than our work in the next chapters, there is only one previous construction of covering arrays that are not orthogonal arrays which uses maximal sequences.
In \cite{raaphorst2014construction} (see also \cite{raaphorst2013variable}) Raaphorst et al.\ study the combinatorial properties of arrays of the form $\ZA{\a}$ where $\a$ is a primitive element of $\fqt$.
Moreover, for the case when $t=3$, they exploit these properties to provide a method to construct covering arrays of strength $3$.
One of their most powerful results is a connection between the positions of the zero entries in the array and the blocks of a BIBD.

\begin{theorem}
{A BIBD from $\A{\a}$ \cite[Theorem 3]{raaphorst2014construction}}
{CyclicArraysIsABIBD}
\index{BIBD}
    For a primitive element $\a \in \fqt$, each row of $\ZA{\a}$ has exactly $z=\w{t-1}$ zeros, and the set
    \[ 
        \mathcal{B}=
        \left\{ 
            \left\{ a_0, \dots, a_{z-1} \right\} 
             \mid \A{\a}_{i,a_0}= \dots = \A{\a}_{i,a_{z-1}}=0
            \text{ for some } i \in [0,k-1]
        \right\}
    \] 
    is the set of blocks of a $(\w{t},\w{t-1},\w{t-2})$-BIBD.
\end{theorem}

The positions of the zero entries are also connected to another combinatorial object, as given in the next proposition.
For positive integers $n,k$ and a set $S\subseteq \Z_n$, the set $\{ k+s \bmod n \mid s \in S \}$ is a \emph{translate} of $S$.
A \emph{$(v,k,\lambda)$-difference set} is a $k$-subset $H$ of an Abelian group $G$ (written additively) of order $v$ such that for every $i\in G\setminus\{ 0 \}$, we can write $i=w-v$ for $\lambda$ distinct choices of $w,v\in H$.

\begin{proposition}
{\cite{raaphorst2014construction}}
{ConnectionOfSeqMatrixAndDifferenceSets}
    Let $q$ be a prime power, $t$ be an integer with $t\geq 2$ and $\a$ be a primitive element of $\fqt$.
    We define
    \[ H_i=\left\{ j\mid j \in [0, \w{t}-1] \text{ such that } \A{\a}_{i,j}=0 \right\}, \quad i \in [0,q^t-1].  \]
    Then, for every $i \in [0,q^t-1]$, we have that $H_i$ is a $( \w{t},\w{t-1}, \w{t-2} )$-difference set.
    Furthermore, these are all the translates of $H_0$ as a subset of $\Z_{\w{t}}$.
\end{proposition}
For the case $t=3$, the sets $H_i$ of \Cref{proposition:ConnectionOfSeqMatrixAndDifferenceSets} are $(q^2+q+1,q+1,1)$-difference sets.
Raaphorst et al.\ use properties of this type of a difference set in order to show that if a triple of columns with indexes $a,b,c \subset [0,q^2+q]$ are not covered in $\ZA{\a}$, then the columns of $\ZA{\a^{-1}}$ with indexes $a,b,c$ are covered. 
This leads to the following result.

\begin{theorem}
{\cite[Theorem 6]{raaphorst2014construction}}
{RaaphorstCAConstruction}
    Let $q$ be a prime power and $\a$ be a primitive element of $\f_{q^3}$.
    Then, the vertical concatenation of $\A{\a}$ with $\A{\a^{-1}}$, as well as a row of zeros, is a $\CA(2q^3-1; 3, q^2+q+1,q)$.
\end{theorem}

Focusing on the fact that $\A{\a^{-1}}$ consists of the columns of $\A{\a}$ in reverse order, the authors of \cite{raaphorst2014construction} considered a generalization where $\A{\a}$, with $\a$ is a primitive element of $\fqt$, $t\geq 4$, is vertically concatenated with copies of $\A{\a}$ with the columns permuted in various ways.
More precisely, search algorithms were ran to find a permutation group of smallest order $s(t,q)$, such that vertically concatenating the $s(t,q)$ permuted copies of $\A{\a}$ and adding a zero row yields a $\CA(s(t,q)(q^t-1)+1; t, \w{t},q)$.
Since $s(2,q)=1$ from \Cref{corollary:RaoHammingOAs}, and $s(3,q)=2$ from \Cref{corollary:OAsOfStrength3FromMSequences} the hope was that $s$ would be a function depending only on $t$ and not $q$, however the experiments seemed to indicate otherwise.
Moreover, in all cases that were examined for $t\geq 4$, due to the large size of $s(t,q)$, the resultant arrays did not improve the known best upper bounds for covering array numbers.
We refer to \cite[Section 4.4]{raaphorst2013variable} for details on the experiments.

\chapter{New covering arrays from maximal sequences and backtracking}
\label{chapter:CAsFirstPaper}

\textsc{In \Cref{section:ConstructionOfCoveringArraysFromCyclicTraceArrays} we discussed} a covering array construction due to Raaphorst et al.\ \cite{raaphorst2014construction} where, for a primitive element $\a\in \f_{q^3}$, the arrays $\A{\a}, \A{\a^{-1}}$ and a row of zeros are vertically concatenated to produce a covering array of strength $3$.
At the end of \Cref{chapter:Preliminaries}, we also discussed how the authors attempted to generalize this construction by considering a primitive element $\a \in \fqt$, $t\geq 4$, and the vertical concatenation of $\A{\a}$ with copies of itself with the columns permuted in various ways.
In this chapter, we consider a generalization where, instead of vertically concatenating copies of a cyclic trace array with the columns permuted, we use cyclic trace arrays corresponding to different primitive elements.
Then, we search for subarrays of that vertical concatenation with the covering array property.

The structure of this chapter is as follows.
In \Cref{section:CAsFirstPaper_Introduction} we give some preliminary definitions and we express our objectives as solving two optimization problems that we state therein.
\Cref{section:ABacktrackAlgoForProblem1} is dedicated to the first problem, which is concerned with finding covering arrays among subarrays of the vertical concatenation of cyclic trace arrays.
We give an algorithmic solution to this problem based on backtracking with several optimizations.
\Cref{section:CAsFirstPaper_ChoiceOfPrimitiveElements} is dedicated to the second problem which is about finding the optimum way of choosing the cyclic trace arrays to concatenate in the first problem.
We also give an algorithmic solution that uses finite field theory.
In \Cref{section:CAsFirstPaper_ImplementationAndNewBounds} we discuss our computer implementation of the algorithms of the previous sections and our experimental results,  that include the improvement of 38 previously best known covering array numbers.

The results of this chapter appear in \cite{tzanakis2016constructing}.

\section{Problem statement}
\label{section:CAsFirstPaper_Introduction}

We begin by generalizing the concept of a cyclic trace array.
\index{Cyclic trace array!generalization}
\index{Cyclic trace array!concatenation}
\begin{definition}
{}
{cyclicAlphaSArrayExtended}
    Let $t,k$ be positive integers, $P=\left\{ \a_0, \dots, \a_{l-1} \right\}$ be a set of primitive elements of $\fqt$, and $C=\left\{ c_0, \dots, c_{k-1} \right\}$ be a nonempty subset of $[0,q^t-2]$.
    We define the following arrays.
    \begin{itemize}
        \item
            $\AA_{q^t/q}(P,C)$ is the $l(q^t-1)\times k$ array that is the vertical concatenation of $\AA_{q^t/q}(\a_i,C)$, $i\in [0,l-1]$.
        \item
            $\AA_{q^t/q, \bm{0}}(P,C)$ is the $\left(l(q^t-1)+1\right)\times k$ array that we obtain by appending a row of zeros to $\AA_{q^t/q}(P,C)$.  
        \item
            We simply write $\AA_{q^t/q}(P)$ to denote $\AA_{q^t/q}(P,[0,\w{t}-1])$ and $\AA_{q^t/q, \bm{0}}(P)$ to denote $\AA_{q^t/q, \bm{0}}(P,[0,\w{t}-1])$.
    \end{itemize}
    When it is clear from the context that the fields are $\fqt$ and $\fq$, we simply write $\mathcal{A}$ instead of $\mathcal{A}_{q^t/q}$ in the above definitions.
\end{definition}

\Cref{definition:cyclicAlphaSArrayExtended} generalizes \Cref{definition:cyclicAlphaSArray} in the sense that, for a singleton $P=\left\{ \a \right\}$, the arrays $\A{P,C}$ (resp. $\ZA{P,C}$) and $\A{\a,C}$ (resp. $\ZA{\a,C}$) are identical.

\begin{table}
\renewcommand{\arraystretch}{0.9}
\small
\centering
\rowcolors{27}{\backgroundshade}{\backgroundshade}
\begin{tabular}{cccccccccccccc}
0&0&1&0&1&2&1&1&2&0&1&1&1 &\\
0&1&0&1&2&1&1&2&0&1&1&1&0 &\\
1&0&1&2&1&1&2&0&1&1&1&0&0 &\\
0&1&2&1&1&2&0&1&1&1&0&0&2 &\\
1&2&1&1&2&0&1&1&1&0&0&2&0 &\\
2&1&1&2&0&1&1&1&0&0&2&0&2 &\\
1&1&2&0&1&1&1&0&0&2&0&2&1 &\\
1&2&0&1&1&1&0&0&2&0&2&1&2 &\\
2&0&1&1&1&0&0&2&0&2&1&2&2 &\\
0&1&1&1&0&0&2&0&2&1&2&2&1 &\\
1&1&1&0&0&2&0&2&1&2&2&1&0 &\\
1&1&0&0&2&0&2&1&2&2&1&0&2 &\\
1&0&0&2&0&2&1&2&2&1&0&2&2 &\\
0&0&2&0&2&1&2&2&1&0&2&2&2 & $\A{\a}$\\
0&2&0&2&1&2&2&1&0&2&2&2&0 &\\
2&0&2&1&2&2&1&0&2&2&2&0&0 &\\
0&2&1&2&2&1&0&2&2&2&0&0&1 &\\
2&1&2&2&1&0&2&2&2&0&0&1&0 &\\
1&2&2&1&0&2&2&2&0&0&1&0&1 &\\
2&2&1&0&2&2&2&0&0&1&0&1&2 &\\
2&1&0&2&2&2&0&0&1&0&1&2&1 &\\
1&0&2&2&2&0&0&1&0&1&2&1&1 &\\
0&2&2&2&0&0&1&0&1&2&1&1&2 &\\
2&2&2&0&0&1&0&1&2&1&1&2&0 &\\
2&2&0&0&1&0&1&2&1&1&2&0&1 &\\
2&0&0&1&0&1&2&1&1&2&0&1&1 &\\
0&2&1&2&2&2&1&0&0&2&2&0&2 &\cellcolor{white}\\
2&1&2&2&2&1&0&0&2&2&0&2&0 &\cellcolor{white}\\
1&2&2&2&1&0&0&2&2&0&2&0&1 &\cellcolor{white}\\
2&2&2&1&0&0&2&2&0&2&0&1&2 &\cellcolor{white}\\
2&2&1&0&0&2&2&0&2&0&1&2&1 &\cellcolor{white}\\
2&1&0&0&2&2&0&2&0&1&2&1&1 &\cellcolor{white}\\
1&0&0&2&2&0&2&0&1&2&1&1&1 &\cellcolor{white}\\
0&0&2&2&0&2&0&1&2&1&1&1&2 &\cellcolor{white}\\
0&2&2&0&2&0&1&2&1&1&1&2&0 &\cellcolor{white}\\
2&2&0&2&0&1&2&1&1&1&2&0&0 &\cellcolor{white}\\
2&0&2&0&1&2&1&1&1&2&0&0&1 &\cellcolor{white}\\
0&2&0&1&2&1&1&1&2&0&0&1&1 &\cellcolor{white}\\
2&0&1&2&1&1&1&2&0&0&1&1&0 &\cellcolor{white}\\
0&1&2&1&1&1&2&0&0&1&1&0&0 &\cellcolor{white} $\A{\a^5}$\\
1&2&1&1&1&2&0&0&1&1&0&0&2 &\cellcolor{white}\\
2&1&1&1&2&0&0&1&1&0&0&2&1 &\cellcolor{white}\\
1&1&1&2&0&0&1&1&0&0&2&1&2 &\cellcolor{white}\\
1&1&2&0&0&1&1&0&0&2&1&2&2 &\cellcolor{white}\\
1&2&0&0&1&1&0&0&2&1&2&2&2 &\cellcolor{white}\\
2&0&0&1&1&0&0&2&1&2&2&2&1 &\cellcolor{white}\\
0&0&1&1&0&0&2&1&2&2&2&1&0 &\cellcolor{white}\\
0&1&1&0&0&2&1&2&2&2&1&0&0 &\cellcolor{white}\\
1&1&0&0&2&1&2&2&2&1&0&0&2 &\cellcolor{white}\\
1&0&0&2&1&2&2&2&1&0&0&2&2 &\cellcolor{white}\\
0&0&2&1&2&2&2&1&0&0&2&2&0 &\cellcolor{white}\\
\rowcolor{white}
0&0&0&0&0&0&0&0&0&0&0&0&0 &$\mathbf{0}$
\end{tabular}
\caption[Example of cyclic trace array corresponding to two elements]{The array $\ZA{\{\a,\a^5\}}$ described in \Cref{example:GeneralizedCyclicTraceArray}.}
\label{table:GeneralizedCyclicTraceArray}
\end{table}


\begin{remark}
{}
{ElementsOfCyclicTraceArray}
    Let $(i,j)\in [0,l(q^t-1)-1] \times [0,k-1]$.
    Then, there exists $m \in [0,l-1]$ such that $m(q^t-1) \leq i < (m+1)(q^t-1)$, and it follows from \Cref{definition:cyclicAlphaSArray,definition:cyclicAlphaSArrayExtended} and the fact that $\a_m^{q^m-1}=1$, that the $(i,j)$-th element of $\A{P,C}$ is given by
    $\A{P,C}_{ij} = \A{\a_m,C}_{ij} = \Tt(\a_m^{i+c_j}).$
\end{remark}

\begin{example}
{}
{GeneralizedCyclicTraceArray}
    For a root $\a$ of the primitive polynomial $x^3+2x+1$ over $\fthree$, we compute
    \[\seq{\a}=00101211201110020212210222.\]
    We have already presented $\ZA{\a}$ in \Cref{table:fullcyclicarray}.
    Now, we observe that $\a^5$ is also a primitive element of $\f_{3^3}$, since the order of $\f_{3^3}^{\times}$ is $3^3-1=26$ and $\gcd(5,26)=1$.
    We have
    \[\seq{\a^5}= 021222100220201211120011010,\]
    and $\ZA{ \left\{ \a,\a^5 \right\}}$ is as shown in \Cref{table:GeneralizedCyclicTraceArray}.
\end{example}

\begin{table}
\rowcolors{2}{\backgroundshade}{white}
\renewcommand{\arraystretch}{1.2}

\[
\begin{array}{llll}
\rowcolor{\tableheadcolor}
\multicolumn{1}{c}{P}&
\multicolumn{1}{c}{C}&
\multicolumn{1}{c}{\ZA{P,C}}&
\multicolumn{1}{c}{\text{Reference}}
\\
\{ \a \}, \a \in \fqt, t\geq 2  
&   [0,\w{t}-1] 
&   \OA_{q^{t-2}}(2, \w{t},q)
&   \text{ \Cref{corollary:RaoHammingOAs}}
\\
\{ \a\}, \a \in \fqt, 4|t
&   \{ \a^{i+j \w{t/2}} \mid j \in [0,q^{t/2}] \}
&   \OA_{q^{t-3}}(3, q^{t/2}+1,q)
&   \text{ \Cref{corollary:OAsOfStrength3FromMSequences}}
\\
\shortstack[l]{$\{ \a\}, \a \in \fqt$ as\\ in \Cref{corollary:DewarOAConstruction}} 
&   [0,2t]
&   \OA_{2^{t-3}}(3, 2t+1,2)
&   \text{ \Cref{corollary:DewarOAConstruction}}
\\
\shortstack[l]{$\{ \a\}, \a \in \fqt$ as\\ in
\Cref{corollary:OlgaOAConstruction}} 
&   [0,3t-1]
&   \OA_{3^{t-3}}(3, 3t,3)
&   \text{ \Cref{corollary:OlgaOAConstruction}}
\\
\{ \a,\a^{-1} \}, \a \in \f_{q^3},
&   [0,q^2+q] 
&   \CA(2q^3-1; 3, q^2+q+1,q)
&   \text{ \Cref{theorem:RaaphorstCAConstruction}}
\end{array}
\]
\caption[The orthogonal and covering arrays described in 
\Cref{section:ConstructionOfOrthogonalArraysFromCyclicTraceArrays,section:ConstructionOfCoveringArraysFromCyclicTraceArrays}.]{The orthogonal and covering arrays described in 
\Cref{section:ConstructionOfOrthogonalArraysFromCyclicTraceArrays,section:ConstructionOfCoveringArraysFromCyclicTraceArrays}.
In every case, $\a$ is a primitive element.
}
\label{table:PreviousOAsAndCAsUsingGeneralizedDefOfA}
\end{table}

The orthogonal and covering arrays described in \Cref{section:ConstructionOfOrthogonalArraysFromCyclicTraceArrays,section:ConstructionOfCoveringArraysFromCyclicTraceArrays} can all be expressed in terms of the cyclic trace arrays in \Cref{definition:cyclicAlphaSArrayExtended}, as shown in \Cref{table:PreviousOAsAndCAsUsingGeneralizedDefOfA}.
In this chapter we extend this table by finding sets $P$ and $C$ such that $\ZA{P,C}$ is a $\CA(|P|(q^t-1)+1; t, |C|,q)$.
Searching for such sets is naturally related to the following two problems.

\begin{problem}
    {}
    {FindMaxSubset}
        Let $q$ be a prime power, $t\geq 2$, $l\geq 1$ and $\a_0, \dots, \a_{l-1}$ be primitive elements of $\fqt$. 
        Find $C\subseteq [0,\w{t}-1]$ of maximum size with the property that, for every $I\subseteq C$ with $|I|=t$, there exists $i \in [0,l-1]$ such that $\ZA{\a_i,I}$ is an $\OA(t, t,q)$.
\end{problem}
\begin{remark}
    {}
    {SolutionToProblemOneIsCA}
    For a solution $C$ of \Cref{problem:FindMaxSubset}, $\ZA{ \{\a_0, \dots, \a_{l-1}\}, C}$ is a $\CA(l(q^t-1)+1; t, |C|,q)$.
\end{remark}

\begin{problem}
    {}
    {ChoiceOfPrimElements}
        Let $q$ be a prime power, $t\geq 2$ and $l\geq 1$.
        Find an $l$-set of primitive elements $P=\{\a_0, \dots, \a_{l-1}\}$ of $\fqt$ such that there exists $C\subseteq [0, \w{t}-1]$ that has the following properties:
        \begin{enumerate}[label=\roman*.]
            \item $C$ is a solution to \Cref{problem:FindMaxSubset} for $P$, and
            \item for every $l$-set $P'$ of primitive elements of $\fqt$ and for every solution $C'$ to \Cref{problem:FindMaxSubset} for $P'$, we have that $|C'|\leq |C|$.
        \end{enumerate}
\end{problem}

For some values of $l$ and $t$, these problems have already been addressed.
For $l=1$, \Cref{problem:FindMaxSubset} is about finding a set of columns of $\ZA{\{\a\}}=\ZA{\a}$ of maximum size such that the subarray they define is a covering array (in this case, an orthogonal array) of strength $t$.
This is equivalent to the problem of finding arcs of maximum size in $PG(t-1,q)$, as we discuss in \Cref{section:OAsFromArcs}.
Moreover, this is independent of the choice of $\a$ and thus \Cref{problem:ChoiceOfPrimElements} is trivial when $l=1$.
For $t=2$, both problems are trivial since $\ZA{\a}=\ZA{\a,C}$ for $C=[0,\w{t}-1]$ is a a covering array of strength 2.
For $t=3$ and $l\geq 2$, the problems are settled from \Cref{proposition:SebastianExtended}. 
Indeed, we have that $\ZA{\{\a,\a^{-1}\}}$ is a covering array (in this case, an orthogonal array) of strength $3$ for any primitive $\a\in\fqthree$, hence the answer to \Cref{problem:FindMaxSubset} is 
\[
    C=[0, \w{3}-1]=[0,q^2+q],
\]
which is of the maximum possible size.
This also implies that, for any primitive $\a\in \fqthree$, the pair $\{\a,\a^{-1}\}$ is an answer to \Cref{problem:ChoiceOfPrimElements}.

From the above we conclude that we can focus on the cases $t\geq 4$ and $l\geq 2$.
In this chapter we give algorithmic answers for these cases that rely on finite field theory and combinatorial exhaustive generation.
We dedicate 
\Cref{section:ABacktrackAlgoForProblem1,section:CAsFirstPaper_ChoiceOfPrimitiveElements}
to 
\Cref{problem:FindMaxSubset,problem:ChoiceOfPrimElements},
respectively.
In
\Cref{section:CAsFirstPaper_ImplementationAndNewBounds}
we discuss the computer implementation of our algorithms, which resulted to $38$ new covering arrays that improve upon previously best upper bounds for covering array numbers of strength $4$, as well as one covering array that improves upon a previously best upper bound \cite{colbournwebsite} of covering arrays of strength $5$.

\begin{algorithm}[t]
    \caption{Generic backtracking algorithm with optimal solution saved globally}
    \label{algorithm:GenericBacktracking}
    \begin{algorithmic}
        \State \textbf{global} $X_{best}$
        \Procedure{Backtracking}{$x_0, \dots, x_{r-1}$}
        \If{$(x_0, \dots, x_{r-1})$ is a feasible solution and it is better than $X_{best}$}
            \State $X_{best} \gets (x_0, \dots, x_{r-1})$
            \label{line:BestIsStoredGlobally}
        \EndIf
        \State Compute $\C(x_0, \dots, x_{r-1})$
        \For{ $x\in \C(x_0, \dots, x_{r-1})$}
            \State \textsc{Backtracking}$(x_0, \dots, x_{r-1},x)$
        \EndFor
        \EndProcedure
        \State \textbf{Main;}
        \State $X_{best}\gets(x_0)$
        \State \textsc{Backtracking}($x_0$)
    \end{algorithmic}
\end{algorithm}

\section{A backtracking algorithm for \Cref{problem:FindMaxSubset}}
\label{section:ABacktrackAlgoForProblem1}
\subsection{Preliminaries}
\label{section:BacktrackingAlgorithms}
In this section, we present the terminology and minimum background related to backtracking algorithms, which is necessary for \Cref{chapter:CAsFirstPaper}.

An \emph{optimization problem} is one that involves finding an optimal solution among the set of all feasible solutions, which we refer to as the \emph{search space}.
Often, the optimal solution
can be represented as a list $S=(s_0, s_1, \dots, s_{n-1})$ in which each $s_i$ is chosen from a finite \emph{possibility set $\mathcal{P}_i$}.
A \emph{backtracking algorithm}
\index{Backtracking algorithm}
is a recursive method for solving such problems, that incrementally builds feasible candidate solutions and abandons (``backtracks'') the ones that do not lead to an optimal solution, as soon as it determines that this is the case.
At every recursion, a backtracking algorithm extends a feasible solution $(s_0, \dots, s_{r-1})$ to a feasible solution $(s_0, \dots, s_{r-1}, x)$, where $x$ is restricted to a subset $\C(s_0, \dots, s_{r-1})\subseteq \mathcal{P}_{r-1}$, according to the problem's constraints.
The set $\C(s_0, \dots, s_{r-1})$ is a \emph{set of candidates}
\index{Set!of candidates}
(or \emph{choice set})
\index{Set!choice}
and its computation is referred to as \emph{pruning}.
\index{pruning}
For all $y \in \mathcal{P}_{r}\setminus \C(s_0, \dots, s_{r-1})$, the cases $(s_0, \dots, s_{r-1},y)$ are not considered.
\index{Solution!feasible}
\index{Feasible solution|see{Solution}}
\index{Optimization problem}
\index{Search space}
\index{Optimal solution|see{Solution}}

In \Cref{algorithm:GenericBacktracking} we show a generic backtracking algorithm where an optimal solution candidate is saved in a global variable $X_{best}$.
The efficiency of this algorithm greatly depends on reducing the size of $\C(s_0,\dots, s_{r-1})$, as well as finding an efficient way of calculating it.

\Cref{algorithm:GenericBacktracking} gives rise to an ordered tree as follows.
\index{Tree}
\index{Backtracking algorithm!tree}
The nodes of the tree are all the feasible solutions, where $(s_0)$ is the root of the tree and the children of node $(s_0, \dots, s_{r-1})$ are precisely the nodes $(s_0, \dots, s_{r-1},x)$, for $x \in \C(s_0, \dots, c_{r-1})$, in the order they are processed by the algorithm. 
A visualization of this is shown in \Cref{figure:GenericBacktrackingTree}.
We often use this tree terminology for convenience.

\Cref{problem:FindMaxSubset} can be formulated as an optimization problem.
A straightforward way to do this is as follows.
Let $P$ be a fixed set of primitive elements of $\fqt$.
A feasible solution to \Cref{problem:FindMaxSubset} is any subset $S$ of $[0,\w{t}-1]$ such that $\ZA{P,S}$ is a covering array of strength $t$.
The corresponding candidate set can be defined as
\begin{equation}
    \label{equation:ChoiceSet}
    \C(S)=
    \left\{ x \mid x> \max(S) \text{ and } \ZA{P,S\cup \{x\}} \text{ is a covering array of strength $t$} \right\}.
\end{equation}
A feasible solution of the maximum size is an answer to \Cref{problem:FindMaxSubset}.
In the following sections we greatly reduce the size of the candidate set in \Cref{equation:ChoiceSet} and determine an efficient way of computing it. 

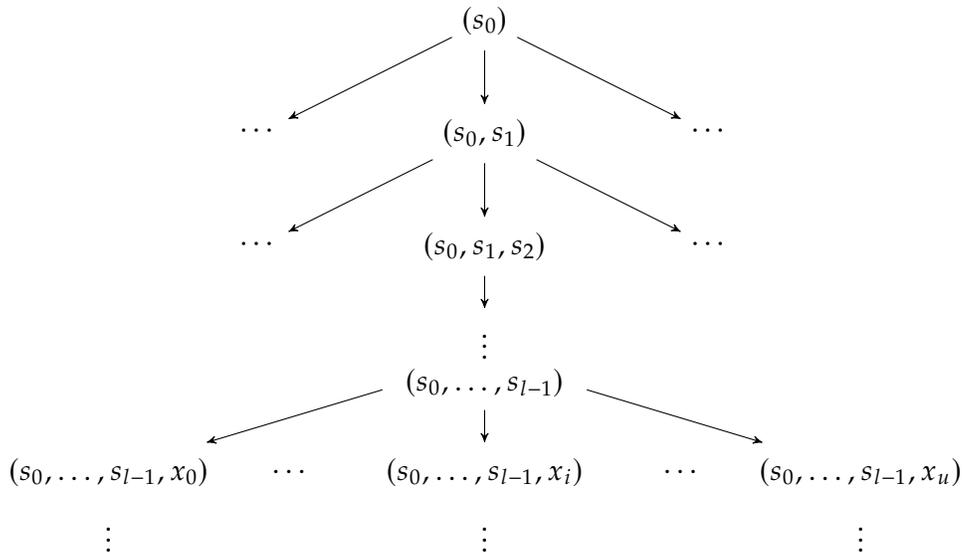
\begin{figure}[t]
\centering
\begin{tikzpicture}[
        align=center,
        ->,
        >=stealth',
        minimum size=2em,
        level/.style={sibling distance = 3cm, level distance = 1.5cm},
        level 4/.style={sibling distance=5cm}
    ] 
  \node [] {$(s_0)$}
      child{node [] {$\cdots$}}
      child{node [] (pair) {$(s_0, s_1)$}
          child{node [] {$\cdots$}}
          child{node [] (triple) {$(s_0, s_1, s_2)$} 
              child{node [] {$\begin{array}{c}\vdots\\(s_0,\ldots,s_{l-1})\end{array}$}
                  child{node [] (pakis) {$(s_0, \dots, s_{l-1},x_0)$}}
                  child{node [] (toulis) {$(s_0, \dots, s_{l-1},x_i)$}}
                  child{node [] (koulis) {$(s_0, \dots, s_{l-1},x_u)$}}
              }
          }
          child{node [] {$\cdots$}}
      }
      child{ node [] {$\cdots$}}
; 
\node [left = 2em of toulis] {$\ldots$};
\node [right = 2em of toulis] {$\ldots$};
\node [below = 0em of toulis] {$\vdots$};
\node [below = 0em of pakis] {$\vdots$};
\node [below = 0em of koulis] {$\vdots$};
\end{tikzpicture}
\caption[Visual representation of the nodes of a backtracking tree]{A visual representation of the nodes of the tree corresponding to \Cref{algorithm:GenericBacktracking} that are at distance $l$ from the root, where $\C(s_0, \dots, s_{l-1})= \{x_0, \dots, x_u\}$.}
\label{figure:GenericBacktrackingTree}
\end{figure}

\subsection{Isomorphism pruning for \Cref{problem:FindMaxSubset}}
\label{section:IsomorphismPruning}

In this section we establish an equivalence relation on sets of positive integers and a criterion for choosing a canonical representative of every equivalence class.
For the special case when the sets of integers represent feasible solutions to the optimization problem that we discuss previously, we show that the elements of these equivalence classes yield equivalent feasible solutions.
As a result, we can reduce the candidate set in \Cref{equation:ChoiceSet} by only considering elements that yield canonical feasible solutions. 

\subsubsection{An equivalence relation on feasible solutions}

The equivalence relation that we want to establish relies on the following notion.
\index{Set!shift modulo $n$ of}
\index{Shift!of set modulo $n$|see{Set}}
\index{Set!shift-equivalent}
\begin{definition}{Shift of a set modulo $n$}{ShiftOfASet}
    Let $n\geq 1$ be an integer and $S\subseteq [0,n-1]$.  
    For any integer $i$, the \emph{shift of $S$ by $i$ modulo $n$} is the set
    \begin{equation*}
        S+_n i=\left\{ (s+i) \bmod{n} \mid s \in S \right\}.
    \end{equation*}
    Two sets $S,T\subseteq [0,n-1]$ are \emph{shift-equivalent modulo $n$} if $S=T+_n i$, for some $i \in [0,n-1]$.
\end{definition}
The relation shift-equivalence modulo $n$ is an equivalence relation on the power set of $[0,n-1].$
Next, we show that this gives rise to a notion of equivalence of feasible solutions to \Cref{problem:FindMaxSubset}.

\begin{lemma}{}{ShiftsOfcolumnsSameCoverage}
    Let $q$ be a prime power, $t$ be an integer with $t\geq 2$ and $\a$ be a primitive element of $\fqt$.
    Suppose that $I,J\subseteq [0, \w{t}-1]$ where $|I|=|J|=t$ and $I$ and $J$ are shift-equivalent modulo $\w{t}$.
    Then, we have that $\ZA{\a,I}$ is an $\OA(t, t,q)$ if and only if $\ZA{\a, J}$ is an $\OA(t,t,q)$.
\end{lemma}

\begin{proof}
    From \Cref{item:b,item:a} of \Cref{proposition:SebastianExtended}, it suffices to show that $\{ \a^i \mid i \in I \}$ is linearly independent if and only if $\{\a^j \mid j \in J \}$ is linearly independent.
    Suppose by means of contradiction that the former is linearly independent but the latter is not.
    Let $I=\{ i_0, \dots, i_{t-1} \}$ and $J=\{ j_0, \dots, j_{t-1} \}$.
    From our assumptions, there exist $c_0, \dots, c_{t-1} \in \fq$ not all zero such that 
    \begin{equation}\label{equation:klanies}
        c_0\a^{j_0} + \dots+c_{t-1}\a^{j_{t-1}}=0.
    \end{equation}
    Now, since $I$ and $J$ are shift-equivalent modulo $\w{t}$, then there exists integer $k$ such that 
    \[
        j_s=(k+i_s) \bmod \w{t},
    \]
    for every $s\in [0,t-1]$. 
    In other words, for every $s \in [0,t-1]$ there exists integer $l_s$ such that
    \[
        j_s = k+i_s+l_s \w{t}.
    \]
    Hence, we have that $\a^{j_s}=\a^k\a^{l_s \w{t}}\a^{i_s}$ and \Cref{equation:klanies} implies that
    \begin{align*}
        0 & = c_0 \a^k \a^{l_0 \w{t}}\a^{i_0}+ \cdots +c_{t-1} \a^k \a^{l_{t-1} \w{t}}\a^{i_{t-1}}\\
        & = \a^k\left(c_0 \a^{l_0 \w{t}}\a^{i_0}+ \cdots +c_{t-1} \a^{l_{t-1} \w{t}}\a^{i_{t-1}}\right).
    \end{align*}
    Thus, denoting $c_s' = c_s \a^{l_s \w{t}}$, we have that
    \begin{equation}\label{equation:papares}
        c_0'\a^{i_0} + \dots +c_{t-1}'\a^{i_{t-1}}=0.
    \end{equation}
    From \Cref{lemma:CharacterizationOfConstantMultiplesInFQM}, we have that $\a^{l_s \w{t}} \in \fqstar$, for every $s$ which means that $c_0', \dots, c_{t-1}'$ are elements of $\fq$, not all zero.
    This, along with \Cref{equation:papares}, implies that $\{\a^i \mid i \in I\}$ is linearly dependent, contradicting our assumptions.
    The proof for the other direction is identical.
\end{proof}
For $S\subseteq [0,n-1]$, we denote by $\eqcl{n}{S}$ the equivalence class of the shift-equivalence modulo $n$ that contains $S$.
In other words,
\begin{equation}
    \label{equation:DefinitionES}
    \eqcl{n}{S}
    = \left\{ S +_n i \mid i \in [0,n-1] \right\}.
\end{equation}

\begin{proposition}
    {}
    {EquivalentClassesSimilarOAProperties}
    Let $q$ be a prime power, $t$ be an integer with $t\geq 2$ and let  $\a_0, \dots, \a_{l-1}$ be primitive elements of $\fqt$.
    Let $S\subseteq [0, \w{t}-1]$ with the property that, for every $I\subseteq S$ with $|I|=t$, there exists $i \in [0, l-1]$ such that $\ZA{\a_i, I}$ is a $\OA(t, t,q)$.
    Then, every $T \in \eqcl{\w{t}}{S}$ has the same property; that is, for every $I\subseteq T$ with $|I|=t$, there exists $i \in [0, l-1]$ such that $\ZA{\a_i, I}$ is a $\OA(t, t,q)$.
\end{proposition}
\begin{proof}
    Let $T\in \eqcl{\w{t}}{S}$ and $J\subseteq T$ with $|T|=t$.
    Then, $J$ is shift-equivalent modulo $\w{t}$ to some $I\subseteq S$, since $S$ and $T$ are shift-equivalent modulo $\w{t}$ as elements of the same equivalence class $\eqcl{\w{t}}{S}$.
    From the assumptions about $S$, we have that $\ZA{\a_i,I}$ is an $\OA(t, q,q)$, by \Cref{lemma:ShiftsOfcolumnsSameCoverage}, $\ZA{\a_i,J}$ is also a covering array of strength $t$.
    This is a subarray of $\ZA{P,J}$, and hence the latter is a covering array of strength $t$.
\end{proof}
It follows from \Cref{proposition:EquivalentClassesSimilarOAProperties} that $S$ is a feasible solution to \Cref{problem:FindMaxSubset} if and only if every set in the equivalence class $\S_{\w{t}}(S)$ is also a feasible solution.
Therefore, we can reduce the search space of the problem to only representatives of these canonical sets.

\subsubsection{Canonical representatives of the equivalent classes}
\label{section:CanonicalRepresentativesForTheEquivalentClasses}
In the following, we establish a criterion that allows us to choose a unique representative from each class of the shift-equivalence modulo $n$.
We do this using the notion of a binary necklace, which we define next.
\index{Necklace}
\index{Binary necklace}
\index{Set!binary representation}
\index{Binary representation|see{Set}}
\begin{definition}{Necklace}{necklace}
    Let $A$ be an ordered set and $\bfa$ be a string of elements from $A$.
    The \emph{necklace of} $\bfa$, denoted by $\neck(\bfa)$, is the
    lexicographically smallest of all cyclic shifts of $\bfa$.
    If $\bfa=\neck(\bfa)$, then $\bfa$ is a necklace.
    If $A = \{0,1\}$, then a necklace is a \emph{binary necklace}; we denote $\Bcal_n$ the set of all binary necklaces of length $n$.
\end{definition}

\begin{example}{}{BinaryNecklaces}
    Let $\bfa=10101$.
    The following are all the cyclic shifts of $\bfa$, listed in lexicographical order:
    \[ 01011 < 01101 < 10101 < 10110 < 11010,\]
    hence, $\neck(\bfa)=01011$.
    Let $\bfb=101010$.
    All the (distinct) shifts of $\bfb$ are $010101$ and $101010$, therefore $\neck(\bfb)=010101$.
\end{example}

    \index{Characteristic vector}
\begin{definition}{The characteristic vector of a set}
    {CharacteristicVectorOfSet}
    Let $n$ be a positive integer, $S\subseteq [0,n-1]$ and, for every $i \in [0,n-1]$ let
    \[
        b_i= 
        \begin{cases}
            1 &  \text{ if } i \in S\\
            0 &  \text{ otherwise.}
        \end{cases}
    \]
    Then, the binary string $\charv_n(S)=b_0 b_1 \dots b_{n-1}$ is the \emph{characteristic vector of $S$}.
\end{definition}

Let $\bfb=b_0 \dots b_{n-1}$ be a binary string.
Similarly to \Cref{definition:LeftShiftOperator}, we denote
\[
    L^i(\bfb)= b_i \dots b_{n-1} b_0 \dots b_{i-1},
\]
\index{Shift!of string}
and refer to $L^i(\bfb)$ as the left cyclic shift of $\bfb$ by $i$.
Furthermore, for $b\in \{0,1\}$, we denote $b^n$ the binary string consisting of the digit $b$ repeated $n$ times.
Apart from the characteristic vector, we need to introduce another binary representation for sets.

\pagebreak
\begin{definition}{Binary representation of sets}{BinaryRepresentationOfSets}
    Let $n$ be a positive integer, $S$ be a nonempty subset of $[0,n-1]$ and, for every $i \in [0, \max(S)]$, let 
    \[
        b_i= 
        \begin{cases}
            1 &  \text{ if } i \in S\\
            0 &  \text{ otherwise.}
        \end{cases}
    \]
    We define the \emph{binary representation of $S$}, denoted by $\bin_n(S)$, to be the binary string
    \begin{align*}
        \bin_n(S)
        &= 0^{n-\max(S)-1} b_0 \dots b_{\max(S)}\\
        &= L^{\max(S)+1}(\charv_n(S)).
    \end{align*}
Moreover, we define $\bin_n(\emptyset)=0^n$.
\end{definition}

Our focus now becomes to show that we can use the above binary representation and the notion of a binary necklace to choose a unique representative from each class of the shift-equivalence modulo $n$.
We begin with an auxiliary result.
\begin{lemma}
    {}
    {EquivalentClassEltsAreTheOnesWithCyclicShiftsOfCharVector}
    Let $n$ be a positive integer and $S\subseteq [0,n-1]$.
    Then, 
    \[
        \left\{\charv_n(T) \mid T \in \eqcl{n}{S} \right\}
        =
        \left\{ L^{i}(\charv_n(S))\mid i \in [0,n-1] \right\}.
    \]
\end{lemma}
\begin{proof}
    Let $T\subseteq [0, n-1]$; for the rest of the proof, we denote
    \begin{align}
        \label{equation:panteloni}
        \notag
        \charv_n(S) &= b_0 \dots b_{n-1}\\
        \charv_n(T) &= c_0 \dots c_{n-1}.
    \end{align}
    
    First, we assume that $T\in \eqcl{n}{S}$, so that $T=S+_n i$ for some $i \in [0, n-1]$.
    Then, the following equivalencies hold for all $j\in [0, n-1]$:
    \begin{align}
        \notag
        c_j = 1 & \Leftrightarrow j \in T \\ 
        \notag
                & \Leftrightarrow j = (s+i) \bmod n, \text{ for some } s \in S\\ 
                & \Leftrightarrow j = s+i+kn, \text{ for some } s \in S \text{ and } k \in \Z.
        \label{equation:kwlos}
    \end{align}
    For $j \in [0,i-1]$, we have that \Cref{equation:kwlos} holds if and only if $k=-1$ and thus $c_j=1$ is equivalent to $n-i+j \in S$, or $b_{n-i+j}=1$;
    this shows that
    \begin{equation}
        \label{equation:pelargos}
        c_0 \dots c_{i-1} = b_{n-i} \dots b_{n-1}.
    \end{equation}
    For  $j \in [i, n-1]$, we have that \Cref{equation:kwlos} holds if and only if $k=0$ and thus $c_j=1$ is equivalent to $j-i \in S$, or $b_{j-i} =1$;
    this shows that
    \begin{equation}
        \label{equation:gaidaros}
        c_{i} \dots c_{n-1} = b_{0} \dots b_{n-1-i}.
    \end{equation}
    From \Cref{equation:pelargos,equation:gaidaros} we conclude that 
    \allowdisplaybreaks[0] 
    \begin{align*}
          \charv_n(T)
        =& c_{0} \dots c_{n-1} \\
        =& b_{n-i} \dots b_{n-1}b_{0} \dots b_{n-1-i} \\
        =& L^{n-i}(b_0 \dots b_{n-1})\\
        =& L^{n-i}(\charv_n(S)),
    \end{align*}
    \allowdisplaybreaks 
    which proves that
    \[
        \left\{\charv_n(T) \mid T \in \eqcl{n}{S} \right\}
        \subseteq
        \left\{ L^{i}(\charv_n(S))\mid i \in [0,n-1] \right\}.
    \]

    Conversely, we assume that $\charv_n(T)=L^i(\charv_n(S))$ for some $i \in [0,n-1]$.
    Then, considering \Cref{equation:panteloni}, we have that 
    \[
        c_0 \dots c_{n-1}=
        b_i \dots b_{n-1} b_0 \dots b_{i-1}
    \]
    or, equivalently,
    \begin{equation}\label{equation:kafes}
        c_j =
        \begin{cases}
            b_{i+j}  &\text{ if } j \in [0, n-1-i]\\
            b_{j-n+i}&\text{ if } j \in [n-i, n-1].
        \end{cases}
    \end{equation}
    \Cref{equation:kafes} for $j \in [0, n-1-i]$, implies that 
    \begin{align*}
        j \in T & \Leftrightarrow i+j \in S\\
        & \Leftrightarrow j = s-i, \text{ for some } s \in S\\
        & \Leftrightarrow j \in S+_n (n-i), \text{ for some } s \in S.
    \end{align*}
    For $j \in [n-i, n-1]$, \Cref{equation:kafes} implies that 
    \begin{align*}
        j \in T & \Leftrightarrow j-n+i \in S\\
        & \Leftrightarrow j \in S+_n (n-i), \text{ for some } s \in S.
    \end{align*}
    We conclude that, for all $j \in [0, n-1]$, we have $j \in T$ if and only if $j\in S+_n (n-i)$, and thus $T=S+_n (n-i)$.
\end{proof}

Let $\bfb=0^s w$, where $w$ is either a binary string that starts with
$1$ or the empty string.
We define $\gtst(\bfb)$ to be the set with characteristic vector $w0^s$.
\begin{proposition}{}{CanonicalWellDefined}
    Let $S$ be a nonempty subset of $[0,n-1]$.
    Then, there exists a unique $T\in \eqcl{n}{S}$ such that $0\in T$ and
    $bin_n(T)$ is a binary necklace.
\end{proposition}

\begin{proof}
    By \Cref{lemma:EquivalentClassEltsAreTheOnesWithCyclicShiftsOfCharVector}, the set $\{\charv_n(T) \mid T\in \eqcl{n}{S}\}$ consists of all the cyclic shifts of $\charv_n(S)$, therefore there exists a necklace $\bfb$ among them.
    Since there does not exist a lexicographically smaller cyclic shift of $\bfb$, this must be of the form $\bfb=0^sw$, where $s\in [0,n-2]$ and $w$ starts and ends with $1$.
    Let 
    $T=\gtst(\bfb)=\gtst(0^s w).$
    Then $T$ is the set with characteristic vector $w 0^s$, which implies that $0 \in T$ and that 
    $\bin_n(T)=0^s w = \bfb,$
    so we conclude that $\bin_n(T)$ is a necklace. 
    Next, we show that $T\in \eqcl{n}{S}$.
    From the definition of $\bfb$, there exists some $U\in \eqcl{n}{S}$ such that 
    \[\charv_n(U)=\bfb=0^sw.\]
    This implies that the minimum element in $U$ is $s$ and thus
    \[\charv_n(U+_n (-s))=w0^s=\charv_n(T),\]
    which means that 
    \begin{equation}\label{equation:makaronia}
        T=U+_n (-s).
    \end{equation}
    Since $U\in \eqcl{n}{S}$, by \Cref{equation:DefinitionES} there exists some $i$ such that 
    \begin{equation}\label{equation:kimas}
        U= S+_n i.
    \end{equation}
    From \Cref{equation:makaronia,equation:kimas}, we conclude that
    $ T=S+_n (i-s) \in \eqcl{n}{S}.  $
    We have shown the existence of the set $T$ in question; it remains to show
    its uniqueness.
    Let $T'\in \eqcl{n}{S}$ such that $0\in T'$ and $\bin_n(T')$ is a necklace.
    Since $T$ and $T'$ are in $\eqcl{n}{S}$, we have that $T'$ we have that $\bin_n(T)$ and
    $\bin_n(T')$ are cyclic shifts of each other and since they are both necklaces, we have $\bin_n(T)=\bin_n(T')$ and
    \[T'=\gtst(\bin_n(T'))=\gtst(\bin_n(T))=T,\]
    which completes the proof
\end{proof}
\Cref{proposition:CanonicalWellDefined} gives a criterion for choosing a unique element from $\eqcl{n}{S}$, for every $S\subseteq [0,n-1]$. 
In other words, it provides a notion of a canonical representative from each equivalent class.

\index{Set!canonical}
\index{Canonical set|see{Set}}
\begin{definition}{Canonical set}{CanonicalSet}
    A set $S\subseteq [0,n-1]$ is \emph{canonical} if either $S=\emptyset$, or
    $0\in S$ and $\bin_n(S)$ is a necklace.
    Furthermore, we denote by $\C_n$ the set that contains all the canonical subsets of $[0,n-1]$.
\end{definition}

\subsubsection{Algorithmic generation of canonical representatives}
\label{section:generatingcanonicalsubsets}

Next, we exploit a correspondence between canonical sets and binary necklaces to efficiently generate all the canonical subsets of $[0,n-1]$.
We recall that $\Bcal_n$ is the set of all the binary necklaces of length $n$.
\begin{proposition}{}{OneToOneNecklacesAndCanonicalSets}
    The mapping $\bin_n: \C_n \rightarrow \Bcal_n$ is a bijection, and its inverse is the mapping $\gtst$.
\end{proposition}

\begin{proof}
    We have that $\bin_n(\emptyset)=0^n$ and $\gtst(0^n)=\emptyset$, so we only need to examine nonempty canonical subsets $S\subseteq [0,n-1]$.
    By \Cref{definition:CanonicalSet}, $\bin_n(S)$ is a nonzero necklace of length $n$, hence $\bin_n$ is indeed a mapping from $C_n$ to $B_n$.
    Let $\bfb\in B_n$.
    Then $\bfb=0^s w$, for some $s\in [0,n-2]$ and $w$ starting with $1$.
    Hence $T=\gtst(\bfb)$ has characteristic vector $w0^s$, which implies that $0\in T$.
    Furthermore, we have that $\bin_n(T)=\bfb$, which is a necklace.
    We conclude that $\gtst$ is a mapping from $B_n$ to $C_n$ and it is the inverse is $\bin_n$ when restricted to $C_n$.
\end{proof}

We take advantage of \Cref{proposition:OneToOneNecklacesAndCanonicalSets} to generate the sets in $\C_n$ relying on an algorithm for the generation of the binary necklaces of length $n$ due to Ruskey et al.\ \cite{ruskey1992generating}, which we present next.
For $b\in {0,1}$, let $\overline{b}$ denote the binary complement of $b$.
For a binary string $\bfb=b_0 \cdots b_{n-1}$ we define
\[
    \label{equation:DefTau}
    \tau(\bfb)= b_0 \cdots b_{n-2} \overline{b_{n-1}}.
\]
In \cite{ruskey1992generating}, Ruskey et al.\ consider a tree 
\index{Tree!of binary necklaces}
$T_{\Bcal_n}$ as follows:
\begin{itemize}
    \item the root of $T_{\Bcal_n}$ is $0^{n-1}1$, and
    \item for every node $\bfb$ of $T_{\Bcal_n}$, its children ---if any--- are the necklaces 
        \[\tau( L^i(\bfb)), i \in [1,r],\] 
    where $r+1\leq n$ is the smallest $i$ such that $L^i(\bfb)$ is not a binary necklace.
\end{itemize}
\begin{algorithm}[t]
    \caption[Generation of binary necklaces]{Generation of all the nonzero elements in $\Bcal_n$ \cite{ruskey1992generating}.}
    \label{algorithm:RuskeyBinaryNecklaces}
    \begin{algorithmic}[1]
        \Procedure{BinaryNecklaces}{$\mathbf{b}$}
        \State Output $\mathbf{b}$ 
        \State $done\gets$ {\bf false}
        \While{not $done$}
            \State $\mathbf{b}\gets  L(\mathbf{b})$
            \State $\mathbf{b}' \gets \tau(\mathbf{b})$
            \If{$b'$ is a necklace}
                \label{line:checkingifnecklace}
                \State \textsc{BinaryNecklaces}($\mathbf{b}'$)
            \Else
                \State $done \gets$ {\bf true}
            \EndIf
        \EndWhile
        \EndProcedure
        \State \textbf{main;}
        \State \textsc{BinaryNecklaces}($0^{n-1}1$)
    \end{algorithmic}
\end{algorithm}

From the definition of $T_{\Bcal_n}$, it follows that its nodes are binary necklaces of length $n$; conversely, every binary necklace is a node of $T_{\Bcal_n}$ \cite[Lemma 7]{ruskey1992generating}.
Furthermore, every necklace appears as a node exactly once \cite[Theorem 5]{ruskey1992generating}.
This is reflected in the fact that a node $\bfb$ does not have children among $\tau (L^i(\bfb))$, $i \in [r,n-1]$, since it is proven that a string of this form is either not a necklace, or it is a necklace but it is also the child of a previous node.

The procedure \textsc{BinaryNecklaces} in \Cref{algorithm:RuskeyBinaryNecklaces}, which is presented in \cite{ruskey1992generating}, is the algorithmic implementation of the definition of $T_{\Bcal_n}$. 
The algorithm traverses the tree recursively from the root and outputs every nonzero binary necklace exactly once.
While generating the children of a necklace $\bfb$ in the tree, at most one is examined which is \emph{not} a necklace.
Thus, the total number of nodes examined is at most $2|\Bcal_n|$.

\begin{theorem}
    {\cite{ruskey1992generating}}
    {RuskeyAlgorithmWorks}
    The output of \Cref{algorithm:RuskeyBinaryNecklaces} consists of all the nonzero binary necklaces of length $n$, without repetitions.
\end{theorem}

\begin{example}
    {\Cref{algorithm:RuskeyBinaryNecklaces} for $n=6$}
    {RuskeyRunExample}
    The output of \textsc{BinaryNecklaces}($000001$) is
    \begin{align*}
        000001, 000011, 000111, 001111, 011111, 001101\\
        011011, 000101, 001011, 010111, 010101, 001001.
    \end{align*}
    The tree $T_{\Bcal_6}$ is shown \Cref{figure:necklacetree}; the sets underneath the binary strings are discussed in a later example and can be ignored for now.
    The crossed out strings are \emph{not} nodes of $T_{\Bcal_6}$; they are the strings that were found not to be necklaces in \Cref{line:checkingifnecklace} of \Cref{algorithm:RuskeyBinaryNecklaces}.
    When encountering such strings, the algorithm backtracks.
%
    We observe that there are $13$ binary necklaces of length $6$, also counting the all-zero $6$-tuple, so $|\Bcal_6|=13$.
    A total of $23$ checks were performed, so indeed the number of checks is less than $2|\Bcal_6|=26$.
\end{example}

\begin{sidewaysfigure}

\def\nodeA{$000001$\\$\{0\}$}
\def\nodeBa{$000011$\\$\{0,1\}$}
\def\nodeBb{$000101$\\$\{0,2\}$}
\def\nodeBc{$001001$\\$\{0,3\}$}
\def\nodeBd{$010001$\\$\{0,4\}$}
\def\nodeCa{$000111$\\$\{0,1,2\}$}
\def\nodeCb{$001101$\\$\{0,1,3\}$} 
\def\nodeCc{$011001$\\$\{0,1,4\}$}
\def\nodeCd{$001011$\\$\{0,2,3\}$}
\def\nodeCe{$010101$\\$\{0,2,4\}$} 
\def\nodeCf{$010011$\\$\{0,3,4\}$}
\def\nodeDa{$001111$\\$\{0,1,2,3\}$} 
\def\nodeDb{$011101$\\$\{0,1,2,4\}$}
\def\nodeDc{$011011$\\$\{0,1,3,4\}$}
\def\nodeDd{$110101$\\$\{0,1,3,5\}$}
\def\nodeDe{$010111$\\$\{0,2,3,4\}$}
\def\nodeDf{$101101$\\$\{0,2,3,5\}$}
\def\nodeDg{$101011$\\$\{0,2,4,5\}$}
\def\nodeEa{$011111$\\$\{0,1,2,3,4\}$} 
\def\nodeEb{$111101$\\$\{0,1,2,3,5\}$}
\def\nodeEc{$110111$\\$\{0,1,3,4,5\}$}
\def\nodeEd{$101111$\\$\{0,2,3,4,5\}$}
\def\nodeFa{$111111$\\$\{0,1,2,3,4,5\}$}

\begin{forest}
for tree={
    align=center,
    l sep+=2em 
}
[\nodeA, s sep=28pt
    [\nodeBa
        [\nodeCa
            [\nodeDa
                [\nodeEa 
                    [\nodeFa, cross out,draw=goldfish]
                ]
                [\nodeEb, cross out,draw=goldfish]
            ]
            [\nodeDb, cross out,draw=goldfish]
        ]
        [\nodeCb
            [\nodeDc [\nodeEc, cross out,draw=goldfish]]
            [\nodeDd, cross out,draw=goldfish]
        ]
        [\nodeCc, cross out,draw=goldfish]
    ]
    [\nodeBb
        [\nodeCd
            [\nodeDe [\nodeEd, cross out,draw=goldfish]]
            [\nodeDf, cross out,draw=goldfish]
        ]
        [\nodeCe [\nodeDg, cross out,draw=goldfish]]
    ]
    [\nodeBc
        [\nodeCf, cross out,draw=goldfish]
    ]
    [\nodeBd, cross out,draw=goldfish] 
]
\end{forest}
\caption{The trees $T_{\Bcal_6}$ and $T_{\C_6}$} 
\label{figure:necklacetree}

\end{sidewaysfigure}

We now turn our focus on applying \Cref{algorithm:RuskeyBinaryNecklaces} to generate the sets in $\C_n$ uniquely.
Since \Cref{algorithm:RuskeyBinaryNecklaces} generates all the elements in $\Bcal_n$ exactly once and, by \Cref{proposition:OneToOneNecklacesAndCanonicalSets}, $\gtst$ is a bijection between $\Bcal_n$ and $\C_n$, we can simply run \Cref{algorithm:RuskeyBinaryNecklaces}, and for each $\bfb$ in the output, obtain the canonical set $\gtst(\bfb)$.
However, we show a slightly different but equivalent way of achieving our goal that is more useful later in the section.
First, we consider a tree $T_{\C_n}$ whose nodes are nonempty canonical subsets of $[0,n-1]$, as follows:
\index{Tree!of canonical sets}
\begin{itemize}
    \item the root is $\{0\}$, and
    \item for every node $S$, its children ---if any--- are the sets \begin{equation}
            \label{equation:noxzema}
            S\cup \{i\}, i \in [\max(S)+1, r]
    \end{equation}
    where $r+1$ is the smallest $j\leq n-1$ such that $S\cup \{i\}$ is not canonical.
\end{itemize}
We show that the nodes of $T_{\C_n}$ are precisely the elements of $\C_n$ and that each node appears exactly once.
\begin{lemma}{}{TranslationNecklaceToSets}
    Let $S$ be a nonempty subset of $[0,n-1]$.
    Then, for all $j \in [1,n-\max(S)-1]$, we have that
    \begin{equation}
        \label{equation:mickey}
        S\cup \{ \max(S)+j \}= \gtst(\tau ( L^j (\bin_n(S))).
    \end{equation}
\end{lemma}

\begin{proof}
    For a nonempty subset $S\subseteq [0,n-1]$, its image under $\bin_n$ is given by
    \[ \bin_n(S) = 0^{n-\max(S)-1} b_0 \cdots b_{\max(S)},\]
    where $b_i = 1$ if and only if $i \in S$.
    Then, for every $j\in [1, n-\max(S)-1]$, we have that
    \begin{align} 
        \notag
        \tau ( L^j (S)) &=0^{n-(\max(S) +j) -1}b_0 \cdots b_{\max(S)} 0^{j-1} 1\\
                         =\bin_n(S\cup \{ \max(S)+j \}).
        \label{equation:goofy}
    \end{align}
    By \Cref{proposition:OneToOneNecklacesAndCanonicalSets}, $\gtst$ is the inverse of $\bin_{n}$, therefore \Cref{equation:goofy} is equivalent to \Cref{equation:mickey}.
\end{proof}
The mapping $\bin_n$ and its inverse $\gtst$ is a bijection between the nodes of the trees $T_{\Bcal_n}$ and $T_{\C_n}$ which, by \Cref{lemma:TranslationNecklaceToSets}, preserves the parent-child relation of nodes.
It follows that, since the nodes of $T_{\Bcal_n}$ are the elements of $\Bcal_n$ without repetitions, then the nodes of $T_{\C_n}$ are all the elements of $\C_n$ without repetitions.

The procedure \textsc{CanonicalSubsets} in \Cref{algorithm:GeneratingCanonicalSubsets} is the algorithmic implementation of the definition of $T_{\C_n}$.
It follows from the above discussion that \textsc{CanonicalSubsets} generates all the nonempty elements of $S_n$ exactly once.


\begin{algorithm}[t]
    \caption[Generation of canonical sets]{Generation of the nonempty sets in $\C_n$}
    \label{algorithm:GeneratingCanonicalSubsets}
    \begin{algorithmic}[1]
        \Procedure{CanonicalSubsets}{$S$,$n$}
        \State Output $S$
        \label{line:OutputS}
        \State $done \gets $ False
        \State $j=\max(S)+1$
            \While{ not $done$ and $j\leq n-1$} 
                \label{line:jfrommaxSandup}
                \If{$\bin(S\cup \left\{ j \right\})$ is a necklace}
                    \label{line:ifbinScupJisanecklace}
                    \State \textsc{CanonicalSubsets}($S\cup \left\{j \right\}$,$n$)
                    \label{line:runcanonicalsubsetsScupJ}
                    \State $j\gets j+1$
                \Else
                \label{line:canonical:Else}
                    \State $done\gets$ True
                    \label{line:canonical:break}
                \EndIf
            \EndWhile
        \EndProcedure
        \State \textbf{Main;}
        \State \textsc{CanonicalSubsets}($\left\{ 0 \right\}$, $n$)
    \end{algorithmic}
\end{algorithm}

    The tree $T_{\Bcal_6}$ is shown \Cref{figure:necklacetree}; the sets underneath the binary strings are discussed in a later example and can be ignored for now.
    The crossed out strings are \emph{not} nodes of $T_{\Bcal_6}$; they are the strings that were found not to be necklaces in \Cref{line:checkingifnecklace} of \Cref{algorithm:RuskeyBinaryNecklaces}.
    When encountering such strings, the algorithm backtracks.

\begin{example}
    {\Cref{algorithm:GeneratingCanonicalSubsets} for $n=6$}
    {RuskeyRunExample2}
    The output of \textsc{CanonicalSubsets}($\{0\}$) is
    \begin{gather*}
        \{0\},\{0,1\}, \{0,1,2\}, \{0,1,2,3\}, \{0,1,2,3,4\}, \{0,1,3\},\\
        \{0,1,3,4\}, \{0,2\}, \{0,2,3\}, \{0,2,3,4\}, \{0,2,4\}, \{0,3\}.
    \end{gather*}
    The tree $T_{\C_6}$ is shown in \Cref{figure:necklacetree}, overlapping with the tree $T_{\Bcal_6}$, which demonstrates how $\bin_6$ and $\gtst$ are bijections between the nodes that respect the parent-child relations.
    The crossed out sets correspond to the sets that are shown not to be canonical in \Cref{line:ifbinScupJisanecklace} of \Cref{algorithm:GeneratingCanonicalSubsets}.
\end{example}

\subsection{A backtracking algorithm to search for covering subarrays}
\label{section:TheBacktrackingAlgorithm}
\subsubsection{Overview}
It follows from \Cref{proposition:EquivalentClassesSimilarOAProperties} that a subset $S$ of $[0, \w{t}-1]$ is a feasible solution to \Cref{problem:FindMaxSubset} if and only if every set in $\S_{\w{t}}(S)$ is also a feasible solution. 
Therefore, considering that all the sets in an equivalence class have the same size, we have for an optimal solution $S$ that the canonical representative in $\S_{\w{t}}(S)$ is also an optimal solution. 
We conclude that there exists a solution of \Cref{problem:FindMaxSubset} in $\C_{\w{t}}$.
In this section, we use this fact and the framework of \Cref{algorithm:GeneratingCanonicalSubsets} to give an algorithm that solves \Cref{problem:FindMaxSubset}.
To introduce this algorithm we need several auxiliary results.

First, we state the optimization problem that we want to solve.
From now on, we fix a set $P=\{\a_0, \dots, \a_{l-1}\}$ of primitive elements of $\fqt$ and use the following terms.
\begin{itemize}
    \item A \emph{feasible solution}
        \index{Solution!feasible}
        is a set $S\in \C_{\w{t}}$ that either has size at most $t-1$, or it has size at least $t$ and for every $I\subseteq S$ with $|I|=t$, there exists $i \in [0, l-1]$ such that $\ZA{\a_i, I}$ is an $\OA(t, t,q)$.
    \item An \emph{optimal solution}
        \index{Solution!optimal}
        is a feasible solution of maximum size; this is a solution to \Cref{problem:FindMaxSubset}.
    \item The \emph{set of candidates}
        \index{Set!of candidates}
        for a feasible solution $S$, denoted $\cand{S}$, is given by 
    \begin{equation*}
    \label{equation:CMdefinition}
    \cand{S}=
       \left\{
           x \in [\max(S)+1, \w{t}-1]
           \mid
           S\cup \{x\} \text{ is a feasible solution}
       \right\}.
    \end{equation*}
    We observe that if $S$ is a feasible solution with size at most $t-2$, then \[\cand{S}=[\max(S)+1, \w{t}-1].\]
\end{itemize}
\subsubsection{Calculation of the set of candidates}
In order to use a backtracking algorithm to find an optimum solution, we first need to determine a way of calculating efficiently the set $\cand{S}$, for a feasible solution $S$.
For this, we use a recursive algorithm that is based on \Cref{lemma:ExtendCM}, that follows.
For $I\subset [1, \w{t}-1]$ we define
\begin{equation*}
    \label{equation:DefUP}
    U_P(I)=
    \left\{
        j \in [\max(I)+1, \w{t}-1] \mid I\cup\{0,j\} \text{ is not a feasible solution}
    \right\}.
\end{equation*}
Furthermore, for a positive integer $j$, we denote
\[
    \label{equation:CandidatesGreaterThanJ}
    \cand{S}_{>j}=\{i \in \cand{S} | i > j\}.
\]

%


\begin{lemma}{}{ExtendCM}
    Let $S$ be a feasible solution with $|S|\geq t-1$, and $j \in \cand{S}$. 
    Furthermore, we denote
    \begin{equation*}
        R(j)=
        \bigcup_{I \in \binom{S}{t-2}}
        \left( 
            U_P(I +_{\w{t}} (-j))
            +_{\w{t}} j
        \right).
    \end{equation*}
    Then, we have that 
    \[\cand{S\cup\{j\}}= \cand{S}_{>j} \setminus R(j).\]
\end{lemma}

\begin{proof}
    ``$\subseteq$''
    Let $d \in \cand{S \cup \{j\}}$.
    Then, $d>j$ and $S \cup \{ j,d \}$ is a feasible solution.
    Hence, $S \cup \{ d \}$ is a feasible solution as a subset of a feasible solution, which means that $d\in \cand{S}_{>j}$.
    It remains to show that $d\not\in R(j)$; assume by means of contradiction that this is not the case.
    Then, there exists $I \subset S$ with $|I|=t-2$ and $i \in U_P(I+_{\w{t}}(-j))$ such that $d=(i+j) \bmod{\w{t}}$.
    Let \[J=\left(I+_{\w{t}}(-j)\right) \cup \{0,i\}.\]
    Then, from the definition of $i$ and $U_P$ we have that $J$ is not a feasible solution.
    By \Cref{proposition:EquivalentClassesSimilarOAProperties}, this implies that $ J+_{\w{t}}j$ is not a feasible solution either.
    However, we have that
    \begin{align*}
        J+_{\w{t}} j
        &= \left(\left(I+_{\w{t}}(-j)\right)\cup \{0,i\}\right)+_{\w{t}}j\\
        &= I\cup \left\{ j, (i+j) \bmod{\w{t}} \right\}\\
        &= I \cup \{j,d\}\\
        &\subseteq S\cup \{j,d\},
    \end{align*}
    hence $J+_{\w{t}} j$ is a feasible solution as a subset of the feasible solution $S\cup \{j,d\}$, a contradiction.

    ``$\supseteq$''
    Let $d\in \cand{S}_{>j}\setminus R(j)$, and assume by means of contradiction that $d\not\in \cand{S\cup \{j\}}$.
    Then, there exists $I\subset S$ with $|I|=t-2$ such that
    $I \cup \{j,d\}$ is not a feasible solution.
    Let 
    \begin{align*}
        J   &= \left(I \cup \{j,d\}\right)+_{\w{t}}(-j)\\
        &= \left(I+_{\w{t}}(-j)\right) \cup \{0,d-j \}
    \end{align*}
    Since $J$ is shift-equivalent to $I\cup \{j,d\}$, then by \Cref{proposition:EquivalentClassesSimilarOAProperties} it is not a feasible solution either and hence
    \[
        d-j  \in U_P(I+_{\w{t}}(-j)), 
    \]
    which implies that $d \in R(j)$, contradicting our initial assumption about $d$.
\end{proof}
\Cref{lemma:ExtendCM} gives a recursive way of calculating the set of candidates for a feasible solution, which we implement in the procedure \textsc{ComputeCand} shown in \Cref{algorithm:ExtendCM}. 
This takes as input a feasible solution $S$, an element $j\in \cand{S}$ as well as the set $\cand{S}$ and outputs $\cand{S\cup \{j\}}$.

\begin{algorithm}[t]
    \caption[Update of the set $\cand{S}$]{Update of the set $\cand{S}$ as per \Cref{lemma:ExtendCM}} \label{algorithm:ExtendCM}
    \begin{algorithmic}[1]
       \State $j \in \cand{S}$ 
       \State Returns $\cand{S\cup \{j\}}$
       \Procedure{ComputeCand}{$S,j,\cand{S}$}
       \algrenewcommand\algorithmicindent{2.0em}%
       \If{$\cand{S}_{>j}=\emptyset$}
            \State \Return $\emptyset$
        \EndIf 
       \If{$|S|<t-1$}
           \State \Return $\cand{S}_{>j}$
        \Else
            \State $C  \gets \cand{S}_{>j}$ 
            \For{ $I\in \binom{S}{t-2}$}
               \State $X\gets \{(i-j) \bmod \w{t} | i \in I\}$
               \Comment $X=I+_{\w{t}}(-j)$
               \State $Y\gets \{(u+j) \bmod \w{t} | u \in U_P(X)\}$
               \Comment $Y=U_P(X)+_{\w{t}}j$
               \label{line:WeNeedUP}
               \State $C\gets C\setminus Y$
            \If{$C=\emptyset$}
                \State \Return $\emptyset$
            \EndIf 
            \EndFor
        \State \Return $C$
        \EndIf
        \EndProcedure
    \end{algorithmic}
\end{algorithm}

\begin{algorithm}[t]
    \caption{Backtracking algorithm for solving \Cref{problem:FindMaxSubset}}
    \label{algorithm:Backtracking}
    \begin{algorithmic}[1] 
    \Procedure{FindCA}{$S,\cand{S}$}
    \Comment Let $\cand{S}= \{s_0, \dots, s_{r-1}\}$
    \State \textbf{global} $Best$	
    \If{$r > |Best|$} 
        \label{line:BegTestIfBest}
        \State $Best \gets S$
    \EndIf
        \label{line:EndTestIfBest}
    \If{$|Best|\geq |S|+|\cand{S}|$}
        \label{line:BegIsThereHope}
        \State\Return 
        \label{line:EndIsThereHope}
    \EndIf

    \While{not $done$ and $i\leq |S|+|\cand{S}| - |Best|$}
    \label{line:While}
        \If{$\bin(S\cup \left\{ s_i \right\})$ is a necklace}
    \label{line:NecklaceCheck}
        \State $C\gets \extend(S, s_i, \cand{S})$
            \label{line:ExtendCM}
            \Comment $C= \cand{S\cup \{s_i\}}$  
            \State \textsc{FindCA}$\left( S\cup\{s_i\}, C \right) $
            \State $i\gets i+1$
        \Else
            \State $done\gets$ True
        \EndIf
    \EndWhile
    \EndProcedure
    \State \textbf{global} $Best \gets \left\{ 0 \right\}$
    \State \textsc{FindCA}($\left\{ 0 \right\}, \left\{ 1,\dots, \w{t}-1 \right\}$)
    \State Output $Best$
    \end{algorithmic}
\end{algorithm}

\subsubsection{The algorithm}
We now review \Cref{algorithm:Backtracking}.
The procedure \textsc{FindCA} takes as input $(S, \cand{S})$, where $S$ is a feasible solution.
It is initiated with input $(\{0\},[1,\w{t}-1])$; by the time it terminates, the global variable $Best$ holds an optimum solution; this is also the solution to \Cref{problem:FindMaxSubset}.
Next, we go through the procedure in more detail.
\begin{itemize}
    \item In every recursive run, the size of the input $S$ is compared to that of the largest feasible solution found at the time, which is stored in the global variable $Best$.
    If the size of $S$ is larger, then it is the new best feasible solution and it is stored in the variable $Best$.
    This is shown in \Cref{line:BegTestIfBest,line:EndTestIfBest}.
\item All the successive runs of the procedure with input $S$ have inputs that are subsets of $S\cup \cand{S}$ and thus have size at most $|S|+|\cand{S}|$.
    If this is no larger than the feasible solution stored in $Best$, then none of the inputs in the successive runs are a better feasible solution and therefore the algorithm backtracks.
    This is accomplished in \Cref{line:BegIsThereHope,line:EndIsThereHope}.
\item In \Cref{line:While} we perform pruning as follows.
    For a feasible solution $S$ with $\cand{S}=\{s_0, \dots, s_{r-1}\}$, the algorithm is not run recursively for any of the children $S\cup \{s_i\}$ with 
    \[i > |S|+|\cand{S}| -|Best|.\]
    This is because, for every such $i$, all the subsequent runs of the procedure with input $S\cup \{s_i\}$ have inputs subsets of $S\cup \{s_{i+1}, \dots, s_{r-1}\}$ which is no larger than the feasible solution stored in $Best$.
\item The check in \Cref{line:NecklaceCheck} causes the algorithm to recurse only for the children that are canonical. 
    This is done to follow framework of \Cref{algorithm:GeneratingCanonicalSubsets}, which guarantees that the search is restricted to canonical sets.
\end{itemize}
The results of our experimental runs of the algorithm are discussed in \Cref{section:CAsFirstPaper_ImplementationAndNewBounds}.

We close the section with a rough estimate for the number of candidate solutions that are examined in \Cref{algorithm:ExtendCM}.
First, we note that for any optimal solution $S$ we have that $\cand{S}= \emptyset$.
Therefore, the nodes of the tree $T_{\C_{\w{t}}}$ visited by the algorithm are among the nodes of size up to $m$, where $m$ is the size of an optimal solution.
Let $N_i$ be the number of nodes of size $i$ in $T_{\C_{\w{t}}}$ ---which is also the number of binary necklaces of length $\w{t}$ and \mbox{weight $i$}--- and let $N_{\leq m}=\sum_{i=0}^{m}N_i$ be the number of nodes with weight at most $m$.
It is known (see for example \cite{gilbert1961symmetry,ruskey1999efficient}) that the number of binary necklaces of length $n$ and weight~$i$ is
\[
    \frac{1}{n}
    \sum_{j\mid \gcd(i,n-i)}\varphi(j)
    \binom{n/j}{i/j},
\]
where $\varphi$ is Euler's phi function, that is, $\varphi(j)$ is the number of integers in $[1,j-1]$ that are coprime to $j$.
It follows that
\begin{equation}
    \label{equation:teleia}
    N_{\leq m}
    = \sum_{i=0}^{m}N_i
    = \frac{1}{\w{t}} \sum_{j\mid \gcd(i,\w{t}-i)}
        \varphi(j) 
        \binom{\w{t}/j}{i/j}.
\end{equation}
This provides an upper bound for the number of candidate solutions in the search space.
Since our algorithm follows the framework of \Cref{algorithm:RuskeyBinaryNecklaces}, this implies that at most $2N_{\leq m}$ nodes are visited by \Cref{algorithm:Backtracking}.
This bound is not tight since not all nodes are feasible solutions and, furthermore, feasible candidates are also eliminated in 
\Cref{line:BegIsThereHope,line:While} of \Cref{algorithm:Backtracking}.
Finally, we note that in comparison a very naive algorithm would require checking $2^{\w{t}}$ subsets of columns, whereas the less naive approach of our algorithm without the isomorphism pruning that uses binary necklaces, would require checking up to $\sum_{i=0}^m \binom{\w{t}}{i}$ nodes, where $m$ is the size of an optimal solution.


\section{The choice of primitive elements}
\label{section:CAsFirstPaper_ChoiceOfPrimitiveElements}

In this section we give an algorithmic answer to \Cref{problem:ChoiceOfPrimElements}.
A straightforward approach is to consider every possible set $P=\left\{ \a_0, \dots, \a_{l-1} \right\}$ of $l$ distinct primitive elements of $\fqt$, and run \Cref{algorithm:Backtracking} in order to find the solution $S_P$ to \Cref{problem:FindMaxSubset}.Then, for a set $S_P$ with maximum size, we have that $P$ is a solution to \Cref{problem:ChoiceOfPrimElements}.
By \Cref{lemma:NumberOfPrimitiveElementsOfFqm} there exist exactly $\phi(q^t-1)$ distinct primitive elements of $\fqt$, hence this approach requires running \Cref{algorithm:Backtracking} for all the ${\binom{\phi(q^t-1)}{l}}$ choices for the set $P$.

In this section we show that by considering a particular subset of primitive elements, it suffices to run \Cref{algorithm:Backtracking} at most $\binom{\phi(\w{t}/te)}{l}$ times, where $p$ is a prime such that $q=p^e$.
More precisely, we create $\varphi(\w{t}/te)$ classes of elements of $\fqt$ with the property that two primitive elements are in the same class if and only if their corresponding cyclic trace arrays have equivalent coverage properties in a sense that we strictly define later.
Then we prove that it suffices to carry out \Cref{algorithm:Backtracking} only for $l$-sets of representatives from the classes that contain primitive elements, which means that 
at most $\binom{\phi(\w{t}/te)}{l}$ runs are required.
To define these classes and their representatives, we need the notion of cyclotomic cosets.

\index{Cyclotomic coset|see{Coset}}
\index{Coset!cyclotomic}
\begin{definition}{Cyclotomic cosets}{CyclotomicCosets}
    Let $p$ be a prime, $w$ be a positive integer coprime to $p$, and $i \in \ZWx$.
    Then, the \emph{cyclotomic coset of $p$ modulo $w$ that contains $i$}
    is the set
    \begin{equation*}
        \cipW
        =
        \left\{
            ip^r \bmod{w} \mid r \in \Z, r\geq 0
        \right\}.
    \end{equation*}
\end{definition}
We can use cyclotomic cosets to partition $\ZWx$ as follows.

\index{Coset!partition into}
\begin{lemma}{Partition of $\ZWx$ using cyclotomic cosets}
    {PartitionZwxwithCyclotomicCosets}
    For a prime power $p$ and positive integer $w$ coprime to $p$, there exists
    a set $\gpW \subset \ZWx$ with the property that
    \begin{equation}
        \label{equation:partitioningofZwxintocosets}
        \ZWx = \bigcup_{i\in \gpW}\cipW,
    \end{equation}
    and $\cipW \cap \cjpW = \emptyset$ for all $i,j \in \gpW$, $i \neq j$.
\end{lemma}
\begin{proof}
    First, for any $i\in \ZWx$ and for every nonnegative integer $r$, we have that $\gcd(ip^r, w)=1$, since both $i$ and $p$ are coprime to $w$, which shows that indeed $\cipW \subset \ZWx$.
    Now, let $j \in \ZWx \setminus \cipW$ and suppose by means of contradiction
    that $\cipW \cap \cjpW \neq \emptyset$.
    Then, there exist nonnegative integers $r_1, r_2$ such that 
    \begin{equation}
        \label{equation:othoni}
        ip^{r_1} \equiv jp^{r_2} \Mod w.
    \end{equation}
    Without loss of generality, we assume that $r_1 \geq r_2$. 
    Then, \Cref{equation:othoni} implies that
    \[p^{r_2}(j-ip^{r_1-r_2}) \equiv 0 \Mod w,\]
    and since $\gcd(p,w)=1$, it follows that $j\equiv ip^{r_1-r_2} \Mod w$. 
    The latter means that $j \in \cipW$, which contradicts our initial assumption about $j$.
\end{proof}

\index{Coset!leaders}
\index{Leader|see{Coset}}
\begin{definition}{Cyclotomic coset leaders}{CyclotomicCosetLeaders}
    The elements of $\gpW$ in
    \Cref{lemma:PartitionZwxwithCyclotomicCosets}
    are \emph{cyclotomic coset leaders of $p$ modulo $w$.}
\end{definition}
Next, we define the notion of equivalent coverage that we mention earlier.

\index{Same coverage}
\begin{definition}
{Arrays with equivalent coverage}
{ArraysWithEquivalentCoverage}
    Let $t\ge 2$, and $A_1$, $A_2$ be $N\times k$ arrays with elements from an alphabet of the same finite size.
    We denote $C_i, D_i$, $i\in [0,k-1]$ their columns, respectively.
    Then, \emph{$A_1$ and $A_2$ have the same $t$-coverage} if for every $I\subseteq [0,k-1]$ with $|I|=t$, we have that $\left\{ C_i\mid i \in I \right\}$ is covered if and only if $\left\{ D_i\mid i\in I \right\}$ is covered.
\end{definition}

The main theorem of this section is the following.
\begin{theorem}{}{CoverageAndCosets}
    Let $q$ be a power of a prime $p$, $t$ be a positive integer with $t\geq 2$, and $\a$ be a primitive element of $\fqt$.
    Then, the following hold.
    \begin{enumerate}
        \item
            For any primitive element $\beta \in \fqt$, there exists $i \in \gpw$ such that $\a^i$ is also primitive and $\ZA{\beta}$ and $\ZA{\a^i}$ have the same $t$-coverage.
        \item
            For all $i, j \in \gpw$ with $i \neq j$, we have that $\ZA{\a^i}$ and $\ZA{\a^j}$ do not have the same $t$-coverage.
    \end{enumerate}
\end{theorem}
\begin{remark}
   {}
   {ImplicationsOfMainTheorem}
    Before we give the proof of \Cref{theorem:CoverageAndCosets}, we discuss its implications regarding \Cref{problem:ChoiceOfPrimElements}.
    First, we observe that we can write $\gpw=X\cup X'$ where
    \begin{align*}
        X&=\{i \in \gpw \mid \text{ there exists } j \in \cipw \text{ such that } \a^j \text{ is primitive}\}\\
        &=\{i \in \gpw \mid \text{ there exists } j \in \cipw \text{ such that } \gcd(j,q^t-1)=1\},
    \end{align*}
    and $X'=\gpw \setminus X$.
    Without loss of generality, we can choose the elements of $\gpw$ so that $\a^i$ is primitive for every $i \in X$.
    Then, it follows from \Cref{theorem:CoverageAndCosets} that in order to solve \Cref{problem:ChoiceOfPrimElements}, it suffices to run \Cref{algorithm:Backtracking} for $P=\{\a^{i_0}, \dots, \a^{i_{l-1}}\}$, for all the $l$-sets $\{i_0, \dots, i_{l-1}\}\in \binom{X}{l}$.
    From the next proposition we have that $|\gpw|=\phi(\w{t})/te$, where $q=p^e$ for a prime $p$. 
    Hence, we conclude that to solve \Cref{problem:ChoiceOfPrimElements}, it suffices to run \Cref{algorithm:Backtracking} at most $\binom{\phi{\w{t}}/te}{l}$ times.
\end{remark}

\begin{proposition}{}{CipwCardinality}
    Let $p$ be a prime, and $e,q,t$ be integers with $e>0$, $q=p^e$, $t\geq 2$.
    Then, we have that $|\cipw|=te$, for all $i \in \Zwx$.
\end{proposition}

\begin{proof}
    The size of $\cipw$ is equal to the smallest positive integer $r$ such that
    $ip^r\equiv i \Mod {\w{t}}$, which is equivalent to ${\w{t}}| p^r-1$, since
    $\gcd(i,\w{t})=1$.
    Now, ${\w{t}}| q^t-1 = p^{te}-1$, and thus $p^{te}\equiv 1 \Mod {\w{t}}$.
    We note that, by its definition, $r$ is the order of $p$ in $\Zwx$ and
    therefore  $p^{te}\equiv 1 \Mod {\w{t}}$ implies that $r | te$, hence $r \leq te$.
    We assume by means of contradiction that $r<te$.
    Then, since $r| te$, it must be $r \leq te/2$.
    On the other hand, since ${\w{t}} | p^r-1$, we have that 
    \[p^r-1 \geq {\w{t}} > p^{te-e}-1,\]
    and thus $r> te-e$.
    Then $te/2 > te-e$ which simplifies to $t<2$, contradicting our assumption that $t\geq 2$.
    We conclude that $r=te$.
\end{proof}

The rest of the section is dedicated to the proof of
\Cref{theorem:CoverageAndCosets}, which we present after
\Cref{lemma:SameCoverageOfPrimitiveElements,lemma:LGoesOutThenIsPowerOfP,lemma:SameZerosThenGammInFqAndLOutSide,lemma:ijEquivalentIffJIsInCosetOfi}.

\begin{lemma}{}{SameCoverageOfPrimitiveElements}
    Let $q$ be a prime power, $t$ be an integer with $t\geq 2$, and
    $\a, \beta$ primitive elements of $\fqt$.
    Then, $\ZA{\a}$ and $\ZA{\beta}$ have the same $t$-coverage if and only if
    there exists $\gamma \in \fqt$ such that, for all $s \in [0,q^t-2]$,
    $
        \Tt(\a^s)=0
        \text{ if and only if }
        \Tt(\gamma\beta^s)=0.
    $
\end{lemma}

\begin{proof}
    We denote the column vectors of $\ZA{\a}$ and
    $\ZA{\beta}$ by
    $A_0, \dots, A_{\w{t}-1}$, and $B_0, \dots, B_{\w{t}-1}$, respectively.

    ``$\Leftarrow$''
    Let $I\subseteq [0,\w{t}-1]$ such that $|I|=t$ and
    $\left\{ A_i\mid i\in I \right\}$
    is covered; we need to show that
    $\left\{ B_i\mid i\in I \right\}$
    is also covered.
    We suppose by means of contradiction that
    this is not the case.
    Then, by \Cref{item:f} of
    \Cref{proposition:SebastianExtended},
    there exists a row of $\A{\beta}$ that has zeros in the entries
    corresponding to the columns $B_i$, $i\in I$.
    In other words, there exists some $r \in [0, q^t-2]$ such that
    \[
        \left( \A{\beta}_{r,i} \right)_{i \in I}
        =\left(  \Tt(\a^r\a^{i})\right)_{i \in I}
        =(0,\dots, 0).
    \]
    From the definition of $\A{\beta}$, the latter means that there
    exists $\gamma\in \fqt^*$ such that $\T(\gamma\beta^i)=0$ for all $i \in I$.
    Setting $\gamma = \a^r$, from our assumptions it follows that
    $\Tt(\a^i)=0$ for all $i\in I$,
    which means that the first row of the columns
    $A_i$, $i \in I$ is a row of zeros.
    Then, by \Cref{item:f} of
    \Cref{proposition:SebastianExtended},
    the columns $A_i$, $i \in I$ is uncovered; this contradicts our
    assumptions.

    ``$\Rightarrow$''
    By
    \Cref{lemma:CharacterizationOfConstantMultiplesInFQM}, 
    $\a^{\w{t}}$ is a primitive element of $\fqstar$, thus for every $s\in [0,
    q^t-2]$ we have that
    $\a^s=c\a^u$, where $u \in [0,\w{t}-1]$, $u\equiv s\Mod w$, and $c \in \fq$.
    Hence, from the linearity of the trace over $\fq$, it is sufficient  to
    prove this direction for all $s \in [0,\w{t}-1]$.
    Now, $\ker(\T)$ is a vector space over
    $\fq$ with dimension $t-1$, as implied by
    \Cref{proposition:TraceDefinesQMinusOneToOneMapping}.
    Since $\a$ is primitive, it follows from the above that there exist $i_1,
    \dots, i_{t-1} \in [0,\w{t}-1]$ such that
    $\mathcal{B}=\{ \a^{i_1}, \dots, \a^{i_{t-1}}\}$
    is also a basis for $\ker(\T)$.
    This means that $\a^{i_1}, \dots, \a^{i_{t-1}}$ are linearly
    independent and the first row of $\A{\a}$ has zeros at the columns
    $A_i$, $i \in \{ i_1, \dots, i_{t-1} \}$.
    Then, from
    \Cref{proposition:SebastianExtended}
    and our assumption that $\A{\a}$ and $\A{\beta}$ have the same
    $t$-coverage, we have that $\beta^{i_1}, \dots, \beta^{i_{t-1}}$ are also
    linearly independent, and there exists a row of $M(\beta)$ with zeros at
    the columns
    $B_i$, $i \in \{ i_1,\dots, i_{t-1} \}$.
    The latter means that there exists $\gamma \in \fqt^*$ such that
    $\T(\gamma\beta^i)=0$ for all $i \in \{ i_1, \dots, i_{t-1} \}$.
    We conclude that $\mathcal{B}'=\{ \gamma\beta^{i_1}, \dots,
    \gamma\beta^{i_{t-1}}\}$ is a basis for $\ker(\T)$.

    Now, suppose that $\T(\a^s)=0$ for some $s \in [0,\w{t}-1]$.
    Then, by
    \Cref{proposition:SebastianExtended},
    the set of columns
    $\{ A_{i_1}, \dots, A_{i_{t-1}}, A_s\}$ of
    $\A{\a}$ is not covered and thus the set of columns
    $\{B_{i_1}, \dots, B_{i_{t-1}}, B_s\}$ of
    $\A{\beta}$ is also not covered, from our assumption that
    $\A{\a}$ and $\A{\beta}$ have the same $t$-coverage.
    Hence there exists a row of $\A{\a}$ with zeros at the columns
    $B_{i_1}, \dots, B_{i_{t-1}}, B_s$, and so there exists
    $\delta \in \fqt^*$ such that
    $\delta\beta^i \in \ker(\T)$ for all
    $i\in \{i_1, \dots, i_{t-1}, s\} $.
    Since $\mathcal{B}'$ is a basis for $\ker(\T)$, we have that $\delta =
    c\gamma$ for some $c \in \fqstar$.
    Then $\delta \beta^s \in \ker(\T)$ implies that
    $\gamma\beta^s \in \ker(\T)$ from the linearity of the trace.
\end{proof}

\begin{lemma}{}{LGoesOutThenIsPowerOfP}
    Let $p$ be a prime, $q,e,t,l$ be integers with $e>0$, $q=p^e$, $t\geq 2$,
    and $l\in [1,q^t-2]$.
    Then, $\Tt(x^l) = \Tt(x)^l$ for all $x \in \fqt$ if and only if
    $l=p^r$ for some $r \in [0,te-1]$.
\end{lemma}

\begin{proof}
    In this proof we denote $\T=\Tt$.
    If $l=p^r$,$r\in [0,te-1]$, then $\T(x^l) = \T(x)^l$ from the properties of
    the Frobenius automorphism in $\fq$.
    Conversely, suppose that $\T(x^l) =(\T(x))^l$ for all $x \in \fqt$.
    For a polynomial $f$ on $x$, we denote by $[x^n]f(x)$ the coefficient of
    $x^n$ in $f$.  We observe that
    \begin{equation}
        \label{equation:coefficientoftr_x_l}
        [x^{1+(l-1)q}]\T(x^l)=
        \begin{cases}
            1,  & \text{ if } l = 1 \\
            0,  & \text{ otherwise},
        \end{cases}
    \end{equation}
    and
    \begin{equation}
        \label{equation:coefficientof_trx_l}
        [x^{1+(l-1)q}]\left(\T(x)\right)^l=
        \begin{cases}
            1,  & \text{ if } l = 1 \\
            l,  & \text{ otherwise}.
        \end{cases}
    \end{equation}
    If $l=1$, then $l=p^r$ with $r=0$.
    If $l>1$, then it follows from
    \Cref{equation:coefficientoftr_x_l,equation:coefficientof_trx_l},
    and our assumption that $\T(x^l) = \T(x)^l$, that $l \equiv 0 \Mod p$.
    Hence $l= k p^r$ for some positive integers $k,r$, with $0 < r < te$, and $p \nmid k$.
    We have that, for all $x \in \fqt$,
   \begin{align}
   \left(T(x^k)\right)^{p^r}
        =  \T(x^{kp^r})
        =  \T(x^l)
    =  \left(\T(x)\right)^l
        =  \left( \T(x)^k \right)^{p^r}.
   \label{equation:kgoesoutofthetracelikel}
   \end{align}
    Taking $p^r$-th roots in
    \Cref{equation:kgoesoutofthetracelikel}
    yields that
$\T(x^k) = (\T(x))^k$ for all $x \in \fqt$.
    By comparing the coefficients of $\T(x^k)$ and $\left(\T(x)\right)^k$ in the same way as
    we did for $\T(x^l)$ and $(\T(x))^l$, we have that either $k=1$, or
    $k\equiv 0 \Mod p$.
    Since we have assumed that $p \nmid k$, it must be $k=1$, and thus $l =p^r$.
\end{proof}

\begin{lemma}{}{SameZerosThenGammInFqAndLOutSide}
    Let $p$ be a prime, $q,e,t$ be integers such that $e>0$, $q=p^e$, $t\geq 2$,
    and $\a,\beta$ be primitive elements of $\fqt$.
    Then $\ZA{\a}$ and $\ZA{\beta}$ have the same coverage if and only if
    $\beta=\a^{p^r}$, for some $r\in [0, te-1]$.
\end{lemma}

\begin{proof}
    If $\beta=\a^{p^r}$ for some $r\in [ 0,te-1]$, then for every
    $s \in [0,q^t-2]$
    we have that
    $\T(\beta^s)=\T(\a^{sp^r})=\T(\a^s)^{p^r}$.
    Hence, $\T(\beta^s)=0$ if and only if $\T(\a^s)=0$, and thus
    $\ZA{\a}$ and $\ZA{\beta}$ have the same $t$-coverage,
    from \Cref{lemma:SameCoverageOfPrimitiveElements}.

    For the converse, we assume that $\ZA{\a}$ and $\ZA{\beta}$
    have the same $t$-coverage.
    Since $\a$ is primitive, there exists $l \in \Z_{q^t-1}^{*}$ such that $\beta=\a^l$.
    Then, from
    \Cref{lemma:SameCoverageOfPrimitiveElements},
    there exists $\gamma\in \fqt$ such that, for all $s \in [ 0,q^t-2 ]$, we
    have $\T(\a^s)=0$ if and only if $\T(\gamma\a^{ls})=0$.
    Again from the primitivity of $\a$, we have that
    \[\fqt^*=\left\{ \a^s \mid s \in [0, q^t-2] \right\},\]
    so we conclude from the above that there exists $\gamma \in \fqt$ such
    that, for all $x \in \fqt^{*}$, we have
    \begin{equation}
        \label{equation:tracexzeroifftracegammaxlzero}
        \T(x)=0 \quad \mbox{if and only if}\quad \T(\gamma x^{l})=0.
    \end{equation}

    Let $y$ be an element in some extension of $\fqt$ such that $\T(\gamma y^l)= 0$.
    Then $\gamma y^l = z \in \ker(\T)\subseteq \fqt$, and $y^l = z/\gamma \in \fqt$.
    Since $\gcd(l, q^t-1)=1$, the $l$-th root of $z/\gamma$ exists, and $y = (z/\gamma)^{1/l} \in \fqt$.
    We have proved that $\T(\gamma x^l)$ splits in $\fqt$.
    Now,
    \begin{equation*}
        \T(\gamma x^l) =
        \prod_{a \in \ker{\left( \T \right)}}
        \left( \gamma x^l -a \right).
    \end{equation*}
    Because $\T(\gamma x^l)$ splits in $\fqt$, so does $\gamma x^l - a$ for all
    $a \in \ker(\T)$.
    Furthermore, the only root of $\gamma x^l -a$ is $(a/\gamma)^{1/l}$, and
    its degree is $l$; it follows that it must be
    \[\gamma x^l -a = \gamma\left(x - (a/\gamma)^{1/l}\right)^l,\]
    and hence
    \begin{equation}
        \label{equation:productfactorshaveloutside}
        \T(\gamma x^l) =
        \prod_{a \in \ker{\left( \T \right)}}
        \gamma\left(x - (a/\gamma)^{1/l}\right)^l.
    \end{equation}
    By
    \Cref{equation:tracexzeroifftracegammaxlzero}
    we have that
    \begin{equation*}
        \ker(\T) =
        \left\{
            \left( a/\gamma \right)^{1/l} \mid \; a \in \ker(\T)
        \right\},
    \end{equation*}
    and by
    \Cref{proposition:TraceDefinesQMinusOneToOneMapping}
    we have that $|\ker(\T)|=q^{t-1}$, hence
    \Cref{equation:productfactorshaveloutside}
    becomes
    \begin{align}
        \label{equation:gqmminus1goesoutside}
        \T(\gamma x^l)
        & = \gamma^{q^{t-1}}
            \prod_{a \in \ker(\T)}
            (x-a)^l
        \nonumber\\
        & = \gamma^{q^{t-1}}
        \left(
        \prod_{a \in \ker(\T)}(x-a)
        \right)^{l}
        \nonumber\\
        & = \gamma^{q^{t-1}}
        \T(x)^l.
    \end{align}
    By comparing the coefficient of $x^l$ in $\T(\gamma x^l)$ and
    $\gamma^{q^{t-1}}(\T(x))^l$, we have that $\gamma = \gamma^{q^{t-1}}$,
    which means that $\gamma \in \f_{q^{t-1}}$.
    However $\gamma \in \fqt$, hence $\gamma \in \fqt \cap \f_{q^{t-1}} = \fq$,
    and from the linearity of the trace over $\fq$,
    $\T(\gamma x^l) = \gamma \T(x^l)$.
    \Cref{equation:gqmminus1goesoutside}
    then implies that $\T(x^l) = (\T(x))^l$, and by
    \Cref{lemma:LGoesOutThenIsPowerOfP}
    we have that $l=p^r$, for some integer $r$ such that $r\in[0,te-1]$.
\end{proof}

\begin{lemma}{}{ijEquivalentIffJIsInCosetOfi}
    Let $p$ be prime, and $q,e,t$ integers with $e>0$ and $q=p^e$.
    For all $i,j \in \Z_{q^t-1}^{*}$, we have that $\ZA{\a^i}$ and
    $\ZA{\a^j}$ have the same $t$-coverage if and only if
    $j\bmod {\w{t}} \in \cipw$.
\end{lemma}

\begin{proof}
    Suppose that $j\bmod {\w{t}} \in \cipw$.
    Then there exist integers $r,h$ such that $j=ip^r+h\w{t}$, and thus
    $\a^j=c\a^{ip^r}$ with $c=\a^{\w{t}h}$.
    We have that 
    \[c^{q-1}=\a^{\w{t}(q-1)h}=\a^{(q^t-1)h}=1,\] 
    which means that $c \in \fq$.
    Then, by
    \Cref{theorem:PropertiesOfTrace} and the properties of the Frobenius
    automorphism
    we have that, for all positive integers $s$,
    \begin{equation*}
        \T(\a^{js})
        =\T(c^s \a^{isp^r})
        =c^s\T(\a^{is})^{p^r}.
    \end{equation*}
    We conclude that $\T(\a^{js}) = 0$ if and only if $\T(\a^{is}) = 0$, which implies from
    \Cref{lemma:SameCoverageOfPrimitiveElements}
    that $\ZA{\a}$ and $\ZA{\beta}$ have the same $t$-coverage.

    Conversely, assume that $\ZA{\a}$ and $\ZA{\beta}$ have the
    same $t$-coverage.
    Then, from
    \Cref{lemma:SameZerosThenGammInFqAndLOutSide}
    we have that $\a^{j}= \a^{ip^r}$ for some $r \in [0, te-1]$, and thus
    $j \equiv ip^r \Mod{q^t-1}$.
    Since $\w{t}|q^t-1$, we also have $j\equiv ip^r \Mod{\w{t}}$, which means that $j
    \bmod {\w{t}} \in \cipw$.
\end{proof}

We now have the necessary background to give the proof of the main result of this
section.
\\

\begin{proof}[Proof of \Cref{theorem:CoverageAndCosets}]
    We begin with the first part.
    Let $\beta$ be a primitive element of $\fqt$.
    From the primitivity of $\a$, we have that there exists $l \in
    \Z_{q^t-1}^{*}$ such that $\beta=\a^l$.
    Let $u=l \bmod {\w{t}}$.
    Then $u=l+h{\w{t}}$ for some integer $h$,
    and thus $\a^{u}=c\a^l$, with $c =\a^{h\w{t}}$.
    From
    \Cref{lemma:CharacterizationOfConstantMultiplesInFQM}
    we have that $c \in \fq$.
    Hence, for any positive integer $s$, we have that
    $\T(\a^{us})=c^s\T(\a^{ls})$ and therefore $\T(\a^{ls})=0$ if and only if
    $\T(\a^{us})=0$.
    It follows from
    \Cref{lemma:SameCoverageOfPrimitiveElements}
    that $\ZA{\a}$ and $\ZA{\beta}$ have the same $t$-coverage.
    Since $\gcd(l,q^t-1)=1$, then also $\gcd(l, \w{t})=1$, hence
    $\gcd(u,\w{t})=1$ as
    well.
    This means $u \in \Zwx$ and thus, from
    \Cref{equation:partitioningofZwxintocosets}
    there exists $i \in \gpw$ such that $u \in \cipw$.
    From \Cref{lemma:ijEquivalentIffJIsInCosetOfi},
    $\ZA{\a^u}$ has the same coverage with $\ZA{\a^i}$
    Since $\ZA{\a^u}$ was shown above to also have the same coverage as
    $\ZA{\beta}$, we conclude that $\ZA{\beta}$  has the same
    coverage with $\ZA{\a^i}$.

    We now prove the second part.
    Suppose by means of contradiction that $i,j\in \gpw$, $i\neq j$, and
    $\ZA{\a^i}$
    has the same $t$-coverage with
    $\ZA{\a^j}$.
    Then, from
    \Cref{lemma:ijEquivalentIffJIsInCosetOfi}
    we have that $j \in \cipw$.
    Thus, $\cipw \cap \cjpw \neq \emptyset$ which means that $\cipw=\cjpw$, as discussed just before
    \Cref{equation:partitioningofZwxintocosets}.
    This contradicts our assumption that $i,j \in \gpw$, and we conclude that
    $\ZA{\a^i}$ and $\ZA{\a^j}$ do not have the same coverage.
\end{proof}

We close this section by giving the steps for solving \Cref{problem:ChoiceOfPrimElements} as a method.

\begin{method}
    {Solving \Cref{problem:ChoiceOfPrimElements} using \Cref{algorithm:Backtracking}}
    {MethodForSolvingProblemTwo}
    Let $p$ be prime, $q=p^e$ where $e$ is a positive integer,  $t\geq 2$, and $l\geq 2$.
        \begin{enumerate}
            \item Calculate the cyclotomic cosets $\cipw$ for all $i \in \Zwx$, and find a set $X$ of coset representatives such that $\cipw$, $i \in X$ are all the distinct cosets with the property that, for every $i \in X$, there exists $j \in \cipw$ with $\gcd(j,q^t-1)$.
                Without loss of generality, choose the elements of $X$ so that they are all coprime to $q^t-1$.
            \item
            Pick \emph{any} primitive polynomial of degree $t$ over $\fq$ and let $\a$ be
            \emph{any} one of its roots.
            \item
                For every $\{i_0, \dots, i_{l-1}\} \in \binom{X}{l}$, set $P=\{ \a^{i_0}, \dots, \a^{i_{l-1}} \}$ and find a solution $S_P$ to \Cref{problem:FindMaxSubset} for $P$, using \Cref{algorithm:Backtracking}.
            \item
                Any $P$ in the previous step that corresponds to a set $S_P$ of the maximum size, is a solution to \Cref{problem:ChoiceOfPrimElements}.
        \end{enumerate}
\end{method}
As discussed before, this method requires running \Cref{algorithm:Backtracking} $\binom{|X|}{l}$ times, $|X|\leq \phi(\w{t}/te)$.
\allowdisplaybreaks

\newpage
\section{Implementation and experimental results}
\label{section:CAsFirstPaper_ImplementationAndNewBounds}

In this section we discuss our implementation of \Cref{method:MethodForSolvingProblemTwo} and present our experimental results.
These are explicit constructions of $CA(l(q^t-1);t,k,q)$ of the form $\ZA{P,S}$, where $P$ is a set of $l$ primitive elements of $\fqt$, and $S$ is a subset of $[0, \w{t}-1]$ of size $k$.
Our results are for $t=4$, with the exception of a few cases for $t=5$. 

We used version 6.8 of the Sage mathematical software \cite{stein2008sage} for the precomputations that involved finite fields, whereas \Cref{algorithm:Backtracking} was implemented in version 2.7 of Python using the Cython extension \cite{behnel2011cython} that translates Python code to C for increased speed.
We used six computers each with an Intel Xeon CPU E5-2667 processor and 16 GB of memory, running the Scientific linux operating system.

Our experiments were as follows.
We ran \Cref{method:MethodForSolvingProblemTwo} for every $q$ and $l$ shown in \Cref{table:ResultsOverview}. 
For the cases marked with an asterisk, the method was complete and the entries in the column denoted $\CA(N;4,k,q)$ are the covering array parameters corresponding to the solution of \Cref{problem:ChoiceOfPrimElements}.
For all the other cases, due to the large search space of the problem, the fourth step was not complete. 
More precisely, it was not possible to consider all the $l$-sets in $\binom{X}{l}$ and for the cases that were considered, \Cref{algorithm:Backtracking} did not terminate.
For these cases, we examined up to $30$ $l$-tuples from $\binom{X}{l}$ at random, and for each of them \Cref{algorithm:Backtracking} ran for up to a month.
By the end of that period, for every choice of $q$ and $l$ we compared the results stored in the global variable $Best$ for all the $l$-sets that were tested; the entries in column denoted $\CA(N;4,k,q)$ are the parameters of the corresponding arrays.

\begin{table}
\renewcommand{\arraystretch}{\genarraystretch}
\centering
\begin{subtable}{1\textwidth}
\rowcolors{1}{\backgroundshade}{white}
\centering
\setlength{\tabcolsep}{7.8pt}
\begin{tabular}{clllrlcllr}
\rowcolor{\tableheadcolor}
 & $q$&$l$ &  $\CA(N;4,k,q)$                  & PrevN         &\,&$q$&$l$& $\CA(N;4,k,q)$                    & PrevN         \\
*& 2  & 2  &  $\CA(        31    ; 4, 6,  2 )$& 21            &\,& 9 & 2 & $\CA(        13121  ; 4, 18, 9 )$ & 13113         \\
*& 3  & 2  &  $\CA(        161   ; 4, 10, 3 )$& 159           &\,& 9 & 3 & $\CA(\textbf{19681} ; 4, 42, 9 )$ &\textbf{30537 }\\
*& 3  & 3  &  $\CA(        241   ; 4, 12, 3 )$& 189           &\,& 9 & 4 & $\CA(\textbf{26241} ; 4, 50, 9 )$ &\textbf{30537 }\\
*& 3  & 4  &  $\CA(        321   ; 4, 12, 3 )$& 189           &\,& 9 & 5 & $\CA(\textbf{32801} ; 4, 82, 9 )$ &\textbf{33129 }\\
*& 4  & 2  &  $\CA(\textbf{511}  ; 4, 17, 4 )$&\textbf{760   }&\,& 11& 2 & $\CA(        29281  ; 4, 21, 11)$ & 29271         \\
 & 4  & 3  &  $\CA(        766   ; 4, 20, 4 )$& 760           &\,& 11& 3 & $\CA(\textbf{43921} ; 4, 37, 11)$ &\textbf{69091 }\\
 & 4  & 4  &  $\CA(        1021  ; 4, 20, 4 )$& 760           &\,& 11& 4 & $\CA(\textbf{58561} ; 4, 77, 11)$ &\textbf{69091 }\\
 & 5  & 2  &  $\CA(\textbf{1249} ; 4, 16, 5 )$&\textbf{1865  }&\,& 11& 5 & $\CA(\textbf{73201} ; 4,125, 11)$ &\textbf{73931 }\\
 & 5  & 3  &  $\CA(\textbf{1873} ; 4, 25, 5 )$&\textbf{2845  }&\,& 13& 2 & $\CA(        57121  ; 4, 24, 13)$ & 57109         \\
 & 5  & 4  &  $\CA(        2497  ; 4, 23, 5 )$& 1865          &\,& 13& 3 & $\CA(\textbf{85681} ; 4, 45, 13)$ &\textbf{136045}\\
 & 7  & 2  &  $\CA(        4801  ; 4, 15, 7 )$& 4795          &\,& 13& 4 & $\CA(\textbf{114241}; 4, 98, 13)$ &\textbf{136045}\\
 & 7  & 3  &  $\CA(        7201  ; 4, 26, 7 )$& 7189          &\,& 13& 5 & $\CA(\textbf{142801}; 4,170, 13)$ &\textbf{146185}\\
 & 7  & 4  &  $\CA(        9601  ; 4, 43, 7 )$& 9583          &\,& 16& 2 & $\CA(\textbf{131071}; 4, 28, 16)$ &\textbf{188401}\\
 & 7  & 5  &  $\CA(        12001 ; 4, 47, 7 )$& 9583          &\,& 16& 3 & $\CA(\textbf{196606}; 4, 55, 16)$ &\textbf{315136}\\
 & 8  & 2  &  $\CA(        8191  ; 4, 17, 8 )$& 8184          &\,& 16& 4 & $\CA(\textbf{262141}; 4,129, 16)$ &\textbf{315136}\\
 & 8  & 3  &  $\CA(        12286 ; 4, 30, 8 )$& 12272         &\,& 17& 2 & $\CA(\textbf{167041}; 4, 29, 17)$ &\textbf{240721}\\
 & 8  & 4  &  $\CA(\textbf{16381}; 4, 48, 8 )$&\textbf{18880 }&\,& 17& 3 & $\CA(\textbf{250561}; 4, 61, 17)$ &\textbf{402577}\\
 & 8  & 5  &  $\CA(        20476 ; 4, 65, 8 )$& 19776         &\,& 17& 4 & $\CA(\textbf{334081}; 4,141, 17)$ &\textbf{402577}\\
 & 8  & 6  &  $\CA(        24571 ; 4, 67, 8 )$& 19776         &\,& 19& 2 & $\CA(\textbf{260641}; 4, 30, 19)$ &\textbf{377227}\\
 &    &    &                                  &               &\,& 23& 2 & $\CA(\textbf{781249}; 4, 35, 23)$ &\textbf{815167}
\end{tabular}
\caption[Overview of experimental results]{Overview of our results, where $N=l(q^4-1)+1$.
The columns denoted PrevN contain the previous smallest upper bounds for $\CAN(4,k,v)$; bold indicates improvement.} 
\label{table:ResultsOverview}
\end{subtable}

\vspace{2em}

\begin{subtable}{1\textwidth}
\centering
\renewcommand{\arraystretch}{\genarraystretch}
\rowcolors{1}{\backgroundshade}{white}
\setlength{\tabcolsep}{6.2pt}
\begin{tabular}{rlllrllllr}
\rowcolor{\tableheadcolor}
$q$ & $l$ & $r$ & $\CA(N-2r;4,k,q-r)$                & PrevN          &$q$ & $l$ & $r$ & $\CA(N-2r;4,k,q-r)$ & PrevN \\
11  & 3   &1    & $\CA(\textbf{43919  }; 4, 37, 10 )$&\textbf{57486}  &17  &3    &2    &$\CA(\textbf{250560}; 4, 61,15 )$ &\textbf{278181} \\
11  & 4   &1    & $\CA(\textbf{58559  }; 4, 77, 10 )$&\textbf{66545}  &16  &2    &1    &$\CA(\textbf{131069}; 4, 28,15 )$ &\textbf{173727} \\
13  & 3   &1    & $\CA(\textbf{85679  }; 4, 45, 12 )$&\textbf{114186} &16  &3    &1    &$\CA(\textbf{196604}; 4, 55,15 )$ &\textbf{277827} \\
13  & 4   &1    & $\CA(\textbf{114239 }; 4, 98, 12 )$&\textbf{129345} &16  &4    &1    &$\CA(\textbf{262139}; 4,129,15 )$ &\textbf{315134} \\
16  & 2   &2    & $\CA(\textbf{131067 }; 4, 28, 14 )$&\textbf{147753} &17  &2    &1    &$\CA(\textbf{167039}; 4, 29,16 )$ &\textbf{188401} \\
16  & 3   &2    & $\CA(\textbf{196602 }; 4, 55, 14 )$&\textbf{226647} &17  &3    &1    &$\CA(\textbf{250559}; 4, 61,16 )$ &\textbf{315136} \\
16  & 4   &2    & $\CA(\textbf{262137 }; 4,129, 14 )$&\textbf{283193} &17  &4    &1    &$\CA(\textbf{334079}; 4,141,16 )$ &\textbf{315136} \\
17  & 2   &2    & $\CA(\textbf{167037 }; 4, 29, 15 )$&\textbf{173800} &19  &2    &1    &$\CA(\textbf{260639}; 4, 30,18 )$ &\textbf{355669} 
\end{tabular}  

\caption[New covering arrays from the fusion operation]{The results of the fusion operation on the arrays in \Cref{table:ResultsOverview}, that improve upon previously best bounds \cite{colbournwebsite}, where $N=l(q^4-1)+1-2r$.
    For every $q$ and $l$, the $\CA(N-2r;4,k,q-r)$ is obtained by applying the fusion operation $r$ times to the $\CA(N;4,k,q)$ for the corresponding $q$ and $l$ in \Cref{table:ResultsOverview}.
\index{Fusion operation}
\index{Covering array!fusion operation on}
}
\label{table:FusionResults}
\end{subtable}
\caption{Overview of our results}
\end{table}

Next, we evaluate our results by comparing them with the state of art as of the time of their publication.
We recall that the covering array number $\CAN(t,k,q)$ is the smallest
$n$ such that a $\CA(n; t,k,q)$ exists.
Hence, a $\CA(N; t,k,q)$ implies that $N$ is an upper bound for $\CAN(t,k,q)$.
In \Cref{table:ResultsOverview}, the columns labeled PrevN contain the previously
smallest known \cite{colbournwebsite} upper bounds for $\CAN(4,k,q)$, for the
$k$ and $q$ of the corresponding array.
Bounds in bold indicate that they are improved by our results, and the numbers
of rows of the corresponding arrays, also indicated in bold, are the new
smallest known upper bounds for $\CAN(4,k,q)$.

More results follow recursively from the fusion operation
\index{Fusion operation}
\index{Covering array!fusion operation on}
\cite{colbourn2010covering}
(see also \cpageref{equation:FusionOperationInequality}), where a $\CA(N; t, k, q)$, a $\CA(N-2r; t, k, q-r)$ can be constructed for any $r < q$.
In several cases, the result of the fusion operation on the arrays in \Cref{table:ResultsOverview} improve upon the previously best bounds in \cite{colbournwebsite};
we list these cases in \Cref{table:FusionResults}.
Although the fusion operation for all possible $r$ was tested, the results that improved upon previous ones were only for $r$ equal to $1$ or $2$.

In \Cref{table:LongTableNewarrays} we give the essential elements for the construction of the 21 arrays displayed in \Cref{table:ResultsOverview} that improve upon previous results.
Every covering array is of the form $\ZA{P,C}$, where $C$ is a subset of $[0,\w{4}-1]$ and $P= \{\a^{i_0}, \dots, \a^{i_{l-1}}\}$, where $\a$ and $\a^{i_0}, \dots, \a^{i_{l-1}}$ are primitive elements of $\fqfour$.
The three columns of \Cref{table:LongTableNewarrays} contain the covering array parameters, the powers $i_0, \dots, i_{l-1}$ and the subset $C$, respectively.
The primitive element $\a\in \fqfour$ depends on $q$; for every $q$ considered in \Cref{table:LongTableNewarrays}, the primitive element $\a$ that is used is the root of the primitive polynomial $P_q(x)\in \fqx$ given in \Cref{table:primitiveelementsused}.
In the next example we demonstrate how the covering arrays of \Cref{table:LongTableNewarrays} can be constructed in practice.
\begin{table}[t]
\centering
\rowcolors{2}{\backgroundshade}{white}
\renewcommand{\arraystretch}{\genarraystretch}
\begin{tabular}{lll}
\rowcolor{\tableheadcolor}
$q$  & 
Minimal polynomial in $\mathbb{F}_q[x]$ of $\a \in \fqfour$ &\\
4  & $P_4(x)= x^4 + (\b + 1)x^3 + \b x^2 + \b$, & $\ffour=\ftwo(\b), \b^2 = \b + 1  $      \\
5  & $P_5(x)= x^4 + x^3 + 2x^2 + 2$         &                                      \\
8  & $P_8(x)= x^4 + \b x^3 + \b$,              & $\feight=\ftwo(\b), \b^3 = \b + 1 $     \\ 
9  & $P_9(x)= x^4 + \b x^3 + \b$,              & $\fnine=\fthree(\b), \b^2 =  \b + 1 $   \\
11 & $P_{11}(x)= x^4 + 4x^3 + 2$               &                                      \\
13 & $P_{13}(x)= x^4 + 6x^3 + 2x^2 + 2$        &                                      \\
16 & $P_{16}(x)= x^4 + \b^2x^3 + \b x^2 + \b$,     & $\fsixteen = \ftwo(\b), \b^4 = \b + 1$  \\
17 & $P_{17}(x)= x^4 + 6x^3 + 3$               &                                      \\
19 & $P_{19}(x)= x^4 + x^3 + 2$                &                                      \\
23 & $P_{23}(x)= x^4 + 9x^3 + 5$               &                                      \\
\end{tabular}
\caption{Minimal polynomials of the primitive elements used in \Cref{table:LongTableNewarrays}.}
\label{table:primitiveelementsused}
\end{table}

\begin{example}
    {Construction of $\CA(1249;4,16,5)$}
    {ConstructionOfCAFromResultsTable}
    To obtain the $\CA(1249; 4, 16,5)$ in the second row of \Cref{table:LongTableNewarrays} we need to construct $\ZA{P,C}$ for $P=\{\a,\a^7\}$ where $\a$ is a root of $P_5(x)=x^4+x^3+2x^2+2$ as in \Cref{table:primitiveelementsused}, and
    \[
        C=\{0,6,9,15,39,45,48,54,78,84,87,93,117,123,126,132\}.
    \]
    To do that, we construct the $(5^4-1)\times 16$ arrays $\AA_{5^4/5}(\a,C)$ and $\AA_{5^4/5}(\a^7,C)$.
    Let $\T=\T_{5^4/5}$.
    Then, for every $i \in [0, 5^4-2]$, the $i$-th row of $\AA_{5^4/5}(\a,C)$ is the vector
    \[
        \left(\T(\a^{i+c})\right)_{c \in C},
    \]
    whereas the $i$-th row of $\AA_{5^4/5}(\a^{7},C)$ is given by
    \[
        \left(\T(\a^{7(i+c)})\right)_{c \in C}.
    \]
    The vertical concatenation of these rows and a row of zeros is a $\CA(1249; 4, 16, 5)$.
\end{example}

Although the results in \Cref{section:ABacktrackAlgoForProblem1,section:CAsFirstPaper_ChoiceOfPrimitiveElements} can be used to search for covering
arrays of any strength, the running time increases significantly for strengths
$t\geq 5$.
We were able to run a few cases of strength $t=5,6$ for small $q$.
One notable result was a $\CA(485; 5, 11, 3)$ which improves the upper bound of $\CAN(5,11,3)$ from $546$ to $485$.
This array can be constructed as $\ZA{ \{\a, \a^{17}\}, C}$, where $\a$ is a root of $x^5+2x^4+1 \in \fthree[x]$, and $C=\{ 11i\mid i =0, \dots, 10 \}$.

\newpage
{
\rowcolors{2}{\backgroundshade}{white}
\renewcommand{\arraystretch}{\genarraystretch}
\begin{longtable}{llP{8.5cm}}
\caption{Components of the new covering arrays.}\label{table:LongTableNewarrays}\\
\rowcolor{\tableheadcolor}
  $\ZA{ \{ \a^{i_1}, \dots, \a^{i_l}\}, C}$ 
& $i_1, \dots, i_l$ 
& \multicolumn{1}{c}{$C$}\\
\endfirsthead
\hiderowcolors
\multicolumn{3}{c}
{\tablename\ \thetable\ -- \textit{Continued from the previous page}} \\
\rowcolor{\tableheadcolor}
  $\ZA{ \{ \a^{i_1}, \dots, \a^{i_l}\}, C}$ 
& $i_1, \dots, i_l$ 
& \multicolumn{1}{c}{$C$}\\
\hiderowcolors
\endhead
\multicolumn{3}{r}{\textit{Continued on the next page}} \\
\endfoot
\endlastfoot
\showrowcolors

$ \CA(511; 4, 17, 4) $  &1, 31  & 
$5i$, $i=0,1,\dots, 16$
\\
$ \CA(1249; 4, 16, 5)$  &1, 7   & 
0, 6, 9, 15, 39, 45, 48, 54, 78, 84, 87, 93, 117, 123, 126, 132
\\
$ \CA(1873; 4, 25, 5) $  &1, 7,17   & 
0, 9, 12, 21, 24, 33, 36, 45, 48, 57, 60, 69, 72, 81, 84, 93, 96, 105, 108, 117, 120, 129, 132, 141, 144
\\
$ \CA(16381 ; 4,  48,  8  )$  &1,  43,  421,  1324  &
0-14, 16, 18, 20, 22, 24, 26, 28, 31, 33, 34, 37, 41, 48, 52, 124, 125, 128, 176, 226, 230, 240, 251, 275, 279, 285, 321, 365, 432, 433, 440, 444, 452, 510, 
\\
$ \CA(19681 ; 4,  42,  9  )$  &1, 7,13 &
$10i$, $i=0, 1, \dots, 41$
\\
$ \CA(26241 ; 4,  50,  9  )$  &1, 1129,  1273,  1329 &
0-3, 5, 6, 8, 9, 11, 13, 15, 16, 18, 19, 22-24, 27-29, 32-34, 38, 43, 46, 49, 54, 56, 57, 60, 65, 67, 70, 80, 97, 102, 117, 168, 201, 226, 310, 335, 358, 367, 369, 391, 458, 468, 482
\\
$ \CA(32801 ; 4,  82,  9  )$  & 1, 29, 43, 47, 139 &
0-81
\\
$ \CA(43921 ; 4,  37,  11 )$  &1, 271, 3491 &
0, 1, 12, 13, 24, 25, 36, 37, 48, 49, 60, 61, 72, 73, 84, 85, 96, 97, 108, 109, 180, 349, 360, 409, 589, 601, 613, 660, 685, 709, 925, 937, 949, 997, 1020, 1189, 1237
\\
$ \CA(58561 ; 4,  77,  11 )$ &1, 271, 3491, 5861 &
0-3, 5, 6, 8, 9, 11, 12, 14, 15, 17, 18, 20, 21, 23, 24, 26, 27, 29, 30, 32, 33, 35, 36, 38, 39, 41, 42, 44, 45, 47, 48, 50, 51, 53, 54, 56, 57, 59, 60, 63, 66, 73, 76, 79, 80, 83, 86, 92, 95, 98, 101, 104, 107, 110, 192, 236, 352, 412, 423, 447, 507, 528, 546, 623, 650, 662, 694, 697, 700, 859, 921, 925, 1078, 1254
\\
$ \CA(73201 ; 4, 125,  11 )$  &1,119, 181, 245, 397 &
0-50, 57-62, 69-74, 81-86, 93-107, 111-114, 176, 177, 197, 230-232, 243, 283, 300, 311, 312, 323, 324, 360-362, 418, 419, 443, 455, 469, 539, 566, 603, 673, 674, 675, 798, 824, 945, 1018, 1066, 1174, 1198, 1308, 1339, 1340
\\
$ \CA(85681 ; 4,  45,  13 )$  &1,  313,  357 &
0, 1, 14, 15, 28, 29, 42, 43, 56, 57, 70, 71, 84, 85, 98, 99, 112, 113, 126, 127, 140, 141, 154, 168, 182, 238, 336, 532, 574, 686, 714, 742, 798, 1051, 1092, 1162, 1387, 1695, 1737, 1792, 1820, 1862, 1946, 1974, 2030
\\
$ \CA(114241 ; 4,  98,  13 )$  &1,  3, 213, 503 &
0-38, 42-44, 48, 72-74, 79-81, 83, 84, 123-126, 131, 132, 149, 150, 159, 164, 165, 183, 197, 203, 223, 225, 227, 229, 237, 240, 247, 273, 274, 292, 327, 333, 403, 406, 572, 601, 609, 617, 625, 776, 847, 966, 1115, 1288, 1299, 1359, 1386, 1480, 1669, 1750, 1866, 1952, 2098
\\
$ \CA(142801 ; 4, 170,  13 )$  &1, 79, 109, 171, 421 &
0-86, 150-169, 243, 245, 247, 264-266, 268, 273, 280, 281, 454, 456, 458-462, 464, 466, 468, 502, 611, 614-619, 642, 773, 782, 797, 803, 810, 811, 828, 829, 965, 975, 977, 979, 983-987, 997, 1158, 1160, 1162, 1163, 1165, 1331, 1447, 1504, 1506, 1643, 1788, 1790, 1792, 2009, 2028, 2152
\\
$ \CA(131071; 4, 28, 16) $ &1, 601  & 
0-3, 5, 6, 8, 11, 12, 17, 22, 23, 25, 36, 45, 46, 50, 157, 184, 352, 661, 1316, 2236, 2736, 3028, 3102, 3126, 3443
\\

$ \CA(196606; 4,  55,  16 )$  &1,  4636,  11086 &
0-3, 5, 6, 8, 11, 12, 17, 20, 22, 26, 29, 34, 35, 39, 40, 45, 49, 54, 69, 73, 78, 91, 100, 102, 105, 111, 120, 122, 137, 146, 155, 164, 184, 208, 239, 332, 333, 395, 399, 404, 537, 598, 858, 1746, 1754, 2020, 2279, 2743, 2751, 2810, 2816, 3189
\\
$ \CA(262141; 4, 129,  16 )$  &1, 295, 475, 883 &
0-53, 87-98, 108-110, 123-125, 129-131, 135-137, 170-173, 182, 185-187, 189-194, 199-201, 210, 223, 308, 337-340, 342, 383, 385, 412, 422, 455, 617, 635, 812, 817, 839, 841, 847, 849, 911, 933, 1438, 1499, 1929, 1938, 1994, 2239, 2758, 2782, 3328, 3383, 3675
\\
$ \CA(167041; 4, 29, 17) $ &1, 18929 &
0-3, 5, 6, 8, 9, 11, 12, 14, 15, 17, 23, 24, 27, 35, 36, 134, 252, 367, 877, 952, 1771, 1871, 2171, 2239, 3184, 4154
\\
$ \CA(250561; 4,  61,  17 )$  &1, 6481, 18929 &
0-3, 5, 6, 8, 9, 11, 12, 14, 15, 17, 18, 20, 21, 23, 24, 26, 27, 29, 30, 32, 33, 35, 36, 38, 40, 41, 43, 46, 49, 52, 54, 57, 60, 82, 93, 98, 110, 115, 120, 123, 151, 168, 194, 219, 248, 264, 371, 709, 910, 1220, 1371, 1428, 1778, 2004, 2324, 2446, 2921, 3623
\\
$ \CA(334081; 4, 141,  17 )$  &1, 707, 739, 989 &
0-61, 63, 65, 67, 69, 71, 73, 102-112, 114, 116, 118, 120, 122, 124, 126, 128, 140, 240, 242, 244, 246, 248, 250, 252, 254, 256-265, 281, 283, 285, 423, 426, 484, 494, 496, 696, 726, 804, 1049, 1127, 1131, 1147, 1149, 1224, 1232, 1237, 1241, 1242, 1245, 1375, 1582, 1913, 2142, 2863, 3061, 3098, 3541, 3576, 3629, 3633, 3863, 3933
\\
$ \CA(260641; 4, 30, 19) $  &1, 32689 &
0-3, 6, 9, 12, 15, 18, 21, 24, 27, 30, 33, 36, 39, 43, 51, 53, 62, 72, 248, 357, 1470, 1779, 2660, 3200, 4355, 5378, 5756
\\
$ \CA(781249; 4,  35,  23) $  &1, 89  &
0-3, 5, 6, 8, 9, 11, 12, 14, 18, 19, 21, 22, 24, 25, 27, 28, 31, 35, 41, 45, 118, 347, 586, 1397, 2394, 2505, 4479, 5556, 6315, 8126, 9124, 9954\\
\hiderowcolors
\multicolumn{3}{c}{}\\
\multicolumn{3}{c}{}
\end{longtable}
}

\chapter{Covering arrays from maximal sequences and character sums}
\label{chapter:CAsFromMSequencesAndCharacterSums}

\textsc{In \Cref{chapter:CombinatorialArraysFromMSequences} we explored} the combinatorial properties of arrays constructed from cyclic shifts of maximal sequences over finite fields.
We showed that cyclic trace arrays have the property that certain subsets of columns -which we characterize in various ways- are uniformly covered.
This is the foundation of several previously established constructions of orthogonal and covering arrays (see \Cref{section:ConstructionOfOrthogonalArraysFromCyclicTraceArrays,section:ConstructionOfCoveringArraysFromCyclicTraceArrays}).

In \Cref{chapter:CAsFirstPaper} we exploited the above-mentioned property to give new covering array constructions using the vertical concatenation of cyclic trace arrays.
In this chapter we exploit that property using a different approach: we reduce the alphabet size (and the number of rows) using the discrete logarithm and remainders modulo some integer. 
We do that by means of character-theoretic arguments and techniques similar to the ones used in the construction of covering arrays from cyclotomy \cite{colbourn2010covering}.
\index{Covering array!from cyclotomy}
Our results include new infinite families of covering arrays of strength $3$ and $4$, as well as an infinite family of covering arrays of arbitrary strength.
To the best of our knowledge, the latter and the covering arrays from cyclotomy, are the only currently known direct constructions that provide a $CA(N;t,k,v)$ for arbitrary $t,k$ and $v$.

The contents of this chapter are as follows.
In \Cref{section:CoveringArraysFromCyclotomy} we give a brief overview of covering arrays from cyclotomy.
In \Cref{section:CharactersChapterMainResults} we present a new type of array that is the result of reducing the alphabet size of a cyclic trace array using the discrete logarithm and taking remainders modulo some integer.
We also state the main theorem of this chapter, which gives a sufficient condition for such an array to be a covering array, and we present the new covering array families that follow.
\Cref{section:ProofOfMainTheoremSecondPaper} is dedicated to the proof of the main theorem.
Finally, in \Cref{section:EvaluatingOurConstruction} we evaluate our construction, which includes a comparison of our covering array families to covering arrays from cyclotomy and a comparison between our theoretical results and computer experiments.

The results from this chapter appear in \cite{tzanakis2017covering}.

\section{Covering arrays from cyclotomy}
\label{section:CoveringArraysFromCyclotomy}
The new covering array constructions that we present in the next section are established using similar techniques to the ones used by Colbourn in \cite{colbourn2010covering}.
In this section we give a brief overview of Colbourn's construction.

Let $q$ be a prime power and $\o$ be a primitive element of $\fq.$
Then, for every $x\in \fqstar$, there exists a unique $i \in [0,q-1]$ such that $x=\o^i$. 
This $i$ is the \emph{discrete logarithm of $x$ with base $\o$}, denoted $i=\logw(x)$.
\index{Discrete logarithm}
\index{Logarithm|see{Discrete logarithm}}
For a divisor $v\geq 2$ of $q-1$, Colbourn considers the $q\times q$ array $M(q,\o,v)$ whose $(i,j)$-th element is given by 
\[
    M(q,\o,v)_{ij}=
    \begin{cases}
        \logw(j-i) \Mod v   & \text{ if } j\neq i,\\
        0                   & \text{ otherwise,}
    \end{cases}
\]
where $(i,j) \in [0,q-1]\times [0,q-1]$.
This is the \emph{cyclotomic matrix associated to $q,v$ and $\o$.}
We note that, for a different primitive element $\o'$, the matrix $M(q,\o',v)$ is either identical to $M(q,\o,v)$, or it is the result of multiplying $M(q,\o,v)$ by some $m$ coprime to $v$ and then reducing modulo $v$. 
For the purposes of constructing covering arrays, such arrays are equivalent; thus the choice of the primitive element $\o$ is irrelevant and we simply write $M(q,v)$ instead of $M(q,\o,v)$.

Cyclotomic matrices are a special case of a type of array with interesting statistical properties.
Let $X=(X_1, \dots, X_k)$ be a random variable over a set $\Omega\subseteq \Z_v^k$.
Then, for any vector $A=(a_1, \dots, a_k)\in \Z_v^k$, the bias of $A$ is
\[
    \mathrm{bias}(A)
    = \frac{1}{g}
    \max_{0\leq l < v/g}\left| \mathrm{Pr}\left[\sum_{i=1}^{k}a_iX_i\equiv lg\Mod v \right]-\frac{g}{v}\right|,
\]
where $g=\gcd(a_1, \dots, a_k, v)$.
For $0\leq \varepsilon \leq 1$, a set $\Omega\subseteq \Z_v^k$ is \emph{$\varepsilon$-biased} if the random variable $X=(X_1, \dots, X_k)$ has the properties that 
\begin{enumerate}
    \item for all $i \in [1,k]$, $X_i$ is distributed uniformly over $\Z_v$, and
    \item $\mathrm{bias}(A) \leq \varepsilon$ for all $A\in \Z_v^k$.
\end{enumerate}
Then, an \emph{$\varepsilon$-based array} is one whose rows are the elements of $\Omega$.
The connection to covering arrays is that a $(2/v^{2t})$-biased array is also a covering array of strength $t$ \cite{colbourn2010covering}.
It turns out that when 
\begin{equation}
    \label{equation:SufficientConditionEpsilonBased}
    q>t^2v^{4t},
\end{equation}
then $M(q,v)$ is a ($2/v^{2t}$)-biased array, and thus a $\CA(q; t, q,v)$
\cite{azar1998approximating}.

A stronger result is known for the binary case:
$M(q,2)$ is a \emph{Paley matrix} which is known to be a $\CA(q; t, q,2)$ when $q>t^2 2^{2t-2}$
\cite{ananchuen1993adjacency,blass1981paley,bollobas1998random,graham1971constructive}.
This weaker condition is due to the fact that the proof does not rely on $M(q,v)$ being an $\varepsilon$-biased array, which is an object with more restrictions than a covering array; instead, different character-theoretic arguments are used.
Similarly, Colbourn employs characters over finite fields to give a sufficient condition for $M(q,v)$ to be a covering array of strength $t$ that is better (weaker) than \Cref{equation:SufficientConditionEpsilonBased}.
\begin{theorem}
    {Covering arrays from cyclotomy \cite[Theorem 3.3.]{colbourn2010covering}}
    {CyclotomicCoveringArray}
    Let $q$ be a prime power and $v\geq 2$ be a divisor of $q-1$.
    If $q>t^2v^{2t}$, then $M(q,v)$ is a $\CA(q; t, q,v)$.
\end{theorem}
In the proof of \Cref{theorem:CyclotomicCoveringArray}, Colbourn provides a character sum which, if larger than a target value, guarantees that a cyclotomic matrix is a covering array.
He then uses a standard technique that uses Weil's theorem (see \cite[Theorem 5.38]{lidl1997finite}) to give a lower bound for that sum.
Comparing that lower bound with the target value yields the sufficient condition of \Cref{theorem:CyclotomicCoveringArray}.

In the following sections we adapt this method by constructing similar arrays, built using discrete logarithms of selected elements from maximal sequences and finding lower bounds for character sums that imply the covering array property.
Using the balanced nature of tuples of elements of maximal sequences, we are able to make connections to Jacobi sums for which we know exact values (cf.\ \Cref{theorem:JacobiSumBound}).
This results in constructions of covering arrays of strength $3$ and $4$, as well as an infinite family of general strength covering arrays.

%

\section{New covering array constructions}
\label{section:CharactersChapterMainResults}
\subsection{Reducing the alphabet size of cyclic trace arrays}
In this chapter we study arrays that resemble those in \Cref{definition:cyclicAlphaSArray}, and which we describe in the next definition.
Before we state the definition we recall that, as per \Cref{lemma:CharacterizationOfConstantMultiplesInFQM}, if $\a$ is a primitive element of $\fqt$, then $\o=\a^{\w{t}}$ is a primitive element of $\fq$.
Thus, for every nonzero $x \in \fqstar$, there exists some $i\in [0,q-2]$ such that $\o^i=x$.
This $i$ is the \emph{discrete logarithm of $x$ with base $\o$}, denoted $\logw{x}$. 
\index{Discrete logarithm}

\index{Cyclic trace array!modulo $v$}
\begin{definition}
{Cyclic trace array modulo $v$}
{ArrayModV}
    Let $t,k$ be positive integers, $\a$ be a primitive element of $\fqt$, $\o=\a^{\w{t}}$, and $C=\{ c_0, \dots, c_{k-1} \}$ be an ordered subset of $[0,q^t-2]$.
    Then, the \emph{cyclic trace array modulo $v$ corresponding to $\a$ and $C$}, denoted $\MV{\a,C}$, is the $v \w{t}\times k$ array with elements 
    \[
        \MV{\a,C}_{ij}, \quad (i,j) \in [0,v \w{t}-1] \times [0,k-1]
    \]
    given by
    \begin{align}
        \label{equation:ElementOfArrayModVWRTA}
        \MV{\a,C}_{ij} &=
        \begin{cases}
            \logw \left(\M{\a,C}_{ij}\right) \bmod v, &\text{ if }\; \M{\a,C}_{ij}\neq0;\\
            0,   & \text{ otherwise;}
        \end{cases}
        \\&=
        \label{equation:ElementOfArrayModVWRTTrace}
        \begin{cases}
            \logw \left(\Tt(\a^{i+c_j})\right) \bmod v, & \text{ if }\; \Tt(\a^{i+c_j})\neq 0;\\
            0,   & \text{ otherwise.}
        \end{cases}
    \end{align}
    Furthermore, we simply write $\MV{\a}$ to denote $\MV{\a, [0, \w{t}-1]}$.
\end{definition}

The construction of $\MV{\a,C}$ in \Cref{definition:ArrayModV} essentially gives a method for reducing the alphabet size of $\M{\a,C}$ from $q$ to $v$.
Indeed, the discrete logarithm in \Cref{equation:ElementOfArrayModVWRTA,equation:ElementOfArrayModVWRTTrace} maps $\fqstar$ to $[0,q-2]$, which is then reduced to $[0,v-1]$ by considering remainders modulo $v$.
As we see in this chapter, this reduction of the alphabet size preserves some of the properties of the cyclic trace arrays.
Apart from the alphabet size, another difference between the arrays in \Cref{definition:cyclicAlphaSArray} and \Cref{definition:ArrayModV} is their dimensions; while $\M{\a,C}$ has $q^t-1$ rows, $\MV{\a,C}$ has $v \w{t}$ rows.
To justify the choice for this number of rows in our definition, we make a connection between cyclic trace arrays modulo $v$ and a new type of sequence.

\begin{lemma}
{}
{MSequenceModVPeriod}
    Let $m$ be a positive integer, $\a$ be a primitive element of $\fqt$, $v\geq 2$ be a divisor of $q-1$, and let $\o= \a^{\w{t}}$.
    Then, the sequence $\seqv{\a} = ( \seqv{\a}_i )_{i\geq 0}$ defined below, has period $v \w{t}$:
    \begin{align}
        \label{equation:DefinitionOfMSeqModV}
        \seqv{\a}_i 
        &=
        \begin{cases}
            \logw \left(\seq{\a}_i\right) \bmod v, & \text{ if \;} \seq{\a}_i\neq 0;\\\notag
            0,& \text{ otherwise;}
        \end{cases}
        \\
        &=
        \begin{cases}
            \logw \left(\Tt(\a^i) \right) \bmod v, & \text{ if \;} \Tt(\a^i)\neq0;\\
            0, & \text{ otherwise.}
        \end{cases}
    \end{align}
\end{lemma}
\begin{proof}
    We denote $\T=\Tt$.
    Let $i$ be a positive integer and $j=i \bmod v \w{t}$; we need to prove that $\seqv{\a}_i=\seqv{\a}_j$.
    Let $n$ be an integer such that $i=j+nv \w{t}$.
    Then, $\T(\a^i) = \T(\a^j\a^{nv\w{t}})$.
    By \Cref{lemma:CharacterizationOfConstantMultiplesInFQM}, we have that $\o=\a^{\w{t}}$ is a primitive element of $\fq$, hence $\a^{\w{t}nv}=\o^{nv} \in \fq$.
    From the linearity of the trace over $\fq$, as per \Cref{theorem:PropertiesOfTrace}, we have 
    \begin{equation}\label{equation:lkjacncnc}
        \T(\a^i) = \a^{\w{t}nv}\T(\a^j)=\o^{nv}\T(\a^j).
    \end{equation}
    It follows from \Cref{equation:lkjacncnc} that
    \begin{equation}
        \label{equation:qwdclkj}
        \T(\a^i)=0 \text{ if and only if } \T(\a^j)=0.
    \end{equation}
    For the case when $\T(\a^i)$ and $\T(\a^j)$ are nonzero, applying $\logw$ in  \Cref{equation:lkjacncnc} yields
    \[
         \logw(\T(\a^i))
        =\logw(\o^{nv}\T(\a^j))
        =nv+\logw(\T(\a^j)).
    \]
    Now, considering remainders modulo $v$ in the above, we have
    \begin{equation}\label{equation:alkkncnc}
        \logw(\T(\a^i)) \bmod{v} =\logw(\T(\a^j)) \bmod{v}.
    \end{equation}
    We recall that $\T(\a^i)$ is the $i$-th element $\seq{\a}_i$ of $\seq{\a}$, therefore it follows from \Cref{definition:MSequenceModV} and \Cref{equation:qwdclkj,equation:alkkncnc} that $\seqv{\a}_i=\seqv{\a}_j$.
\end{proof}

\index{Maximal sequence!modulo $v$}
\begin{definition}{Maximal sequence modulo $v$}{MSequenceModV}
    The sequence $\seqv{\a}$ described in \Cref{lemma:MSequenceModVPeriod} is the \emph{maximal sequence modulo $v$ corresponding to $\a$}.
\end{definition}

\begin{example}
{}
{SeqModVFirstExample}
Let $q=4$, $t=2$ and $\b$ be a primitive element of $\fqt$.
    We recall that in \Cref{example:ConstructionFFourAndFSixteen} we construct the finite field with $16$ elements as follows.
    \begin{itemize}
        \item We first identify $\ffour$  as $\ftwo(\a)$, where $\a$ is a root of the (primitive) polynomial $x^2+x+1 \in \ftwox$.
        \item We then identify $\fsixteen$ as $\ffour(\b)$, where $\b$ is a root of the (primitive) polynomial $x^2+x+\a \in \ffour[x]$.
        \item We list the nonzero elements of $\fqt=\ffour(\b)$ in \Cref{table:PowersOfBetaInExample}.
    \end{itemize}
    Now, we have 
    \[\w{t}=[2]_4=\frac{2^4-1}{4-1}=5,\]
    so $\o=\b^{\w{t}}=\b^5$ is a primitive element of $\fq$. 
    In \Cref{table:PowersOfBetaInExample} we see that $\b^{5}=\a$, which is indeed primitive as mentioned previously.
    We have that 
    \[\Tt(\b^i)=\b^i+\b^{4i} \in \fq,\]
    for all integers $i$.
    We calculate these values for $i \in [0,15]$ using \Cref{table:PowersOfBetaInExample} and we list them in the row labeled $\Tt(\b^i)$ in \Cref{table:ExampleSequenceModV}.
    Since $\o=\a$ is a root of $x^2+x+1$, we have that $\o^2=\o+1$ thus $\logw(\o+1)=2$; furthermore, $\logw(\o)=1$ and $\logw(1)=0$.
    Using the above, we list the values of $\logw(\Tt(\b^i))$ for $i \in [0,15]$ in \Cref{table:ExampleSequenceModV} when $\T(\b^i)$ is nonzero.
    Then, setting $v=3$, which is a divisor of $q-1$, we also show the corresponding values of  $\seqv{\b}_i$ for $i \in [0,15]$, where the zeros in bold are the zeros that are correspond to $i$ such that  $\Tt(\b^i)=0$. 
    Furthermore, $\seqv{\b}$ has period $v\w{t}=3\cdot 5=15$; we compare the elements in a period of the two sequences below, omitting the parentheses and commas in the sequence notation:
    \[
        \begin{array}{r*{15}c}
        \seq{\b}:&\o&\o&1&\o&0&\o+1&\o+1&\o&\o+1&0&1&1&\o+1&1&0\\
        \seqv{\b}:&1&1&0&1&\mathbf{0}&2&2&1&2&\mathbf{0}&1&1&2&1&\mathbf{0}
        \end{array}
    \]
\end{example}
\begin{table}[t]
    \[
    \renewcommand{\arraystretch}{\genarraystretch}
    \begin{array}{llllll}
    \rowcolor{\tableheadcoloralt}
    \multicolumn{1}{c}{i} & 0     & 1         & 2         & 3         & 4 \\
    \Tt(\b^i)        &  \o &\o       &1 &\o &0   \\
    \logw(\Tt(\b^i)) & 	1 &1       &0 &1 & -   \\
    \seqv{\b}_i      & 	1 &1       &0 & 1 & \mathbf{0}   \\
    \rowcolor{\tableheadcoloralt}
     \multicolumn{1}{c}{i} &    	5    &6    &7    &8    &9 \\
    \Tt(\b^i)        &  \o+1 &\o+1       &\o &\o+1 &0   \\
    \logw(\Tt(\b^i)) & 	2 &2       &1 &2 &-   \\
    \seqv{\b}_i      & 	2 &2       &1 &2 &\mathbf{0}   \\
    \rowcolor{\tableheadcoloralt}
    \multicolumn{1}{c}{i} &10 &11 &12 &13 &14  \\
    \Tt(\b^i)        &  1 &1       &\o+1 &1 &0   \\
    \logw(\Tt(\b^i)) & 	0 &0       &2 &0 &-   \\
    \seqv{\b}_i      &  0 & 0      &2 &0 & \mathbf{0}
    \end{array}
    \]
    \caption{Demonstration of the sequences in \Cref{example:SeqModVFirstExample}}
    \label{table:ExampleSequenceModV}
\end{table}

In \Cref{example:SeqModVFirstExample}, $\seqv{\b}$ and $\seq{\b}$ have the same period. 
This is not always the case, as we show in the following less detailed example.

\begin{table}[t]
    \[
    \renewcommand{\arraystretch}{\genarraystretch}
    \begin{array}{*{9}l}
    \rowcolor{\tableheadcoloralt}
            \multicolumn{1}{c}{i} & 0 & 1 & 2 & 3 & 4 & 5 & 6 & 7 \\
            \seq{\a}_i & 2& 1& 2& 6& 0& 3& 3& 1\\
            \logw(\seq{\a}_i) & 2& 0& 2& 3& -& 1& 1& 0\\
            \seqv{\a}_i & 2& 0& 2& 0& \mathbf{0}& 1& 1& 0\\
    \rowcolor{\tableheadcoloralt}
            \multicolumn{1}{c}{i} & 8 & 9 & 10 & 11 & 12 & 13 & 14 & 15 \\
            \seq{\a}_i & 6& 3& 6& 4& 0& 2& 2& 3\\
            \logw(\seq{\a}_i) & 3& 1& 3& 4& -& 2& 2& 1\\
            \seqv{\a}_i & 0& 1& 0& 1& \mathbf{0}& 2& 2& 1\\
    \rowcolor{\tableheadcoloralt}
            \multicolumn{1}{c}{i} & 16 & 17 & 18 & 19 & 20 & 21 & 22 & 23 \\
            \seq{\a}_i & 4& 2& 4& 5& 0& 6& 6& 2\\
            \logw(\seq{\a}_i) & 4& 2& 4& 5& -& 3& 3& 2\\
            \seqv{\a}_i & 1& 2& 1& 2& \mathbf{0}& 0& 0& 2\\
    \rowcolor{\tableheadcoloralt}
            \multicolumn{1}{c}{i} & 24 & 25 & 26 & 27 & 28 & 29 & 30 & 31 \\
            \seq{\a}_i & 5& 6& 5& 1& 0& 4& 4& 6\\
            \logw(\seq{\a}_i) & 5& 3& 5& 0& -& 4& 4& 3\\
            \seqv{\a}_i & 2& 0& 2& 0& \mathbf{0}& 1& 1& 0\\
    \rowcolor{\tableheadcoloralt}
            \multicolumn{1}{c}{i} & 32 & 33 & 34 & 35 & 36 & 37 & 38 & 39 \\
            \seq{\a}_i & 1& 4& 1& 3& 0& 5& 5& 4\\
            \logw(\seq{\a}_i) & 0& 4& 0& 1& -& 5& 5& 4\\
            \seqv{\a}_i & 0& 1& 0& 1& \mathbf{0}& 2& 2& 1\\
    \rowcolor{\tableheadcoloralt}
            \multicolumn{1}{c}{i} & 40 & 41 & 42 & 43 & 44 & 45 & 46 & 47\\
            \seq{\a}_i & 3& 5& 3& 2& 0& 1& 1& 5\\
            \logw(\seq{\a}_i) & 1& 5& 1& 2& -& 0& 0& 5\\
            \seqv{\a}_i & 1& 2& 1& 2& \mathbf{0}& 0& 0& 2
        \end{array}
    \]
    \caption{Demonstration of the sequences in \Cref{example:SeqModVSecondExample}}
    \label{table:ExampleSequenceModV2}
    \end{table}

\begin{example}
{}
{SeqModVSecondExample}
    Let $q=7$, $t=2$, and $\a$ be a root of the primitive polynomial $x^2+6x+3 \in \fseven[x]$.
    Then, $\seq{\a}$ is a maximal sequence with period $7^2-1=48$; we list the first $48$ elements in \Cref{table:ExampleSequenceModV2}.
    Let $v=3$, which is a divisor of $q-1$.
    We calculate $\w{t}=[2]_7=(7^2-1)/(7-1)=8$, so we have that $\o=\a^8=3$ is a primitive element of $\fseven$. 
    We have that $3^2=2$, $3^3=6$, $3^4=4$, $3^5=5$, and $3^6=1$; using this information we list the first $48$ elements of $\seqv{\a}$ in \Cref{table:ExampleSequenceModV2}.
    As before, the zeros in bold correspond to zero elements of $\seq{\a}$.
    Furthermore, by \Cref{lemma:MSequenceModVPeriod}, $\seqv{\a}$ has period $v \w{t}=3\cdot 8=24$, which is indeed reflected in the table.
\end{example}

The reader may have noticed that in \Cref{table:ExampleSequenceModV,table:ExampleSequenceModV2},  the rows labeled $\seq{\a}_i$ are nonzero constant multiples of each other. 
This pattern is due to the projective property of maximal sequences, that we describe in \Cref{proposition:ProjectivePropertyOfMaximalSequences}.

We recall that, as per \Cref{definition:LeftShiftOperator}, for integers $j,n$, we denote
\[
    L^{j}_n\left(\seqv{\a}\right) = \left(\seqv{\a}_{i+j}\right)_{i=0}^{n-1}.
\]

\begin{remark}
{}
{ArraysModVAndCyclicShifts}
The following analogues of the statements in \Cref{remark:ColumnsOfCyclicTraceArrayAreCyclicShiftsOfMseq} are straightforward implications of \Cref{definition:ArrayModV,definition:MSequenceModV}.

    \begin{enumerate}
        \item 
            The columns of $\MV{\a,C}$ are the left cyclic shifts of $\seqv{\a}$ by the elements in $C$, as follows:
            \[
                \MV{\a,C}=
                \left[ 
                    L^{c_0}_{v\w{t}}\left( \seq{\a; v}\right) \mid
                    \dots\mid
                    L^{c_{k-1}}_{v\w{t}} \left( \seq{\a; v}\right)
                \right].
            \]
        \item If $C=\{ c_0, \dots, c_{k-1} \}$, then for $(i,j) \in [0,v\w{t}-1]\times[0,k-1]$, the $(i,j)$-th element of $\MV{\a,C}$ is
            \begin{equation}\label{equation:IJthElementOfModVArrayInTermsOfSeqq}
                \MV{\a,C}_{i,j}= \seqv{\a}_{i+c_j}.
            \end{equation}
        \item 
            The rows of $\MV{\a}$ are the vectors 
            $L_{\w{t}}^i( \seqv{\a} )$, $i \in [0,v \w{t}-1]$.
        \item 
            For $C\subseteq C'$, we have that $\MV{\a,C}$ is a subarray of $\MV{\a,C'}$.
    \end{enumerate}
\end{remark}

\begin{table}
        \[
            \M{\a}=
            \begin{array}{lllll}
               \o	&\o	 &1	   &\o	 &0	 	\\
               \o	&1	 &\o   &0	 &\o+1 	\\
               1	&\o	 &0	   &\o+1 &\o+1 	\\
               \o	&0	 &\o+1 &\o+1 &\o	\\
               0	&\o+1&\o+1 &\o	 &\o+1 	\\
               \o+1	&\o+1&\o   &\o+1 &0	 	\\
               \o+1	&\o	 &\o+1 &0	 &1	 	\\
               \o	&\o+1&0	   &1	 &1	 	\\
               \o+1	&0	 &1	   &1	 &\o+1 	\\
               0	&1	 &1	   &\o+1 &1	 	\\
               1	&1	 &\o+1 &1	 &0	 	\\
               1	&\o+1&1	   &0	 &\o	\\
               \o+1	&1	 &0	   &\o	 &\o	\\
               1	&0	 &\o   &\o	 &1	 	\\
               0	&\o	 &\o   &1	 &\o	
            \end{array}
            \quad
            \quad
            \quad
            \MV{\a}=
            \begin{array}{lllll}
               1	&1	 &0	   &1	 &\mathbf{0}	 	\\
               1	&0	 &1   &\mathbf{0}	 &2    	\\
               0	&1	 &\mathbf{0}	   &2    &2    	\\
               1	&\mathbf{0}	 &2    &2    &1	\\
               \mathbf{0}	&2   &2    &1	 &2    	\\
               2   	&2   &1   &2    &\mathbf{0}	 	\\
               2   	&1	 &2    &\mathbf{0}	 &0	 	\\
               1	&2   &\mathbf{0}	   &0	 &0	 	\\
               2   	&\mathbf{0}	 &0	   &0	 &2    	\\
               \mathbf{0}	&0	 &0	   &2    &0	 	\\
               0	&0	 &2    &0	 &\mathbf{0}	 	\\
               0	&2   &0	   &\mathbf{0}	 &1	\\
               2   	&0	 &\mathbf{0}	   &1	 &1	\\
               0	&\mathbf{0}	 &1   &1	 &0	 	\\
               \mathbf{0}	&1	 &1   &0	 &1	
            \end{array}
        \]
        \caption[Demonstration of a cyclic trace array modulo $v$]{The arrays $\M{\a}$ and $\MV{\a}$ for $q=t=4$, $v=3$ and $\a$ as in \Cref{example:SeqModVFirstExample}.}
        \label{table:ArrayModVFirstExample}
\end{table}

\begin{table}
\renewcommand{\arraystretch}{0.8}
\[
        \M{\a}=
        \begin{array}{llllllll}
2&1&2&6&0&3&3&1\\
1&2&6&0&3&3&1&6\\
2&6&0&3&3&1&6&3\\
6&0&3&3&1&6&3&6\\
0&3&3&1&6&3&6&4\\
3&3&1&6&3&6&4&0\\
3&1&6&3&6&4&0&2\\
1&6&3&6&4&0&2&2\\
6&3&6&4&0&2&2&3\\
3&6&4&0&2&2&3&4\\
6&4&0&2&2&3&4&2\\
4&0&2&2&3&4&2&4\\
0&2&2&3&4&2&4&5\\
2&2&3&4&2&4&5&0\\
2&3&4&2&4&5&0&6\\
3&4&2&4&5&0&6&6\\
4&2&4&5&0&6&6&2\\
2&4&5&0&6&6&2&5\\
4&5&0&6&6&2&5&6\\
5&0&6&6&2&5&6&5\\
0&6&6&2&5&6&5&1\\
6&6&2&5&6&5&1&0\\
6&2&5&6&5&1&0&4\\
2&5&6&5&1&0&4&4\\
5&6&5&1&0&4&4&6\\
6&5&1&0&4&4&6&1\\
5&1&0&4&4&6&1&4\\
1&0&4&4&6&1&4&1\\
0&4&4&6&1&4&1&3\\
4&4&6&1&4&1&3&0\\
4&6&1&4&1&3&0&5\\
6&1&4&1&3&0&5&5\\
1&4&1&3&0&5&5&4\\
4&1&3&0&5&5&4&3\\
1&3&0&5&5&4&3&5\\
3&0&5&5&4&3&5&3\\
0&5&5&4&3&5&3&2\\
5&5&4&3&5&3&2&0\\
5&4&3&5&3&2&0&1\\
4&3&5&3&2&0&1&1\\
3&5&3&2&0&1&1&5\\
5&3&2&0&1&1&5&2\\
3&2&0&1&1&5&2&1\\
2&0&1&1&5&2&1&2\\
0&1&1&5&2&1&2&6\\
1&1&5&2&1&2&6&0\\
1&5&2&1&2&6&0&3\\
5&2&1&2&6&0&3&3\\
\end{array}
    \quad\quad
        \begin{array}{rllllllll}
            \multirow{21}{*}{$\MV{\a}=$\hspace{0.8em}}    
            2&0&2&0&\mathbf{0}&1&1&0\\
            0&2&0&\mathbf{0}&1&1&0&0\\
            2&0&\mathbf{0}&1&1&0&0&1\\
            0&\mathbf{0}&1&1&0&0&1&0\\
            \mathbf{0}&1&1&0&0&1&0&1\\
            1&1&0&0&1&0&1&\mathbf{0}\\
            1&0&0&1&0&1&\mathbf{0}&2\\
            0&0&1&0&1&\mathbf{0}&2&2\\
            0&1&0&1&\mathbf{0}&2&2&1\\
            1&0&1&\mathbf{0}&2&2&1&1\\
            0&1&\mathbf{0}&2&2&1&1&2\\
            1&\mathbf{0}&2&2&1&1&2&1\\
            \mathbf{0}&2&2&1&1&2&1&2\\
            2&2&1&1&2&1&2&\mathbf{0}\\
            2&1&1&2&1&2&\mathbf{0}&0\\
            1&1&2&1&2&\mathbf{0}&0&0\\
            1&2&1&2&\mathbf{0}&0&0&2\\
            2&1&2&\mathbf{0}&0&0&2&2\\
            1&2&\mathbf{0}&0&0&2&2&0\\
            2&\mathbf{0}&0&0&2&2&0&2\\
            \mathbf{0}&0&0&2&2&0&2&0\\
            0&0&2&2&0&2&0&\mathbf{0}\\
            0&2&2&0&2&0&\mathbf{0}&1\\
            2&2&0&2&0&\mathbf{0}&1&1\\
             & & & & & & & \\
             & & & & & & & \\
             & & & & & & & \\
             & & & & & & & \\
             & & & & & & & \\
             & & & & & & & \\
             & & & & & & & \\
             & & & & & & & \\
             & & & & & & & \\
             & & & & & & & \\
             & & & & & & & \\
             & & & & & & & \\
             & & & & & & & \\
             & & & & & & & \\
             & & & & & & & \\
             & & & & & & & \\
             & & & & & & & \\
             & & & & & & & \\
             & & & & & & & \\
             & & & & & & & \\
             & & & & & & & \\
             & & & & & & & \\
             & & & & & & & \\
             & & & & & & &
        \end{array}
    \]
    \caption[Demonstration of a cyclic trace array modulo $v$]{The arrays $\M{\a}$ and $\MV{\a}$ for $q=7$, $t=2$, $v=3$, and $\a$ as in 
    \Cref{example:SeqModVSecondExample}.}
    \label{table:ArrayModVSecondExample}
\end{table}

In \Cref{table:ArrayModVFirstExample,table:ArrayModVSecondExample}, we give $\M{\a}$ and $\MV{\a}$ for $q,t,v$ and $\a$ as in \Cref{example:SeqModVFirstExample,example:SeqModVSecondExample}, respectively;
the statements in \Cref{remark:ArraysModVAndCyclicShifts} can be verified for these arrays using \Cref{table:ExampleSequenceModV,table:ExampleSequenceModV2}. 

We close this section commenting on the choice of the number of rows in
\Cref{definition:ArrayModV}.
Let $j$ be a positive integer with $j\geq v \w{t}$ and $i=j \bmod{v \w{t}}$.
Since $\seqv{\a}$ has period $v \w{t}$, then, for any positive integer $n$, we have that
\[
    L_{n}^{j}\left( \seqv{\a} \right)=L_{n}^{i}\left( \seqv{\a} \right).
\]
In view of the second statement of \Cref{remark:ArraysModVAndCyclicShifts}, this means that if $\MV{\a}$ had been defined with more than $v \w{t}$ rows, the rows with indexes greater than $v \w{t}$ would be identical to some rows with indexes less than $v \w{t}$.
For the purpose of constructing covering arrays, such rows  are redundant.

\subsection{Main results}
In this section we state the main theorem of 
\Cref{chapter:CAsFromMSequencesAndCharacterSums} and present the implied results. 
The main theorem's proof relies on several lengthy results and is the topic of
\Cref{section:ProofOfMainTheoremSecondPaper}.

\begin{theorem}
    {Main theorem of \Cref{chapter:CAsFromMSequencesAndCharacterSums}}
    {MainTheoremSecondPaper}
    Let $q$ be a prime power, $t\geq 2$, $\a$ be a primitive element of $\fqt$, and $v$ be a positive divisor of $q-1$.
    Let $C\subseteq [0,q^t-2]$, such that the following equivalent (cf.\ \Cref{theorem:EquivalenceOAandNTsetsExtended} for $s=t-1$) statements hold:
    \begin{enumerate}
        \item The $q^t \times |C|$ array $\ZA{\a,C}$ is a linear $\OA_q(t-1, |C|,q)$.
        \item The set $\{ \a^c\mid c \in C \}$ is a $(|C|,t-1)$-set over $\fq$.
        \item The set $ \{L_{q^t-1}^c(\seq{\a}) \mid c \in C  \}$ is a $(|C|,t-1)$-set over $\fq$.
        \item For every $c_0, \dots, c_{t-2} \in C$ and $d_0, \dots, d_{t-2} \in \fq$ not all zero, the minimal polynomial of $\a$ does not divide $\sum_{j=0}^{t-2}d_j x^{c_j}$.
        \item For every $t-1$ points in $\{[\a^c] \mid c \in C\}$, there is no $(t-3)$-flat in $PG(t-1,q)$ that contains them.
    \end{enumerate}
    If furthermore
    \begin{equation}
    \label{equation:MainTheoremSecondPaper}
    q^{\frac{t}{2}-2}(q-tv) \geq v^{t-1},
    \end{equation}
    then $\MV{\a,C}$ is a 
    $\CA(v \w{t}; t, |C|,v)$.
\end{theorem}

\begin{remark}{}{AsymptoticBehaviour}
The bound on $q$ improves with respect to the number of columns as $t$
increases.
Indeed, in 
\Cref{equation:MainTheoremSecondPaper}
it is required that $q>tv$, so a sufficient condition for 
\Cref{equation:MainTheoremSecondPaper}
to hold is $q^{\frac{t}{2}-2} \geq v^{t-1}$, which is equivalent to
$ q\geq v^{2+\frac{6}{t-4}} $.
We conclude that a sufficient condition for 
\Cref{theorem:MainTheoremSecondPaper} is 
\[q>\max\{ tv, v^{2+\frac{6}{t-4}} \},\]
which means that $q$ needs to be larger than $tv$.
\end{remark}

We now focus on special cases of $t$ in the main theorem.
We consider some of the orthogonal arrays that we present in \Cref{chapter:CombinatorialArraysFromMSequences}, of strength $t-1$ for the cases when $t$ is equal to $3$, $4$, or it is arbitrary.
Then, for each case, we give a corollary of the main theorem that describes a family of covering arrays of strength $t$.

\begin{corollary}{Covering arrays of strength 3}
    {CharactersCAsOfStrength3}
    Let $q$ be a prime power, $\a$ be a primitive element of $\fqthree$, and $v$ be a positive divisor of $q-1$, such that $q\geq v^4+6v^3+9v^2$.
    Then, $\mathcal{A}_{q^3/q}(\a;v)$ is a $\CA(v(q^2+q+1); 3, q^2+q+1,v)$.
\end{corollary}

\begin{proof}
    Let $C=[0, \w{3}-1]=[0,q^2+q]$. 
    By \Cref{corollary:RaoHammingOAs}, we have that $\mathcal{A}_{q^3/q,\mathbf{0}}(\a,C)$ is an $\OA_{q^{3-2}}(2, \w{3},q)= \OA_q(2, q^2+q+1,q)$.
    Hence the second statement of \Cref{theorem:MainTheoremSecondPaper} holds for $t=3$.
    Furthermore, we claim that \Cref{equation:MainTheoremSecondPaper} holds for $t=3$.
    Indeed, substituting $t=3$ in \Cref{equation:MainTheoremSecondPaper}, we have that an equivalent condition in that case is $q^{3/2-2}(q-3v)\geq v^2$, which is equivalent to $q^{1/2} \geq v^2+3v/\sqrt{q}$.
    A sufficient condition for the latter to hold is $q^{1/2}\geq v^2+3v$; squaring both sides gives $q\geq (v^2+3v)^{2}=v^4+6v^3+9v^2$, which holds by our assumptions. 
    We conclude that all the conditions of \Cref{theorem:MainTheoremSecondPaper} hold for $t=3$, and thus $\MV{\a,C}=\MV{\a}$ is a $\CA(v \w{t}; t, |C|,v)$ for $t=3$, i.e. a $\CA(v(q^2+q+1); 3, q^2+q+1,v)$.
\end{proof}

We note that the bound in \Cref{corollary:CharactersCAsOfStrength3} can be improved to $q \geq v^4 + 6v^3$ using the better but more complicated conditions used in the proofs in \Cref{section:ProofOfMainTheoremSecondPaper}.
We discuss the details in the end of that section.
\begin{corollary}
    {Covering arrays of strength 4}
    {CharactersCAsOfStrength4}
    Let $q$ be a prime power, $\a$ be a primitive element of $\fqfour$, and 
    \[C=\left\{ j(q+1) \mid j \in [0,q^{2}] \right\}.\]
    Furthermore, let $v$ be a positive divisor of $q-1$, such that $q \geq v^3+4v$.
    Then, $\mathcal{A}_{q^4/q}(\a, C ; v)$ is a $\CA(v(q^3+q^2+q+1); 4, q^2+1,v)$.
\end{corollary}

\begin{proof}
    By \Cref{corollary:OAsOfStrength3FromMSequences} we have that $\mathcal{A}_{q^4/q}(\a,C)$ is an $\OA_{q^{4-3}}(3, q^2+1,q)=\OA_q(3, q^2+1,q)$.
    Hence, the second statement of \Cref{theorem:MainTheoremSecondPaper} holds for $t=4$.
    Furthermore, we claim that \Cref{equation:MainTheoremSecondPaper} holds for $t=4$.
    Indeed, substituting $t=4$ in \Cref{equation:MainTheoremSecondPaper}, we have that an equivalent condition in that case is $q \geq v^3-4v$.
    We conclude that all the conditions of \Cref{theorem:MainTheoremSecondPaper} hold for $t=4$, and thus $\MV{\a,C}$ is a $\CA(v \w{t}; t, |C|,v)$ for $t=4$, i.e. a $\CA(v(q^3+q^2+q+1); 4, q^2+1,v)$.
\end{proof}

Suppose that $q=p^r$, for a prime $p$ and some positive integer $r$.
We recall that in \Cref{equation:RationalPointsNQ}, we define
\begin{equation*}
    N_q=
    \begin{cases}
        q+\lfloor 2\sqrt{q}\rfloor,
        & \text{if $p|2\sqrt{q}$ and $r\geq 3$, $r$ odd};\\
        q+\lfloor 2\sqrt{q}\rfloor+1, & \text{otherwise}.
    \end{cases}
\end{equation*}

\begin{corollary}
    {Covering arrays of strength $t$}
    {CharactersCAsOfStrengthT}
    Let $q$ be a prime power, $t\geq 2$, $\a$ be a primitive element of $\fqt$, and $v$ be a positive divisor of $q-1$, such that 
    \[q^{t/2-2}(q-tv) \geq v^t-1.\]
    Then, there exists $C\subset [0, \w{t}-1]$ with $|C|=N_q$, such that $\MV{\a,C}$ is a $\CA(v \w{t}; t, N_q, v)$.
\end{corollary}

\begin{proof}
    By \Cref{corollary:OAsFromNMDS}, there exists $C\subset [0, \w{t}-1]$ with $|C|=N_q$, such that $\M{\a,C}$ is an $\OA_{q}(t-1, N_q,q)$.
    Hence, the second statement of \Cref{theorem:MainTheoremSecondPaper} holds.
    From our assumptions, \Cref{equation:MainTheoremSecondPaper} also holds. 
    We conclude that all the conditions of \Cref{theorem:MainTheoremSecondPaper} are satisfied, and hence $\MV{\a,C}$ is a $\CA(v \w{q}; t, N_q,v)$.
\end{proof}

\section{Proof of the main theorem}
\label{section:ProofOfMainTheoremSecondPaper}

\subsection{Overview}
As we mention earlier, the proof of \Cref{theorem:MainTheoremSecondPaper} relies on several results; before we delve into this, we give an overview of the strategy of our proof.
For convenience, we denote 
\begin{equation}
    \label{equation:AiForConvenience}
    S_i= S_i(\a,q,v,t) = L_{v\w{t}}^i\left( \seqv{\a; q} \right),\quad i \in [0, \w{t}-1].
\end{equation}
Let $C\subseteq [0, \w{t}-1]$. 
By \Cref{remark:ArraysModVAndCyclicShifts}, the $v \w{t} \times |C|$ array $\MV{\a,C}$ consists of precisely the columns $S_c, c\in C$.
Thus, to prove that this is a covering array of strength $t$, we need to show that, for every $I \subseteq C$ with $|I|=t$, the set $\{ S_i \mid i \in I \}$ is covered.
In view of that, we establish the following two results:

\begin{itemize}
    \item \textbf{\Cref{proposition:LinearlyIndependentMeansCoveredSecondPaper}:}
        For every $t$-set $I$, if $\{ \a^i\mid i\in I \}$ is linearly independent, then $\{ S_i\mid i \in I \}$ is covered.
    \item \textbf{\Cref{proposition:LDcolumnsarecovered}:}
        For every $t$-set $I$, if $\{ \a^i\mid i\in I \}$ is a linearly dependent $(t,t-1)$-set and \Cref{equation:MainTheoremSecondPaper} is satisfied, then $\{ S_i\mid i \in I \}$ is covered.
\end{itemize}
We claim that establishing the above two propositions essentially proves \Cref{theorem:MainTheoremSecondPaper}.
Indeed, assume that the five equivalent statements of \Cref{theorem:MainTheoremSecondPaper} hold.
Then, in particular, $\{ \a^c \mid c\in C \}$ is a $(|C|,t-1)$-set.
This implies that for every $I\subseteq C$, since $\{ \a^i \mid i \in I \}$ is a subset of a $(|C|, t-1)$-set, there are two cases:
\begin{enumerate}
    \item $\{ \a^i\mid i \in I \}$ is linearly independent over $\fq$, or
    \item $\{ \a^i\mid i \in I \}$ is a linearly dependent $(t,t-1)$-set over $\fq$.
\end{enumerate}
If we further assume that \Cref{equation:MainTheoremSecondPaper} holds, then in either case we have from \Cref{proposition:LinearlyIndependentMeansCoveredSecondPaper,proposition:LDcolumnsarecovered} that the set $\{ S_i \mid i \in I\}$ is covered, which shows that $\MV{\a,q}$ is a covering array of strength $t$.

The next sections are as follows: 
In \Cref{section:APreliminaryResult} we prove an auxiliary result that we need often.
In \Cref{section:FirstPartOfProofMainThmSecondPaper}, we prove \Cref{proposition:LinearlyIndependentMeansCoveredSecondPaper}, which is the first (and most straightforward) part of the proof of \Cref{theorem:MainTheoremSecondPaper}.
In \Cref{section:SecondPartOfProofMainThmSecondPaper}, we give the second part of the proof, that consists of some preliminary results on multiplicative characters, and the proof of \Cref{proposition:LDcolumnsarecovered}, that relies on them.

\subsection{A preliminary result}
\label{section:APreliminaryResult}
We state a corollary of \Cref{proposition:SebastianExtended} that comes in handy throughout the rest of this chapter.

\begin{corollary}
{}
{NumberOfTraceTuplesThatAgree}
    Let $\a$ be a primitive element of $\fqt$, $i_0, \dots, i_{t-1}\in [0, \w{t}-1]$ and $b_0, \dots, b_{t-1} \in \fq.$
    We denote by $R$ the number of elements $r \in [0,q^t-2]$ such that 
    \begin{equation}
        \label{equation:klania}
    \left(
        \Tt(\a^{r+i_0}),\Tt(\a^{r+i_1}),\dots,\Tt(\a^{r+i_{t-1}})
    \right)
    = (b_0, b_1, \dots, b_{t-1}).
    \end{equation}
    Then, we have the following values for $R$.
    \begin{enumerate}
        \item If $\{\a^{i_0}, \dots, \a^{i_{t-1}}\}$ is linearly independent, then 
              \[R=
                  \begin{cases}
                      1& \text{ if } (b_0, \dots, b_{t-1})\neq (0, \dots, 0)\\
                      0& \text{ otherwise.}
                  \end{cases}
              \]
          \item If $\{\a^{i_0}, \dots, \a^{i_{t-1}}\}$ is a linearly dependent $(t,t-1)$-set, then there exist $y_0, \dots, y_{t-1} \in \fqstar$ such that $\sum_{k=0}^{t-1}\a^{i_k}=0$ and
            \[R=
                \begin{cases}
                    q   & \text{ if } \sum_{k=0}^{t-1}y_kb_k=0 \text{ and } (b_0, \dots, b_{t-1})\neq (0, \dots, 0)\\
                    q-1 & \text{ if } (b_0, \dots, b_{t-1}) =   (0, \dots, 0)\\
                    0   & \text{ otherwise.}
                \end{cases}
            \]
    \end{enumerate}
\end{corollary}
\begin{proof}We prove each case separately.
    \begin{enumerate}
        \item 
            If $\{\a^{i_0}, \dots, \a^{i_{t-1}} \}$ is linearly independent, then by \Cref{proposition:SebastianExtended} we have that $\AA=\ZM{\a, \{i_0, \dots, i_{t-1}\}}$ is an $\OA_{\lambda}(t, t,q)$ with $\lambda=q^{t-t}=1$.
            This means that, for every $b_0, \dots, b_{t-1}\in \fq$ there exists a unique row of $\AA$ equal to $(b_0, \dots, b_{t-1}).$
            Indexing the rows of $\AA$ by $[0,q^t-1]$ and considering \Cref{definition:cyclicAlphaSArray}, we have that the row of index $r\in [0,q^t-2]$ is given by the left hand side of \Cref{equation:klania}, while the row with index $r=q^t-1$ is the last row that contains only zeros.
            Therefore, if $(b_0, \dots, b_{t-1})$ is not the zero $t$-tuple, then there exists a unique $r\in [0,q^t-2]$ such that \Cref{equation:klania} holds; otherwise, such an $r$ does not exist.
        \item If $\{ \a^{i_0}, \dots, \a^{i_{t-1}} \}$ is a linearly dependent $(t,t-1)$-set, then from its linear dependence we have that there exist $y_0, \dots, y_{t-1}\in \fq$ not all zero such that $\sum_{k=0}^{t-1}y_k\a^{i_k}=0$.
        Suppose by means of contradiction that the coefficients are not all nonzero and, without loss of generality, assume that $y_{t-1}=0$.
    Then, we have that $\sum_{k=0}^{t-2}y_k\a^{i_k}=0$, which means that $\a^{i_0}, \dots, \a^{i_{t-2}}$ is a linearly dependent subset of $\{ \a^{i_0}, \dots, \a^{i_{t-1}} \}$ of size $t-1$, contradicting the fact that the latter is a $(t,t-1)$-set.
    We conclude that $y_0, \dots, y_{t-1}$ are all nonzero, i.e. $y_0, \dots, y_{t-1}\in \fqstar$.

    Indexing the rows of $\AA'$ by $[0, q^t-1]$, we have two cases:
    \begin{enumerate}
        \item If $(b_0, \dots, b_{t-2})\neq(0, \dots, 0)$, these $q$ rows are among the ones with indexes in $[0,q^t-2]$, as the last row is the all-zero row.
        In other words, by \Cref{definition:cyclicAlphaSArray}, there exist exactly $q$ elements $r \in [0,q^t-2]$ such that 
        \begin{equation}
            \label{equation:pordes}
            \left( \Tt(\a^{r+i_0}),\Tt(\a^{r+i_1}),\dots,\Tt(\a^{r+i_{t-2}}) \right)
            = (b_0, b_1, \dots, b_{t-2}).
        \end{equation}
        \item If $(b_0, \dots, b_{t-2})=(0, \dots, 0)$, one among these $q$ rows is the last, with index $q^t-1$, while the remaining $q-1$ rows are among those with indexes in $[0,q^t-2]$.
        Hence, there exist exactly $q-1$ elements $r \in [0,q^t-2]$ such that \Cref{equation:pordes} holds.
    \end{enumerate}
    Multiplying the equation $\sum_{k=0}^{t-1}y_k\a^{i_k}=0$ by $\a^r$ and solving for $\a^r\a^{i_{t-1}}=\a^{r+i_{t-1}}$ yields
    \begin{equation}
        \label{equation:VanHalen}
        \a^{r+i_{t-1}}= - \sum_{k=0}^{t-2} \frac{y_k}{y_{t-1}}\a^{r+i_k}.
    \end{equation}
    For all $r \in [0,q^t-2]$ satisfying \Cref{equation:pordes} in any of the above cases, we have that \Cref{equation:klania} holds if and only if
    \[
        b_{t-1} 
        = \Tt(\a^{r+i_{t-1}})
        = \Tt\left(- \sum_{k=0}^{t-2} \frac{y_k}{y_{t-1}}\a^{r+i_k}\right)
        =    - \sum_{k=0}^{t-2} \frac{y_i}{y_{t-1}}\T(\a^{r+i_k})
        =    - \sum_{k=0}^{t-2} \frac{y_i}{y_{t-1}}b_k,
    \]
    that is, $\sum_{k=0}^{t-1}y_kb_k=0$.
    We conclude that if this last equation holds, the number of $r\in [0,q^t-2]$ satisfying \Cref{equation:klania} is $q$ if $(b_0, \dots, b_{t-1})$ is not the all-zero $t$-tuple, and $q-1$ if it is.
    If that equation does not hold, then there are no such $r$.
    \end{enumerate}
\end{proof}

\subsection{First part of the proof: the linear independence case}
\label{section:FirstPartOfProofMainThmSecondPaper}

\begin{proposition}{}
    {LinearlyIndependentMeansCoveredSecondPaper}
    Let $t$ be a positive integer, $\a$ be a primitive element of $\fqt$, $v\geq 2$ be a divisor of $q-1$, and $S_i=S_i(\a,q,v,t)$, $i \in [0, \w{q}-1]$ as defined in \Cref{equation:AiForConvenience}.
    For $ I \subseteq [0, \w{t}-1]$ such that $\{ \a^i \mid i \in I \}$ is linearly independent over $\fq$, we have that $\{ S_i \mid i \in I \}$ is covered.
\end{proposition}
\begin{proof}
    Let $I=\{ i_0, \dots, i_{t-1} \} \subseteq [0,\w{t}-1]$, such that $\{ \a^i\mid i\in I \}$ is linearly independent.
    We recall that, for $i\in I$, $S_i$ is a vector of length $v \w{t}$ and elements from $[0,v-1]$.
    By \Cref{definition:CoveredColumns}, in order to show that $\{ S_i\mid i \in I \}$ is covered, we need to prove that, for every $(l_0, \dots, l_{t-1}) \in [0,v-1]^t$, there exists some row of the $v \w{t}\times t$ array with columns $S_i$, $i \in I$, that is equal to $(l_0,\ldots, l_{t-1})$.

    Let $(l_0, \dots, l_{t-1}) \in [0, v-1]^t$.
    By \Cref{lemma:CharacterizationOfConstantMultiplesInFQM}, we have that $\o = \a^{\w{t}}$ is a primitive element of $\fq$.
    We observe that, since $v|q-1$, then for every $l \in [0,v-1]$ there exists some $y \in [0,q-2]$ such that $l=y\bmod v$.
    Furthermore, since $x \mapsto \logw(x)$ is a bijection between $\fqstar$ and $[0,q-2]$, we have that, for every $y \in [0,q-2]$, there exists $b\in \fqstar$ such that $\logw(b)=y$.
    We conclude that, for every $l \in [0,v-1]$, there exists $b \in \fqstar$ such that $l=\logw(b)\bmod v$.
    In particular, there exist $b_0, \ldots, b_{t-1} \in \fqstar$ such that 
    \begin{equation}
        \label{equation:sjwjdjcj}
        (l_0, \dots, l_{t-1})
        = \left( \logw(b_0)\bmod v, \dots, \logw(b_{t-1})\bmod v \right).
    \end{equation}
    Now, since $\{ \a^i \mid i \in I \}$ is linearly independent, and $(b_0, \dots, b_{t-1})$ is not the zero $t$-tuple, then,  by \Cref{corollary:NumberOfTraceTuplesThatAgree} we have that there exists $r\in [0,q^t-2]$ such that 
    \[
        (b_0, \dots, b_{t-1}) 
        = \left( \Tt(\a^{r+i_0}), \dots, \Tt(\a^{r+i_{t-1}}) \right).
    \]
    Since $b_0, \dots, b_{t-1} \in \fqstar$, there are no zeros among the elements of the $t$-tuples in the last equation and thus we can apply $\logw$, which yields
    \[
        (\logw(b_0), \dots, \logw(b_{t-1})) = \left( \logw(\Tt(\a^{r+i_0})), \dots, \logw(\Tt(\a^{r+i_{t-1}})) \right).
    \]
    Taking remainders modulo $v$, we have
    \begin{align*}
            & (\logw(b_0)\bmod v, \dots, \logw(b_{t-1})\bmod v) \\
        =   &\left( \logw(\Tt(\a^{r+i_0}))\bmod v, \dots, \logw(\Tt(\a^{r+i_{t-1}})) \bmod v\right),
    \end{align*}
    and from \Cref{equation:DefinitionOfMSeqModV,equation:sjwjdjcj}, we have
    \begin{equation}\label{equation:cjcamlk}
        (l_0, \dots, l_{t-1})
        = \left( \seqv{\a}_{r+i_0}, \dots, \seqv{\a}_{r+i_{t-1}} \right).
    \end{equation}
    Now, let $s\in [0, v \w{t}-1]$ such that $s=r\bmod{v \w{t}}$.
    By \Cref{lemma:MSequenceModVPeriod}, we know that $\seqv{\a}$ has period $v \w{t}$, and therefore $\seqv{\a}_{r+i}=\seqv{\a}_{s+i}$ for all positive integers $i$.
    This, along with \Cref{equation:cjcamlk}, implies that
    \begin{equation*}
        (l_0, \dots, l_{t-1})
        = \left( \seqv{\a}_{s+i_0}, \dots, \seqv{\a}_{s+i_{t-1}} \right).
    \end{equation*}
    The right hand side is the row with index $s$ of the $v \w{t} \times t$ array with columns $S_i, i \in I$; 
    from our discussion in the beginning, this completes the proof.
\end{proof}

\subsection{Second part of the proof: the linear dependence case}
\label{section:SecondPartOfProofMainThmSecondPaper}
\subsubsection{Preliminaries on characters}
We start with the study of some useful character sums that the proof of \Cref{proposition:LDcolumnsarecovered} relies on.
Let $q$ be a prime power, $\o$ be a primitive element of $\fq$ and $v$ be a positive divisor of $q-1$.
We denote by $\xov$, or simply $\xv$ when $\o$ is clear from the context, the multiplicative character of $\fq$ defined by 
\[
    \label{equation:DefinitionXV}
    \xv(\o^j)=\chi_{\o,v}(\o^j)=e^{\frac{2\pi i}{v}j}, \quad j \in \Z.
\]
\index{Character!of order $v$}
We know from \Cref{corollary:DescriptionOfMultiplicativeCharactersOfAFiniteField} that the set $\widehat{\fqstar}$ of multiplicative characters of $\fq$ is a cyclic group of order $q-1$ and generator $\chi$, defined by
\begin{equation}\label{equation:DefinitionOfChiV}
    \chi(\o^j)=e^{\frac{2\pi i}{q-1}j}, \quad j \in \Z.
\end{equation}
\begin{lemma}
    {}
    {OrderOfChiV}
    For a prime power $q$ and $v$ a positive divisor of $q-1$, the order of $\xv$ in $\widehat{\fqstar}$ is $v$.
\end{lemma}
\begin{proof}
    The order of $\xv$ is the smallest positive integer $s$ such that $\xv^s$ is the identity mapping, that is,  such that for every integer $k$ we have
    \[ \xv^s(\o^k)=e^{\frac{2\pi i}{v}sk}=1.  \]
    This is satisfied for $s=v$, but not for $s \in [1,v-1]$, hence the order of $\xv$ is indeed $v$.
\end{proof}
\begin{lemma}{}{CharSumColbourn}
    Let $q$ be a prime power, $v$ be a divisor of $q-1$ and $\o$ be a primitive element of $\fq$.
    Then, for $y \in \fqstar$ and integer $l$, we have that
    \[
        \sum_{j=0}^{v-1}\xvj(\o^{v-l})\xvj(y)=
        \begin{cases}
            v & \text{ if } \logw(y) \equiv l \Mod v;\\
            0 & \text{ otherwise.}
        \end{cases}
    \]
\end{lemma}
\begin{proof}
    We have that
    \[ \xv^{j}(\o^{v-l})\xv^{j}(y)= \xv^{j}(\o^{v-l})\xv^{j}(\o^{\logw(y)})= \xv^{j}(\o^{v-l+\logw(y)}).\]
    Let $s=v-l+\logw(y)$ and
    \[ S=
        \sum_{j=0}^{v-1}\xvj(\o^{s})=
        \sum_{j=0}^{v-1}e^{\frac{2\pi i }{v}sj}.
    \]
    If $s\equiv 0 \Mod v$ then
    $e^{\frac{2\pi i }{v}sj}=1$ and $S=v$. 
    Otherwise, we have that $e^{\frac{2\pi i}{v}sj} \neq 1$ and thus we can apply the formula for a finite geometric sum, which yields
    \[
        S= \sum_{j=0}^{v-1}e^{\frac{2\pi i}{v}sj}
        = \frac{(e^{\frac{2\pi i}{v}s})^v-1}{e^{\frac{2\pi i}{v}s}-1}
        = \frac{1-1}{e^{\frac{2\pi i}{v}s}-1}
        = 0.
    \]
\end{proof}

\subsubsection{A useful function}
We now introduce a function that is used to indicate when a given set of column vectors of a cyclic trace array modulo $v$ is covered.
Let $\a$ be a primitive element of $\fqt$, $v$ be a positive divisor of $q-1$, $\o=\a^{\w{t}}$, and $I=\{ i_0, \dots, i_{t-1} \}\subseteq [0,\w{t}-1]$.
For $L=(l_0, \dots, l_{t-1})\in [0, v-1]^{t}$ and $r \in [0,q^t-2]$, we define

\begin{equation}
    \label{equation:DefinitionHr}
    \hr=\prod_{k=0}^{t-1} \left( \sum_{j=0}^{v-1}\xvj(\o^{v-l_k})\xvj\left(\Tt(\a^{r+i_k})\right) \right). 
\end{equation}
The values of this function, which we give in the next lemma, can be expressed using the set
\begin{equation}
    \label{equation:DefinitionNr}
    \nr=\left\{ k \in [0,t-1] \mid \Tt(\a^{r+i_k})\neq 0  \right\}.
\end{equation}

\begin{lemma}{}{EvaluateHr}
    Let $q$ be a prime power, $v$ a positive divisor of $q-1$, $\o=\a^{\w{t}}$, $I=\{ i_0, \dots, i_{t-1} \}$, $L=(l_0, \dots, l_{t-1})$, and $r\in [0,q^t-2]$.
    Then, we have 
    \[
        \hr=
        \begin{cases}
            v^{|\nr|} & 
            \text{ if }\, \logw(\T(\a^{r+i_k})) \equiv l_k \Mod v, \text{ for all } k\in \nr; \\
            0   & \text{ otherwise.}
        \end{cases}
    \]
    In particular, if $\nr=\emptyset$ then $\hr=1$.
\end{lemma}

\begin{proof}
    We denote $\T=\Tt$.
    For $k \in [0,t-1]$, we show that
    \begin{equation}
    \label{equation:evaluateproducttermsofhr}
        \sum_{j=0}^{v-1}\xvj(\o^{v-l_k})\xvj\left(\T(\a^{r+i_k})\right)
        =
        \begin{cases}
            v& \text{ if } \T(\a^{r+i_k}) \neq 0 \text{ and } \logw(\T(\a^{r+i_k})) \equiv l_k \Mod{v} ; \\
            0& \text{ if } \T(\a^{r+i_k}) \neq 0 \text{ and } \logw(\T(\a^{r+i_k})) \not\equiv l_k \Mod{v}; \\
            1& \text{ if } \T(\a^{r+i_k}) =    0.
        \end{cases}
    \end{equation}
    The first two cases are a straightforward application of \Cref{lemma:CharSumColbourn} for $l=l_k$ and $y=\T(\a^{r+i_k})$.
    For the case when $\T(\a^{r+i_k})=0$ we recall that, by \Cref{definition:MultiplicativeCharacterOfAFiniteField}, we have 
    \[
    \xv^j(0)=
    \begin{cases}
        0 & \text{ if } j \not\equiv 0 \Mod v;\\
        1 & \text{ otherwise},
    \end{cases}
    \]
    and thus
    \[
    \sum_{j=0}^{v-1}\xvj(\o^{v-l_k})\xvj\left(\T(\a^{r+i_k})\right)
    =
    \sum_{j=0}^{v-1} \xv^j(\o^{v-l_k}) \xv^j(0)
    =
    \xv^0(\o^{v-l_k})\xv^0(0) 
    = 1,
    \]
    which proves the third case of \Cref{equation:evaluateproducttermsofhr}.
    Now, denoting $N_r=\nr$, we have
    \begin{align*}
        \hr
        & =
        \prod_{\substack{k=0\\k\not\in N_r}}^{t-1}\left( \sum_{j=0}^{v-1}\xvj(\o^{v-l_k})\xvj\left(\Tt(\a^{r+i_k})\right) \right)
        \prod_{\substack{k=0\\k\in N_r}}^{t-1}\left( \sum_{j=0}^{v-1}\xvj(\o^{v-l_k})\xvj\left(\Tt(\a^{r+i_k})\right) \right).
    \end{align*}
    By \Cref{equation:evaluateproducttermsofhr}, the terms of the first product are all equal to $1$, therefore,  
    \begin{align*}
        \hr
        =
        \prod_{\substack{k=0\\k\in N_r}}^{t-1}\left( \sum_{j=0}^{v-1}\xvj(\o^{v-l_k})\xvj\left(\Tt(\a^{r+i_k})\right) \right).
    \end{align*}
    If for every $k \in \nr$ we have $\logw(\T(\a^{r+i_k})) \equiv l_k\Mod v$, then, by \Cref{equation:evaluateproducttermsofhr}, every term of the product is equal to $v$ and hence $\hr=v^{|\nr|}$.
    Otherwise, if for some $k\in \nr$ we have $\logw(\T(\a^{r+i_k}))\not \equiv l_k\Mod v$, then, again by \Cref{equation:evaluateproducttermsofhr}, the corresponding term of the product is equal to $0$ and hence $\hr=0$.
\end{proof}

\subsubsection{The proof of \Cref{proposition:LDcolumnsarecovered}}
It follows from \Cref{lemma:EvaluateHr} that the values of $\hr$ are nonnegative integers.
Next, we show that a specific value implies that one of the rows of $\MV{\a,I}$ is equal to $L$.

\begin{lemma}
    {}
    {WhenHrIsVToTheTThenLIsCovered}
    Let $q$ be a prime power, $v$ be a positive divisor of $q-1$, $I\subseteq [0,\w{t}-1]$ with $|I|=t$, and $L\in [0,v-1]^t$.
    If $\hr =v^t$ for some $r\in [0,q^t-2]$, then there exists a row of $\MV{\a,I}$ equal to $L$.
\end{lemma}
\begin{proof}
    We denote $\sqv{\a}=\seqv{\a}$ and $\T=\Tt$.
    Let $r \in [0,q^t-2]$ so that $\hr=v^t$.
    We index the rows of $\MV{\a,I}$ by $[0,v \w{t}-1]$ and let $s=r \bmod v \w{t}$.
    We show that the row of $\MV{\a,I}$ with index $s$ is equal to $L$; in other words, if $I=\{ i_0, \dots, i_{t-1} \}$ and $L=\{ l_0, \dots, l_{t-1} \}$, we show that 
    \[
        \left( \sqv{\a}_{s+i_0}, \dots, \sqv{\a}_{s+i_{t-1}} \right)
        =\left( l_0, \dots, l_{t-1} \right).
    \]
    By \Cref{lemma:MSequenceModVPeriod}, $\sqv{\a}$ has period $v \w{t}$, which implies that $\sqv{\a}_{s+i}= \sqv{\a}_{r+i}$ for every $i$, therefore we can equivalently show that 
    \[
        \left( \sqv{\a}_{r+i_0}, \dots, \sqv{\a}_{r+i_{t-1}} \right)
        =\left( l_0, \dots, l_{t-1} \right).
    \]
    Since $h(r)=v^t$, by \Cref{lemma:EvaluateHr}, we can make two deductions:
    first, $|\nr|=t$, and therefore $\nr=[0,t-1]$; secondly, 
    \begin{equation}
        \label{equation:prasinaloga}
        \logw(\T(\a^{r+i_k})) \equiv l_k \Mod v, \text{ for all } k\in [0,t-1].
    \end{equation}
    By the definition of $\nr$ in \Cref{equation:DefinitionNr}, the fact that $\nr=[0,t-1]$ means that $\T(\a^{r+i_k})\neq 0$, for all $k\in[0,t-1]$.
    By \Cref{equation:DefinitionOfMSeqModV}, this in turn implies that 
    \begin{equation}
        \label{equation:patousa}
        \sqv{\a}_{r+i_k}= \logw(\T(\a^{r+i_k})), \text{ for all } k \in [0,t-1].
    \end{equation}
    Combining \Cref{equation:prasinaloga,equation:patousa}, we conclude that $\sqv{\a}_{r+i_k}=l_k$, for all $k \in [0,t-1]$; from our previous discussion, this completes the proof.
\end{proof}

\begin{lemma}{}{SufficientConditionForSigma}
    Let $q$ be a prime power, $v$ be a positive divisor of $q-1$, $I\subseteq [0, \w{t}-1]$ with $|I|=t$, and $S_i=S_i(\a,q,v,t)$, $i \in [0, \w{t}-1]$ as defined in \Cref{equation:AiForConvenience}.
    Furthermore, we denote
    \[
        \sigma(\a,I,v,L) =
        \sum_{r=0}^{q^t-2} \hr. 
    \]
    If for every $L \in [0,v-1]^t$ we have that  
    \begin{equation}\label{equation:SufficientConditionForSigma}
        \sigma(\a,I,v,L) \geq tv q^{t-1}-1,
    \end{equation}
    then the set $\{ S_i \mid i \in I \}$ is covered.
\end{lemma}

\begin{proof}
    Let $I\subseteq [0,\w{t}-1]$, $|I|=t$.
    By \Cref{definition:CoveredColumns}, the set $\{ S_i\mid i \in I \}$ is covered if, for every $L\in [0, v-1]^t$, there exists a row of the $v \w{t}\times t$ array with columns $S_i$, $i \in I$, that is equal to $L$.
    However, by \Cref{remark:ArraysModVAndCyclicShifts} and the definition of $S_i$, this array is $\MV{\a,I}$.
    Hence, by \Cref{lemma:WhenHrIsVToTheTThenLIsCovered}, a sufficient condition for $\{ S_i \mid i \in I \}$ to be covered is that, for every $L$, there exists some $r$ with $\hr=v^t$.
    Let 
    \begin{equation}
        \label{equation:SigmaNDefinition}
        \sigma_n(\a,I,v,L)=\sum_{\substack{r=0\\|\nr|=n}}^{q^t-2}\hr, \quad n \in [0,t].
    \end{equation}
    We observe that if $|\nr|=t$ for some $r$, then by \Cref{lemma:EvaluateHr} we have that $\hr$ is equal to either $0$ or $v^t$.
    Therefore, if $\sigma_t(\a,I,v,L)>0$, then there exists $r$ such that $\hr=v^t$.
    It follows from the above discussion that, to prove this lemma, it is sufficient to show for arbitrary $L$ that, if \Cref{equation:SufficientConditionForSigma} holds, then $\sigma_t(\a,I,v,L)>0$.

    Let $I=\{ i_0, \dots, i_{t-1}\} \subseteq [0, \w{t}-1]$, and $L=(l_0, \dots, l_{t-1}) \in [0,v-1]^t$.
    For the rest of the proof, we fix $\a, I, v$ and $L$, and we denote
    $ h(r) = \hr$, $N_r  = \nr$, $\sigma=\sigma(\a,I,v,L)$, and $\sigma_n=\sigma_n(\a,I,v,L)$.
    We observe that $\sigma=\sum_{n=0}^t\sigma_n$, and hence 
    \[
        \sigma_t=\sigma-\sum_{n=0}^{t-1}\sigma_n.
    \]
    We evaluate the terms of the right hand side.
    First, for every $r$ such that $|N_r|=0$, we have that $N_r=\emptyset$ and thus $h(r)=1$, by \Cref{lemma:EvaluateHr}.
    This implies that $\sigma_0$ is equal to the number of $r$ such that
    $N_r=\emptyset$ or, equivalently, the number of $r$ such that
    \[
        \left( \T(\a^{r+i_0}), \dots, \T(\a^{r+i_{t-1}}) \right) = (0,\dots, 0),
    \]
    which is $q-1$, by \Cref{corollary:NumberOfTraceTuplesThatAgree}.
    We conclude that $\sigma_0=q-1$.

    Next we examine $\sigma_{n}$ for $n \in [1,t]$.
    We rewrite $\sigma_n$ as
    \begin{equation}
        \sigma_n 
        =
        \sum_{\substack{r=0\\|N_r|=n}}^{q^t-2}h(r)
        \label{equation:sigman}
        =\sum_{\substack{J\subseteq [0,t-1]\\|J|=n}}
         \sum_{\substack{r=0\\ N_r = J}}^{q^t-2} h(r).
    \end{equation}
    We focus on the inner sum of the right hand side; for clarity, we assume
    without loss of generality that $J=[0,n-1]$.  
    Then,
    \begin{align*}
        \sum_{\substack{r=0 \\ N_r = J}}^{q^t-2} h(r) 
        = \sum_{\substack{ 
                    r=0 \\
                    \T(\a^{r+i_k}) \neq 0, \text{ all } k\in J \\
                    \T(\a^{r+i_k}) =    0, \text{ all } k \not\in J \\
                 }}
                 ^{q^t-2}
             h(r)
        &=
        \sum_{b_0, \dots, b_{n-1} \in \fqstar}
        \;
        \sum_{\substack{ 
                r=0 \\
                \T(\a^{r+i_k}) =  b_i, \text{ all } k\in I \\
                \T(\a^{r+i_k}) =    0, \text{ all } k \not\in I \\
              }}
              ^{q^t-2}
             h(r).
    \end{align*}
    Since $\{ \alpha^{i_0}, \ldots, \alpha^{i_{t-1}} \}$ is a linearly dependent $(t,t-1)$-set, there exist $y_i \in \fqstar$ such that $\sum_{k=0}^{t-1} y_k \alpha^{i_k}=0$ (see \Cref{corollary:NumberOfTraceTuplesThatAgree}). 
    Defining $b_k=0$ for $k\in [n,t-1]$, we have from \Cref{corollary:NumberOfTraceTuplesThatAgree} that there are no $r$ satisfying the conditions of the inner sum unless
    \begin{equation}
        \label{equation:patates}
        \sum_{k=0}^{t-1} y_kb_k=\sum_{k=0}^{n-1} y_kb_k = 0.
    \end{equation}
    Now, since $J=[0,n-1]$ and $n\geq 1$, we have that $b_0\neq 0$; since $y_0$ is also nonzero, \Cref{equation:patates} does not hold when $n=1$.
    We conclude that there are no $r$ satisfying the condition of the inner sum of the right hand side of \Cref{equation:sigman}, and hence $\sigma_1=0$.

    We have shown that $\sigma_0=q-1$ and $\sigma_1=0$; next, we consider the case $n\in [2,t]$.
    We have
    \begin{align*}
        \sum_{\substack{r=0 \\ N_r = J}}^{q^t-2} h(r) 
        =
        \sum_{\substack{b_0, \dots, b_{n-1} \in \fqstar \\ \sum_{k=0}^{n-1} y_kb_k=0}}
        \;
        \sum_{\substack{ 
                r=0 \\
                \T(\a^{r+i_k}) = b_k, \text{ all } k\in J \\
                \T(\a^{r+i_k}) = 0,   \text{ all } k \not\in J \\
              }}
              ^{q^t-2}
             h(r).
    \end{align*}
    We can break the sum on the right hand side as
    \begin{align*}
        \sum_{\substack{b_0, \dots, b_{n-1} \in \fqstar 
                        \\ \sum_{k=0}^{n-1} y_kb_k=0
                        \\ \text{ and }
                        \\ \logw(b_k) \equiv l_k \!\!\Mod v
                        \\ \text{ for all } k \in [0,n-1]
                    }\quad}
        \sum_{\substack{ 
                r=0 \\
                \T(\a^{r+i_k}) = b_k, \text{ all } k\in J \\
                \T(\a^{r+i_k}) = 0,   \text{ all } k \not\in J \\
              }}
              ^{q^t-2}
             h(r)
         +
         \sum_{\substack{b_0, \dots, b_{n-1} \in \fqstar 
                     \\ \sum_{k=0}^{n-1} y_kb_k=0
                        \\ \text{ and }
                        \\ \logw(b_k) \not\equiv l_k\!\!\Mod v
                        \\ \text{ for some } k \in [0,n-1]
                    }\quad}
        \sum_{\substack{ 
                r=0 \\
                \T(\a^{r+i_k}) = b_k, \text{ all } k\in J \\
                \T(\a^{r+i_k}) = 0,   \text{ all } k \not\in J \\
              }}
              ^{q^t-2}
             h(r).
    \end{align*}
    From \Cref{lemma:EvaluateHr}, we have that $h(r)=v^n$ for all $r$ in the
    first double sum, and $h(r)=0$ for all $r$ in the second double sum. 
    Therefore
    \begin{align*}
        \sum_{\substack{r=0 \\ N_r = J}}^{q^t-2} h(r) 
        &=
        \sum_{\substack{b_0, \dots, b_{n-1} \in \fqstar 
                        \\ \sum_{k=0}^{n-1} y_kb_k=0
                        \\ \text{ and }
                        \\ \logw(b_k) \equiv l_k \!\!\Mod v
                        \\ \text{ for all } k \in [0,n-1]
                    }\quad}
        \sum_{\substack{ 
                r=0 \\
                \T(\a^{r+i_k}) = b_k, \text{ all } k\in J \\
                \T(\a^{r+i_k}) = 0,   \text{ all } k \not\in J \\
              }}
              ^{q^t-2}
             v^n.
    \end{align*}
    It follows from \Cref{corollary:NumberOfTraceTuplesThatAgree} that there exist exactly $q$ elements $r\in [0,q^t-2]$ that satisfy the conditions of the inner sum, and therefore the equality becomes
    \begin{align}
        \label{equation:aggouri}
        \sum_{\substack{r=0 \\ N_r = J}}^{q^t-2} h(r) 
        &=
        \sum_{\substack{b_0, \dots, b_{n-1} \in \fqstar 
                        \\ \sum_{k=0}^{n-1} y_kb_k=0
                        \\ \text{ and }
                        \\ \logw(b_k) \equiv l_k \!\!\Mod v
                        \\ \text{ for all } k \in [0,n-1]
                    }\quad}
                q v^n.
    \end{align}
    We claim that there are at most $((q-1)/v)^{n-1}$ $n$-tuples $(b_0, \dots, b_{n-1})$ satisfying the conditions of the sum on the right hand side.
    Since $v|q-1$, for every $l \in [0,v-1]$ there exist exactly $(q-1)/v$ elements $y \in [0,q-2]$ such that $y=l\bmod v$.
    For each such $y$, there exists a unique $b \in \fqstar$ such that $\logw(b)=y$.
    From the above, we conclude that for every $l \in [0,v-1]$ there exist precisely $(q-1)/v$ elements $b \in \fqstar$ such that $l=\logw(b)\bmod v$.
    Thus, there exist $((q-1)/v)^{n-1}$ choices of $b_0, \ldots, b_{n-2} \in \fqstar$ such that 
    \begin{equation}
        \label{equation:kakakia}
        (l_0, \dots, l_{n-2})
        = \left( \logw(b_0)\bmod v, \dots, \logw(b_{n-2})\bmod v \right).
    \end{equation}
    For every such choice, consider the unique $b_{n-1} \in \fqstar$ such that $\sum_{k=0}^{n-1}y_kb_k=0$.
    Then, the condition of the outer sum is satisfied if and only if $\logw(b_{n-1})\equiv l_{n-1} \Mod v$, which holds if and only if
    \[
        \logw(b_{n-1})\logw\left( -\sum_{k=0}^{n-2} \frac{y_k}{y_{n-1}}b_k\right)\equiv l_{n-1}\Mod v.
    \]
    This does not necessarily hold for all $((q-1)/v)^{n-1}$ choices of $b_0, \dots, b_{n-2}$, therefore there are indeed at most that many choices of $n$-tuples $(b_0, \dots, b_{n-1})$ satisfying the conditions of the last sum, as we claimed.
    It follows from the above and \Cref{equation:aggouri} that
    \[
    \sum_{\substack{r=0 \\ N_r = J}}^{q^t-2} h(r) \leq \left( \frac{q-1}{v} \right)^{n-1} qv^n.
    \]
    From this and \Cref{equation:sigman} we conclude that, for all $n \in [2,t-1]$, we have
    \[
        \sigma_n 
        \leq 
        \sum_{\substack{J\subseteq [0,t-1]\\|J|=n}}
        \left( \frac{q-1}{v} \right)^{n-1} qv^n
        =
        \binom{t}{n}(q-1)^{n-1}qv.
    \] 
    Now, solving for $\sigma_t$ in $\sigma = \sum_{n=0}^{t}\sigma_n$, yields
    \begin{align*}
        \sigma_t 
        &=
        \sigma - \sigma_0 - \sigma_1 - \sum_{n=2}^{t-1}\sigma_n
        \\&=
        \sigma - (q-1) - 0  - \sum_{n=2}^{t-1}\sigma_n
        \\&\geq 
        \sigma - (q-1) - \sum_{n=2}^{t-1} \binom{t}{n}(q-1)^{n-1}qv.
    \end{align*}
    Thus, a sufficient condition for $\sigma_t > 0$ to hold is for the last expression to be positive, or equivalently
    \begin{align*}
        \sigma  &> q-1 + \sum_{n=2}^{t-1} \binom{t}{n}(q-1)^{n-1}qv\\
                &= q-1 + qv\left(\frac{q^t-1}{q-1}-(q-1)^{t-1}-t\right).
    \end{align*}
    We observe that $tvq^{t-1}-1 > q-1 + qv(\frac{q^t-1}{q-1}-(q-1)^{t-1}-t)$, so a simpler sufficient condition for $\sigma_t$ to be positive is $\sigma \geq tvq^{t-1}-1$ as given in \Cref{equation:SufficientConditionForSigma}.
    From the discussion in the beginning, this completes the proof.
\end{proof}

\pagebreak
We now have everything we need in order to prove the analogue of
\Cref{proposition:LinearlyIndependentMeansCoveredSecondPaper}
for the case of a linearly dependent $(t,t-1)$-set.

\begin{proposition}{}
    {LDcolumnsarecovered}
    Let $t$ be a positive integer, $\a$ be a primitive element of $\fqt$, $v\geq 2$ be a divisor of $q-1$, and $S_i=S_i(\a,q,v,t)$, $i \in [0, \w{t}-1]$ as defined in \Cref{equation:AiForConvenience}.
    For $I\subseteq [0, \w{t}-1]$ with $|I|=t$ such that $\{ \a^i\mid i \in I \}$ is a linearly dependent $(t,t-1)$-set.
    If furthermore
    \begin{equation}
        \label{equation:LDcolumnsarecovered}
        q^{\frac{t}{2}-2}(q-tv) > v^{t-1},
    \end{equation}
    then the set $\{ S_i\mid i \in I \}$ is covered.
\end{proposition}

\begin{proof}
To prove this proposition, it is sufficient to show that if \Cref{equation:LDcolumnsarecovered} holds, then \Cref{equation:SufficientConditionForSigma} holds for every $L \in [0,v-1]^t$; the proof then is complete by \Cref{lemma:SufficientConditionForSigma}.
Throughout the proof we fix $\a, v, I=\{i_0, \dots, i_{t-1}\}$, $L=(l_0, \dots, l_{t-1})$, and we denote $\sigma = \sigma(\a, I, v, L)$, $h(r)=\hr$, $\T=\Tt$, and $\o_k=\o^{v-l_k}$ where, as usual, $\o=\a^{\w{t}}$.
We rewrite $\sigma$ as 
\begin{align}
    \notag
    \sigma&= 
          \sum_{r=0}^{q^t-2}h(r)
          = \sum_{r=0}^{q^t-2}
                \prod_{k=0}^{t-1}
                \left( 
                    \sum_{j=0}^{v-1}\xvj(\o_{k})\xvj\left(\T(\alpha^{r+i_k})\right)
                \right)\\
          \notag
          &= \sum_{r=0}^{q^t-2}
                \prod_{k=0}^{t-1}
                \left( 
                    1+\Sj\xvj(\o_{k})\xvj\left(\T(\alpha^{r+i_k})\right)
                \right)\\
          \notag
          &=  q^t-1+ 
          \sum_{\substack{U\subseteq [0,t-1] \\U\neq \emptyset}}
                \sum_{r=0}^{q^t-2}
                \prod_{u\in U} \Sj \xvj\left( \o_{u} \right)\xvj\left(\T(\a^{r+i_u})\right)
          \\&= q^t-1 +
              \sum_{n=1}^{t}
              \sum_{ \substack{
                      U\subseteq [0,t-1]\\
                      U=\left\{u_0, \dots, u_{n-1} \right\}
                  }}
              \sum_{r=0}^{q^t-2}
              \prod_{s=0}^{n-1}
              \sum_{j=1}^{v-1}
                  \xvj(\o_{u_s})
                  \xvj(\a^{r+i_{u_s}})
          \notag
          \\&= q^t-1 +
              \sum_{n=1}^{t}
              \sum_{ \substack{
                      U\subseteq [0,t-1]\\
                      U=\left\{u_0, \dots, u_{n-1} \right\}
                  }}
              \sum_{ \substack{
                      j_0, \dots, j_{n-1} \\
                      \in [1,v-1]
                  }}
              \sum_{r=0}^{q^t-2}
              \prod_{s=0}^{n-1}
                  \xv^{j_s}(\o_{u_s})
                  \xv^{j_s}\left(\T( \a^{r+u_s})\right)
          \notag
          \\&= q^t-1 +
              \sum_{n=1}^{t}
              \sum_{ \substack{
                      U\subseteq [0,t-1]\\
                      U=\left\{u_0, \dots, u_{n-1} \right\}
                  }}
              \sum_{ \substack{
                      j_0, \dots, j_{n-1} \\
                      \in [1,v-1]
                  }}
              \prod_{s=0}^{n-1}
                  \xv^{j_s}(\o_{u_s})
              \sum_{r=0}^{q^t-2}
              \prod_{s=0}^{n-1}
                  \xv^{j_s}\left(\T( \a^{r+u_s})\right)
              \label{equation:biglast}
\end{align}
We first examine the terms of the expression in \Cref{equation:biglast} corresponding to $n\leq t-1$.
We focus on the last sum of the right hand side of \Cref{equation:biglast}.
Let $N(b_0, \dots, b_{n-1})$ be the number of $r\in [0,q^t-2]$ such that
\begin{equation}
    \label{equation:baklavas}
    \left( \T(\a^{r+i_{u_0}}), \dots, \T(\a^{r+i_{u_{n-1}}}) \right)
    =
    (b_0, \dots, b_{n-1}).
\end{equation}
Then, we can rewrite the right hand side of \Cref{equation:biglast} as 
\begin{align}
    \label{equation:orange}
    \sum_{r=0}^{q^t-2}
    \prod_{s=0}^{n-1}
        \xv^{j_s}\left(\T( \a^{r+i_{u_s}})\right)
        \notag
    &=
    \sum_{r=0}^{q^t-2}
    \xv^{j_0}\left(\T(\a^{r+i_{u_0}})\right) 
    \cdots 
    \xv^{j_{n-1}}\left(\T(\a^{r+i_{u_{n-1}}})\right) \\
    &=
    N(b_0, \dots, b_{n-1})
    \sum_{\substack{b_0,\dots, b_{n-1}\in \fq\\ \text{Eq. (\ref{equation:baklavas}), some }r}}
    \xv^{j_0}(b_0) \cdots \xv^{j_{n-1}}(b_{n-1}),
\end{align}
where the sum runs over all $b_0, \dots, b_{n-1}\in \fq$ for which there exists some $r\in [0,q^t-2]$ such that \Cref{equation:baklavas} holds.

Since $\xvj(0)=0$ for all $j\in [1,v-1]$, all the sum terms in \Cref{equation:orange} where any of $b_0, \dots, b_{n-1}$ is zero, vanish. 
Hence, we have that 
\begin{equation}
    \label{equation:ahladi}
    \sum_{r=0}^{q^t-2}
    \prod_{s=0}^{n-1}
        \xv^{j_s}\left(\T( \a^{r+i_{u_s}})\right)
    =
    N(b_0, \dots, b_{n-1})
    \sum_{\substack{b_0,\dots, b_{n-1}\in \fqstar\\ \text{Eq. (\ref{equation:baklavas})}}}
    \xv^{j_0}(b_0) \cdots \xv^{j_{n-1}}(b_{n-1}),
\end{equation}
where the sum runs over all $b_0, \dots, b_{n-1} \in \fqstar$ for which there exists $r\in [0,q^t-2]$ such that \Cref{equation:baklavas} holds.
Now, we observe that the left hand side of \Cref{equation:baklavas} is the row of index $r$ of $\M{\a, \{ i_{u_0}, \dots, i_{u_{n-1}} \}}$.
This is an $\OA_{q^{t-n}}(n, n,q)$ by \Cref{proposition:SebastianExtended}, since $\{ \a^{i_{u_0}}, \dots, \a^{i_{u_{n-1}}} \}$ is linearly independent as a subset of size $n\leq t-1$ of $\{ \a^{i_0}, \dots, \a^{i_{t-1}} \}$, which is a $(t,t-1)$-set from our assumptions.
We conclude that, for every $b_0, \dots, b_{n-1}\in \fqstar$, there exist precisely $N(b_0, \dots, b_{n-1})=q^{n-t}$ elements $r \in [0,q^t-2]$ such that \Cref{equation:baklavas} holds, and therefore \Cref{equation:ahladi} becomes
\begin{equation}
    \label{equation:germas}
    \sum_{r=0}^{q^t-2}
    \prod_{s=0}^{n-1}
        \xv^{j_s}\left(\T( \a^{r+i_{u_s}})\right)
    =
    q^{n-t}
    \sum_{b_0,\dots, b_{n-1} \in \fqstar}
    \xv^{j_0}(b_0) \cdots \xv^{j_{n-1}}(b_{n-1}),
\end{equation}
where this time the sum runs over all elements $b_0, \dots, b_{n-1} \in \fqstar$.
By \Cref{lemma:directproductofcharacters}, we have 

\begin{equation}
    \label{equation:pontikia}
    \sum_{b_0,\dots, b_{n-1} \in \fqstar}
    \xv^{j_0}(b_0) \cdots \xv^{j_{n-1}}(b_{n-1})=0,
\end{equation}
hence, from \Cref{equation:germas,equation:pontikia} we have that, for all $n\in [1,t-1]$, 
\[
    \sum_{r=0}^{q^t-2}
    \prod_{s=0}^{n-1}
        \xv^{j_s}\left(\T( \a^{r+i_{u_s}})\right)
        =0.
\]
This shows that the terms corresponding to $n\in [1,t-1]$ of the sum in \Cref{equation:biglast} vanish.
Therefore, the only term remaining is for $n=t$, which means that $U=[0,t-1]$, and
\begin{equation}
\label{equation:vothros}
  \sigma = q^t -1 +
  \sum_{ \substack{
          j_0, \dots, j_{t-1} \\
          \in [1,v-1]
      }}
      \prod_{s=0}^{t-1}
      \xv^{j_s}(\o_{i_s})
      \sum_{r=0}^{q^t-2}
      \prod_{s=0}^{t-1}
      \xv^{j_s}\left(\T( \a^{r+i_s})\right).
\end{equation}
Now, since $\{ \a^{i_0}, \dots, \a^{i_{t-1}} \}$ is a linearly dependent
$(t,t-1)$-set, there exist $y_0, \dots, y_{t-1} \in \fqstar$ such that
$\sum_{s=0}^{t-1} y_s\a^{i_s}=0$ (see \Cref{corollary:NumberOfTraceTuplesThatAgree}).
Multiplying the equation by $\a^r$ yields $\sum_{s=0}^{t-1} y_s\a^{r+i_s}=0$,
and applying the trace, which is linear over $\fq$, yields
$\sum_{s=0}^{t-1} y_s\T(\a^r\a^{i_s})=0$. 
Solving for the term corresponding to $s=t-1$, we get

\begin{equation}
    \label{equation:mykses}
    \T(\a^{r+i_{t-1}}) = -\sum_{s=0}^{t-2} \frac{y_s}{y_{t-1}} \T(\a^{r+i_s}).
\end{equation}
Hence, \Cref{equation:vothros} becomes
\[
  \sigma - q^t +1 
  =
  \sum_{ \substack{
          j_0, \dots, j_{t-1} \\
          \in [1,v-1]
      }}
      \prod_{s=0}^{t-1}
      \xv^{j_s}(\o_{s})
  \sum_{r=0}^{q^t-2}
  \xv^{j_{t-1}}\left( 
          \T\left( \a^{r+i_{t-1}}\right)
            \right)
  \prod_{s=0}^{t-2}
  \xv^{j_s}\left(\T( \a^{r+i_s})\right),
\]
and by
\[
  \sigma - q^t +1 
  =\sum_{ \substack{
          j_0, \dots, j_{t-1} \\
          \in [1,v-1]
      }}
      \prod_{s=0}^{t-1}
      \xv^{j_s}(\o_{s})
  \sum_{r=0}^{q^t-2}
  \xv^{j_{t-1}}\left( -\sum_{s=0}^{t-2} \frac{y_s}{y_{t-1}} \T(\a^{r+i_s})\right)
  \prod_{s=0}^{t-2}
  \xv^{j_s}\left(\T( \a^{r+i_s})\right).
  \]
We note that $\{ \a^{i_0}, \dots, \a^{i_{t-2}} \}$ is linearly independent  as a subset of size $t-1$ of the $(t,t-1)$-set $\{ \a^{i_0}, \dots, \a^{i_{t-1}} \}$.
Hence, by \Cref{proposition:SebastianExtended},  $\ZM{\a, \{i_0, \dots, i_{t-1}\}}$ is an $\OA_{q}(t-1, t-1,q)$.
This implies that, for every $b_0, \dots, b_{t-2}\in \fqstar$, there exist exactly $q$ rows equal to $(b_0, \dots, b_{t-2})$, i.e. there exist exactly exactly $q$ values of $r\in [0, q^t-2]$ such that
\[
    \left( \T(\a^{r+i_0}), \dots, \T(\a^{r+i_{t-2}}) \right)
    =
    (b_0, \dots, b_{t-2}).
\]
Thus, we have that 
\begin{align*}
  \sigma - q^t +1 
  &=
  \sum_{ \substack{
          j_0, \dots, j_{t-1} \\
          \in [1,v-1]
      }}
      \prod_{s=0}^{t-1}
      \xv^{j_s}\left(\o_{s}\right)
  \;q
  \sum_{b_0, \dots, b_{t-2} \in \fqstar}
  \xv^{j_{t-1}}\left( 
              -\sum_{s=0}^{t-2} 
              \frac{y_s}{y_{t-1}} b_s
            \right)
  \prod_{s=0}^{t-2}
  \xv^{j_s}\left(b_s\right)
  \\&=
  \sum_{ \substack{
          j_0, \dots, j_{t-1} \\
          \in [1,v-1]
      }}
      \prod_{s=0}^{t-1}
      \xv^{j_s}\left(\o_{s}\right)
  \;q
  \sum_{z\in \fqstar}
  \xv^{j_{t-1}}\left(z\right)
  \sum_{\substack{
      b_0, \dots, b_{t-2} \in \fqstar\\
      z = -\sum_{s=0}^{t-2} 
      \frac{y_s}{y_{t-1}}b_s
      }
  }
  \prod_{s=0}^{t-2}
  \xv^{j_s}\left(b_s\right).
\end{align*}
We recall that $y_s\neq 0$ for all $s$, so we can normalize the last sum by substituting $b_s$ with $-\frac{zy_{t-1}}{y_s}b_s$. This yields
\begin{equation}\label{equation:ksylo}
\sigma-q^t+1=
  \sum_{ \substack{
          j_0, \dots, j_{t-1} \\
          \in [1,v-1]
      }}
      \prod_{s=0}^{t-1}
      \xv^{j_s}(\o_{s})
  \;q
  \sum_{z\in \fqstar}
  \xv^{j_{t-1}}(z)
  \sum_{\substack{
      b_0, \dots, b_{t-2} \in \fqstar\\
      b_0+\dots+b_{t-2} = 1
      }
  }
  \prod_{s=0}^{t-2}
  \xv^{j_s}\left(-\frac{zy_{t-1}}{y_s}b_s\right).
\end{equation}
Since the characters are multiplicative, we have that 
\begin{align*}
    \prod_{s=0}^{t-2}
    \xv^{j_s}\left(-\frac{zy_{t-1}}{y_s}b_s\right)
    &=
    \prod_{s=0}^{t-2}
    \xv^{j_s}\left(z\right)
    \xv^{j_s}\left(\frac{-y_t}{y_s}\right)
    \xv^{j_s}\left(b_s\right)\\
    &=
    \xv^{j_0+\cdots + j_{t-2}}(z)
    \prod_{s=0}^{t-2}
    \xv^{j_s}\left(\frac{-y_{t-2}}{y_s}\right)
    \xv^{j_s}\left(b_s\right),
\end{align*}
Substituting this in \Cref{equation:ksylo}, we get
\[
  \sigma - q^t +1 = \;q
  \sum_{ \substack{
          j_0, \dots, j_{t-1} \\
          \in [1,v-1]
      }}
      \xv^{j_{t-1}}(\o_{t-1})
  \prod_{s=0}^{t-2}
  \xv^{j_s}\left( \frac{-y_{t-1}\o_{s}}{y_s} \right)
  \sum_{z\in \fqstar}
  \xv^{j_0+\cdots+ j_{t-1}}(z)
  \sum_{\substack{
      b_0, \dots, b_{t-2} \in \fqstar\\
      b_0+\dots+b_{t-2} = 1
      }
  }
  \prod_{s=0}^{t-2}
  \xv^{j_s}(b_s).
\]
By \Cref{definition:JacobiSum}, the last sum is a Jacobi sum and we have
\[
  \sigma - q^t +1 = \;q
  \sum_{ \substack{
          j_0, \dots, j_{t-1} \\
          \in [1,v-1]
      }}
      \xv^{j_{t-1}}(\o_{t-1})
  \prod_{s=0}^{t-2}
      \xv^{j_s}\left( \frac{-y_{t-1}\o_s}{y_s} \right)
  \sum_{z\in \fqstar}
      \xv^{j_0+\cdots+ j_{t-1}}(z)
      \J(\xv^{j_0}, \dots, \xv^{j_{t-2}}).
\]
Since $\xv$ is a character of order $v$, we have that $\xv^{n}$ is the trivial character if and only if $n\equiv 0\Mod v$.
Thus, from \Cref{theorem:CharSumIsZeroClassicResult} we have that
\[ 
    \sum_{z\in \fqstar} \xv^{j_0+\cdots+ j_{t-1}}(z)=
    \begin{cases}
        q-1 & \text{ if } j_0+\dots+j_{t-1} \equiv 0 \Mod v;\\
      0 & \text{ otherwise}.
    \end{cases}
\]
Therefore our equation becomes
\begin{equation}
  \sigma - q^t +1
  =
  q(q-1)
  \sum_{ \substack{
          j_0, \dots, j_{t-1} \in [1,v-1]     \\
          j_0 + \dots + j_{t-1} \equiv\, 0\!\! \Mod v
      }}
      \xv^{j_{t-1}}(\o_{t-1})
  \prod_{s=0}^{t-2}
  \xv^{j_s}\left( \frac{-y_{t-1}\o_{s}}{y_s} \right)
  \J(\xv^{j_0}, \dots, \xv^{j_{t-2}}).
  \label{equation:sumofjsiszero}
\end{equation}
For all $1\leq j_0, \dots, j_{t-1}\leq v-1$ such that $j_0+\dots+j_{t-1}\equiv 0 \Mod v$, we have 
\[j_0+\dots+j_{t-2} \equiv -j_{t-1} \not \equiv 0 \Mod v,\] since $j_{t-1} \in [1, v-1]$.
Hence $\prod_{i=0}^{t-1} \xv^{j_i}$ is a nontrivial character, and by \Cref{theorem:JacobiSumBound} we have that 
$|\J(\xv^{j_0}, \dots, \xv^{j_{t-2}})|=q^{\frac{t}{2}-1}$.
Furthermore, the codomain of characters is $\mathbb{C}^{\times}$, so they are absolutely bounded by $1$ and hence
\[ 
  \left|
  \xv^{j_{t-1}}(\o_{t-1})
  \prod_{s=0}^{t-2}
      \xv^{j_s}\left( \frac{-y_{t-1}\o_s}{y_s} \right)
  \right|
  \leq 1.
\]
Thus, applying absolute values to \Cref{equation:sumofjsiszero} yields
\begin{equation*}
  \left|
  \sigma - q^t +1
  \right|
  \leq
  q(q-1)
  \sum_{ \substack{
          j_0, \dots, j_{t-1} \in [1,v-1]     \\
          j_0 + \dots + j_{t-1} \equiv 0 \Mod v
      }}
  q^{\frac{t}{2}-1}.
\end{equation*}
Using inclusion-exclusion we deduce that the number of $j_0, \dots, j_{t-1}$ satisfying the conditions of the sum above is 
\[
    S=\sum_{n=1}^{t-1}(-1)^{t-n-1}(v-1)^n.
\] 
Thus
\begin{align*}
    |\sigma - q^t+1|    &\leq q(q-1) q^{\frac{t}{2}-1}S\\
                        &\leq q^{\frac{t}{2}}(q-1)\sum_{n=1}^{t-1}(-1)^{t-n-1}(v-1)^n\\
                        & =
                        q^{\frac{t}{2}}(q-1)\frac{(v-1)^t+(-1)^t(v-1)}{v},
\end{align*}
which implies that
\begin{equation}
    \sigma \geq q^t-1 -
    q^{\frac{t}{2}}(q-1)\frac{(v-1)^t+(-1)^t(v-1)}{v}.
    \label{equation:lowerboundforsigma}
\end{equation}
From the lower bound for $\sigma$ in \Cref{equation:lowerboundforsigma}, it
follows that \Cref{equation:SufficientConditionForSigma} is satisfied if
\[
    q^t-1 - q^{\frac{t}{2}}(q-1)\frac{(v-1)^t+(-1)^t(v-1)}{v}
    \geq tvq^{t-1}-1.
\]
Using the fact that 
\[v^{t-1} \geq \frac{(v-1)^t+(-1)^t(v-1)}{v},\]
we have that a simpler sufficient condition is $q^{\frac{t}{2}-1}(q-tv) \geq v^{t-1}$.
\Cref{equation:SufficientConditionForSigma} is satisfied.
From the discussion in the beginning, this completes the proof.
\end{proof}

\section{Evaluation of our construction}
\label{section:EvaluatingOurConstruction}
\subsection{Comparison with covering arrays from cyclotomy}
Colbourn \cite{colbourn2010covering} constructs covering arrays using cyclotomic generators over finite fields. 
These are $\CA( q; t,q,v )$ for $t,q,v$ satisfying $v|q-1$ and $q>t^2v^{2t}$, and are referred to as  \emph{covering arrays from cyclotomy}.
\index{Covering array!from cyclotomy}
The techniques used to provide covering arrays from cyclotomy and {\em covering arrays from maximal sequences} are similar and, as far as we know, they are the only direct constructions that provide a $\CA(N;t,k,v)$ for any arbitrary $t,k$ and $v$. 

The main similarity of the two constructions is that they start with a function whose values are uniformly distributed over the field, and take the logarithm.
Furthermore, in both constructions, the conditions that guarantee the covering array property rely on the evaluation of similar character sums.  
On the other hand, the resulting arrays have different dimensions and existence conditions.

In what follows, we consider $v$ and $t$ fixed and analyze when maximal sequences yield ``better" covering arrays than covering arrays from cyclotomy. 
Let $\ktv{q}$ and $\Ntv{q}$ be the number of columns and number of rows in the maximal sequence construction, respectively, that is
\begin{align*}
    \ktv{q}
           &=
        \begin{cases} 
            q^2+q+1 & \text{ if } t=3,\\
            q^2+1 & \text{ if } t=4,\\
            \sim q+\lfloor \sqrt{2q}\rfloor & \text{ if } t\geq 5,
        \end{cases}
         \qquad \mbox{and} \qquad \Ntv{q}=v\frac{q^t-1}{q-1}.
\end{align*}
We define
\begin{align*}
\qmin(t,v)&=
\min\left\{ 
        q \mid
        \text{ a $\CA(\Ntv{q}; t, \ktv{q},v)$ can be constructed from maximal sequences}
    \right\},\\
\qcol(t,v)&=
\min\left\{ 
        q \mid
        \text{ a $\CA(q; t,q,v)$ can be constructed from cyclotomy}
    \right\}.
\end{align*}    
If $\Ntv{\qmin(t,v)} < \qcol(t,v)$, we further define   
\begin{align*}
\qmax(t,v)&=
\max\left\{ 
        q \mid
        q\geq \qmin(t,v) \text{ and } \Ntv{q} < \qcol(t,v)
    \right\}.
\end{align*}
and note that for all prime powers $q$ such that 
$\qmin(t,v) \leq q \leq \qmax(t,v)$, covering arrays from maximal sequences have
fewer rows than
covering arrays from cyclotomy for the same number of columns.

\begin{figure}[t]
\centering
\includegraphics[scale=0.5]{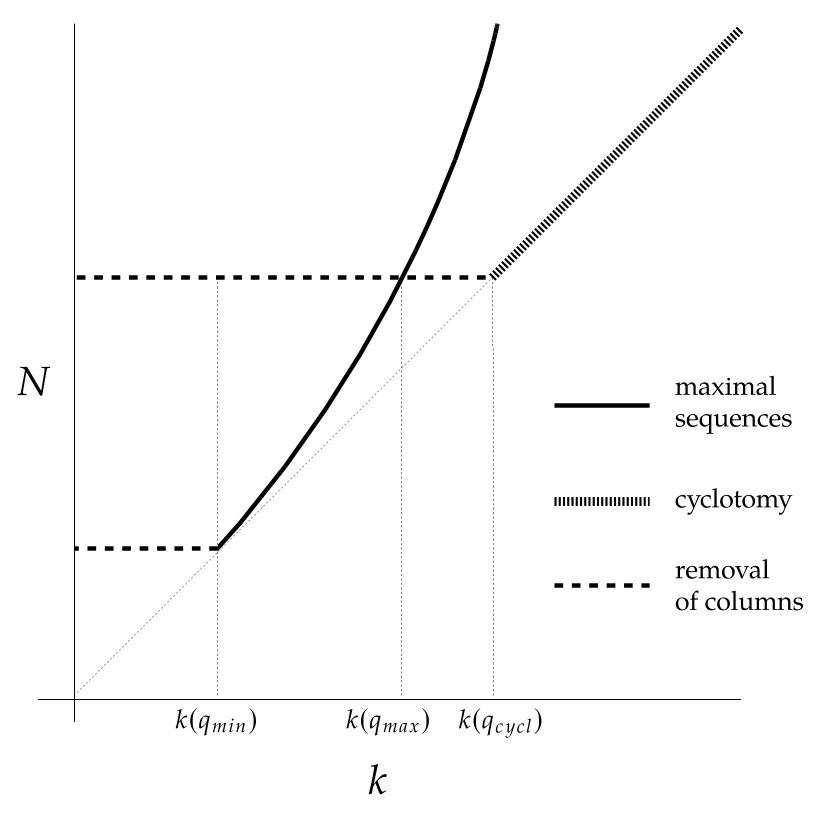}
\caption[Comparison of covering arrays from maximal sequences and cyclotomy]{Comparison of the sizes of covering arrays from maximal sequences and covering arrays from cyclotomy, for fixed $t$ and $v$.} 
\label{figure:comparison}
\end{figure}
In \Cref{figure:comparison}, we give the graph of $N$ as a function of $k$ for the $\CA(N;t,k,v)$ that can be constructed from cyclotomy and maximal sequences. 
The horizontal doted lines are justified since when we remove some columns from a covering array we obtain another covering array.

\begin{table}
    \centering
\renewcommand{\arraystretch}{\genarraystretch}
\rowcolors{2}{\backgroundshade}{white}
\setlength{\tabcolsep}{8.5pt}
\begin{tabular}{rrrrcrrrrcrrrr}
\rowcolor{\tableheadcolor}
$v$ & $t$ & $\qmin$ & $\qmax$ && 
$v$ & $t$ & $\qmin$ & $\qmax$ &&
$v$ & $t$ & $\qmin$ & $\qmax$ \\
3&3&16&16    && 12&8&313&337&    & 12&9&277&313\\                 
4&4&37&37    && 12&9&277&313&    & 12&10&277&289\\                           
4&5&37&37    && 12&10&277&289&   & 13&7&443&443\\                
4&6&29&29    && 13&7&443&443&    & 13&10&313&313\\               
5&5&61&61    && 13&10&313&313&   & 14&8&449&463\\                
6&5&97&97    && 14&8&449&463&    & 14&9&421&421\\                
6&6&73&79    && 14&9&421&421&    & 14&10&379&379\\                
6&7&67&79    && 14&10&379&379&   & 15&7&601&601\\                 
6&8&67&73    && 15&7&601&601&    & 15&8&541&541\\                 
6&9&67&67    && 15&8&541&541&    & 15&10&421&421\\                
6&10&67&67   && 15&10&421&421&   & 16&8&593&641\\                 
7&6&113&127  && 16&8&593&641&    & 16&9&529&593\\                 
7&7&113&113  && 16&9&529&593&    & 16&10&529&529\\                
8&5&193&193  && 16&10&529&529&   & 17&9&613&647\\                 
8&6&169&169  && 17&9&613&647&    & 17&10&613&613\\                
8&7&137&137  && 17&10&613&613&   & 18&8&811&829\\               
8&8&121&137  && 18&8&811&829&    & 18&9&739&757\\               
8&9&113&121  && 18&9&739&757&    & 18&10&631&631\\              
8&10&113&121 && 18&10&631&631&   & 19&10&761&761\\                   
9&7&181&199  && 19&10&761&761&   & 20&8&1021&1021\\     
9&8&163&181  && 20&8&1021&1021&  & 20&9&881&961\\       
9&9&163&163  && 20&9&881&961&    & 20&10&821&881\\      
9&10&163&163 && 20&10&821&881&   & 21&9&1009&1051\\      
10&6&271&281 && 21&9&1009&1051&  & 21&10&967&967\\       
10&7&241&251 && 21&10&967&967&   & 22&8&1277&1277\\      
10&8&211&211 && 22&8&1277&1277&  & 22&9&1123&1123\\      
10&9&181&211 && 22&9&1123&1123&  & 22&10&1013&1013\\     
10&10&181&191&& 22&10&1013&1013& & 23&9&1289&1289\\       
11&6&353&353 && 23&9&1289&1289&  & 23&10&1151&1151\\     
11&8&243&243 && 23&10&1151&1151& & 24&9&1369&1369\\      
11&9&243&243 && 24&9&1369&1369&  & 24&10&1249&1321\\     
11&10&243&243&& 24&10&1249&1321& & &&&\\
12&7&349&373 && 12&8&313&337&    & &&&                
\end{tabular}

    \caption[Sample of parameter values for which covering arrays from maximal sequences have fewer rows that covering arrays from cyclotomy]{A sample of values of $v,t$ where $v\leq 24$, $t\leq 10$, and the corresponding $\qmin=\qmin(v,t)$ and $\qmax=\qmax(v,t)$.
    For any prime power $q$ such that $q|v-1$ and $\qmin\leq q \leq \qmax$, covering arrays from maximal sequences can be constructed with fewer rows than covering arrays from cyclotomy with the same number of columns.
}
\label{table:comparisoncyclotomy}
\end{table}
In \Cref{table:comparisoncyclotomy}, we list for $2\leq v \leq 25$ and $2\leq t \leq 10$, the nonempty intervals $q_{min}(v,t) \leq q \leq q_{max}(v,t)$, $q$ a prime power, where  a covering array from maximal sequences yields fewer rows than the corresponding covering array from cyclotomy.
We note that both in this work and in \cite{colbourn2010covering}, the conditions that guarantee the existence of the covering arrays  can be replaced with slightly better (weaker) but more complicated conditions found in the proofs.
The compilation of \Cref{table:comparisoncyclotomy} was done by taking into account these better conditions for both constructions.

\subsection{Experimental results}
\label{section:ExperimentalResults}
\Cref{equation:MainTheoremSecondPaper} gives a theoretical lower bound on $q$ to guarantee the existence of a covering array.
Values of $q$ smaller than this theoretical bound can yield covering arrays.
To test this, we created the arrays described in \Cref{corollary:CharactersCAsOfStrength3,corollary:CharactersCAsOfStrength4} for values of $q$ and $v$, $v|q-1$, for which experiments were computationally feasible and checked if they were covering arrays.
This was often true for values of $q$ smaller than the required condition, as demonstrated in 
\Cref{table:experiments_strength3,table:experiments_strength4}.
As before, the theoretical guarantee is calculated using the weaker, more complicated conditions for $q$ that follow from the proofs in \Cref{section:ProofOfMainTheoremSecondPaper}, rather than the conditions in
\Cref{corollary:CharactersCAsOfStrength3,corollary:CharactersCAsOfStrength4}.

\begin{table}
\begin{subtable}{1\textwidth}
    \centering
\[
\renewcommand{\arraystretch}{\genarraystretch}
    \rowcolors{2}{\backgroundshade}{white}
    \begin{array}{rlll}
    \rowcolor{\tableheadcolor}
    v
    &\multicolumn{1}{l}{\text{Successful $q$ below theoretical guarantee}}
    &\multicolumn{1}{l}{\text{Values of $q$ tested}}
    &\multicolumn{1}{l}{\text{Theoretical guarantee}}
    \\
    2   & \text{None}                                      &q <    7   & q\geq 7         \\
    3   & \text{None}                                      &q <    19  & q\geq 19        \\
    4   & q\geq 37,  \text{ except } 41                    &q <    61  & q\geq 61        \\
    5   & q\geq 81,  \text{ except } 101                   &q <    181 & q\geq 181       \\
    6   & q\geq 103, \text{ except } 109,139,157,223,277   &q <    439 & q\geq 439       \\
    7   & q \geq 169                                       &q \leq 379 & q\geq 953       \\
    8   & 193,241,281,313                                  &q \leq 337 & q\geq 1849      \\
    9   & 289                                              &q \leq 523 & q\geq 3259      \\
    10  & 361                                              &q \leq 491 & q\geq 5261      \\
    11  & 243,397                                          &q \leq 463 & q\geq 8273      \\
    12  & \text{None}                                      &q \leq 457 & q\geq 12241     
    \end{array}
\]
    \caption[Values for $q$ for which we obtain covering arrays of strength $3$]{Ranges for $q$ such that $v|q-1$ and $q$
        being a prime power is sufficient for the
        construction of a 
        $\CA(v(q^2+q+1); 3,q^2+q+1,v)$ as per
        \Cref{corollary:CharactersCAsOfStrength3}}
    \label{table:experiments_strength3}
\end{subtable}

\vspace{2em}
 
\begin{subtable}{1\textwidth}
    \centering
\[
    \renewcommand{\arraystretch}{1.1}
    \rowcolors{2}{\backgroundshade}{white}
    \begin{array}{rlll}
    \rowcolor{\tableheadcolor}
    v
    &\multicolumn{1}{l}{\text{Successful $q$ below theoretical guarantee}}
    &\multicolumn{1}{l}{\text{Values of $q$ tested}}
    &\multicolumn{1}{l}{\text{Theoretical guarantee}}
    \\
    2 & q \geq 5                                         &q <    9   & q\geq 9         \\
    3 & q \geq 13                                        &q <    19  & q\geq 19        \\
    4 & q \geq 25                                        &q <    37  & q\geq 37        \\
    5 & 31 \leq q \leq 61                                &q \leq    61  & q\geq 81       \\
    6 & 49                                               &q \leq 49 & q\geq 139       \\
    7 & \text{None}                                      &q \leq 43  & q\geq 239       \\
    8 & \text{None}                                      &q \leq 49  & q\geq 337
    \end{array}
\]
    \caption[Values for $q$ for which we obtain covering arrays of strength $4$]{Ranges for $q$ such that $v|q-1$ and $q$ being a prime power is sufficient for the construction of a $\CA(v(q^3+q^2+q+1); 4,q^2+1,v)$ as per \Cref{corollary:CharactersCAsOfStrength4}.}
    \label{table:experiments_strength4}
\end{subtable}
\caption{Experimental results and comparison with theoretical guarantee}
\end{table}

\chapter[A preliminary computation of the number of $t$-tuples in cyclic trace arrays $\bmod v$]{A preliminary computation of the number of $t$-tuples in cyclic trace arrays modulo $v$}
\label{chapter:Index}

\def\Sj{\sum_{j=1}^{v-1}}
\def\Sjone{\sum_{j_1=1}^{v-1}}
\def\Sjtwo{\sum_{j_2=1}^{v-1}}
\def\Sjthree{\sum_{j_3=1}^{v-1}}
\def\nr{N_r}
\def\xj{\chi^{j}}
\def\xji{\chi^{j_i}}
\def\xjone{\chi^{j_1}}
\def\xjtwo{\chi^{j_2}}
\def\xjthree{\chi^{j_3}}

\textsc{Let $q$ be a prime power}, $t>2$ be an integer, $v\geq 2$ be a divisor of $q-1$, $\a$ be a primitive element of $\fqt$, and $C\subseteq [0,\w{t}-1]$ such that $\{\a^c\mid c\in C\}$ is a $(|C|,t-1)$-set.
In \Cref{chapter:CAsFromMSequencesAndCharacterSums} we showed by means of a character sum argument that when 
\begin{equation}
    \label{equation:goma}
    q^{\frac{t}{2}-2}(q-tv)\geq v^{t-1},
\end{equation}
then $\MV{\a,C}$ is a $\CA(v \w{t}; t, |C|,v)$.
As we demonstrated in \Cref{section:ExperimentalResults}, the bound in \Cref{equation:goma} appears to be weak, that is, covering arrays are produced for $q$ smaller than this bound.
In this chapter, we develop another approach to determine when covering arrays are obtained, based on bounding the number of rows that contain different $t$-tuples.
In \Cref{section:PreliminariesLambdas} we give the necessary preliminaries and in \Cref{section:UpperLowerBoundsOnLambda} we give upper and lower bounds for this number, that we express as the solutions of a linear program.

The contents of this chapter are preliminary work which has also been submitted as an extended abstract \cite{kalamata}.

\section{Preliminaries}
\label{section:PreliminariesLambdas}
For the rest of the chapter, we fix the notation of $q,t,v$ and $\a$ as in the introduction and let $\o=\a^{\w{t}}$, which is a primitive element of $\fq$ as shown in \Cref{lemma:CharacterizationOfConstantMultiplesInFQM}.
Furthermore, we define $M$ to be the $(q-1)\times (q-1)$ array, whose $(i,j)$-th element for $(i,j)\in [0,q^t-2]^2$ is given by 
\begin{align*}
    M_{ij}=
    \begin{cases}
        \logw(\Tt(\a^{i+j})) \bmod v,& \text{ if } \Tt(\a^{i+j})\neq 0;\\
        0, & \text{ otherwise.}
    \end{cases}
\end{align*}
We index the rows and columns of $M$ by $[0,q^t-2]$.

We start with some preliminary observations.
Let $I\subseteq [0,q^t-2]$ with $|I|=s$ and let $M_I$ be the $(q^t-1)\times s$ subarray of $M$ defined by the columns with indexes in $I$.
For $\bfb \in \Z_v^s$, let $\lambda_{\bfb}$ denote the number of rows in $M_I$ equal to $\bfb$.
From a similar argument as the one we used for \Cref{equation:kakakia} we observe that, if there are $r$ zeros in $\bfb$, then there are $(\frac{q-1}{v})^{s-r} (\frac{q+v-1}{v})^{r}$ pre-images under the $\logw$.
It follows that
\begin{equation}
    \label{lambda-linind}
    \lambda_{\bfb} = \left( \frac{1}{v} \right)^s (q-1)^{s-r} (q+v-1)^{r}
\end{equation}
if $\{\a^i \mid i \in I\}$ is linearly independent, and
\[
  0 \leq \lambda_{\bfb} \leq 
    q\left(\frac{1}{v}\right)^s (q-1)^{s-r} (q+v-1)^{r} 
\]
otherwise.
Our goal is to improve these bounds.
Let $C\subseteq [0,q^t-2]$ such that $\{\a^c\mid c\in C\}$ is a $(t,t-1)$-set.
We recall that $d(\bfx,\bfy)$ denotes the Hamming distance between the vectors $\bfx$ and~$\bfy$.
\begin{lemma}
{}{}
Let $q$ be a prime power, $v | q-1$, and $\bfb=(b_1,\ldots,b_t) 
\in \Z_v^t$ with $r$ $0$-entries. Let $\lambda_{\bfb}$ be the number of rows equal to $(b_1,\ldots,b_t)$ of the subarray of $M$ defined by the columns in $C$.
Then, we have
\begin{align} 
  \lambda_{\bfb} &= \left( \frac{1}{v} \right)^{t} (q-1)^{t-r} (v+q-1)^{r} \nonumber\\
 &\qquad -  (-1)^t \left( \frac{1}{v} \right)^{t} 
    (q-1)^{r} (v-1)^{r} \left(v + (q-1)(v-1)\right)^{t-r} \nonumber \\
    &\qquad +  (-1)^t \sum_{\scriptsize \begin{array}{l}\bfc \in \Z_v^t \\ d(\bfb,\bfc)=t\end{array}} \lambda_{\bfc},\label{lambda-exact}.
\end{align}
\end{lemma} 

\begin{proof}We start by computing $\lambda_{\bfb}$ using an inclusion-exclusion argument.
    \[
    \lambda_{\bfb} = \sum_{s=0}^{t} (-1)^s \sum_{\scriptsize \begin{array}{l}C' \subseteq C\\ |C'| = s\end{array}} \sum_{\scriptsize \begin{array}{l}\bfc \in \Z_v^s\\d(\bfc,\bfb|_{C'}) = s\end{array}} \lambda_{\bfc}
  \]
  Now we use the fact that if $C' \subsetneq C$ then $C'$ is a linearly independent set of columns and therefore \Cref{lambda-linind} allows us to explicitly compute this sum for all but $C'=C$.  To use \Cref{lambda-linind} to compute $\lambda_{\bfc}$, with $d(\bfc,\bfb|_{C'})=s$ we need to know the number of zero entries in $\bfc$.  We count these by supposing that $\bfb$ has $j$ zero entries and $i$ non-zero entries among the $C'$ with $i+j=s \leq t$.  If $\bfc$ disagrees everywhere with $\bfb|_{C'}$ then for all the zero entries in $\bfb|_{C'}$, $\bfc$ can be any of the $v-1$ non-zero elements of $\Z_{v}$.  For a position where $\bfb|_{C'}$ is non-zero, say $z$, $\bfc$ could have a zero entry or one of the $v-2$ other non-zero entries of $\Z_{v}\setminus\{z\}$.  We count the number of zero entries in $\bfc$ with $k \leq i$.  Thus 
\begin{align*}
    \lambda_{\bfb} &= \sum_{\scriptsize \begin{array}{c}0 \leq j \leq r\\0 \leq i \leq t-r\\ i+j<t\end{array}} 
    {r \choose j} {t-r \choose i} (-1)^{i+j} \\
 &\qquad \sum_{k=0}^{i} {i \choose k} (v-1)^j (v-2)^{i-k} 
    \left( \frac{q-1}{v} \right)^{i+j-k}
    \left( \frac{q+v-1}{v} \right)^{k} q^{t-i-j} \\
  &\qquad + (-1)^t \sum_{\scriptsize \begin{array}{l}\bfc \in \Z_v^t \\ d(\bfb,\bfc)=t\end{array}} \lambda_{\bfc}
\end{align*}
To simplify the first sum we include $i+j = t$ in the sum and remove it afterwards.

\begin{align*}
    \lambda_{\bfb} 
&= \sum_{j=0}^{r} \sum_{i=0}^{t-r} 
    {r \choose j} {t-r \choose i} (-1)^{i+j} \\
 &\qquad \sum_{k=0}^{i} {i \choose k} (v-1)^j (v-2)^{i-k} 
    \left( \frac{q-1}{v} \right)^{i+j-k}
    \left( \frac{q+v-1}{v} \right)^{k} q^{t-i-j} \\
&\qquad - (-1)^t \sum_{k=0}^{t-r} {t-r \choose k} 
    (v-1)^r (v-2)^{t-r-k} \left( \frac{q-1}{v} \right)^{t-k}
    \left( \frac{q+v-1}{v} \right)^{k} \\
&\qquad + (-1)^t \sum_{\scriptsize \begin{array}{l}\bfc \in \Z_v^t \\ d(\bfb,\bfc)=t\end{array}} \lambda_{\bfc}
\end{align*}
Manipulating the above expression we get
\begin{align*}
    \lambda_{\bfb} 
&= \sum_{j=0}^{r} \sum_{i=0}^{t-r} 
    {r \choose j} {t-r \choose i} (-1)^{i+j}
    (v-1)^j \left( \frac{q-1}{v} \right)^{i+j} q^{t-i-j}\\
 &\qquad \sum_{k=0}^{i} {i \choose k} (v-2)^{i-k} 
    \left( \frac{q-1}{v} \right)^{-k}
    \left( \frac{q+v-1}{v} \right)^{k} \\
&\qquad - (-1)^t \left( \frac{1}{v} \right)^{t} (q-1)^{t} 
    (v-1)^{r} (v-2)^{t-r} \sum_{k=0}^{t-r} {t-r \choose k} 
    \left( \frac{q+v-1}{(q-1)(v-2)} \right)^{k} \\
&\qquad + (-1)^t \sum_{\scriptsize \begin{array}{l}\bfc \in \Z_v^t \\ d(\bfb,\bfc)=t\end{array}} \lambda_{\bfc}\\
&= \sum_{j=0}^{r} \sum_{i=0}^{t-r} 
    {r \choose j} {t-r \choose i} (-1)^{i+j}
    (v-1)^j \left( \frac{q-1}{v} \right)^{i+j} q^{t-i-j}
    \left(\frac{qv}{q-1}-1\right)^i \\
&\qquad - (-1)^t \left( \frac{1}{v} \right)^{t} (q-1)^{t} (v-1)^{r} 
    (v-2)^{t-r} \left( 1+ \frac{q+v-1}{(q-1)(v-2)} \right)^{t-r} \\
&\qquad + (-1)^t \sum_{\scriptsize \begin{array}{l}\bfc \in \Z_v^t \\ d(\bfb,\bfc)=t\end{array}} \lambda_{\bfc}\\
&= q^{t}\sum_{i=0}^{t-r} {t-r \choose i} (-1)^i 
    \left(1-\frac{q-1}{qv}\right)^i \sum_{j=0}^{r} {r \choose j}
    (-1)^j (v-1)^j 
    \left( \frac{q-1}{qv} \right)^{j} \\
&\qquad - (-1)^t \left( \frac{1}{v} \right)^{t} (q-1)^{t} 
    (v-1)^{r} (v-2)^{t-r} 
    \left( \frac{v+ (q-1)(v-1)}{(q-1)(v-2)} \right)^{t-r} \\
&\qquad + (-1)^t \sum_{\scriptsize \begin{array}{l}\bfc \in \Z_v^t \\ d(\bfb,\bfc)=t\end{array}} \lambda_{\bfc}\\
\end{align*}
Finally we obtain 
\begin{align*}
  \lambda_{\bfb} 
&= \left( \frac{1}{v} \right)^{t} (q-1)^{t-r} (v+q-1)^r \\
&\qquad - (-1)^t \left( \frac{1}{v} \right)^{t} (q-1)^{r} 
    (v-1)^{r} \left( v+(q-1)(v-1) \right)^{t-r} \\
&\qquad + (-1)^t \sum_{\scriptsize \begin{array}{l}\bfc \in \Z_v^t \\ d(\bfb,\bfc)=t\end{array}} \lambda_{\bfc}.
\end{align*}
\end{proof}

To simplify our expression in the remainder of the paper we define
\begin{align} 
  c_r &= \left( \frac{1}{v} \right)^{t} (q-1)^{t-r} (v+q-1)^r, \label{CRlambda} \\
  c_{r,t-1} &= (-1)^{t-1} q \left( \frac{1}{v} \right)^{t-1} ((v-1)(q-1))^{r-1}(qv-q+1)^{t-r-1} \nonumber \\ &\qquad (tvq-tq-tv+t+rv) \label{CRlambda_penultimate}, \\
  c_{r,t} &= (-1)^t \left( \frac{1}{v} \right)^{t} (q-1)^{r} 
  (v-1)^{r} \left( v+(q-1)(v-1) \right)^{t-r}, \label{CRlambda_last} \\
  d_{r,k}   &= (-1)^t {t-r \choose k} (v-1)^r (v-2)^{t-r-k}. \label{DKlambda}
\end{align}
Note that $c_{r,t-1}$ and $c_{r,t}$ are the $i+j=t-1$ and $i+j=t$ summands of $c_r$ respectively.

\section{Upper and lower bounds on $\lambda$} \label{linprog}
\label{section:UpperLowerBoundsOnLambda}

In this section we compute more precise upper lower and upper bounds on $\lambda_{\bfb}$. We define
\[
  \lambda_{r}^{-} = \min\{\lambda_{\bfb} : \bfb\in \Z_v^t, \bfb \mbox{ has $r$ zeros} \}
\]
and
\[
  \lambda_{r}^{+} = \max\{\lambda_{\bfb} : \bfb\in \Z_v^t, \bfb \mbox{ has $r$ zeros} \}.
\]
The very simplest bounds we know when $\bfb$ has $r$ zeros, are
\[
  0 \leq \lambda_r^- \leq \lambda_{\bfb} \leq \lambda_r^+ \leq qc_r.
\]

One improvement on this is to take advantage of the fact that the successive terms of a inclusion-exclusion computation are each upper and lower bounds depending on the sign of the last term in the summation. We use the fact that truncating before the last and before the second-to-last terms are the best possible of these upper and lower bounds.

\begin{lemma}
{}{}
  Let $\{A_i: i \in I\}$ be subsets of $U$ with $|I|= t$ and for $0 \leq s \leq t$, define
  \[
    c_s = \sum_{i=0}^{s} \sum_{\begin{array}{c} S \subseteq I \\ |S| = i\end{array}} \left | \bigcap_{i \in S} A_i \right |
  \]
  Then for any  $0\leq j \leq t$ we have

  \begin{equation}\label{truncation-bounds}
(-1)^{t-j} c_t \leq (-1)^{t-j} c_{t-j},
  \end{equation}
  for all odd $1 \leq j \leq t$ we have 
  \begin{equation}\label{penultimate-bound}
  (-1)^{t-1} c_{t-1} \leq (-1)^{t-j} c_{t-j},
  \end{equation}
  and for all even $2 \leq j < t$ we have
  \begin{equation}\label{antepenultimate-bound}
  (-1)^{t-2} c_{t-2} \leq (-1)^{t-j} c_{t-j}.
  \end{equation}
\end{lemma}
\begin{proof}
We observe that  $c_0 = |U|$ and $c_t = |U \setminus \bigcup_{i \in I} A_i|$
and $c_s$ is just the computation of $|U \setminus \bigcup_{i \in I} A_i|$
by inclusion-exclusion truncated at step $s$.  The first conclusion of the
theorem is simply that the successive truncations are upper and lower bounds
depending on the sign of the last summand.  It is proven by considering the
contribution of any single point of $U$ to each side of the equation.  If a
point $x$ is in none of the $A_i$ then it contributes 1 to the computation
of $c_s$ for all $0 \leq s \leq t$ and thus contributes equally to both
sides of Inequality~(\ref{truncation-bounds}).  If a point $x$ is in
exactly $r>0$ of the subsets then its contribution to $c_s$ is
  \[
    \sum_{i=0}^{s} (-1)^i \binom{r}{i} = (-1)^{s} \binom{r-1}{s}.
  \]
This last equality is easily established. (Also recall that
$(-1)^{s} \binom{-1}{s} = 1$.) Then Inequality~(\ref{truncation-bounds})
is established by noting that the contribution of a point in
$\bigcup_{i \in I} A_i$ (and thus $r \geq 1$) is
  \[
    (-1)^{t-j} (-1)^t \binom{r-1}{t} = 0 \leq \binom{r-1}{t-j}
  = (-1)^{t-j} (-1)^{t-j} \binom{r-1}{t-j}
  \]
since $t > r-1 \geq 0$.

Now we prove Inequality~(\ref{penultimate-bound}). Again the contribution
of any point $x$ in none of the $A_i$ is 1 to both $c_{t-1}$ and to
$c_{t-j}$, so we can restrict our attention to points that belong to
$r>0$ of the $A_i$.  If point $x$ is in $0< r < t$ of the $A_i$ then
its contribution to $c_{t-1}$ is
  \[
    (-1)^{t-1} \binom{r-1}{t-1} = 0
  \]
and its contribution to $c_{t-j}$ is
\[
  (-1)^{t-j}\binom{r-1}{t-j}.
\]
The needed inequality follows from
\[
  (-1)^{t-1} 0  = 0 \leq \binom{r-1}{t-j}
= (-1)^{t-j} (-1)^{t-j} \binom{r-1}{t-j}.
\]
If $x$ is a point in all $t$ of the $A_i$ then its contribution to
$c_{t-1}$ is
  \[
    (-1)^{t-1} \binom{t-1}{t-1} = (-1)^{t-1}
  \]
and it contribution to $c_{t-j}$ is
\[
  (-1)^{t-j}\binom{t-1}{t-j}.
\]
The needed inequality follows from
\[
  (-1)^{t-1} (-1)^{t-1}  = 1 \leq \binom{t-1}{t-j}
= (-1)^{t-j} (-1)^{t-j} \binom{t-1}{t-j},
\]
unless $t-j<0$, that is $j>t$, or $j<1$ which are not possible.

Now we prove Inequality~(\ref{antepenultimate-bound}). Again the
contribution of any point $x$ in none of the $A_i$ is 1 to both
$c_{t-1}$ and to $c_{t-j}$, so we can restrict our attention to
points that belong to $r>0$ of the $A_i$.  If point $x$ is in
$0< r < t-1$ of the $A_i$ then its contribution to $c_{t-2}$ is
  \[
    (-1)^{t-2} \binom{r-1}{t-2} = 0
  \]
and its contribution to $c_{t-j}$ is
\[
  (-1)^{t-j}\binom{r-1}{t-j}.
\]
The needed inequality follows from
\[
  (-1)^{t-1} 0  = 0 \leq \binom{r-1}{t-j}
= (-1)^{t-j} (-1)^{t-j} \binom{r-1}{t-j}.
\]
If $x$ is a point in $t-1$ of the $A_i$ then its contribution
to $c_{t-2}$ is
  \[
    (-1)^{t-2} \binom{t-2}{t-2} = (-1)^{t-2}
  \]
and its contribution to $c_{t-j}$ is
\[
  (-1)^{t-j}\binom{t-2}{t-j}.
\]
The needed inequality follows from
\[
  (-1)^{t-2} (-1)^{t-2}
= 1 \leq \binom{t-2}{t-j} = (-1)^{t-j} (-1)^{t-j} \binom{t-2}{t-j},
\]
unless $j$ is out of range ($j>t$ or $j<2$). If $x$ is a point in
all $t$ of the $A_i$ then its contribution to $c_{t-2}$ is
  \[
    (-1)^{t-2} \binom{t-1}{t-2} = (-1)^{t-2}(t-1)
  \]
and its contribution to $c_{t-j}$ is
\[
  (-1)^{t-j}\binom{t-1}{t-j}.
\]
The needed inequality follows from
\[
  (-1)^{t-2} (-1)^{t-2}(t-1)  = t-1 \leq \binom{t-1}{t-j}
= (-1)^{t-j} (-1)^{t-j} \binom{t-1}{t-j},
\]
unless $j$ is out of range ($j<2$, $j\geq t$).  When $j=t$ this
inequality can be violated indicating that our method does not
prove the bound in Inequality~(\ref{antepenultimate-bound}).
However when $j=t$ and even, and letting $n_i$ be the number
of points that appear in exactly $i$ of the $A$, we get that
Inequality~(\ref{antepenultimate-bound}) holds except when
\[
n_t(t-2) > \sum_{i=1}^{t-2} n_i.
\]
It is also worth noting that when $j=t$,  $c_{t-j} = |U|$ and so
the only time that $c_{t-2}$ is a worse upper bound for $c_t$
than $c_0$ is when $c_{t-2} > |U|$.
\end{proof}

This lemma implies that the best bounds obtainable by truncating 
an inclusion-exclusion computation are obtained by the last two 
truncations, except possibly in the case that $t$ is even and 
there exist points in every $A_i$, when it might be the case that 
\[
  c_t < c_0=|U| < c_{t-2}.
\]
Thus when $t$ is even and $\bfb$ has $r$ zeros,
\begin{equation}
  c_r - c_{r,t} \leq \lambda_{\bfb}  \leq 
 \min(c_r -c_{r,t} - c_{r,t-1},q^t) \\
\end{equation}
and when $t$ is odd and $\bfb$ has $r$ zeros, 
\begin{equation}
  c_r - c_{r,t} - c_{r,t-1} \leq \lambda_{\bfb} \leq c_r - c_{r,t}.
\end{equation}

To derive even more precise bounds on $\lambda_{\bfb}$ we use 
Linear Programming by substituting $\lambda_r^-$ and 
$\lambda_r^+$ for $\lambda_{\bfc}$ in \Cref{lambda-exact}. 
How we do so will again depend on the sign of the last term 
which depends on the parity of $t$.

\subsection{Bounds on $\lambda_{\bfb}$ when $t$ is even}
When $t \bmod 2 = 0$ \Cref{lambda-exact} is
\begin{align*} \label{lambda-exact-t-0}
    \lambda_{\bfb} 
= \left( \frac{1}{v} \right)^{t} (q-1)^{t-r} (v+q-1)^{r}
- \left( \frac{1}{v} \right)^{t} 
(q-1)^{r} (v-1)^{r} \left(v + (q-1)(v-1)\right)^{t-r} 
+ \sum_{\scriptsize \begin{array}{c}\bfc \in \Z_v^t \\d(\bfb,\bfc)=t\end{array}} \lambda_{\bfc},
\end{align*}
Since $\lambda_k^- \leq \lambda_{\bfc} \leq \lambda_k^+$ for 
any $\bfc$ with $k$ zeros we have that   
\begin{align*}
  c_r - c_{r,t} + \sum_{k=0}^{t-r} {t-r \choose k} (v-1)^r (v-2)^{t-r-k} \lambda_k^- 
 \leq  \lambda_{\bfb} 
 \leq  c_r - c_{r,t} + \sum_{k=0}^{t-r} {t-r \choose k} 
  (v-1)^r (v-2)^{t-r-k} \lambda_k^+, 
\end{align*}
for any $\bfb$ with $r$ zeros, and thus
\begin{eqnarray*}
c_r - c_{r,t} &\leq &  \lambda_r^- - \sum_{k=0}^{t-r} d_{r,k} \lambda_k^-, \qquad 0\leq r\leq t \\
-c_r +c_{r,t} &\leq & -\lambda_r^+ + \sum_{k=0}^{t-r} d_{r,k} \lambda_k^+, \qquad 0\leq r\leq t \\
0 &\leq & \lambda_r^- \leq \lambda_r^+ \leq qc_r,     \qquad 0\leq r\leq t.
\end{eqnarray*}
Let $\bfl\in \Z^{2t+2}$,
$\bfC\in \Z^{5t+5}$ and 
$D\in \Z^{(t-1)\times (t+1)}$ defined by
\begin{align*}
   \bfl 
&= \left(-\lambda_0^-,\ldots,-\lambda_t^-,\lambda_0^+, \ldots,
   \lambda_t^+\right)\\
   \bfC 
&= \left (c_0-c_{0,t},\ldots,c_t-c_{t,t},-c_0+c_{0,t},\ldots,
-c_t+c_{t,t},0_1, \ldots,0_{2t+2},-qc_0,\ldots,-qc_{t}\right )\\
D_{i,j} &=   d_{i,j}={t-i \choose j} (v-1)^i (v-2)^{t-i-j},
\end{align*}
and $B \in \Z^{5t+5 \times 2t+1}$ defined in $(t+1) \times (t+1)$ blocks as 
\[ B = 
\left ( \begin{array}{c|c}
    D-I_{t+1}    &  {\bf 0} \\ \hline
    {\bf 0}      &  D-I_{t+1}\\ \hline
    I_{t+1}      &  I_{t+1} \\ \hline
    -I_{t+1}     &  {\bf 0} \\ \hline
    {\bf 0} &  -I_{t+1} \\ 
\end{array} \right ),
\]
We have the linear program

\begin{equation}
    \label{equation:LinearProgram1}
    B\bfl \geq \bfC.
\end{equation}

\subsection{Bounds on $\lambda_{\bfb}$ when $t$ is odd}
When $t \bmod 2 = 1$ \Cref{lambda-exact} is
\begin{align*} \label{lambda-exact-t-1}
    \lambda_{\bfb} 
= \left( \frac{1}{v} \right)^{t} (q-1)^{t-r} (v+q-1)^{r} 
+ \left( \frac{1}{v} \right)^{t} 
(q-1)^{r} (v-1)^{r} \left(v + (q-1)(v-1)\right)^{t-r}
- \sum_{\scriptsize \begin{array}{c}\bfc \in \Z_v^t\\ d(\bfb,\bfc)=t\end{array}} \lambda_{\bfc},
\end{align*}
Substituting in the bounds for $\lambda_{\bfc}$, we get    
\begin{align*}
 c_r - c_{r,t}  - \sum_{k=0}^{t-r} {t-r \choose k} 
  (v-1)^r (v-2)^{t-r-k} \lambda_k^+ 
  \leq  \lambda_{\bfb}
\leq  c_r - c_{r,t} - \sum_{k=0}^{t-r} {t-r \choose k} 
  (v-1)^r (v-2)^{t-r-k} \lambda_k^-, 
\end{align*}
and thus
\begin{eqnarray*}
c_r - c_{r,t} &\leq &  \lambda_r^- + \sum_{k=0}^{t-r} d_{r,k} \lambda_k^+, \qquad 0\leq r\leq t \\
-c_r + c_{r,t} &\leq & -\lambda_r^+ - \sum_{k=0}^{t-r} d_{r,k} \lambda_k^-, \qquad 0\leq r\leq t \\
0 &\leq & \lambda_r^- \leq \lambda_r^+ \leq qc_r,     \qquad 0\leq r\leq t.
\end{eqnarray*}
Let $\bfl\in \Z^{2t+2}$,
$\bfC\in \Z^{5t+5}$ and 
$D\in \Z^{(t-1)\times (t+1)}$ defined by
\begin{align*}
   \bfl 
&= \left(-\lambda_0^-,\ldots,-\lambda_t^-,\lambda_0^+, \ldots,
          \lambda_t^+\right)\\
   \bfC 
&= \left (c_0-c_{0,t},\ldots,c_t-c_{t,t},-c_0+c_{0,t},\ldots,-c_t+c_{t,t},0_1, \ldots,0_{2t+2},-qc_0,\ldots,-qc_{t}\right ) \\
D_{i,j} &=   d_{i,j}={t-i \choose j} (v-1)^i (v-2)^{t-i-j},
\end{align*}
and $B \in \Z^{5t+5 \times 2t+1}$ defined in $(t+1) \times (t+1)$ blocks as 
\[ B = 
\left ( \begin{array}{c|c}
    -I_{t+1}    &  D \\ \hline
    D      & -I_{t+1}\\ \hline
    I_{t+1}      &  I_{t+1} \\ \hline
    -I_{t+1}     &  {\bf 0} \\ \hline
    {\bf 0} &  -I_{t+1} \\ 
\end{array} \right ),
\]
We have the linear program
\begin{equation}
    \label{equation:LinearProgram2}
    B\bfl \geq \bfC.
\end{equation}

As we discuss in \Cref{chapter:FutureDirections}, we plan to implement a computer program to solve the above linear program.

\titlespacing*{\chapter}{0cm}{1.7cm}{2cm} 
\chapter{Future directions}
\label{chapter:FutureDirections}

\textsc{We conclude} with a number of directions for future research related to our thesis.

\subsubsection[Cyclic trace arrays and finite geometry]{Concatenation of cyclic trace arrays and finite geometry}
\label{section:fasolakia}
In \Cref{chapter:CAsFirstPaper} we consider the vertical concatenation of cyclic trace arrays and we search for subarrays of columns that are covering arrays.
Here we restate the problem in terms of finite geometry, we analyze some of our experimental results from this point of view and conclude with some questions that naturally arise.

For simplicity, let us assume the case of vertically concatenating two cyclic trace arrays. 
The problem is the following: let $\a$ be a primitive element of $\fqt$ and let $j$ be coprime to $q^t-1$, which means that $\a^j$ is also primitive.
We want to find $S\subseteq [0,\w{t}-1]$ with the property that, for every $I \subseteq S$ with $|I|=t$, we have that $\ZA{\a,I}$ or $\ZA{\a^j,I}$ is an $\OA(t, t,q)$.
For such $S$, we have that $\ZA{ \{\a,\a^j\}, S}$ is a $\CA(2(q^t-1)+1; t, |S|,q)$.
This problem is studied in \Cref{chapter:CAsFirstPaper} in more generality, and our results are presented in \Cref{section:CAsFirstPaper_ImplementationAndNewBounds}.

We now restate the problem in terms of finite geometry.
First, we note that as discussed in \Cref{section:FiniteGeometry}, we consider the points of $PG(t-1,q)$ to be 
\[
    \label{equation:OvoidFuture}
    PG(t-1,q)=
    \{[\a^i] \mid i \in [0, \w{t}-1]\}.
\]
However, since $\a^j$ is also a primitive element of $\fqt$, we also have that
\[
    PG(t-1,q)=
    \{[\a^{ji}] \mid i \in [0, \w{t}-1]\}, 
\]
which means that the mapping 
\[
\begin{array}{ll}
    \phi_j: & PG(t-1,q) \rightarrow PG(t-1,q)\\
            & [\a^i]  \mapsto  [\a^{ji}]
\end{array}
\]
is a permutation of the points of $PG(t-1,q)$.
Furthermore, we note that, by \Cref{item:b,item:e} of \Cref{proposition:SebastianExtended}, $\ZA{\a,I}$ (respectively $\ZA{\a^j,I}$) is an $\OA(t,t,q)$ if and only if $\{[\a^{i}] \mid i \in I\}$ (respectively $\{[\a^{ij}]\mid i \in I\}$) is a set of points of $PG(t-1,q)$ that are not contained in a hyperplane.
Hence, we can restate our problem as follows: we want to find a permutation $\phi_j$ of the points of $PG(t-1,q)$ and a subset $S\subseteq PG(t-1,q)$ with the property that, for every $I \subseteq S$ with $|I|=t$, we have that $I$ or $\phi_j(I)$ is a set of points that are not contained in a hyperplane. 

From the above point of view, some of the experimental results of \Cref{chapter:CAsFirstPaper} are of particular interest.
We define
$
    I(q)= \{i(q+1) \mid i \in [0,q^2]\}.
$
From our discussion after \Cref{corollary:OAsOfStrength3FromMSequences}, for a primitive element $\a \in \f_{q^4}$ the set $\Omega(\a,q) = \{[\a^i] \mid i \in I(q)\}$ is an ovoid
in $PG(3,q)$, which means that every three points from the set are not colinear.\index{Ovoid}
For $q=4$, we have that $I(4)=\{5i \mid i \in [0,16]\}$.
We observe that the first experimental result in \Cref{table:LongTableNewarrays} is the array $\ZA{ \{\a,\a^{31}\}, I(4)}$, where $\a$ is given in \Cref{table:primitiveelementsused}.
From the geometrical interpretation above, we have that $\phi_{31}$ is a permutation of the points of $PG(3,4)$ that maps the ovoid $\Omega(\a,4)$ to the ovoid $\Omega(\a^{31}, 4)$ and, moreover, for every set $I$ of points in $\Omega(\a,4)$ with $|I|=4$, we have that $I$ or $\phi_{31}(I)$ is a set of points that are not contained in the same plane (that is the same as a hyperplane in $PG(3,q))$.

Another interesting experimental result can be described in similar terms.
For $d \in [1,q]$ such that $d|q^2+1$ we define
$
    I_d(q) = \{id(q+1)\mid i \in [0,q^2]\}.
    $
One of our experimental results is the $\CA(161; 4, 10,3)$ shown in \Cref{table:ResultsOverview}. 
This is constructed as a $\ZA{ \{\a,\a^{17}\}, S}$, where
\begin{align*}
    S&= \{0,1,8,9,16,17,24,25,32,33\}\\
     &= \{0,8,16,24,32\} \cup \{1,9,17,25,33\}\\
     &= I_2(3)\cup \left( I_2(3)+1 \right).
\end{align*}
This could be described as a construction using two \emph{half ovoids}.

Inspired by the above, we constructed arrays $\ZA{P,S}$ where $P$ is a set of primitive elements of $\f_{q^4}$ and $S$ is an ovoid or the union of fractions of an ovoid.
Our tests were for $q$ up to $11$.
The successful cases are listed in \Cref{table:CAsFromOvoidParts}.
We note that the cases for $q=3$ and $q=4$ were found using \Cref{algorithm:Backtracking}, whereas in all the other cases the covering array definition was tested directly.

We propose investigating the above constructions in the context of finite geometry.
One aspect of this research would be to find permutations of the points of $PG(3,q)$ such that they map $4$-sets of coplanar points to $4$-sets of non-coplanar points, when restricted to parts of ovoids.
For support in this direction, we suggest studying two previous works for possible connections:
Ebert \cite{ebert1985partitioning} describes the partitioning of $PG(3,q)$ into disjoint ovoids, and Wilson and Qing \cite{wilson1997cyclotomy} use half ovoids for the construction of a family of linear codes.

\begin{table}
\renewcommand{\arraystretch}{\genarraystretch}
    \rowcolors{2}{\backgroundshade}{white}
    \centering
    \begin{tabular}{lllll}
    \rowcolor{\tableheadcolor}
$q$&$   P $&$   S $&$   \ZA{P,S}  $& Ovoids\\
$3 $&$\{\a,\a^{17}\} $&$ I_2(3)\cup \left( I_2(3)+1 \right) $&$ \CA(161; 4, 10,3)$&Two halfs\\
$4 $&$\{\a,\a^{31}\} $&$ I(4) $&$ \CA(511; 4, 17,4)$&Ovoid\\
$5 $&$\{\a,\a^{37}\} $&$ I_2(5) $&$ \CA(1249; 4, 13,5)$&Half\\
$5 $&$\{\a,\a^{37},\a^{71}\} $&$ I_2(5)\cup \left( I_2(5)+3 \right) $&$ \CA(1873; 4, 26,5)$&Two half\\
$7 $&$\{\a,\a^{223}\} $&$ I_{10}(7)\cup \left( I_{10}(7)+1 \right) \cup \left( I_{10}(7)+3 \right)$&$ \CA(4801; 4, 15,7)$&Three tenths\\
$7 $&$\{\a,\a^{37},\a^{223}\} $&$ I_2(7)$&$ \CA(7201; 4, 25,7)$&Half\\
$8 $&$\{\a,\a^{43}\} $&$ I_5(8)\cup \left( I_5(8)+1 \right)\cup \left( I_5(8)+3 \right) $&$ \CA(8191; 4, 39,8)$&Three fifths\\
$9 $&$\{\a,\a^{7},\a^{13}\} $&$ I_2(9) $&$ \CA(19681; 4, 41,9)$&Half 
\end{tabular}
    \caption[Strength $4$ covering arrays from concatenation of cyclic trace arrays and ovoids.]{Strength $4$ covering arrays from vertical concatenation of cyclic trace arrays and parts of ovoids.}
    \label{table:CAsFromOvoidParts}
\index{Cyclic trace array!concatenation}
\end{table}

\subsubsection{Alternative algorithmic search for covering arrays from concatenation of cyclic trace arrays}
In \Cref{chapter:CAsFirstPaper} we consider the problem of vertically concatenating cyclic trace arrays and finding subarrays of columns that yield covering arrays as an optimization problem to which we give an algorithmic solution that uses backtracking.
One problem with the backtracking approach is that, for large values of $q$ such as $q\geq 15$, due to the increased size of the search space, \Cref{algorithm:Backtracking} takes a very long time to fully traverse branches of the tree corresponding to the search space.
We suggest adapting \Cref{algorithm:Backtracking} so that branches with more potential have higher probability of being examined first.
This could be done by ranking the nodes of the tree according to the size of the corresponding candidate set.
Another option would be to use this type of ranking with metaheuristics such as tabu search or simulated annealing.

\subsubsection{Other covering array constructions from character sums}
In \Cref{chapter:CAsFromMSequencesAndCharacterSums} we establish a criterion for cyclic trace arrays using a character sum argument.
A crucial part of this is bound a character sum that involves the function $h$, defined in \Cref{equation:DefinitionHr}.
A similar function is used by Colbourn \cite{colbourn2010covering} in order to construct covering arrays from cyclotomy, where instead of values of the trace function, linear shifts of the base field elements are used.
We propose extending this idea further by substituting the trace function in $h$ with some other function that distributes the values uniformly over the field, and examine whether this yields another covering array construction.

\subsubsection{Randomness properties of maximal sequences modulo $v$}
In \Cref{chapter:CAsFromMSequencesAndCharacterSums} we introduce the sequence $\seqv{\a}$.
To the best of our knowledge, this type of sequence has not been studied before. 
We propose researching these sequences further for randomness criteria similar to the ones studied in \cite{golomb2005signal} for LFSR sequences.
For example, examine their autocorrelation and crosscorrelation properties and, if $v$ is prime, find values or bounds for their linear complexity.

\subsubsection{Counting the rows of cyclic trace arrays modulo $v$ by linear programming.}
In \Cref{chapter:Index} we consider arrays constructed by applying the discrete logarithm modulo $v$ to the nonzero elements of maximal sequences, similar to the cyclic trace arrays of \Cref{chapter:CAsFromMSequencesAndCharacterSums}.
We provide bounds for the number of occurrences of the different rows in subarrays of columns, by expressing this number as the solution of a linear program, as shown in \Cref{equation:LinearProgram1,equation:LinearProgram2}.
We propose investigating this further by solving these linear program for some feasible cases, and compare the lower bounds to the conditions established in \Cref{chapter:CAsFromMSequencesAndCharacterSums} for cyclic trace arrays modulo $v$ to be covering arrays.

\addcontentsline{toc}{chapter}{List of figures}
\listoffigures
\addcontentsline{toc}{chapter}{List of tables}
\listoftables
\addcontentsline{toc}{chapter}{List of algorithms}
\listofalgorithms
\addcontentsline{toc}{chapter}{Nomenclature}
{
\renewcommand{\chaptermark}[1]{\markboth{\MakeUppercase{#1}}{}}

\chapter*{Nomenclature}
\chaptermark{Nomenclature}
Symbols that appear only in a restricted context are not listed.
Wherever appropriate, a page reference is given.

\begin{center}
\renewcommand{\arraystretch}{1.1}
\begin{longtable}{p{0.18\textwidth}p{0.82\textwidth}}
& \endfirsthead & \endhead \endfoot \endlastfoot

%



$0^n$ &
The binary string consisting of the digit $0$ repeated $n$ times,
\pageref{definition:BinaryRepresentationOfSets} \\

$1^n$ &
The binary string consisting of the digit $1$ repeated $n$ times,
\pageref{definition:BinaryRepresentationOfSets} \\

$\mathcal{A}_{q^t/q}(\a,C)$ &
Cyclic trace array for primitive $\a \in \fqt$ and $C\subseteq [0, \w{t}-1]$,
\pageref{definition:cyclicAlphaSArray}\\

$\A{\a,C}$ &
Simplified notation for $\AA_{q^t/q}(\a,C)$,
\pageref{definition:cyclicAlphaSArray}\\

$\AA_{q^t/q}(\a)$ &
The cyclic trace array $\AA_{q^t/q}(\a,[0,\w{t}-1])$,
\pageref{definition:cyclicAlphaSArray}\\

$\A{\a}$ &
Simplified notation for $\AA_{q^t/q}(\a)$,
\pageref{definition:cyclicAlphaSArray}\\

$\AA_{q^t/q, \bm{0}}(\a,C)$ &
The array $\AA_{q^t/q}(\a,C)$ with a row of zeros appended,
\pageref{definition:cyclicAlphaSArray}\\

$\ZA{\a, C}$ &
Simplified notation for $\AA_{q^t/q, \bm{0}}(\a,C)$,
\pageref{definition:cyclicAlphaSArray}\\

$\AA_{q^t/q, \bm{0}}(\a)$ &
The cyclic trace array $\AA_{q^t/q, \bm{0}}(\a,[0,\w{t}-1])$,
\pageref{definition:cyclicAlphaSArray}\\

$\ZA{\a}$ &
Simplified notation for $\AA_{q^t/q, \bm{0}}(\a)$,
\pageref{definition:cyclicAlphaSArray}\\

$\MV{\a,C}$ &
Cyclic trace array modulo $v$ corresponding to $\a \in \fqt$,
\pageref{definition:ArrayModV}\\

$\MV{\a}$ &
Simplified notation for $\MV{\a, [0, \w{t}-1]}$,
\pageref{definition:ArrayModV}\\

$\Bcal_n$ &
Set of all the binary necklaces of length $n$
\pageref{definition:necklace}\\

$c_r$ &
See \Cref{CRlambda}, 
\pageref{CRlambda}\\

$c_{r,t-1}$ &
See \Cref{CRlambda_penultimate}, 
\pageref{CRlambda_penultimate}\\

$c_{r,t}$ &
See \Cref{CRlambda_last}, 
\pageref{CRlambda_last}\\

$\charv_n(S)$ &
Characteristic vector of set $S$,
\pageref{definition:CharacteristicVectorOfSet} \\

$d_{r,k}$ &
See \Cref{DKlambda},
\pageref{DKlambda}\\

$d(\mathbf{x},\mathbf{y})$ &
The Hamming distance of the vector $\mathbf{x}$ and $\mathbf{y}$,
\pageref{definition:HammingDistance}\\

$\cipW$ &
Cyclotomic coset of $p$ modulo $w$ that contains $i$,
\pageref{definition:CyclotomicCosets}\\

$\C_n$ &
Set containing all canonical subsets of $[0,n-1]$,
\pageref{definition:CanonicalSet}\\

$\CA_{\lambda}(N; t, k,v)$ &
A covering array of index $\lambda$,
\pageref{definition:CoveringArray}\\

$\CA(N; t, k,v)$ &
A covering array $\CA_{1}(N; t, k,v)$,
\pageref{definition:CoveringArray}\\

$\CAN(t, k,v)$ &
The covering array number for $t, k$ and $v$,
\pageref{definition:CAN}\\

\textsc{Candidates}$(S)$ &
The set of candidates for a feasible solution $S$, 
\pageref{equation:CMdefinition}\\

\textsc{Candidates}$(S)_{>j}$ &
The set of candidates for a feasible solution $S$, 
\pageref{equation:CandidatesGreaterThanJ}\\

$\fq$   & 
The finite field with $q$ elements,
\pageref{theorem:SubfieldCriterionOfFiniteFields} \\

$\fqstar$   &
The nonzero elements of $\fq$,
\pageref{definition:GroupsOfFiniteField} \\

$\fqx$   & 
The ring of polynomials on the variable $x$ over $\fq$\\

$\gtst(\bfb)$ &
See \cpageref{proposition:CanonicalWellDefined}\\

$\hr$ &
See \Cref{equation:DefinitionHr},
\pageref{equation:DefinitionHr}\\

$\logw(x)$ &
The discrete logarithm of $x$ with base $\o$,
\pageref{section:CharactersChapterMainResults}\\

$m(d,q)$ &
The size $n$ of a maximum $n$-arc in $PG(d,q)$,
\pageref{fornomenclaturearc}\\

$M(d,q)$ &
The size $n$ of a maximum $n$-track in $PG(d,q)$,
\pageref{item:computer}\\

$\neck(\bfa)$ &
Necklace of string $\bfa$,
\pageref{definition:necklace}\\

$\nr$ &
See \Cref{equation:DefinitionNr},
\pageref{equation:DefinitionNr}\\

$\ord(a)$&
The order of a group element $a$\\

$\bigo$    &
Big O notation\\

$\OA(t, k,v)$ &
An orthogonal array of index unity,
\pageref{definition:OrthogonalArray}\\

$\OA_{\lambda}(t, k,v)$ &
An orthogonal array of index $\lambda$,
\pageref{definition:OrthogonalArray}\\

$PG(d, q)$ &
The $d$-dimensional projective space over $\fq$,
\pageref{definition:TheFiniteProjectiveSpacePGDQ}\\

$\S_n(S)$ &
Equivalence class containing $S$ and its shift-equivalent sets modulo $n$,
\pageref{equation:DefinitionES}\\

$\seq{\a}$ &
Sequence associated with primitive element $\a \in \fqt$,
\pageref{definition:SequenceAssociatedWithFFElement}\\

$\sq{\a}$ &
Simplified notation for $\seq{\a}$,
\pageref{definition:SequenceAssociatedWithFFElement}\\

$\seqv{\a}$ &
Maximal sequence modulo $v$,
\pageref{equation:DefinitionOfMSeqModV} \\

$T_{\Bcal_n}$ &
The tree defined in \cpageref{equation:DefTau}\\

$T_{\C_n}$ &
The tree defined in \cpageref{equation:noxzema}\\

$\Tt$ &
The trace function on $\fqt$ over $\fq$,
\pageref{definition:TraceFunction} \\

$\J$    &
Jacobi sum,
\pageref{definition:JacobiSum} \\

$L^j$ &
Left shift operator on sequence or string,
\pageref{definition:LeftShiftOperator}, \pageref{definition:necklace}\\

$L^j_n$ &
Subinterval of length $n$ of left shift by $j$
\pageref{definition:LeftShiftOperator}\\

$U_P(I)$ &
See \cpageref{equation:DefUP}\\

$w(\mathbf{x})$ &
The Hamming weight of the vector $\mathbf{x}$,
\pageref{definition:HammingDistance}\\

$\Z_n^{\times}$   & 
Set of elements of $\Z_n$ that have a  multiplicative inverse\\

&\\

$\gpW$ &
Set of cyclotomic coset leaders of $p$ modulo $w$,
\pageref{definition:CyclotomicCosetLeaders}\\

$\lambda_{\bfb}$ &
Number of rows of array in context, equal to vector $\bfb$,
\pageref{lambda-linind}\\

$\lambda_{r}^{-}$ &
Lower bound for $\lambda_{\bfb}$, 
\pageref{section:UpperLowerBoundsOnLambda}\\

$\lambda_{r}^{+}$ &
Upper bound for $\lambda_{\bfb}$, 
\pageref{section:UpperLowerBoundsOnLambda}\\

$\sigma$ &
See \Cref{lemma:SufficientConditionForSigma}, 
\pageref{lemma:SufficientConditionForSigma}\\

$\sigma_n$ &
See \Cref{equation:SigmaNDefinition}, 
\pageref{equation:SigmaNDefinition}\\

$\tau(\bfb)$ &
The binary string $b_0 b_1 \cdots b_{n-2}\overline{b_{n-1}}$,
where $\bfb= b_0 b_1 \cdots b_{n-1}$,
\pageref{equation:DefTau} \\

$\phi$&
Euler's function,
\pageref{lemma:NumberOfPrimitiveElementsOfFqm}\\

$\chi_0$&
The trivial character of a group,
\pageref{definition:Character}\\

$\xv$ &
Multiplicative character of $\fq$ of order $v$, 
\pageref{equation:DefinitionXV}\\

&\\

$\overline{b}$ &
The binary complement of $b$,
\pageref{equation:DefTau} \\

$\overline{\chi}$&
The conjugate of a character $\chi$,
\pageref{definition:Character}\\

$\overline{(s_0, \dots, s_{r-1})}$  &
Infinite periodic sequence with period $r$, \pageref{example:SomeLinearSequencesOverFtwo}\\

$\widehat{G}$ &
The group of the characters of a group $G$\\

$C^{\perp}$ &
The dual code of the code $C$,
\pageref{definition:LinearCodeOverFF}\\

&\\
$(v,s,\lambda)$-BIBD    &
A balanced incomplete block design,
\pageref{definition:BIBD}\\

$(k,s)$-set &
Set of $k$ vectors, every $s$ of them linearly independent,
\pageref{definition:NTSet}\\

&\\

$x = y \bmod{n}$ & 
$x$ is the remainder of the Euclidean division of $y$ by $n$\\

$x \equiv y \Mod{n}$ & 
Congruence of integers $x$ and $y$ modulo $n$\\

$S +_n i $ &
Shift of the set $S$ modulo $n$,
\pageref{definition:ShiftOfASet}\\
&\\

$[v]$   &
The point of a projective space represented by the vector $v$,
\pageref{equation:RepresentationOfPointsInPGAsClassesOfVectors}\\

$\w{t}$ &
$(q^t-1)/(q-1)$ \pageref{equation:DefinitionW}\\

$\lceil x \rceil$   &
The smallest integer that is greater or equal to $x$\\

$\lfloor x \rfloor$   &
The largest integer that is smaller or equal to $x$\\

$[a,b]$ &
The set $\{a,a+1, \dots, b\}$ where $a,b$ are integers such that $a< b$\\

$[n,k,d]_q$ &
A linear code over $\fq$ of length $n$, dimension $k$ and minimum distance $d$,
\pageref{definition:LinearCodeOverFF}\\

$\binom{a}{b}$    &
The binomial coefficient $\frac{a!}{(a-b)!b!}$\\

${a \brack b}_q$    &
Gaussian binomial coefficient,
\pageref{proposition:FlatsAsBlocksOfABIBD}\\

&\\

AMDS &
Almost maximum distance separable code,
\pageref{definition:SingletonDefectAndRelatedCodes}\\

BIBD &
Balance incomplete block design,
\pageref{definition:BIBD}\\

LFSR & Linear Feedback Shift Registers\\

MDS &
maximum distance separable code,
\pageref{definition:SingletonDefectAndRelatedCodes}\\

NMDS &
near maximum distance separable code,
\pageref{definition:SingletonDefectAndRelatedCodes}\\

\end{longtable}
\end{center}
}

\printbibliography
\printindex

\backmatter

\afterpage{\thispagestyle{empty}}

\clearpage
\pagenumbering{gobble} 
\mbox{}

\newpage
\newpagecolor{beach!40}
\begin{tikzpicture}[remember picture, overlay, transform shape]
  \node [anchor=north west, inner sep=0pt]
    at (current page.north west)
    {
        \begin{tikzpicture}[scale=1.11]
        \node (first) at (0,0) [black!50,draw,minimum width=1.5cm,minimum height=1.5cm,fill=aoi] {};
        \node [right = 0cm of first,black!50,draw,minimum width=1.5cm,minimum height=1.5cm,fill=pond] {};
        \node [right = 1.5cm of first,black!50,draw,minimum width=1.5cm,minimum height=1.5cm,fill=goldfish] {};
        \node [right = 3cm of first,black!50,draw,minimum width=1.5cm,minimum height=1.5cm,fill=goldfish] {};
        \node [right = 4.5cm of first,black!50,draw,minimum width=1.5cm,minimum height=1.5cm,fill=aoi] {};
        \node [right = 6cm of first,black!50,draw,minimum width=1.5cm,minimum height=1.5cm,fill=goldfish] {};
        \node [right = 7.5cm of first,black!50,draw,minimum width=1.5cm,minimum height=1.5cm,fill=pond] {};
        \node [right = 9cm of first,black!50,draw,minimum width=1.5cm,minimum height=1.5cm,fill=beach] {};
        \node [right = 10.5cm of first,black!50,draw,minimum width=1.5cm,minimum height=1.5cm,fill=beach] {};
        \node [right = 12cm of first,black!50,draw,minimum width=1.5cm,minimum height=1.5cm,fill=goldfish] {};
        \node [right = 13.5cm of first,black!50,draw,minimum width=1.5cm,minimum height=1.5cm,fill=beach] {};
        \node [right =15cm of first,black!50,draw,minimum width=1.5cm,minimum height=1.5cm,fill=pond] {};
        \node [right =16.5cm of first,black!50,draw,minimum width=1.5cm,minimum height=1.5cm,fill=aoi] {};
        \node [right =18cm of first,black!50,draw,minimum width=1.5cm,minimum height=1.5cm,fill=aoi] {};
        \node [right =19.5cm of first,black!50,draw,minimum width=1.5cm,minimum height=1.5cm,fill=beach] {};

        \node (second) [below = 0cm of first,black!50,draw,minimum width=1.5cm,minimum height=1.5cm,fill=pond] {};
        \node [right =    0cm of second,black!50,draw,minimum width=1.5cm,minimum height=1.5cm,fill=goldfish] {};
        \node [right =  1.5cm of second,black!50,draw,minimum width=1.5cm,minimum height=1.5cm,fill=goldfish] {};
        \node [right =    3cm of second,black!50,draw,minimum width=1.5cm,minimum height=1.5cm,fill=aoi] {};
        \node [right =  4.5cm of second,black!50,draw,minimum width=1.5cm,minimum height=1.5cm,fill=goldfish] {};
        \node [right =    6cm of second,black!50,draw,minimum width=1.5cm,minimum height=1.5cm,fill=pond] {};
        \node [right =  7.5cm of second,black!50,draw,minimum width=1.5cm,minimum height=1.5cm,fill=beach] {};
        \node [right =    9cm of second,black!50,draw,minimum width=1.5cm,minimum height=1.5cm,fill=beach] {};
        \node [right = 10.5cm of second,black!50,draw,minimum width=1.5cm,minimum height=1.5cm,fill=goldfish] {};
        \node [right =   12cm of second,black!50,draw,minimum width=1.5cm,minimum height=1.5cm,fill=beach] {};
        \node [right = 13.5cm of second,black!50,draw,minimum width=1.5cm,minimum height=1.5cm,fill=pond] {};
        \node [right =   15cm of second,black!50,draw,minimum width=1.5cm,minimum height=1.5cm,fill=aoi] {};
        \node [right = 16.5cm of second,black!50,draw,minimum width=1.5cm,minimum height=1.5cm,fill=aoi] {};
        \node [right =   18cm of second,black!50,draw,minimum width=1.5cm,minimum height=1.5cm,fill=beach] {};

        \node (third) [below = 0cm of second,black!50,draw,minimum width=1.5cm,minimum height=1.5cm,fill=goldfish] {};
        \node [right =      0 of third,black!50,draw,minimum width=1.5cm,minimum height=1.5cm,fill=goldfish] {};
        \node [right =  1.5cm of third,black!50,draw,minimum width=1.5cm,minimum height=1.5cm,fill=aoi] {};
        \node [right =    3cm of third,black!50,draw,minimum width=1.5cm,minimum height=1.5cm,fill=goldfish] {};
        \node [right =  4.5cm of third,black!50,draw,minimum width=1.5cm,minimum height=1.5cm,fill=pond] {};
        \node [right =    6cm of third,black!50,draw,minimum width=1.5cm,minimum height=1.5cm,fill=beach] {};
        \node [right =  7.5cm of third,black!50,draw,minimum width=1.5cm,minimum height=1.5cm,fill=beach] {};
        \node [right =    9cm of third,black!50,draw,minimum width=1.5cm,minimum height=1.5cm,fill=goldfish] {};
        \node [right = 10.5cm of third,black!50,draw,minimum width=1.5cm,minimum height=1.5cm,fill=beach] {};
        \node [right =   12cm of third,black!50,draw,minimum width=1.5cm,minimum height=1.5cm,fill=pond] {};
        \node [right = 13.5cm of third,black!50,draw,minimum width=1.5cm,minimum height=1.5cm,fill=aoi] {};
        \node [right =   15cm of third,black!50,draw,minimum width=1.5cm,minimum height=1.5cm,fill=aoi] {};
        \node [right = 16.5cm of third,black!50,draw,minimum width=1.5cm,minimum height=1.5cm,fill=beach] {};

        \node (fourth) [below =      0 of third,black!50,draw,minimum width=1.5cm,minimum height=1.5cm,fill=goldfish] {};
        \node [right =    0cm of fourth,black!50,draw,minimum width=1.5cm,minimum height=1.5cm,fill=aoi] {};
        \node [right =  1.5cm of fourth,black!50,draw,minimum width=1.5cm,minimum height=1.5cm,fill=goldfish] {};
        \node [right =    3cm of fourth,black!50,draw,minimum width=1.5cm,minimum height=1.5cm,fill=pond] {};
        \node [right =  4.5cm of fourth,black!50,draw,minimum width=1.5cm,minimum height=1.5cm,fill=beach] {};
        \node [right =    6cm of fourth,black!50,draw,minimum width=1.5cm,minimum height=1.5cm,fill=beach] {};
        \node [right =  7.5cm of fourth,black!50,draw,minimum width=1.5cm,minimum height=1.5cm,fill=goldfish] {};
        \node [right =    9cm of fourth,black!50,draw,minimum width=1.5cm,minimum height=1.5cm,fill=beach] {};
        \node [right = 10.5cm of fourth,black!50,draw,minimum width=1.5cm,minimum height=1.5cm,fill=pond] {};
        \node [right =   12cm of fourth,black!50,draw,minimum width=1.5cm,minimum height=1.5cm,fill=aoi] {};
        \node [right = 13.5cm of fourth,black!50,draw,minimum width=1.5cm,minimum height=1.5cm,fill=aoi] {};
        \node [right =   15cm of fourth,black!50,draw,minimum width=1.5cm,minimum height=1.5cm,fill=beach] {};
        \node [right = 16.5cm of fourth,black!50,draw,minimum width=1.5cm,minimum height=1.5cm,fill=aoi] {};
%
    \end{tikzpicture}
    };
\end{tikzpicture}

\end{document}